\documentclass[a4paper]{book}

\usepackage{amssymb,amsbsy,amsthm,amsmath,graphicx,epsfig,xcolor}
\usepackage{mathabx,times,hyperref,eucal}
\usepackage{appendix}
\numberwithin{equation}{chapter}

\DeclareFontFamily{U}{FdSymbolC}{}
\DeclareFontShape{U}{FdSymbolC}{m}{n}{<-> s * FdSymbolC-Book}{}
\DeclareSymbolFont{fdarrows}{U}{FdSymbolC}{m}{n}
\DeclareMathSymbol{\spoon}{\mathrel}{fdarrows}{"6B}

\usepackage{sectsty}

\chapterfont{\Large}
\sectionfont{\large}

\subsectionfont{\normalsize}

\newtheorem{thm}{Theorem}[chapter]
\newtheorem{lem}[thm]{Lemma}
\newtheorem{prop}[thm]{Proposition}
\newtheorem{cor}[thm]{Corollary}

\newtheorem{dfn}[thm]{Definition}
\newtheorem{prob}[thm]{Problem}

\newtheorem{mainthm}{Theorem}

\theoremstyle{remark}
\newtheorem{ex}[thm]{Example}
\newtheorem{rmk}[thm]{Remark}
\newtheorem{rmks}[thm]{Remarks}

\newcommand{\bs}[1]{\boldsymbol{#1}}
\renewcommand{\bf}[1]{\mathbf{#1}}
\renewcommand{\rm}[1]{\mathrm{#1}}
\renewcommand{\cal}[1]{\mathcal{#1}} 

\newcommand{\bbC}{\mathbf{C}}
\newcommand{\bbD}{\mathbf{D}}
\newcommand{\bbE}{\mathbf{E}}

\newcommand{\bbN}{\mathbf{N}}
\newcommand{\bbP}{\mathbf{P}}

\newcommand{\bbR}{\mathbf{R}}
\newcommand{\bbT}{\mathbf{T}}
\newcommand{\bbZ}{\mathbf{Z}}

\newcommand{\rmS}{\mathbf{S}}

\newcommand{\rmH}{\mathrm{H}}
\newcommand{\rmI}{\mathrm{I}}
\newcommand{\rmM}{\mathbf{M}}
\newcommand{\rmU}{\mathbf{U}}

\newcommand{\rmh}{\mathrm{h}}

\newcommand{\rmp}{\mathrm{p}}

\newcommand{\F}{\mathcal{F}}

\newcommand{\calK}{\mathcal{K}}

\newcommand{\V}{\mathcal{V}}
\newcommand{\X}{\mathcal{X}}
\newcommand{\Y}{\mathcal{Y}}

\newcommand{\A}{\mathfrak{A}}
\newcommand{\B}{\mathfrak{B}}

\newcommand{\I}{\mathfrak{I}}

\newcommand{\K}{\mathfrak{K}}

\newcommand{\N}{\mathfrak{N}}

\newcommand{\G}{\Gamma}

\renewcommand{\S}{\Sigma}

\renewcommand{\a}{\alpha}
\renewcommand{\b}{\beta}
\newcommand{\eps}{\varepsilon}
\renewcommand{\phi}{\varphi}
\newcommand{\g}{\gamma}
\renewcommand{\l}{\lambda}
\newcommand{\s}{\sigma}

\renewcommand{\hat}[1]{\widehat{#1}}
\newcommand{\ol}[1]{\overline{#1}}
\newcommand{\ul}[1]{\underline{#1}}
\renewcommand{\tilde}[1]{\widetilde{#1}}
\renewcommand{\t}[1]{\widetilde{#1}}

\newcommand{\fin}{\nolinebreak\hspace{\stretch{1}}$\lhd$}

\newcommand{\bspi}{\boldsymbol{\pi}}
\newcommand{\ann}{\mathrm{ann}}
\newcommand{\hann}{\mathrm{h}^\mathrm{ann}}

\newcommand{\vol}{\mathrm{vol}}

\newcommand{\Rep}{\mathrm{Rep}}

\newcommand{\PPSD}{\mathrm{PPSD}}

\newcommand{\sa}{\mathrm{sa}}
\newcommand{\tr}{\overline{\mathrm{tr}}}
\newcommand{\Tr}{\mathrm{tr}}
\newcommand{\Det}{\mathrm{det}}

\renewcommand{\Rep}{\mathrm{Rep}}

\begin{document}

\title{Annealed almost periodic entropy}

\author{Tim Austin}


\maketitle

\setcounter{tocdepth}{1}
\tableofcontents

\pagestyle{plain}

\chapter{Introduction}


\section*{\emph{Random almost periodic sequences and entropy}}

The recent work~\cite{APE4} studied various notions of entropy for unitary representations of countable groups and representations of separable C*-algebras.  These notions were defined based on an analogy between this setting and that of ergodic theory, where the Kolmogorov--Sinai entropy for single transformations now has various generalizations to measure-preserving actions of different classes of countable groups, and in particular beyond amenable groups (see~\cite{Bowen--survey} for an introduction).  This analogy led to new quantities that can be defined for positive definite functions associated to various kinds of representations.  The most substantial theorems in~\cite{APE4} show that these entropy values for representations are often given by certain Fuglede--Kadison determinants.  These results are interpreted as various non-commutative generalizations of Szeg\H{o}'s limit theorem for positive definite Toeplitz determinants.

In~\cite{APE4}, the analog of sofic entropy is called almost periodic (`AP') entropy.  It is defined in the generality of matrix-valued completely positive maps on a separable, unital C*-algebra $\A$.  To define it, one must first introduce an auxiliary `AP sequence', meaning a sequence of finite-dimensional representations of $\A$ whose dimensions diverge.  Along this AP sequence, one looks for finite-dimensional `models' of a desired positive functional $\phi$ on $\A$, and then AP entropy is defined in terms of the growth rates of the volumes of the sets of such models.  This is the analog of how a sofic approximation is used to define sofic entropy, and is the reason for the term `almost periodic'.

We can extend the theory of AP entropy from~\cite{APE4} by allowing AP sequences $\bspi = (\pi_n)_{n\ge 1}$ that are themselves random.  The main point is the resulting interaction between (i) the law of the chosen random AP sequence and (ii) the infinitary positive functional $\phi$ that one is trying to approximate.

In Part~\ref{part:general} of this work, we first recall some necessary background from various fields, and then in Chapter~\ref{chap:random-AP} we introduce two new entropy notions for a random AP sequence $\bspi$ and a completely positive map $\phi:\A\to \rmM_k$.  The first is the `annealed AP entropy' $\rmh^\ann_{\bspi}(\phi)$, which is introduced carefully in Section~\ref{sec:ann-APent}.  It is obtained by simply inserting an expectation over $\bspi$ into the definition of deterministic AP entropy.  Annealed AP entropy is the analog of Bowen's `annealed sofic entropy' from~\cite{Bowen10free} (originally referred to as the `f-invariant'; see the discussion of terminology in~\cite{AusBowShr}).  Many of our results about it lie parallel to that theory, but sometimes with quite different proofs.

To be more explicit, let $\B(\A,\rmM_k)_+$ be the space of completely positive maps from $\A$ to $\rmM_k$.  For each $n$ and any subset $O$ of $\B(\A,\rmM_k)_+$, let
\[\X(\pi_n,O) := \big\{[v_1,\dots,v_k] \in \bbC^n:\ [\langle \pi_n(\bullet)v_j,v_i\rangle]_{i,j=1}^k \in O\big\}.\]
Now fix one element $\phi$ of $\B(\A,\rmM_k)_+$. We define its \textbf{annealed almost periodic entropy} along $\bspi$ to be
\begin{equation}\label{eq:hann-preview}
\rmh^\ann_{\bspi}(\phi) := \inf_O\limsup_{n\to\infty}\frac{1}{n}\log \bbE\frac{\vol_{2kn}\X(\pi_n,O)}{v(n)^k},
\end{equation}
where $v(n)$ is the volume of the unit ball in $2n$ real dimensions, and where the infimum ranges over all neighbourhoods $O$ of $\phi$.  This quantity takes values in $[-\infty,\infty)$.  The expectation refers to the random choice of $\pi_n$.  We borrow the term `annealed' for these averages from similar calculations in statistical physics: see the discussion following Definition~\ref{dfn:annAPent} and in~\cite[Sec. I.1]{MezParVir--book}.

Our second new entropy notion is the `zeroth-order AP entropy' $\rmh^0_{\bspi}(\phi)$.  It governs the probability as $n\to\infty$ that $\phi$ approximately appears as the positive functional associated to $\pi_n$ by some vector in that representation: see Section~\ref{sec:0-APent} for the precise definition.  Zeroth-order AP entropy is more novel.  Unlike annealed AP entropy, it depends only on the equivalence class of the GNS representation of $\phi$, and this leads to a general definition of zeroth-order AP entropy for any separable representation.

If the random tracial states $\tr_{d_n}\circ \pi_n$ converge sufficiently fast in probability to a limiting tracial state $\tau$ on $\A$, then $\rmh^\ann_{\bspi}$ and $\rmh^0_{\bspi}$ are related to each other via another Fuglede--Kadison determinant: see Theorem~\ref{thm:three-entropy}.  This `three-entropy formula' is a generalization of the formula for deterministic AP entropy in~\cite[Thm. C]{APE4} to random AP sequences.


\section*{\emph{Random AP sequences for free groups}}

With this general theory in hand, in Part~\ref{part:free} we restrict our attention to unitary representations of a group $\Gamma$ that is freely generated by a finite set $S$ of size $r$.  (Much of our work should have straightforward extensions to the free group of countable infinite rank, but we do not develop these here.)

The free group $\G$ has many unitary representations in any desired dimension, since independent generators can be chosen arbitrarily.  In particular, for each positive integer $n$, we can let $(\pi_n(s):\ s\in S)$ be a tuple of $n$-by-$n$ unitary matrices chosen independently at random from Haar measure, and then let $\pi_n$ be the resulting unitary representation of the whole of $\G$.  These are the random AP sequences that we focus on in Part~\ref{part:free}.  (They are also central examples for Voiculescu's theory of `free probability'~\cite{VoiDykNic92,Voi02}, which already includes a notion of `free entropy', but the variants of AP entropy that we study here are quite different.)  To simplify notation, let us shorten $\rmh^\ann_{\bspi}$ to $\hann$ and $\rmh^0_{\bspi}$ to $\rmh^0$ for this sequence.

For this particular $\bspi$, quite precise calculations are possible that extend the general theory from Part~\ref{part:general}, and these lead to rich connections with random matrix theory, particularly through various large deviations principles.  Among the consequences we obtain new probability laws for such independent random unitary matrices, including a large deviations principle that enhances the Collins--Male strong convergence theorem for such matrices (as well as giving a new, independent proof of that theorem itself).

First, a kind of first-moment calculation (explained further below) leads to concrete formulas for $\hann$.  The next theorem states two of these; some others are derived in Section~\ref{sec:ann-AP-forms}.  Let $B_n$ denote the closed $n$-ball around the identity in the left Cayley graph of $\G$.  Also, if $\phi:\G\to\rmM_k$ is positive definite and $F$ is a finite subset of $\G$, then define a positive semidefinite $(k|F|)$-by-$(k|F|)$ matrix by
\begin{equation}\label{eq:Qpin}
\phi[F] := [\phi(g^{-1}h):\ g,h \in F].
\end{equation}

\begin{mainthm}\label{mainthm:annealed-formula}
The value $\hann(\phi)$ is equal to both
	\begin{equation}\label{eq:LDP-formula-1}
		\lim_{n\to\infty}\Big(\log\Det
		\,\phi[B_n] - \sum_{s \in S}\log\Det\,\phi[B_n\cap sB_n]\Big)
	\end{equation}
	and
	\begin{equation}\label{eq:LDP-formula-2}
		\lim_{n\to\infty}\Big(\sum_{s \in S}\log\Det\,\phi[B_n\cup sB_n] - (2r-1)\log\Det\,\phi[B_n]\Big)
	\end{equation}
	(interpreting these as $-\infty$ if any of the determinants appearing here equals $0$).
\end{mainthm}


Our proof of Theorem~\ref{mainthm:annealed-formula} is split into two parts: formula~\eqref{eq:LDP-formula-1} is contained in Proposition~\ref{prop:annealed-formula-1}, and formula~\eqref{eq:LDP-formula-2} is contained in Proposition~\ref{prop:annealed-formula-2}.

Formula~\eqref{eq:LDP-formula-2} is the analog of Bowen's original definition of annealed sofic entropy in ergodic theory (originally called the `f-invariant'): see~\cite{Bowen10free}.  The analogy is clear upon recognizing the log-determinant of a positive semidefinite matrix as an analog of Shannon entropy itself.  We explain this analogy further in Chapter~\ref{chap:PSD} (and also introduce a notation which reflects which we use in the remainder of the book).  More recently, essentially the same formulas as Bowen's have appeared independently in the study of certain random matrix models constructed over sparse locally-tree-like random graphs.  See, for example, the first displayed equation on~\cite[p4]{BordGuiMal--heavytail}, and the further references discussed in that paper.  In fact, all of these formulas can be seen as descendants of the Bethe and Kikuchi entropy functionals from statistical physics: see the discussion in~\cite[Subs. 3.3]{AusBowShr}.


\section*{\emph{Large deviations principles}}

One of our first results about $\hann$ shows that, in case $\phi$ is unital, we can replace $\vol_{2kn}/v(n)^k$ in~\eqref{eq:hann-preview} with the unique unitary-invariant probability measure $m_{\bf{U}(k,n)}$ on the space $\bf{U}(k,n)$ of all orthonormal $k$-tuples in $\bbC^n$ (see Lemma~\ref{lem:normalized3}).  Having done so, the Fubini--Tonelli theorem lets us exchange probability and expectation like this:
\[\bbE m_{\bf{U}(k,n)}\X(\pi_n,O) = \bbE\int 1_{\X(\pi_n,O)}\,dm_{\bf{U}(k,n)} = \int\bbP(\Phi^{
\pi_n}_V \in O)\,dm_{\bf{U}(k,n)}(V).\]
Since the distribution of $\pi_n$ is invariant under the transitive action of $\bf{U}(n)$ on $\bf{U}(k,n)$, the last integrand above is actually constant, and so it simply equals $\bbP(\Phi^{\pi_n}_V \in O)$ for some fixed choice of $V \in \bf{U}(k,n)$. 	We now fix a choice of $V \in \bf{U}(k,n)$ for each $n$, suppressing its dependence on $n$.  Intuitively, the random function
\[\Phi^{\pi_n}_V(g) := V^\ast \pi_n(g)V \qquad (g \in \G)\]
describes how the random unitary matrices that generate $\pi_n$ `sit together' through their actions on a fixed orthonormal $k$-tuple of vectors $V$. As a result, we see that $\hann$ is really an exponent that governs certain probabilities relating to $\Phi^{\pi_n}_V$ as $n\to\infty$.  This is a crucial advantage of working with annealed averages.

Let $\S^\rm{u}_k(\G)$ be the convex set of all unital $\rmM_k$-valued positive definite functions on $\G$ (see Sections~\ref{sec:reps-and-maps} and~\ref{sec:ctble-groups}).  This set is compact and metrizable in its weak* topology, and $\Phi^{\pi_n}_V$ is a random element of it.  It turns out that $\Phi^{\pi_n}_V$ obeys the large deviations principle in this space with $-\hann$ as rate function, and this provides a route to proving Theorem~\ref{mainthm:annealed-formula}.  A standard formulation of large deviations principles is recalled in Definition~\ref{dfn:LDP}.

To state this principle formally, define the following function of $\phi \in \S^\rm{u}_k(\G)$ for any finite subset $F$ of $\G$:
\[h_F(\phi) := \left\{\begin{array}{ll} \log\Det\, \phi[F] - \sum_{s \in S}\log\Det\, \phi[F\cap sF] &\quad \hbox{if $\phi[F]$ is nonsingular} \\ -\infty & \quad \hbox{if $\phi[F]$ is singular.}\end{array}\right.\]
Notice that $h_F(\phi) > -\infty$ if and only if $\phi[F]$ is nonsingular, and that $h_F(\phi)$ agrees with the expression inside the limit~\eqref{eq:LDP-formula-1} when $F = B_n$.  Later we use these functions when $F$ is a general `grounded' subset of $\G$, meaning that $F$ is finite, contains $e$, and is connected in the left Cayley graph of $\G$: see Section~\ref{sec:group-geom}.  In particular, balls are grounded.

\begin{mainthm}\label{mainthm:LDP}
 The random positive definite function $\Phi^{\pi_n}_V$ obeys the large deviations principle in the space $\S^\rm{u}_k(\G)$ with rate function equal to
	\begin{equation}\label{eq:LDP-formula-inf}
		-\inf\big\{h_F(\phi):\ F\ \hbox{grounded}\} = -\lim_{n\to\infty}h_{B_n}(\phi) \qquad (\phi \in \S^\rm{u}_k(\G)).
	\end{equation}
\end{mainthm}

Because $V$ is fixed but $\pi_n$ is random in this theorem, we refer to Theorem~\ref{mainthm:LDP} as the `chosen-tuple' large deviations principle.  We prove it in Section~\ref{sec:completed-LDP}, and then use it to prove Theorem~\ref{mainthm:annealed-formula} in Section~\ref{sec:ann-AP-forms}.  Related manipulations also lead to some other formulas for $\hann$, including two different series expansions for it in Corollaries~\ref{cor:first-expansion} and~\ref{cor:Sew}.  The second of these is the analog of Seward's formula for annealed sofic entropy from~\cite[Thm. 1.7]{Seward--freestab}, so we call it the `Seward expansion'.  It plays an important role later in the work.

The proof of Theorem~\ref{mainthm:LDP} depends on a way of building up a positive definite function on the whole of $\G$ through its restrictions to larger and larger grounded subsets.  At each stage, a space of possible extensions of a partially-defined positive definite function is parameterized by certain contraction matrices.  These provide an analog for free groups of the classical Verblunsky coefficients for a sequence of orthogonal polynomials on the unit circle: see, for instance,~\cite{SimOPUCI}.  We therefore refer to these contraction matrices as Verblunsky coefficients in our setting as well (Definition~\ref{dfn:Verb}).  (Beware that they also depend on choosing a certain enumeration of the group $\G$ in advance.)

When we obtain $\Phi^{\pi_n}_V$ at random as above, it is described by a random sequence of Verblunsky coefficients.  Our main technical step towards Theorem~\ref{mainthm:LDP} is the result that any finite initial segment of this random sequence of Verblunsky coefficients becomes independent once $n$ is large enough: see Theorem~\ref{thm:dil-dist} and Corollary~\ref{cor:KilNenfree}.  This is the analog for free groups of a theorem of Killip and Nenciu from~\cite{KilNen04} for single random unitary matrices.  Then our use of it to prove Theorem~\ref{mainthm:LDP} resembles recent work on a large-deviations approach to Verblunsky's form of Szeg\H{o}'s limit theorem: see~\cite{GamRou14,GamNagRou16} and also the exposition in~\cite{BreSimZei18}, which include the case $r=1$ of Theorem~\ref{mainthm:LDP}.  Using this description of the joint distribution of Verblunsky coefficients, we first prove a large deviations principle for the restriction $\Phi^{\pi_n}_V[F]$ for any fixed grounded set $F$ (see Theorem~\ref{thm:pre-LDP}), and then turn this into Theorem~\ref{mainthm:LDP} by a standard argument about inverse limits.  Theorem~\ref{mainthm:LDP} is also somewhat analogous to several arguments in the literature that consider joint large deviations for the largest eigenvalue of a random matrix together with the corresponding unit eigenvector: see~\cite{GuiHusRek25} for a recent example.

Our next results combine these basic calculations for $\hann$ with more of the general theory of annealed and zeroth-order AP entropy from Part~\ref{part:general}.  First we study the probability that a given positive definite function $\phi:\G\to \rmM_k$ is asymptotically associated to our random AP sequence $(\pi_n)_{n\ge 1}$ at all.  For the sake of brevity, we describe our main results here less formally than above, referring to the main text for details.  Assume that $\phi$ is unital for simplicity.  Then, roughly put, we show that small neighbourhoods $O$ of $\phi$ satisfy
\[\bbP(\X(\pi_n,O)\ne \emptyset) \approx e^{\rmh^0(\phi)\cdot n + o(n)} \qquad \hbox{as}\ n\to\infty,\]
where the exponent can be obtained from Theorem~\ref{mainthm:annealed-formula} and the relation
\[\rmh^0(\phi) := \lim_{t\downarrow 0}\hann(t\phi + \tau\otimes I_k)\]
(see Lemma~\ref{lem:mollify}), and where $\tau$ is the regular character on $\G$.

With this interpretation of $\rmh^0$ in mind, our next major theorem identifies which positive definite maps $\phi$ have $\rmh^0(\phi) = 0$, or equivalently which ones are asymptotically associated to the random sequence $(\pi_n)_{n\ge 1}$ with high probability.  Recall that a positive definite map on $\G$ is `tempered' if it is a pointwise limit of maps associated to the regular representation (see Section~\ref{sec:regular}).

\begin{mainthm}\label{mainthm:tempered}
	Let $\phi:\G\to \rmM_k$ be positive definite. If $\phi$ is tempered, then for every neighbourhood $O$ of $\phi$ we have
	\[\bbP(\X(\pi_n,O) \ne \emptyset) \to 1.\]
	On the other hand, if $\phi$ is not tempered, then there are a neighbourhood $O$ of $\phi$ and a positive value $c$ such that
	\[\bbP(\X(\pi_n,O) \ne \emptyset) \le e^{-cn + o(n)}.\]
\end{mainthm}

Theorem~\ref{mainthm:tempered} is not really a new result.  The case $k=1$ is equivalent to the Collins--Male theorem from~\cite{ColMal14}, and the case of larger $k$ is reduced to the case $k=1$ fairly easily.  The case $k=1$ is deduced from the analogous theorem for GUE random matrices due to Haagerup and Thorbj\o rnsen~\cite{HaaTho05}, which was the first real breakthrough in this area when it appeared. Other proofs have been given more recently as well: see, in particular,~\cite{CGVTvH,CGVvH} and the survey~\cite{vanHan--strong-survey}.

However, our new proof is independent of previous works, and seems to have little in common with them.  Very roughly, the existing methods in this area depend on controlling the expected traces of powers or other transforms of a random matrix $\pi_n(a)$ so well that one can infer bounds on expected operator norms as a consequence.  By contrast, annealed AP entropy lets us work directly with positive definite functions on $\Gamma$, without ever picking an element $a$ of $C^\ast \G$ and considering the resulting random operator $\pi_n(a)$ or its norm.

The second part of Theorem~\ref{mainthm:tempered} requires most of the work.  The key is to prove that $\rmh^0(\phi) < 0$ whenever $\phi$ is not tempered.  We do this by combining the Seward expansion of $\hann$ with an explicit realization of the GNS representation $\pi_\phi$ as a compact perturbation of an inflation of the regular representation.

To place this argument in its proper context, we need some of the general theory of approximate equivalence and containment for separable unital representations of a separable C*-algebra $\A$, which goes back to classic work of Voiculescu, Arveson, and Hadwin.  More recently, Ab\'ert and Elek have introduced a natural topology on the space $\Rep^\sim_{\rm{a}}(\A)$ of approximate equivalence classes of separable unital representations of $\A$ (or on the corresponding space $\Rep^\sim_{\rm{a}}(\G)$ for a countable group $\G$).  They show that this topology is compact and metrizable.  Following~\cite{APE4}, we call it the `strong quotient' topology, because it simultaneously strengthens both Fell's quotient topology from~\cite{Fel60b,Fel62} and the mode of `strong convergence' for representations~\cite{Magee--survey,vanHan--strong-survey}.  We recall the background that we need from this theory in Chapter~\ref{chap:approx}, although our presentation of it has some novel features.  Later, Corollary~\ref{cor:h0-almost-contained} shows that zeroth-order AP entropy is invariant under approximate equivalence, and non-increasing under approximate containment.

Finally, with Theorem~\ref{mainthm:tempered} at our disposal, we can show how it implies the Collins--Male theorem, and also prove another large deviations principle to accompany that theorem.  Consider again our finitely generated free group $\G$ and the uniformly random AP sequence $(\pi_n)_{n\ge 1}$, and let $\tilde{\pi_n}$ be the approximate equivalence class of $\pi_n$.

\begin{mainthm}\label{mainthm:sq-LDP}
The sequence $(\t{\pi_n})_{n\ge 1}$ obeys a large deviations principle in $\Rep^\sim_{\rm{a}}(\G)$ with speed $n$ and rate function
\[I(\t{\pi}) := \left\{\begin{array}{ll}-\rmh^0(\pi) &\qquad \hbox{if}\ \pi \ \hbox{approximately contains the regular representation} \\ \infty &\qquad \hbox{otherwise.}\end{array}\right.\]
\end{mainthm}

Theorem~\ref{mainthm:sq-LDP} provides the exponential decay rate of upper-tail probabilities for the random operator norms $\|\pi_n(a)\|$ for any $a \in C^\ast \G$.  This large deviations principle is a natural companion to the Collins--Male theorem~\cite{ColMal14} giving strong asymptotic freeness of these random representations.  Concurrently with the present work, a large deviations principle for the largest eigenvalues of some closely related random matrices is approached in a different way in~\cite{GuiHusRek25}.  We discuss this a little further in Section~\ref{sec:op-norm-LDP}, and there may be more to learn by comparing that work with this one.

Our study of $\rmh^0$ also leads to a generalization of Veblunsky's form of Szeg\H{o}'s theorem to free groups, extending the discussion of this analogy from~\cite{APE4}: see the disucssion in Section~\ref{sec:three-entropy-cors}.


Theorems~\ref{mainthm:annealed-formula}--\ref{mainthm:sq-LDP} can be illustrated with special examples of positive definite functions on free groups.  We do this using the classical family of Haagerup functions, which are defined by a simple multiplicative formula along the edges of the group's Cayley tree: see Section~\ref{sec:Haa} for their precise definition and history.  In Section~\ref{sec:Haa-ent} we obtain a new characterization of Haagerup functions as the tangent functionals to the `annealed almost periodic pressure', which is the Legendre transform of $\hann$.  This depends on some more general theory for the Legendre transforms of our entropy notions, developed in Sections~\ref{sec:AP-Legendre} and~\ref{sec:annAP-Legendre}. We then obtain explicit formulas for $\hann$ and $\rmh^0$ of Haagerup functions in Section~\ref{sec:Haa-as-equil}.  Eventually, these calculations allow us to find an explicit large deviations principle for the largest eigenvalues of the random matrices $\pi_n(a)$ when $a$ is an element of $\bbC[\G]$ that has a simple nearest-neighbour form in the Cayley tree: see~\eqref{eq:rate-eg} in Section~\ref{sec:op-norm-LDP}.  We refer to Haagerup functions as a running example to illustrate the general theory at several other points below as well.


\subsection*{Acknowledgements}

This work depended on insightful conversations and correspondence with many people.  In this regard I am particularly grateful to Nir Avni, Uri Bader, Lewis Bowen, Peter Burton, Amir Dembo, Alice Guionnet, Jonathan Husson, Michael Magee, Magdalena Musat, Narutaka Ozawa, Sorin Popa, Jana Reker, Mikael R\o rdam, Alain Rouault, Brandon Seward, Dima Shlyakhtenko, Dan Timotin, Hugo Woerderman, and Ofer Zeitouni.

In addition, some parts of this work grew out of discussions with various participants at two conferences: \emph{Spectral Gaps 2025} in Portoro\v{z}, Slovenia, in September 2025, and \emph{Entropy and Random Matrices} at CIRM, Marseille, France, in March 2026.

\subsection*{Copyright}

For the purpose of open access, the author has applied a Creative Commons Attribution (CC-BY) license to any Author Accepted Manuscript version arising from this submission.

\subsection*{Tool use declaration}

The graphs in Figures~\ref{fig:phase-trans},~\ref{fig:zbeta},~\ref{fig:three-ent} and~\ref{fig:rate-function} were plotted using the Desmos\texttrademark{} free online graphing calculator at \url{www.desmos.com}. No LLMs were used at any stage during this project.

\part{PREPARATIONS}\label{part:general}

\chapter{Basic notation and conventions}\label{chap:basic-notn}

\section{Linear algebra and linear operators}\label{sec:lin-alg}

We assume several standard ideas from linear algebra and matrix analysis in the sequel.  For definiteness I use~\cite{HorJohMA} as a reference wherever possible

Throughout this work, our focus is restricted to linear algebra and functional analysis over $\bbC$ rather than $\bbR$.  This is the appropriate choice for studying unitary representations and C$\sp*$-algebras later.  Much of our work could be re-fashioned over $\bbR$ without requiring major new ideas.

For any positive integer $k$, we write $\bbC^{\oplus k}$ for the space of height-$k$ column vectors with complex entries.  More generally, if $H$ is a vector space, then we write $H^{\oplus k}$ for the $k$-fold \textbf{inflation} of $H$, which is the vector space of height-$k$ column vectors with entries in $H$.  If $S$ is a set, possible infinite, and $H$ is a Hilbert space, then we extend this notation further by writing $H^{\oplus S}$ for the Hilbert-space direct sum of an $S$-indexed family of copies of $H$, still regarded as a space of column vectors. This insistence on column vectors is slightly unusual in functional analysis, but for finite $k$ it enables us to use matrix-vector notation from linear algebra in places where it greatly simplifies some reasoning

If $P$ is an orthogonal projection from a Hilbert space to a closed subspace, then we sometimes write $P^\perp:=1-P$.

We write $\rmM_{n,k}$ for the vector space of $n$-by-$k$ matrices over the complex numbers, and identify these with linear maps from $\bbC^{\oplus k}$ to $\bbC^{\oplus n}$ using matrix-vector multiplication.  By writing such a matrix as $[v_1,\dots,v_k]$, where $v_1$, \dots, $v_k$ are its columns, we identify it with a $k$-tuple of vectors in $\bbC^{\oplus n}$.  We generalize this notation further by allowing columns from a general vector space $H$, meaning that a linear map from $\bbC^{\oplus k}$ to $H$ may still be written in the form $[v_1,\dots,v_k]$.  This identification of tuples with linear maps is so convenient that we often abuse notation by calling the linear map $V$ itself a `$k$-tuple of vectors in $H$'.  If $H$ is an inner product space, the adjoint $V^\ast$ is the map from $H$ to $\bbC^{\oplus k}$ whose output coordinates are given by the inner products with the vectors $v_i$.

If $S$ is any set (not necessarily finite), $M$ is an $S$-by-$S$ matrix, and $E,F \subset S$, then we write $M[E,F]$ for the submatrix of $M$ indexed by $E\times F$.  We usually abbreviate $M[E,E]$ to $M[E]$.

We abbreviate $\rmM_{k,k}$ to $\rmM_k$ and regard it as a $\ast$-algebra over $\bbC$ in the usual way.  We write $I_k$ for the $k$-by-$k$ identity matrix. We write $\Tr_k$ and $\Det$ for the usual trace and determinant on any such algebra, and we set
\[\tr_k M := k^{-1}\Tr_k\,M \qquad (M \in \rmM_k).\]

We write $\rmM_{k,\sa}$ for the subspace of self-adjoint members of $\rmM_k$, $\rmM_{k+}$ for the subset of positive semidefinite members of $\rmM_k$, and $\rmM^\circ_{k+}$ for the further subset of positive definite members.  The set $\rmM_{k+}$ is a closed cone in $\rmM_k$, and $\rmM^\circ_{k+}$ is the relative interior of $\rmM_{k+}$ in $\rmM_{k,\sa}$. This means that, if $A$ is positive definite, then any sufficiently small self-adjoint perturbation of $A$ is still positive definite; and this stability property characterizes the positive definite matrices among the positive semidefinite ones.

For a linear operator on an inner product space, or a matrix that can be regarded as such, the notation $\|\cdot\|$ means the operator norm by default.  We occasionally write it as $\|\cdot\|_{\rm{op}}$ where disambiguation may be helpful.

If $H$ is a Hilbert space, then we write $\B(H)$ for the C*-algebra of its bounded linear operators and $\mathfrak{K}(H)$ for the closed ideal of compact operators.  We also write $\|\cdot\|_{\rm{HS}}$ for the \textbf{Hilbert--Schmidt norm}.  If $(e_i:\ i\in I)$ is any orthonormal basis for $H$, then this norm is given by
\[\|T\|_{\rm{HS}}^2 := \sum_i \|Te_i\|^2 \qquad (T \in \B(H)).\]
Here we allow any element of $T \in \B(H)$, but the resulting sum may be infinite. If it is finite, then $T$ is an operator of \textbf{Hilbert--Schmidt class}, and the restriction of $\|\cdot\|_{\rm{HS}}$ to the set of these defines a Banach space.  See, for instance,~\cite[App. 2]{FolAHA}.  Because that choice of orthonormal basis is arbitrary, these norms satisfy
\begin{equation}\label{eq:op-HS-ineq}
\|T\| \le \|T\|_{\rm{HS}} \qquad (T \in \B(H)).
\end{equation}

We also use the following less standard notations.  When $k\le n$, we write $\bf{GL}(k,n)$ for the set of all linear injections from $\bbC^{\oplus k}$ to $\bbC^{\oplus n}$, or equivalently $n$-by-$k$ matrices with linearly independent columns.  When $k=n$ this is the usual general linear group $\bf{GL}(n,\bbC)$.  We write $\bf{U}(k,n)$ for the subset of unitary embeddings, or equivalently matrices whose columns are orthonormal.  When $k=n$ this gives the unitary group $\rmU(n)$.  Finally, we write $\Xi(k,n)$ for the subset of elements of $\rmM_{n,k}$ that are contractions for the Euclidean distances on $\bbC^{\oplus k}$ and $\bbC^{\oplus n}$, and $\Xi^\circ(k,n)$ for the further subset of strict contractions.  These are the closed and open unit balls of $\rmM_{n,k}$ in the operator norm, respectively.

\section{Real analysis}

We use Landau's `big-$O$' and `little-$o$' notation widely, but only for sequences indexed by natural numbers.  In case these sequences involve other parameters as well, each function hidden behind this asymptotic notation may depend freely on those other parameters.  Thus, for example, if $X$ is any set and $f_1$, $f_2$, \dots and $f$ are real-valued functions on $X$, then the assertion
\[``\ \ f_n(x) = f(x) + o(1) \ \ "\]
means that $f_n \to f$ pointwise, whereas uniform convergence could be written
\[``\ \ \sup_x |f_n(x) - f(x)| = o(1)\ \ ".\]

For any dimension $d$, we write $\vol_d$ for Lebesgue measure on $\bbR^d$.  For any positive integers $d$ and $k$, we also write $\vol_{2kd}$ for the measure on $\rmM_{k,d}(\bbC)$ obtained by identifying this space with $\bbR^{2kd}$.

If $G$ is any compact metric group, then $m_G$ is its Haar probability measure.

\chapter{Large deviations principles}\label{chap:LDP-prelims}

Large deviations theory is treated in several standard texts, including~\cite{Var--LDPbook2} and~\cite{DemZei--LDPbook}.  This short chapter recalls the formulation that we use in the sequel, and a few technical results that we need.

Let $\X$ be a complete separable metric space, let $(\mu_n)_{n\ge 1}$ be a sequence of Borel probability measures on it, and let $I$ be a function from $\X$ to $[0,\infty]$.

\begin{dfn}[Large deviations principle]\label{dfn:LDP}
	The sequence $(\mu_n)_{n\ge 1}$ obeys the \textbf{large deviations principle} (`\textbf{LDP}') with \textbf{rate function} $I$ if the following hold:
	\begin{itemize}
		\item[i.] (Upper bound) If $x \in \X$ and $a < I(x)$, then there is a neighbourhood $U$ of $x$ such that
		\[\mu_n U \le e^{-an + o(n)}.\]
		\item[ii.] (Lower bound) If $x \in \X$ and $I(x) < a < \infty$, then any neighbourhood $U$ of $x$ satisfies
		\[\mu_n U \ge e^{-a n - o(n)}.\]
		\item[iii.] (Exponential tightness) For every $a > 0$ there is a compact subset $K$ of $\X$ such that
		\[\mu_n (\X\setminus K) \le e^{-an + o(n)}.\]
	\end{itemize}
\end{dfn}

\begin{rmk}
	In the main LDPs prove later, the relevant space $\X$ is compact.  In this case condition (iii) is superfluous.  However, during the proofs we must sometimes restrict attention from $\X$ to an open subset of it with its induced topology, and then the additional condition of exponential tightness becomes important.  See Section~\ref{sec:completed-LDP}, in particular. \fin
\end{rmk}

Put together, conditions (i) and (ii) from Definition~\ref{dfn:LDP} imply that
\[I(x) = -\inf_U \limsup_{n\to \infty} \frac{1}{n}\log \mu_n U = -\inf_U \liminf_{n\to \infty} \frac{1}{n}\log \mu_n U,\]
where $U$ runs over all neighbourhoods of $x$.  These in turn imply that $I$ must be lower semicontinuous because of the following general principle.

\begin{lem}\label{lem:upper-semicts}
	Let $\X$ be a topological space, let $\cal{U}$ be an open cover of it, let $F$ be a function from $\cal{U}$ to $[-\infty,\infty)$, and let
	\[f(x) := \inf\{F(U):\ U\in \cal{U},\,U\ni x\} \qquad (x \in \X).\]
	Then $f$ is upper semicontinuous.
\end{lem}

\begin{proof}
	If $f(x) < a$, then there must be some $U \in \cal{U}$ such that $U \ni x$ and ${F(U) < a}$, and this implies that $f(x') < a$ for all $x' \in U$.
\end{proof}

We apply Lemma~\ref{lem:upper-semicts} to various other examples later, sometimes with the opposite sign to conclude lower semicontinuity.

In addition, conditions (i)--(iii) from Definition~\ref{dfn:LDP} imply that $I$ must be proper, meaning that its level set $\{I \le a\}$ is compact for every $a \in [0,\infty)$.

An alternative to Definition~\ref{dfn:LDP} is perhaps more standard.  This alternative assumes explicitly that $I$ is lower semicontinuous and proper, and then asserts that
\begin{equation}\label{eq:LDP2-1}
	\mu_n C \le \exp\big(-\inf_{x \in C}I(x)\cdot n + o(n)\big)
\end{equation}
for every closed subset $C$ of $\X$, and that
\begin{equation}\label{eq:LDP2-2}
\mu_n U \ge \exp\big(-\inf_{x \in U}I(x)\cdot n - o(n)\big)
\end{equation}
for every open subset $U$ of $\X$.  The equivalence of this definition with Definition~\ref{dfn:LDP} can be found in~\cite[Sec. 2.1]{Var--LDPbook2} or~\cite[Subs. 4.1.2]{DemZei--LDPbook}.  (Beware that~\cite[Sec. 2.1]{Var--LDPbook2} starts with an apparently weaker definition of `exponential tightness' than our condition (iii), but then derives this equivalence as well.)

Now suppose that $\Y$ and $\X$ are complete separable metric spaces and that $\pi:\Y \to \X$ is continuous.  Suppose in addition that $(\nu_n)_{n\ge 1}$ is a sequence of Borel probability measures on $\Y$, and let $\mu_n := \pi_\ast\nu_n$. If the sequence $(\nu_n)_{n\ge 1}$ obeys the LDP with rate function $I$, then the sequence $(\mu_n)_{n\ge 1}$ obeys the LDP with rate function
\begin{equation}\label{eq:contraction}
	J(x) := \inf\{I(y):\ \pi(y) = x\}.
\end{equation}
This is the \textbf{contraction principle}: see~\cite[Sec. 2.3]{Var--LDPbook2} or~\cite[Subs. 4.2.1]{DemZei--LDPbook}.  Since $I$ is proper and lower semicontinuous, the infimum in~\eqref{eq:contraction} is always attained.  If it is attained at a unique point, then we can also deduce the following.

\begin{lem}\label{lem:cond-lim}
In the setting above, assume that $J(x) < \infty$ and that $y$ is the unique point in $\pi^{-1}\{x\}$ such that $I(y) = J(x)$.  For any neighbourhood $U$ of $y$ there is a neighbourhood $V$ of $x$ such that
\begin{equation}\label{eq:cond-lim}
\nu_n(U\mid \pi^{-1}[V]) \to 1 \qquad \hbox{as}\ n\to\infty.
\end{equation}
\end{lem}

\begin{proof}
Given $U$, we first show that there are a neighbourhood $W$ of $x$ and a constant $c > J(x)$ such that $I(z) \ge c$ for all $z \in \pi^{-1}[W]\setminus U$.  Towards a contradiction, suppose otherwise. By the second countability of $\X$, we can then find a sequence $(z_n)_{n\ge 1}$ in $\Y\setminus U$ such that $\pi(z_n)\to x$ and $\limsup_n I(z_n) \le J(x)$.  Since $I$ is proper, the second property traps this sequence in a compact subset of $\Y$, and so we may pass to a subsequence in order to assume also that $(z_n)_{n\ge 1}$ has a limit $z$ in $\Y\setminus U$.  Now the continuity of $\pi$ gives $\pi(z) = x$, while the lower semicontinuity of $I$ gives $I(z) \le J(x)$.  This contradicts the uniqueness of $y$.  So the desired $W$ and $c$ must exist.

Now let $V$ be a neighbourhood of $x$ such that $\ol{V}\subset W$.  On the one hand, the subset $U\cap \pi^{-1}[V]$ is a neighbourhood of $y$, so it satisfies
\begin{equation}\label{eq:cond-lim-1}
\mu_n(U\cap \pi^{-1}[V]) \ge e^{-I(y)n-o(n)}\qquad \hbox{as}\ n\to\infty
\end{equation}
by Definition~\ref{dfn:LDP}(ii).  On the other hand, the subset $\pi^{-1}[\ol{V}]\setminus U$ is closed in $\Y$ and is contained in $\pi^{-1}[W]\setminus U$, and so~\eqref{eq:LDP2-1} and our choice of $W$ and $c$ give
\begin{equation}\label{eq:cond-lim-2}
\mu_n(\pi^{-1}[V]\setminus U) \le \mu_n(\pi^{-1}[\ol{V}]\setminus U) \le e^{-cn + o(n)} \qquad \hbox{as}\ n\to\infty.
\end{equation}
Since $c > I(y)$, comparing~\eqref{eq:cond-lim-1} and~\eqref{eq:cond-lim-2} gives~\eqref{eq:cond-lim}.
\end{proof}

Historically, the first special cases of Lemma~\ref{lem:cond-lim} were the conditional limit theorem for a long string of independent random letters conditioned on its empirical distribution~\cite[Sec. 11.6]{CovTho06}, and results comparing canonical and microcanonical distributions in statistical mechanics~\cite[Sec. III.4]{SimSMLG}.

The next lemma can simplify the family of open sets that we need to check when proving a LDP.

\begin{lem}\label{lem:LDP-base-enough}
	Let $\cal{U}$ be a base for the topology of $\X$.  Let (i)---(iii) be the conditions from Definition~\ref{dfn:LDP}.  If condition (i) holds with the restriction that $U \in \cal{U}$, then it holds in full, and similarly with condition (ii).
\end{lem}

\begin{proof}
	Among the neighbourhoods of a given point $x$, both conditions get stronger as $U$ gets smaller, so these implications hold because $\cal{U}$ is a base.
\end{proof}

Two simple applications of Lemma~\ref{lem:LDP-base-enough} are used in the sequel.

\begin{lem}\label{lem:LDP-product}
	Suppose $(\mu_n)_{n\ge 1}$ and $(\nu_n)_{n\ge 1}$ obey the LDPs on $\cal{X}$ and $\cal{Y}$ with rate functions $I$ and $J$, respectively.  Then $(\mu_n \times \nu_n)_{n\ge 1}$ obeys the LDP on $\cal{X}\times \cal{Y}$ with the rate function
	\[I(x) + J(y) \qquad ((x,y) \in \cal{X}\times \cal{Y}).\] \qed
\end{lem}

Lemma~\ref{lem:LDP-product} holds by applying Lemma~\ref{lem:LDP-base-enough} to the base of product open sets, and by observing that products of compact sets can be used to verify exponential tightness of the product measures.  If we assume that $I$ and $J$ both have minimum value $0$, then the reverse of the implication in Lemma~\ref{lem:LDP-product} also holds, by applying the contraction principle to each coordinate projection.

Our second application of Lemma~\ref{lem:LDP-base-enough} requires a little more preamble.  Let $\cal{X}_i$ be a compact metric space for every $i\ge 1$, let $\cal{X} := \prod_{i\ge 1}\cal{X}_i$, and let $\Y$ be a nonempty closed subset of $\X$. For any subsets $F \subset G\subset \bbN$ let
\[\pi_F:\prod_{i \in G}\cal{X}_i\to \prod_{i \in F}\cal{X}_i\]
be the coordinate projection, and let $\Y_F := \pi_F[\Y]$.  Let $(\mu_n)_{n\ge 1}$ be a sequence of Borel probability measures on $\cal{Y}$, and let $\mu_{F,n} := \pi_{F\ast}\mu_n$ for any $n\in \bbN$ and $F\subset \bbN$. Finally, let $\cal{F}$ be an upwards-directed cover of $\bbN$ by finite subsets.

\begin{lem}\label{lem:LDP-inf-product}
	If $(\mu_{F,n})_{n\ge 1}$ obeys the LDP with rate function $I_F:\cal{Y}_F\to [0,\infty]$ for every $F\in \cal{F}$, then
	\begin{itemize}
		\item[a.] these rate functions are related by
		\[I_F(x) = \inf\{I_G(y):\ y \in \cal{Y}_G,\ \pi_F(y) = x\}\]
		whenever $F,G \in \cal{F}$, $F\subset G$, and $x \in \cal{Y}_F$, and
		\item[b.] the original sequence $(\mu_n)_{n \ge 1}$ obeys the LDP with rate function
		\[I(y) := \sup\{I_F(\pi_F(y)):\ F\in \cal{F}\} \qquad (y \in \Y).\] \qed
	\end{itemize}
\end{lem}

Part (a) of Lemma~\ref{lem:LDP-inf-product} is a case of the contraction principle. Part (b) follows from part (a) by applying Lemma~\ref{lem:LDP-base-enough} to the base of open subsets of $\cal{Y}$ that have the form $\Y\cap \pi_F^{-1}[U]$ for some $F\in\cal{F}$ and some open subset $U$ of $\cal{Y}_F$.  Indeed, for these sets we have
\[\mu_n(\Y\cap \pi_F^{-1}[U]) = \mu_{F,n}U,\]
and for each fixed $F$ the right-hand measure here is governed by $I_F$ as $n\to\infty$.

The last result of this section concerns how an LDP depends on the underlying space on which it is formulated.  Let $\cal{X}$ be a separable and completely metrizable space with an open subset $\cal{X}_0$.  Let $(\mu_n)_{n\ge 1}$ be a sequence of Borel probability measures on $\cal{X}$ such that ${\mu_n \cal{X}_0 = 1}$ for all sufficiently large $n$.  Then we can regard $\cal{X}_0$ as a separable and completely metrizable space in its own right, and choose a sequence $(\nu_n)_{n\ge 1}$ of Borel probability measures on it such that $\nu_n = \mu_n(\,\cdot\,\cap \cal{X}_0)$ for all sufficiently large $n$.  It is usually easy to pass probability limit laws between the measures $\mu_n$ on $\cal{X}$ and $\nu_n$ on $\cal{X}_0$.  The next lemma gives general conditions for an LDP to survive such a passage.  The important change here is not between one sequence of measures and the other, but rather between choosing $\X_0$ or $\X$ as the ambient space.

\begin{lem}\label{lem:LDP-open-subset}
	In the situation above, let $I:\cal{X}\to [0,\infty]$ be lower semicontinuous, and assume that $I(x) = \infty$ for every $x \in \cal{X}\setminus \cal{X}_0$. Then $(\mu_n)_{n\ge 1}$ obeys the LDP on $\cal{X}$ with rate function $I$ if and only if $(\nu_n)_{n\ge 1}$ obeys the LDP on $\cal{X}_0$ with rate function $I|\cal{X}_0$.
\end{lem}

\begin{proof}
	By omitting finitely many terms of the sequence, we may assume that $\mu_n \cal{X}_0 = 1$ and $\nu_n = \mu_n(\,\cdot\, \cap \cal{X}_0)$ for all $n$.  In both directions below we refer conditions (i)--(iii) from Definition~\ref{dfn:LDP}.
	
	\vspace{7pt}
	
	\emph{Step 1: ($\Rightarrow$)}.\quad If conditions (i) and (ii) hold for neighbourhoods of arbitrary points in $\cal{X}$, then they also hold at points in $\cal{X}_0$ if we use neighbourhoods contained in $\cal{X}_0$.  These are always available because $\cal{X}_0$ is open.
	
	Now let $a > 0$, and let $K$ be a compact subset of $\X$ witnessing condition (iii) of Definition~\ref{dfn:LDP} for the sequence $(\mu_n)_{n\ge 1}$.  By our assumptions on $I$, the compact set $\{I \le a\}$ is disjoint from the closed set $\cal{X}\setminus \cal{X}_0$, and so the set $\{I\le a\}$ has an open neighbourhood $U$ whose closure is still disjoint from $\cal{X}\setminus \cal{X}_0$.  Now our choice of $K$ and an application of~\eqref{eq:LDP2-1} give
\[\mu_n(\cal{X}\setminus (K\cap \ol{U})) \le \mu_n(\cal{X}\setminus K) + \mu_n(\cal{X}\setminus U) \le e^{-an + o(n)} + e^{-an + o(n)}. \]	
So $K\cap \ol{U}$ is a compact subset of $\cal{X}_0$ that witnesses condition (iii) for $(\nu_n)_{n\ge 1}$.
	
	\vspace{7pt}
	
	\emph{Step 2: ($\Leftarrow$)}.\quad Conditions (i) and (ii) for $(\nu_n)_{n\ge 1}$ imply the same conditions for $(\mu_n)_{n\ge 1}$ around any any point that lies in $\cal{X}_0$.  In addition, for any $a > 0$, condition (iii) for $(\nu_n)_{n\ge 1}$ gives a compact subset $K$ of $\cal{X}_0$ such that
	\[\nu_n (\cal{X}_0\setminus K) \le e^{-an + o(n)}.\]
	Since $K$ is also compact as a subset of $\cal{X}$, this verifies condition (iii) for the sequence $(\mu_n)_{n\ge 1}$ as well.  However, it also verifies condition (i) for $(\mu_n)_{n\ge 1}$ at every point of $\cal{X}\setminus \cal{X}_0$, because $I(x) = \infty$ for such points, and the set $\cal{X}\setminus K$ is an open neighbourhood of them.  Lastly, condition (ii) for $(\mu_n)_{n\ge 1}$ is vacuous at any $x$ satisfying $I(x) = \infty$, so we have now verified all three conditions for $(\mu_n)_{n \ge 1}$ everywhere.
\end{proof}

Now let $f$ be a bounded and continuous function on $\X$.  If $(\mu_n)_{n\ge 1}$ obeys the large deviations principle on $\X$ with rate function $I$, then that principle also tells us the asymptotic behaviour of these exponential integrals:
\begin{equation}\label{eq:Var}
\frac{1}{n}\log \int e^{nf(x)}\ d\mu_n(x) \to \sup\{f(x) - I(x):\ x \in \X\}.
\end{equation}
This asymptotic is due to Varadhan; see~\cite[Thm. 2.5]{Var--LDPbook2}.

\subsection*{\emph{Notes and further references}}

Most of the work in this appendix is subsumed by more general machinery that can be found in dedicated textbooks such as~\cite{DemZei--LDPbook}.  See~\cite[Sec. 1.2]{DemZei--LDPbook} for a discussion of the difference between a `full' LDP (as we have defined it) and a `weak LDP', which promises~\eqref{eq:LDP2-1} only for compact sets and does not assume exponential tightness.  General forms of the conditional limit theorem for strings of independent random symbols are given in~\cite[Secs. 3.3 and 7.3]{DemZei--LDPbook}. Lemma~\ref{lem:LDP-inf-product} is really a special case of the Dawson--G\"artner theorem about large deviations theory on inverse limit spaces: see, for instance,~\cite[Thms. 4.6.1 and 4.6.9]{DemZei--LDPbook}.

\chapter{Positive definite matrices and their determinants}\label{chap:PSD}

\section{Gram matrices and log-determinant entropy}\label{sec:Gram}

Let $H$ be a complex inner product space and let $V = [v_1, \dots, v_k]$ be $k$-tuple in $H$, interpreted as a linear map from $\bbC^{\oplus k}$ to $H$. The \textbf{Gram matrix} $Q$ of this tuple is the $k$-by-$k$ matrix of their inner products:
\begin{equation}\label{eq:Gram}
	Q = [\langle v_j,v_i\rangle]_{i,j=1}^k = V^\ast V.
\end{equation}

Gram matrices are positive semidefinite, and every positive semidefinite matrix can be written as a Gram matrix.  The choice of this representation is not unique, but given $Q$ one can make the canonical choice $H := \bbC^{\oplus k}$ and $V:= Q^{1/2}$: this is the only choice for which $V$ is again positive semidefinite.

The Gram matrix in~\eqref{eq:Gram} nonsingular if and only if the vectors are linearly independent, hence if and only if $V$ is injective.  More generally, a positive semidefinite matrix $Q$ is the Gram matrix of a tuple of vectors in some inner product space $H$ if and only if $\dim H \ge \rm{rank}\,Q$, and in this case the choice of those vectors is unique up to a unitary transformation of $H$.  See~\cite[Thm. 7.2.10]{HorJohMA}, for example.  The Gram matrix $Q$ has all diagonal entries equal to $1$ if and only if every $v_i$ is a unit vector.

At many points in this book we need to manipulate logarithms of determinants of positive semidefinite matrices.  This work is made more transparent by an analogy between a tuple of vectors $V$ in an inner product space $H$ and a finite tuple of discrete random variables.  Within this analogy, the Gram matrix $V^\ast V$ is the analog of the joint distribution, and then the log-determinant of $V^\ast V$ is the analog of the joint Shannon entropy of the random variables.  This analogy is discussed at greater length in~\cite{APE4}, where some of the main theorems provide `equivariant' extensions of it for actions of countable groups.  Just a few key differences must be kept in mind, particularly these: (i) differential entropy depends on the ambient dimension in which a random vector takes it values; (ii) differential entropy can have either sign, or even equal $-\infty$.  These warnings apply directly to log-determinants as well.

\begin{dfn}
For any finite positive semidefinite matrix $Q$, its \textbf{log-determinant entropy} is
\[\rmH(Q) := \log \det Q,\]
with the usual convention that $\log 0 = -\infty$.

If $S$ is any set (not necessarily finite), $Q$ is a positive semidefinite $S$-by-$S$ matrix, and $F$ is a finite subset of $S$, then we define
\[\rmH_Q(F) := \rmH(Q[F]),\]
where $Q[F]$ is the $F$-by-$F$ submatrix of $Q$. (Recall our notation for submatrices from Section~\ref{sec:lin-alg}.)
\end{dfn}

If $V = [v_1,\dots,v_k]$ is a $k$-tuple in a Hilbert space $H$ and $Q = V^\ast V$, then we can also express $\rmH(Q)$ directly in terms of the tuple using exterior algebra:
	\[\rmH(Q) = 2\log\|v_1\wedge \cdots \wedge v_k\|.\]
See~\cite[Exer. 2.13]{WarFDMLG}, for example.  If we work over $\bbR$ rather than $\bbC$, then the norm of the wedge product on the right-hand side above is the intrinsic $k$-dimensional volume of the parallelepiped generated by $v_1$, \dots, $v_k$.  So this formula connects the `entropy' $\rmH(Q)$ directly to a geometric feature of the tuple $v_1$, \dots, $v_k$.

Let $V = [v_1, \dots, v_k]$ and $Q = V^\ast V$ be as above.
Then $V$ has a polar decomposition $WQ^{1/2}$, where $W:\bbC^{\oplus k} \to H$ is a partial isometry with initial space $\rm{ran}\,Q = (\ker V)^\perp$ and final space $M := \rm{ran}\, V$.  Using this partial isometry, the orthogonal projection $P$ from $H$ onto $M$ is given by $WW^\ast$. In case $Q$ is invertible, we can write $W = VQ^{-1/2}$.  In general, we can write instead
\begin{equation}\label{eq:new-basis}
	W = V(Q^\dag)^{1/2},
\end{equation}
where $Q^\dag$ is the Moore--Penrose generalized inverse of $Q$~\cite[Exer. 7.3.P7]{HorJohMA}.  Since $Q$ is positive semidefinite, its image and its kernel are orthogonal complements and both $Q$-invariant, and so we can define the generalized inverse by
\[Q^\dag|\rm{ran}\,Q := (Q|\rm{ran}\,Q)^{-1} \qquad \hbox{and} \qquad Q^\dag|\ker Q := 0.\]
So now we obtain the formula
\begin{equation}\label{eq:proj-onto-span}
P = WW^\ast = VQ^\dag V^\ast.
\end{equation}

\section{Two-by-two block Gram matrices}\label{sec:block-Gram}

Let $V = [v_1,\dots,v_k]$ and $U = [u_1, \dots, u_\ell]$ be two tuples in the same Hilbert space $H$.  The combined tuple $[v_1, \dots, v_k, u_1, \dots, u_\ell]$ has a $(k+\ell)$-by-$(k+\ell)$ Gram matrix which we can write in the block form
\begin{equation}\label{eq:two-block}
	Q = \left[\begin{array}{cc} Q_{11} & R\\ R^\ast & Q_{22}\end{array}\right].
\end{equation}
The diagonal blocks $Q_{11}$ and $Q_{22}$ are the separate Gram matrices of $V$ and $U$, and
\[R= [\langle u_j,v_i\rangle]_{1 \le i \le k,\,1\le j \le \ell} = V^\ast U.\]

Let $M$ be the span of $v_1$, \dots, $v_k$ and $P$ the orthogonal projection onto $M$.  We can separate each $u_i$ into $Pu_i$ and $P^\perp u_i$, and then form two new Gram matrices from these tuples of components.  Because $M$ and $M^\perp$ are orthogonal, the full Gram matrix $Q_{22}$ of $u_1$, \dots, $u_k$ is the sum of these two Gram matrices of projections.

To find the Gram matrix of $[Pu_1,\dots,Pu_\ell]$, we can use~\eqref{eq:proj-onto-span} to express each $Pu_i$:
\begin{equation}\label{eq:proj-Gram}
	[\langle Pu_{i'},Pu_i\rangle]_{i,i'=1}^\ell = U^\ast P^2 U = U^\ast P U = U^\ast VQ^\dag_{11} V^\ast U = R^\ast Q^\dag_{11} R.
\end{equation}
It follows that the Gram matrix of the orthogonal projections $P^\perp u_1$, \dots, $P^\perp u_\ell$ is equal to the difference
\begin{equation}\label{eq:comp-Gram}
	Q_{22} - R^\ast Q^\dag_{11} R.
\end{equation}
This is the \textbf{generalized Schur complement} of $Q_{11}$ in $Q$~\cite[Exers. 7.1.P28 and 7.3.P8]{HorJohMA}.  It reduces to the usual Schur complement $Q_{22} - R^\ast Q_{11}^{-1}R$ when $Q_{11}$ is nonsingular~\cite[Subs. 0.7--8]{HorJohMA}.

Now consider a general a $2$-by-$2$ block self-adjoint matrix $Q$ written as in~\eqref{eq:two-block}, and assume that $Q_{11}$ and $Q_{22}$ are both positive semidefinite.  The next proposition characterizes those off-diagonal blocks for which the whole of $Q$ is still positive semidefinite.

\begin{prop}\label{prop:two-block-completion}
	Let $Q$ be as in~\eqref{eq:two-block}, and assume that $Q_{11}$ and $Q_{22}$ are positive semidefinite. Then the following are equivalent:
	\begin{itemize}
	\item $Q$ is positive semidefinite;
	\item $R$ is equal to $Q_{11}^{1/2}CQ_{22}^{1/2}$ for some contraction $C \in \Xi(\ell,k)$ (recall this notation from Section~\ref{sec:lin-alg});
	\item the generalized Schur complement $Q_{22} - R^\ast Q^\dag_{11} R$ is positive semidefinite.
	\end{itemize}
	
If $Q_{11}$ and $Q_{22}$ are both nonsingular, then the contraction $C$ promised above is unique, and the following additional conditions are equivalent:
	\begin{itemize}
		\item the whole of $Q$ is nonsingular;
		\item the contraction $C$ is strict;
		\item the Schur complement $Q_{22} - R^\ast Q^{-1}_{11} R$ is nonsingular. \qed
	\end{itemize}
\end{prop}

See~\cite[Thms. 7.7.7 and 7.7.9 and Exer. 7.7.P20]{HorJohMA}.  If $Q_{11}$ and $Q_{22}$ are nonsingular, then the contraction $C$ must be equal to $Q_{11}^{-1/2}RQ_{22}^{-1/2}$.  In general, this contraction need not be unique, but if it exists then a suitable choice is always given by $(Q^\dag_{11})^{1/2}R(Q^\dag_{22})^{1/2}$.

If we know that $Q$ is positive semidefinite, then it is a Gram matrix, say of the concatenated tuple $[V,U]$ as before.  In this case the various conclusions above all have natural geometric interpretations.  For example, the generalized Schur complement $Q_{22} - R^\ast Q^\dag_{11} R$ must be positive semidefinite because it is the Gram matrix of the projected tuple $P^\perp u_1$, \dots, $P^\perp u_\ell$.  Also, using the polar decompositions $V = WQ_{11}^{1/2}$ and $U = YQ_{22}^{1/2}$, we have
\[R = V^\ast U = Q_{11}^{1/2}(W^\ast Y)Q_{22}^{1/2},\]
and $W^\ast Y$ is a contraction as a composition of two partial isometries.  Expressing $W$ and $Y$ as in~\eqref{eq:new-basis}, we obtain the formula for this particular contraction in terms of Moore--Penrose generalized inverses.

\begin{ex}\label{ex:k1}
	Suppose that $\ell = 1$, so $Q$ is the $(k+1)$-by-$(k+1)$ Gram matrix of a tuple $[v_1,\dots,v_k,u]$. Assume that $v_1$, \dots, $v_k$ are linearly independent and $u \ne 0$.  In this case the top right block of $Q$ is a column vector in $\bbC^{\oplus k}$ which specifies the inner product of $u$ with each $v_i$.  This column vector is parametrized by a contraction from $\bbC$ to $\bbC^{\oplus k}$, or equivalently by a vector in $\bbC^{\oplus k}$ of length at most $1$.  Geometrically, we have a canonical choice of orthonormal basis for $\rm{span}\,\{v_1,\dots,v_k\}$ given by the columns of $W := VQ_{11}^{-1/2}$, and the parametrizing vector is the orthogonal projection of $u/\|u\|$ onto that span written in this basis.
	
	Most simply of all, if also $k = 1$, then the parametrizing contraction is simply a complex number of modulus $\cos \theta$, where $\theta$ is the angle between $v$ and $u$. \qed
\end{ex}

The generalized Schur complement also appears in Schur's determinantal formula:
\begin{equation}\label{eq:Schur-det}
	\det Q = \det Q_{11} \cdot \det (Q_{22} - R^\ast Q^\dag_{11} R).
\end{equation}
This is usually derived assuming that $Q_{11}$ is nonsingular~\cite[Sec. 0.8.5]{HorJohMA}, in which case $Q_{11}^\dag = Q_{11}^{-1}$.  But it also holds vacuously if $Q_{11}$ is singular, as then both sides vanish.  The next definition connects this with our `entropy' interpretation of log-determinants.

\begin{dfn}\label{dfn:cond-log-det-ent}
If $S$ is any set (not necessarily finite), $Q$ is a positive semidefinite $S$-by-$S$ matrix, and $E,F$ are disjoint finite subsets of $S$, then we write
\[Q[F\mid E] := Q[F] - Q[F,E]Q[E]^\dag Q[E,F],\]
the generalized Schur complement of $Q[E]$ in $Q[E\cup F]$.  Now we define the \textbf{conditional log-determinant entropy of $F$ given $E$ in $Q$} by
\[\rmH_Q(F\mid E) := \log \det Q[F\mid E],\]
and the \textbf{log-determinant mutual information between $E$ and $F$ in $Q$} by
\[\rmI_Q(E\,;\,F) := -\log \det \big(I_{|F|} - (Q[F]^\dag)^{1/2}Q[F,E]Q[E]^\dag Q[E,F](Q[F]^\dag)^{1/2}\big).\]
\end{dfn}

Regarding $Q[E\cup F]$ as a $2$-by-$2$ block matrix in the natural way, we have
\begin{equation}\label{eq:mut-inf-contraction}
\rmI_Q(E\,;\,F) = -\log \det(I_{|F|} - C^\ast C),
\end{equation}
where $C$ is the contraction from $\bbC^{\oplus F}$ to $\bbC^{\oplus E}$ promised by Proposition~\ref{prop:two-block-completion}. Since $I_{|F|} - C^\ast C$ is also a contraction, its log-determinant is non-positive, and a standard symmetry of determinants also shows that it has the same determinant as $I_{|E|} - CC^\ast$.  If $Q$ is the Gram matrix of $[v_s:\ s \in S]$, then our discussion above expresses $C$ as $W^\ast Y$, where $W$ and $Y$ are partial isometries with final spaces equal to ${\rm{span}\{v_s:\ s \in E\}}$ and ${\rm{span}\{v_s:\ s \in F\}}$, respectively.  These facts lead to the following.

\begin{lem}\label{lem:mut-inf-nonneg}
	In the setting of Definition~\ref{dfn:cond-log-det-ent}, we have $\rmI_Q(E\,;\,F) = \rmI_Q(F\,;\,E)$, this value is non-negative, and it vanishes if and only if $Q[E,F] = 0$. If $Q$ is the Gram matrix of $[v_s:\ s \in S]$, then $\rmI_Q(E\,;\,F)$ depends only on the subspaces ${\rm{span}\{v_s:\ s \in E\}}$ and ${\rm{span}\{v_s:\ s \in F\}}$, and it vanishes if and only if those subspaces are orthogonal. \qed
	\end{lem}

\begin{ex}\label{ex:1D-cond-ent}
In Definition~\ref{dfn:cond-log-det-ent}, suppose that $Q$ is the Gram matrix of $[v_s:\ s \in S]$.  As a heuristic principle, `conditioning' on the subset $E$ means \emph{projecting everything to $M := \{v_s:\ s \in E\}^\perp$}.  For example, if $F$ is a singleton $\{t\}$ disjoint from $E$, then the geometric interpretation of Schur complements above gives
\[\rmH_Q(t\mid E) =  2\log\|P^\perp v_t\| \qquad \hbox{and} \qquad \rmI_Q(t\,;\,E) = - 2\log (\|P^\perp v_t\|/\|v_t\|),\]
where $P$ is the orthogonal projection onto $M$. \fin
\end{ex}

\begin{rmk}
	If $Q[E\cup F]$ is nonsingular, then $\rmI_Q(E\,;\,F)$ is equal (up to normalization) to the mutual information between $X_E$ and $X_F$ when $X = (X_E,X_F)$ is a complex Gaussian random vector with covariance matrix $Q[E\cup F]$.  See~\cite[Chap. 8]{CovTho06} for the real case, which is easily adapted.  In this setting, the last part of Lemma~\ref{lem:mut-inf-nonneg} implies that this mutual information is the same for $(aX_E,aX_F)$ whenever $a > 0$: this is the `scale invariance' property of mutual information in this setting~\cite[Sec. 11.3]{MacKay--book}. \fin
\end{rmk}

Consider again a positive semidefinite $S$-by-$S$ matrix $Q$ as in Definition~\ref{dfn:cond-log-det-ent}.  The analogy with Shannon entropy continues through the following analogs of the chain rule (see~\cite[Secs. 2.4 and 2.5]{CovTho06}).

\begin{prop}[Chain rules]\label{prop:chain1}
In the setting of Definition~\ref{dfn:cond-log-det-ent}, if $E,F$ are disjoint finite subsets of $S$, then
	\[\rmH_Q(E\cup F) = \rmH_Q(E) + \rmH_Q(F\mid E)\]
(including the possibility of $-\infty$ on either side), and
	\[\rmH_Q(F) - \rmI_Q(F\,;\,E) = \rmH_Q(F\mid E)\]
(including the possibility of $-\infty$ on either side).
\end{prop}

\begin{proof}
The first formula is simply the logarithm of Schur's determinantal formula~\eqref{eq:Schur-det}.

For the second formula, notice that both sides equal $-\infty$ if $Q[F]$ is singular, so we may assume it is nonsingular.  In that case we can write
\begin{multline*}
Q[F\mid E] =  Q[F]^{1/2}\Big(I_{|F|} - Q[F]^{-1/2}Q[F,E]Q[E]^\dag Q[E,F]Q[F]^{-1/2}\Big)Q[F]^{1/2}.
\end{multline*}
Now take determinants and logarithms again.
\end{proof}

\begin{cor}\label{cor:chain1}
In the setting of Definition~\ref{dfn:cond-log-det-ent}, if $E,F$ are disjoint finite subsets of $S$, then
	\begin{align}
	\rmH_Q(E\cup F) &= \rmH_Q(E) + \rmH_Q(F) - \rmI_Q(E\,;\,F) \nonumber\\
	&\le \rmH_Q(E) + \rmH_Q(F)  \label{eq:subadd}
\end{align}
(including the possibility that any of these expressions could be $-\infty$).  If these expressions are finite, then equality holds in~\eqref{eq:subadd} if and only if $Q[E,F] = 0$.
\end{cor}

\begin{proof}
Combine the two formulas from Proposition~\ref{prop:chain1} with the non-negativity and conditions for vanishing from Lemma~\ref{lem:mut-inf-nonneg}.
\end{proof}

Alternatively, the inequality~\eqref{eq:subadd} is the logarithm of Fischer's inequality~\cite[Thm. 7.8.5]{HorJohMA}, and then~\eqref{eq:strong-subadd} is the generalization of it to the Koteljanski\u{\i} (also called Hadamard--Fischer) inequality~\cite[Thm. 7.8.9]{HorJohMA}.  So these results of matrix analysis are analogs of the subadditivity of Shannon entropy and its conditional variant, respectively (see~\cite[Thm. 2.6.6]{CovTho06}).  Several other results of matrix analysis and information theory also fit into this analogy: for example Sz\'asz's inequality~\cite[Thm. 7.8.11]{HorJohMA} corresponds to one of Han's inequalities~\cite[Thm. 17.6.1]{CovTho06}.

Here is one further addition to our analogy between entropy and log-determinants:

\begin{dfn}\label{dfn:cond-log-det-mut-inf}
If $S$ is any set (not necessarily finite), $Q$ is a positive semidefinite $S$-by-$S$ matrix, and $E,F,G$ are pairwise disjoint finite subsets of $S$, then the \textbf{conditional log-determinant mutual information between $E$ and $F$ given $G$ in $Q$} is defined by
\[\rmI_Q(E\,;\,F\mid G) := \rmI_{Q'}(E\,;\,F)\]
with $Q' := Q[E\cup F\mid G]$.
\end{dfn}

If $Q$ is the Gram matrix of $[v_s:\ s \in S]$, then Definition~\ref{dfn:cond-log-det-mut-inf} means the following.  We let $L$ be the span of $\{v_t:\ t \in G\}$, $P$ be the orthogonal projection onto $L$, and $Q'$ be the Gram matrix of $[P^\perp v_s:\ s \in E\cup F]$.  Then we use this new Gram matrix to evaluate $\rmI_Q(E\,;\,F\mid G)$.

\begin{cor}\label{cor:cond-chain}
In the setting of Definition~\ref{dfn:cond-log-det-ent}, if $E,F,G$ are pairwise disjoint finite subsets of $S$, then
\begin{align}
\rmH_Q(E\cup F\mid G) &= \rmH_Q(E\mid G) + \rmH_Q(F\mid E\cup G) \label{eq:cond-chain1}\\
&\le \rmH_Q(E\mid G) + \rmH_Q(F\mid G) \label{eq:strong-subadd}
\end{align}
(including the possibility that any of these expressions could be $-\infty$), and
\begin{equation}\label{eq:cond-chain2}
\rmH_Q(E\mid G) - \rmI_Q(E\,;\,F\mid G) = \rmH_Q(E\mid F\cup G)
\end{equation}
(including the possibility of $-\infty$ on either side).

Let $Q$ be the Gram matrix of $[v_s:\ s \in S]$, let $L := \rm{span}\{v_t:\ t \in G\}$, and let $P$ be the orthogonal projection onto $L$.  If both sides of~\eqref{eq:strong-subadd} are finite, then equality holds there if and only if the sets of vectors $\{P^\perp v_s:\ s \in E\}$ and $\{P^\perp v_s:\ s \in F\}$ are orthogonal.
\end{cor}

\begin{proof}
The identities and inequality follow by applying Proposition~\ref{prop:chain1} and Corollary~\ref{cor:chain1} to $Q' := Q[E\cup G\mid G]$.  The extra ingredient one needs is the quotient property for Schur complements:
\[Q'[E\mid F] = Q[E\mid F\cup G].\]
See~\cite[Eqn. (0.8.5.12)]{HorJohMA}.  That reference is phrased for normal Schur complements, so it assumes nonsingularity, but the proof given there works for generalized Schur complements as well.  Finally, the condition for equality in~\eqref{eq:strong-subadd} holds by applying Lemma~\ref{lem:mut-inf-nonneg} and Corollary~\ref{cor:chain1} to $Q'$.
\end{proof}

\section{Three-by-three-block Gram matrices}\label{sec:three-block-Gram}

This section extends the work of Section~\ref{sec:block-Gram} by allowing three blocks rather than two.   Let $H$ be a complex inner product space as before. Let $k$,$\ell$ and $m$ be non-negative integers, and in $H$ consider tuples of vectors
\begin{equation}\label{eq:full-tuple}
	V = [v_1,\dots,v_k], \quad U = [u_1,\dots,u_\ell] \quad \hbox{and} \quad X = [x_1,\dots,x_m].
\end{equation}
Combine them into a single $(k+\ell+m)$-tuple, and write their joint Gram matrix explicitly in $3$-by-$3$ block form
\begin{equation}\label{eq:full-Gram}
	Q := \left[\begin{array}{ccc}V^\ast V & V^\ast U & V^\ast X\\ U^\ast V & U^\ast U & U^\ast X \\ X^\ast V & X^\ast U & X^\ast X\end{array}\right].
\end{equation}

Let $M$ be the span of $u_1$, \dots, $u_\ell$ (not $v_1$, \dots, $v_k$, as previously), and let $P$ be the orthogonal projection onto $M$. Let us decompose the matrix $Q$ into two summands by splitting each vector into its components in $M$ and $M^\perp$.  In terms of the linear maps in~\eqref{eq:full-tuple}, this simply means we write $V = PV + P^\perp V$, and similarly.  Since $PU = U$ and $P^\perp U = 0$, we find that $Q$ is equal to
\begin{equation}\label{eq:three-block-decomp}
	\left[\begin{array}{ccc} V^\ast P V & V^\ast  U & V^\ast P X\\ U^\ast V & U^\ast U & U^\ast X \\ X^\ast P V & X^\ast U & X^\ast P X\end{array}\right] + \left[\begin{array}{ccc} V^\ast P^\perp V & 0 & V^\ast P^\perp X\\ 0 & 0 & 0 \\ X^\ast P^\perp V & 0 & X^\ast P^\perp X\end{array}\right].
\end{equation}
The first summand here is the Gram matrix of three tuples of vectors that are all contained in $M$.  On the other hand, by ignoring the middle row and column, the second summand is effectively just a two-by-two-block Gram matrix.

This decomposition leads naturally to a generalization of Proposition~\ref{prop:two-block-completion} to three-by-three-block self-adjoint matrices.

\begin{prop}\label{prop:three-block-completion}
	Consider a $(k+\ell+m)$-by-$(k+\ell+m)$ self-adjoint matrix
	\begin{equation}\label{eq:full-Gram2}
		Q := \left[\begin{array}{ccc} Q_{11} & Q_{12} & R\\ Q_{12}^\ast & Q_{22} & Q_{23} \\ R^\ast & Q_{23}^\ast & Q_{33}\end{array}\right].
	\end{equation}
	Assume that the submatrices
	\[Q_{(1\cup 2)} := \left[\begin{array}{cc} Q_{11} & Q_{12}\\ Q_{12}^\ast & Q_{22}\end{array}\right] \qquad \hbox{and} \qquad Q_{(2\cup 3)} := \left[\begin{array}{cc} Q_{22} & Q_{23}\\ Q_{23}^\ast & Q_{33}\end{array}\right]\]
	are both positive semidefinite.  Then $Q$ is positive semidefinite if and only if it has the form
	\begin{equation}\label{eq:three-block-decomp2}
		\left[\begin{array}{ccc} Q_{12}Q^\dag_{22} Q_{12}^\ast & Q_{12} & Q_{12} Q^\dag_{22} Q_{23}\\ Q_{12}^\ast & Q_{22} & Q_{23} \\ Q_{23}^\ast Q^\dag_{22}Q_{12}^\ast  & Q_{23}^\ast & Q_{23}^\ast Q^\dag_{22} Q_{23}\end{array}\right] + \left[\begin{array}{ccc} S_{11} & 0 & S_{11}^{1/2}CS_{33}^{1/2}\\ 0 & 0 & 0 \\ S_{33}^{1/2}C^\ast S_{11}^{1/2} & 0 & S_{33}\end{array}\right],
	\end{equation}
	where
	\begin{equation}\label{eq:Schurs}
		S_{11} := Q_{11} - Q_{12}Q^\dag_{22}Q_{12}^\ast \quad \hbox{and} \quad S_{33} := Q_{33} - Q_{23}^\ast Q^\dag_{22}Q_{23}
	\end{equation}
	are generalized Schur complements and $C \in \Xi(m,k)$.
	
	If $Q_{(1\cup 2)}$ and $Q_{(2\cup 3)}$ are both nonsingular, then the contraction $C$ promised here is unique, and the whole of $Q$ is nonsingular if and only if $C$ is a strict contraction.
\end{prop}

%

\begin{proof}
	\emph{Necessity.} \quad Let $n := k+\ell + m$. If the whole of $Q$ is positive semidefinite, then it is the Gram matrix of some tuples of vectors in $\bbC^{\oplus n}$ as in~\eqref{eq:full-tuple}, and so it decomposes as in~\eqref{eq:three-block-decomp}.
	
	Applying~\eqref{eq:proj-Gram} to the four corner blocks of the first summand in~\eqref{eq:three-block-decomp}, we find that it equals the first summand in~\eqref{eq:three-block-decomp2}.  On the other hand, by~\eqref{eq:comp-Gram}, the two nonzero diagonal blocks in the second summand in~\eqref{eq:three-block-decomp} are equal to $S_{11}$ and $S_{22}$.  Finally, Proposition~\ref{prop:two-block-completion} tells us that this second summand is positive semidefinite if and only if it has the asserted form for some contraction $C$.
	
	\vspace{7pt}
	
	\emph{Sufficiency.} \quad Suppose that $Q$ has the form in~\eqref{eq:three-block-decomp2}.  The second summand is positive semidefinite by Proposition~\ref{prop:two-block-completion}.  It therefore suffices to prove that $Q$ is positive semidefinite in case it equals the first summand in~\eqref{eq:three-block-decomp2}, meaning that $Q_{ii} = Q_{i2}^\ast Q^\dag_{22} Q_{2i}$ for $i \in \{1,3\}$.  By~\eqref{eq:proj-Gram}, this means that $Q_{(1\cup 2)}$ and $Q_{(2\cup 3)}$ are unchanged by the procedure of writing each as a Gram matrix and then projecting all the vectors onto the span of the vectors corresponding to the block $Q_{22}$.  We may therefore pick a tuple $U$ whose Gram matrix is $Q_{22}$ and then two more tuples $V$ and $X$ contained in $\rm{ran}\,U$ so that $Q_{(1\cup 2)}$ and $Q_{(2\cup 3)}$ are the Gram matrices of $[V,U]$ and $[U,X]$ respectively.  Finally, applying~\eqref{eq:proj-Gram} for the projection onto $\rm{ran}\,U$ and the whole triple tuple $[V,U,X]$, we find that under these conditions the Gram matrix of $[V,U,X]$ is the first summand in~\eqref{eq:three-block-decomp2}, so that summand must be positive semidefinite.
	
	\vspace{7pt}
	
	\emph{Uniqueness and nonsingularity.} \quad These follow from the conditions for uniqueness and nonsingularity in our application of Proposition~\ref{prop:two-block-completion} above.
\end{proof}

We can illustrate the parametrization by contractions in Proposition~\ref{prop:three-block-completion} with the following continuation of Example~\ref{ex:k1}.

\begin{ex}\label{ex:kell1}
	Suppose that $m=1$, and let $Q$ as above be the Gram matrix of the combined tuple $[v_1,\dots,v_k,u_1,\dots,u_\ell,x]$ in $\bbC^{\oplus n}$.  Let $P$ be the orthogonal projection onto $M = \rm{span}\,\{u_1,\dots,u_\ell\}$.  
	
	The parametrizing contraction $C$ appears in the second summand in~\eqref{eq:three-block-decomp2}, which corresponds to the second summand in~\eqref{eq:three-block-decomp}.  The nonzero blocks of this summand form the $(k+1)$-by-$(k+1)$ Gram matrix of the tuple $[P^\perp v_1,\dots,P^\perp v_k,P^\perp x]$.  As in Example~\ref{ex:k1}, the top right block here is a column vector in $\bbC^{\oplus k}$ which specifies the inner product of $P^\perp x$ with each $P^\perp v_i$.  This column vector is parametrized by an element of $\Xi(1,k)$, meaning simply a vector in $\bbC^{\oplus k}$ of length at most $1$.  Under the canonical choice of orthonormal basis for $\rm{span}\,\{P^\perp v_1,\dots,P^\perp v_k\}$, this vector corresponds to the orthogonal projection of $P^\perp x/\|P^\perp x\|$ onto that span. \qed
\end{ex}

Now suppose that $Q$ is an $S$-by-$S$ positive semidefinite matrix for some set $S$, and that $E,F,G$ are pairwise disjoint finite subsets of $S$.  Then $Q[E\cup F\cup G]$ has a natural $3$-by-$3$ block structure, and so Proposition~\ref{prop:three-block-completion} expresses $Q[E\cup F\cup G]$ in terms of $Q[E\cup F]$, $Q[F\cup G]$, and a auxiliary contraction matrix $C$ from $\bbC^{\oplus G}$ to $\bbC^{\oplus E}$.  By construction, this is the same contraction that appears when we use Proposition~\ref{prop:two-block-completion} to express $Q[E\cup G\mid F]$ starting from $Q[E\mid F]$ and $Q[G\mid F]$.  Therefore, in the present setting, equation~\eqref{eq:mut-inf-contraction} becomes
\begin{equation}\label{eq:cond-mut-inf-contraction}
\rmI_Q(E\,;\,G\mid F) = -\log \det (I_{|G|} - C^\ast C).
\end{equation}

Proposition~\ref{prop:three-block-completion} can be regarded as the solution to a `matrix completion problem'.  Referring to~\eqref{eq:full-Gram2}, imagine that we know all the blocks of this self-adjoint matrix apart from $R$ and $R^\ast$, and that the submatrices $Q_{(1\cup 2)}$ and $Q_{(2\cup 3)}$ are positive semidefinite.  The `completion problem' asks whether we can choose $R$ so that the whole of $Q$ remains positive semidefinite.  Assuming that $Q_{22}$ is nonsingular, Proposition~\ref{prop:three-block-completion} tells us that the answer is Yes, and moreover that the possible choices for $R$ are all the matrices of the form
\begin{equation}\label{eq:param}
	R = Q_{12}Q^\dag_{22}Q_{23} + S_{11}^{1/2} C S_{33}^{1/2}
\end{equation}
as $C$ ranges over $\Xi(m,k)$, where $S_{11}$ and $S_{33}$ are as in~\eqref{eq:Schurs}.

The solution of this problem establishes a relationship between whole spaces of matrices.  Recall that $\rmM_{k+}$ denotes the closed cone of positive semidefinite elements of $\rmM_{k,sa}$, and that $\rmM^\circ_{k+}$ denotes the subset of positive definite elements, which is the relative interior of $\rmM_{k+}$ in $\rmM_{k,\sa}$.

The next space is less standard, so we define it formally.  Let $k$, $\ell$ and $m$ be non-negative integers, as before.

\begin{dfn}
	A \textbf{partial positive semidefinite} matrix with \textbf{block sizes} $k$, $\ell$ and $m$ is a $5$-tuple
	\[(Q_{11},Q_{12},Q_{22},Q_{23},Q_{33}) \in \rmM_{k+} \times \rmM_{k,\ell}\times \rmM_{\ell+} \times \rmM_{\ell,m}\times \rmM_{m+}\]
	with the property that both of the two-by-two-block matrices
	\[\left[\begin{array}{cc} Q_{11} & Q_{12}\\ Q_{12}^\ast & Q_{22}\end{array}\right] \qquad \hbox{and} \qquad \left[\begin{array}{cc} Q_{22} & Q_{23}\\ Q_{23}^\ast & Q_{33}\end{array}\right]\]
	are positive semidefinite.
	
	A partial positive semidefinite matrix is \textbf{partially nonsingular} if both of those two-by-two-block matrices are nonsingular; otherwise it is \textbf{partially singular}.
	
	We write $\PPSD_{k,\ell,m}$ for the space of partial positive semidefinite matrices with block sizes $k$, $\ell$ and $m$, and $\PPSD^\circ_{k,\ell,m}$ for the further subset of its partially nonsingular members.
\end{dfn}

Rather than writing out a $5$-tuple as above, henceforth we write the corresponding partial positive semidefinite matrices as a three-by-three-block matrix with some unknown entries:
\[\left[\begin{array}{ccc} Q_{11} & Q_{12} & ?\\ Q_{21}^\ast & Q_{22} & Q_{23} \\ ? & Q_{23}^\ast & Q_{33}\end{array}\right].\]
We often indicate such an entity by a raised question mark, as in ``$Q^?$''.  There is a natural map
\[\pi^?:\rmM_{(k+\ell+m)+}\to \PPSD_{k,\ell,m}\]
that simply erases the $k$-by-$m$ blocks in the top right and bottom left of its argument.  If $Q^? \in \PPSD_{k,\ell,m}$, then let us write
\[\Delta(Q^?) := \{Q \in \rmM_{(k+\ell+m)+}:\ \pi^?(Q) = Q^?\},\]
and refer to this fibre as the set of \textbf{completions} of $Q^?$.

\begin{cor}\label{cor:three-block-completion}
	The map $\pi^?$ is a surjection from $\rmM_{(k+\ell+m)+}$ onto $\PPSD_{k,\ell,m}$.  Its restriction to $\rmM^\circ_{(k+\ell+m)+}$ is surjective onto $\PPSD^\circ_{k,\ell,m}$.
	
	We can define a map
	\begin{equation}\label{eq:param-map}
		\PPSD_{k,\ell,m}\times \Xi(m,k) \to \rmM_{(k+\ell+m)+}
	\end{equation}
	 according to
	\[(Q^?,C) \mapsto \left[\begin{array}{ccc}Q_{11} & Q_{12} & R\\ Q_{12}^\ast & Q_{22} & Q_{23}\\ R^\ast & Q_{23}^\ast & Q_{33}\end{array}\right],\]
	where $S_{11}$, $S_{22}$ are given by~\eqref{eq:Schurs} and then $R$ is given by~\eqref{eq:param}.  For each $Q^? \in \PPSD_{k,\ell,m}$, this map sends $\{Q^?\}\times \Xi(m,k)$ onto the fibre $\Delta(Q^?)$, and the restriction of this map to $\PPSD^\circ_{k,\ell,m}\times \Xi^\circ(m,k)$ is a homeomorphism onto $\rmM^\circ_{(k+\ell+m)+}$.
\end{cor}

\begin{proof}
This is mostly just a reformulation of Proposition~\ref{prop:three-block-completion}.  The restriction of~\eqref{eq:param-map} to $\PPSD^\circ_{k,\ell,m}\times \Xi^\circ(m,k)$ is a homeomorphism because the relevant inverses and Schur complements in~\eqref{eq:Schurs} and~\eqref{eq:param} are continuous on this open subset.
\end{proof}

For any partially nonsingular $Q^?$, the map in~\eqref{eq:param-map} defines a homeomorphic parametrization of $\Delta(Q^?)$ by $\Xi(m,k)$ under which the nonsingular completions are identitied with the strict contractions.

If $Q_{22}$ is nonsingular, then a canonical choice of completion for $Q^?$ results by setting $C$ to be $0$.  Equivalently, this is the choice for which the second summand in~\eqref{eq:three-block-decomp2} vanishes.  If $Q$ is the combined Gram matrix of the tuples in~\eqref{eq:full-tuple}, then this decomposition is the same as~\eqref{eq:three-block-decomp}, and there the second term vanishes if and only if the subspaces $\rm{ran}\,V + \rm{ran}\,U$ and $\rm{ran}\,U + \rm{ran}\,X$ are relatively orthogonal over their common further subspace $\rm{ran}\,U$.  This particular choice of $Q$ is called the \textbf{central completion} of $Q_{(1\cup 2)}$ and $Q_{(2\cup 3)}$.  With our analogy between tuples of vectors and random variables in mind, the central completion is the analog of a relatively independent joint distribution for a triple of random variables (compare, for instance, the main technical construction in~\cite{Boz89}).  By~\eqref{eq:cond-mut-inf-contraction}, it is characterized by the vanishing of the relevant conditional log-determinant mutual information.

\subsection*{\emph{Notes and further references}}

Regarding log-determinants as a notion of entropy is widespread across information theory and matrix analysis.  An early development for applied purposes appeared in Burg's thesis~\cite{Burg75}.  The paper~\cite{ChoCov84} broadened Burg's work and demonstrated the convenience of information theoretic inequalities as basic techniques.

Extending this connection, many results in matrix analysis identify solutions to natural problems about positive semidefinite matrices through a variational principle for a log-determinant or a related quantity.  These include certain matrix completion problems, which may be phrased in terms of maximizing determinants (as in~\cite{GroJohSaWol84}) or maximizing `entropies' (as in~\cite{BakWoe92}), even if no operational meaning for entropy is intended.  An analogous use of `entropy' for certain integrals of logarithms appears in the study of orthogonal polynomials, interpolation problems and the `commutant lifting theorem': see, for instance,~\cite{FoiFraCL} or~\cite[Sec. 2.3]{SimOPUCI}. A modern account of such `maximum entropy' arguments is given in~\cite[Thm. 2.1.4 and Sec. 2.6]{BakWoeMCetc}.

Completion problems for positive semidefinite matrices, as well as many other classes of matrices, have been studied in great depth, both within pure linear algebra and for the sake of numerous engineering applications.  
A succinct classic can be found in~\cite{Joh90}, and the textbooks~\cite{FoiFraCL} and~\cite{BakWoeMCetc} cover it thoroughly.  In particular, positive two-by-two-block and three-by-three-block matrix completion are both treated thoroughly in~\cite[Chap. XVI]{FoiFraCL}, and both that chapter and~\cite[Chap. 2]{BakWoeMCetc} consider the natural further generalization of Proposition~\ref{prop:three-block-completion} to `banded' matrices with larger numbers of blocks.  For these, possible completions are parametrized by sequences of contractions in a way that generalizes~\eqref{eq:param}: see~\cite[Sec. XVI.8]{FoiFraCL} or~\cite[Thm. 2.1.2]{BakWoeMCetc}.
	
This is also an area of longstanding interaction between matrix analysis and abstract operator theory: see, for instance,~\cite[Chap. 3]{PauCB} for an introduction, including a generalization of Proposition~\ref{prop:two-block-completion} to $2\times 2$ matrices over a C$\sp*$-algebra~\cite[Lem. 3.1]{PauCB}.
	

\chapter{Some probability on spaces of matrices}

\section{Polar decomposition of integrals}\label{sec:polar-decomp-int}

Recall the notations $\bf{GL}(k,n)$, $\bf{U}(k,n)$, $\rmM_{k+}$ and $\rmM^\circ_{k+}$ from Section~\ref{sec:lin-alg}.  The set $\bf{GL}(k,n)$ is a dense open subset of $\rmM_{n,k} \cong (\bbC^{\oplus n})^k$. Its complement is defined by the vanishing of determinants of submatrices, so it has zero volume.  Therefore we may regard Lebesgue measure $\vol_{2kn}$ as a measure on $\bf{GL}(k,n)$.

The set $\rmU(k,n)$ has a natural action of $\rmU(n)$ by post-multiplication, and this action is transitive. So $\rmU(k,n)$ is a homogeneous space of the compact Lie group $\rmU(n)$; the stabilizer of any fixed choice of $k$ orthonormal vectors in $\bbC^{\oplus n}$ is a copy of ${\rmU(n-k)}$.  As a result, $\rmU(k,n)$ inherits a Haar probability measure $m_{\rmU(k,n)}$ from $\rmU(n)$.  For example, if $k = 1$, then $\rmU(1,n)$ is the unit sphere $\rmS^{2n-1}$, and $m_{\bf{U}(1,n)}$ is the normalized surface-area measure.

Also, the set $\rmM^\circ_{k+}$ is relatively open in the real-linear subspace $\rmM_{k,\rm{sa}}$ of $\rmM_k$, and $\rmM_{k,\rm{sa}}$ has real dimension $k^2$, so the restriction of $\vol_{k^2}$ to $\rmM^\circ_{k+}$ gives a nontrivial sigma-finite measure.

If $k\le n$, then the continuity of spectral calculus lets us define a continuous map
\begin{equation}\label{eq:polar-decomp-map}
	\bf{U}(k,n)\times \rmM^\circ_{k+} \to \bf{GL}(k,n):(V,Q)\mapsto VQ^{1/2}.
\end{equation}
This is a homeomorphism with inverse is given by polar decomposition of elements of $\bf{GL}(k,n)$.
It to the following formula for integration over $\bf{GL}(k,n)$.

\begin{prop}\label{prop:int-form}
	If $n\ge k$, then any positive Borel function $f$ on $\bf{GL}(k,n)$ satisfies
	\begin{multline}\label{eq:int-form}
		\int_{\bf{GL}(k,n)} f(T)\,d\vol_{2kn}(T) \\ = v(k,n)\int_{\rmM^\circ_{k+}} (\det Q)^{n-k} \int_{\bf{U}(k,n)} f(VQ^{1/2})\ dm_{\bf{U}(k,n)}(V)\ d\vol_{k^2}(Q),
	\end{multline}
	where
	\begin{equation}\label{eq:ckn}
		v(k,n) := \frac{\pi^{nk}}{\pi^{k(k-1)/2}\prod_{j=1}^k(n-j)!}.
	\end{equation}
\end{prop}

When $k=1$,~\eqref{eq:int-form} is simply the formula for integration over a complex vector space using polar coordinates:
\begin{align*}
	\int_{\bbC^n} f(z)\,d\vol_{2n}(z) &= \frac{\pi^n}{(n-1)!}\int_0^\infty q^{n-1}\int_{\rmS^{2n-1}}f(\sqrt{q}u)\ d\s_{2n-1}(u)\ dq \\
	&= \frac{2n\pi^n}{n!}\int_0^\infty r^{2n-1}\int_{\rmS^{2n-1}}f(ru)\ d\s_{2n-1}(u)\ dr,
\end{align*}
where the second line follows by substituting $q := r^2$. See, for instance,~\cite[Subs. 1.4.3 and 1.4.9]{RudinFTUB}.  The same reference also gives the identification
\begin{equation}\label{eq:ball-vol-2}
	v(1,n) = \frac{\pi^n}{(n-1)!} = n\cdot v(n),
\end{equation}
where $v(n)$ is the ball-volume function defined in~\eqref{eq:ball-vol}.

%
%

\begin{proof}[Proof of Proposition~\ref{prop:int-form}]
We can reduce~\eqref{eq:int-form} to standard results about complex Wishart distributions: see~\cite{Goo63}, for example.  Let $T$ be a random matrix in $\rmM_{n,k}$ whose entries are independent standard complex Gaussian random variables. This means that the overall distribution of $T$ is
\begin{equation}\label{eq:C-Gauss}
	\frac{1}{\pi^{nk}}e^{-\Tr_k(T^\ast T)}\ d\vol_{2kn}(T) \qquad (T \in \rmM_{n,k}).
\end{equation}
This is a product of $n$ copies of~\cite[eqn. (1.5)]{Goo63}.  Then the resulting random Gram matrix $Q := T^\ast T$ has the \textbf{complex Wishart distribution}
\begin{equation}\label{eq:C-Wish}
	\frac{1}{\pi^{k(k-1)/2}\prod_{j=1}^k(n-j)!}(\det Q)^{n-k}e^{-\Tr_k Q}\ d\vol_{k^2}(Q) \qquad (Q \in \rmM_{k+}):
\end{equation}
see~\cite[eqn. (1.6)]{Goo63}.

Observe that both sides of~\eqref{eq:int-form} are invariant under the action of $\bf{U}(n)$ by left multiplication.  By averaging over that action first, we reduce to the case when $f(T) = g(T^\ast T)$ for some positive Borel function $g$ on $\rmM_{k+}$.  Then, letting $h(Q) := e^{\Tr_k Q}g(Q)$, the distribution formulas~\eqref{eq:C-Gauss} and~\eqref{eq:C-Wish} give the equality of expectations
\begin{multline*}
	\frac{1}{\pi^{nk}}\int_{\bf{GL}(k,n)} e^{-\Tr_k(T^\ast T)}h(T^\ast T)\ d\vol_{2kn}(T) \\ = \frac{1}{\pi^{k(k-1)/2}\prod_{j=1}^k(n-j)!}\int_{\rmM^\circ_{k+}} (\det Q)^{n-k}e^{-\Tr_k Q}h(Q)\ d\vol_{k^2}(Q).
\end{multline*}
This is~\eqref{eq:int-form} after multiplying through by $\pi^{nk}$.
\end{proof}

We refer to the normalizing constants $v(k,n)$ frequently in the sequel. Rather than the exact formula~\eqref{eq:ckn}, we usually need only the following asymptotic behaviour for $k$ fixed and $n\to\infty$, which follows from~\eqref{eq:ckn} and Stirling's approximation:
\begin{equation}\label{eq:ckn-asymp}
	v(k,n) = \frac{(e\pi)^{nk}}{n^{nk}}\cdot e^{o(n)} \qquad \hbox{and hence} \qquad v(k,n) = v(n)^k\cdot e^{o(n)}.
\end{equation}

\section{A method of types for Gram matrices}\label{sec:method-of-types}


Section~\ref{sec:Gram} discusses log-determinants as an analog of discrete Shannon entropy.  One aspect of this analogy is a `method-of-types' interpretation for log-determinants.  It estimates the volumes of certain sets of `typical tuples' of vectors in high dimensions.  This is similar to the method of types in information theory~\cite[Sec. 11.1]{CovTho06}, which interprets Shannon entropy in terms of some basic counting problems.  When log-determinant entropy appears as the differential entropy of a jointly Gaussian random vector, it also has a standard method-of-types interpretation in terms of volumes: the basic idea is covered in~\cite[Sec. 8.2]{CovTho06}.

For any $n\ge k$ and $O \subset \rmM_{k+}$, let
\begin{equation}\label{eq:TnO}
	T(n,O) := \{X \in \rmM_{n,k}:\ X^\ast X \in O\}.
\end{equation}

\begin{thm}\label{thm:types-1}
	Let $Q$ be a $k$-by-$k$ positive semidefinite matrix.
	\begin{itemize}
		\item[a.] (Lower bound) If $O$ is any neighbourhood of $Q$ in $\rmM_{k+}$, then
		\[\frac{\vol_{2kn}T(n,O)}{v(n)^k} \ge (\det Q)^{n - o(n)}.\]
		\item[b.] (Upper bound) For any $a > \det Q$ there is a neighbourhood $O$ of $Q$ in $\rmM_{k+}$ such that
		\[\frac{\vol_{2kn}T(n,O)}{v(n)^k} \le a^{n + o(n)}.\]
	\end{itemize}
\end{thm}

This is~\cite[Thm. 6.3]{APE4}.  Here we give an alternative deduction from Proposition~\ref{prop:int-form}.

\begin{proof}
	When $Q$ is nonsingular, we can prove both parts together rather directly from Proposition~\ref{prop:int-form}.  Indeed, in that case we can shrink $O$ if necessary to assume that it is contained in $\rmM^\circ_{k+}$, and then that proposition gives
	\begin{align*}
		\vol_{2kn}T(n,O) &= v(k,n)\int_O (\det Q)^{n-k}\int_{\bf{U}(k,n)} 1\ dm_{\bf{U}(k,n)}(V)\ d\vol_{k^2}(Q)\\
		&=v(k,n)\int_O (\det Q)^{n-k}\ d\vol_{k^2}(Q).
	\end{align*}
	If $\eps > 0$, then we may choose $O$ so small that every $Q' \in O$ satisfies
	\begin{equation}\label{eq:O-above-and-below}
		e^{-\eps} \det Q \le \det Q' \le e^\eps \det Q,
	\end{equation}
	at which point the above integral implies that
	\[(e^{-\eps}\det Q)^{n-k}\vol_{k^2}O \le \frac{\vol_{2kn}T(n,O)}{v(k,n)} \le (e^\eps \det Q)^{n-k}\vol_{k^2}O.\]
	This implies both the lower and upper bounds, because (i) $\eps$ is arbitrary, (ii) $k$ is fixed while $n\to\infty$, (iii) the asymptotic~\eqref{eq:ckn-asymp} lets us replace $v(k,n)$ with $v(n)^k$, and (iv) the lower bound for a smaller choice of $O$ implies it for a larger choice of $O$, so $\eps n$ can be improved to $o(n)$ in the lower bound.
	
	The proof of the upper bound when $Q$ is singular is analogous, except this time we choose $O$ so small that every $Q' \in O$ satisfies $\det Q' \le a$, rather than~\eqref{eq:O-above-and-below}.
\end{proof}

We can generalize Theorem~\ref{thm:types-1} by considering the combined Gram matrix of one fixed tuple together with another that may vary.  This leads to a method-of-types interpretation of the log-determinant entropy of a Schur complement, which can be regarded as a conditional log-determinant entropy according to Definition~\ref{dfn:cond-log-det-ent}.   Assume that $n\ge \ell + k$, let $e_1$, \dots, $e_n$ be a fixed orthonormal basis for $\bbC^{\oplus n}$, and let $E := [e_1,\dots,e_\ell]$. For any subset $O$ of $\rmM_{(k+\ell)+}$, let
\begin{equation}\label{eq:T'}
	T'(n,O) := \{X \in \rmM_{n,k}:\ [E,X]^\ast [E,X] \in O.\}
\end{equation}
Let
\[Q = \left[\begin{array}{cc}I_\ell & R\\ R^\ast & Q_{22}\end{array}\right].\]
be a fixed $(k+\ell)$-by-$(k+\ell)$ positive semidefinite matrix, and let $S$ be the Schur complement $Q_{22} - R^\ast R$.

\begin{thm}\label{thm:types-2}
	The following hold:
	\begin{itemize}
		\item[a.] (Lower bound) If $Q$ is nonsingular and $O$ is any neighbourhood of $Q$ in $\rmM_{(k+\ell)+}$, then
		\[\frac{\vol_{2kn}T'(n,O)}{v(n)^k} \ge (\det S)^{n - o(n)}.\]
		\item[b.] (Upper bound) For any $a > \det S$ there is a neighbourhood $O$ of $Q$ in $\rmM_{(k+\ell)+}$ such that
		\[\frac{\vol_{2kn}T'(n,O)}{v(n)^k} \le a^{n + o(n)}.\]
	\end{itemize}
\end{thm}

This reduces to Theorem~\ref{thm:types-1} when $\ell = 0$.  We prove it in general by a reduction to that theorem.

\begin{proof}
	Denote a typical element of $\rmM_{(k+\ell)+}$ by
	\begin{equation}\label{eq:Q'blocks}
		Q' = \left[\begin{array}{cc}Q_{11}' & R'\\ (R')^\ast & Q'_{22}\end{array}\right].
	\end{equation}
	In this space, $Q$ has a neighbourhood base consisting of the sets of the form
	\[O = \{Q':\ Q_{11}' \in O_{11},\ R' \in O_{12},\ Q'_{22} - (R')^\ast R' \in O_{22}\}\]
	as $O_{11}$, $O_{12}$ and $Q_{22}$ range over neighbourhoods of $I_\ell$ in $\rmM_{\ell+}$, of $R$ in $\rmM_{\ell,k}$, and of $S$ in $\rmM_{k+}$ respectively.  It therefore suffices to prove parts (a) and (b) using sets of this form.  However, for such a set $O$, we have
	\[T'(n,O) = \{X \in\rmM_{n,k}: E^\ast X \in O_{12} \ \hbox{and}\ (PX)^\ast(PX) \in O_{22}\},\]
	where $P$ is the orthogonal projection of $\bbC^{\oplus n}$ onto the last $n-k$ coordinates, regarded as spanning a copy of $\bbC^{\oplus (n-k)}$.  As a result, $T'(n,O)$ has the same volume as the Cartesian product
	\[O_{12} \times T(n-\ell,O_{22}) \subset \rmM_{\ell,k} \times \rmM_{(n-\ell),k},\]
	where $T(n-\ell,O_{22})$ is defined as in~\eqref{eq:TnO}. The resulting factor of $\vol_{2k\ell} O_{12}$ does not depend on $n$, and the factor of $\vol_{2k(n-\ell)}T(n-\ell,O_{22})$ is governed by Theorem~\ref{thm:types-1}. This completes the proof when combined with~\eqref{eq:ckn-asymp}, which now gives the asymptotic
	\[v(k,n-\ell) = e^{o(n)}v(n)^k \qquad \hbox{as}\ n\to\infty\ \hbox{for fixed}\ k\ \hbox{and}\ \ell.\]
\end{proof}

In the rest of this section, we show how Theorems~\ref{thm:types-1} and~\ref{thm:types-2} imply some related large deviations principles for certain random Gram matrices.  

Let $k$, $\ell$ and $n$ be positive integers with $n\ge k+\ell$.  In $\bbC^{\oplus n}$, let $e_1$, \dots, $e_n$ be the standard basis, and let $E := [e_1,\dots,e_\ell]$. In addition, let $V$ be a random element of $\rmU(k,n)$ with distribution $m_{\bf{U}(k,n)}$.  The Gram matrices of $E$ and $V$ are $I_\ell$ and $I_k$, respectively.  The matrix of inner products $E^\ast V$ is a random element of $\Xi(k,\ell)$: it is simply given by the first $\ell$ rows of $V$.  We denote its distribution by
\begin{equation}\label{eq:projected-contraction-dist}
	\s_{n,\ell,k}A := m_{\rmU(k,n)}\{V:\ E^\ast V \in A\} \qquad (A \subset \Xi(k,\ell),\ A\ \hbox{Borel}).
\end{equation}

\begin{ex}\label{ex:k=1}
	Suppose that $k = 1$, so that $V$ is given by a single random vector $v$ distributed uniformly over the sphere $\rmS^{2n-1}$, and $E^\ast v$ is the vector of its first $\ell$ coordinates.  When $\ell=n$, $\s_{n,\ell,1}$ is simply the usual normalized surface measure of $\rmS^{2n-1}$. When $\ell < n$, it is given by
	\[\frac{(n-1)!}{\pi^\ell(n-\ell-1)!}\cdot (1 - |y|^2)^{n-\ell-1}\cdot 1_{\{|y| \le 1\}}\cdot d\vol_{2\ell}(y) \qquad (y \in \bbC^{\oplus \ell}).\]
	See~\cite[Sec. 1.4]{RudinFTUB}, for example.
	
	One can also derive an explicit formula for $\s_{n,\ell,k}$ when $\ell$ and $k$ are both at least $2$, but it is more complicated, and we avoid using it in the sequel. \qed
\end{ex}

\begin{thm}\label{thm:matrix-LDP1}
	For fixed $k$ and $\ell$, the sequence of distributions $(\s_{n,\ell,k})_{n\ge 1}$ obeys the large deviations principle on $\Xi(k,\ell)$ with speed $n$ and rate function
\[-\log\det(I_\ell - C^\ast C) \qquad (C \in \Xi(k,\ell)).\]
\end{thm}

\begin{proof}
	Since $\Xi(k,\ell)$ is compact, exponential tightness holds automatically.  It remains to prove conditions (i) and (ii) from Definition~\ref{dfn:LDP}.
	
	Let $O_{11}$ be a neighbourhood of $I_\ell$ in $\rmM_{\ell+}$.  Let $C\in \Xi(k,\ell)$ and let $O_{12}$ be a neighbourhood of it in that space.  Finally, let $\eps > 0$, and let $O_{22}$ be a neighbourhood of $I_k$ in $\rmM_{k+}$ which is so small that
	\begin{equation}\label{eq:det-control}
		e^{-\eps} \le \Det\, Q'_{22} \le e^\eps \qquad \hbox{for every}\ Q'_{22} \in O_{22}.
	\end{equation}
	(so, in particular, $O_{22} \subset \rmM^\circ_{k+}$).
	
	Write a typical element $Q'$ of $\rmM_{(k+\ell)+}$ in $2$-by-$2$ block form with $R'$ off the diagonal (see~\eqref{eq:Q'blocks}), and let
	\[O := \{Q' \in \rmM_{(k+\ell)+}:\ Q'_{11} \in O_{11},\ (Q_{11}')^{-1/2}R'(Q_{22}')^{-1/2} \in O_{12},\ Q_{22}' \in O_{22}\}.\]
	As we vary $O_{11}$, $O_{12}$ and $O_{22}$, such sets form a neighbourhood base in $\rmM_{(k+\ell)+}$ around the element
	\[Q := \left[\begin{array}{cc} I_\ell & C\\ C^\ast & I_k\end{array}\right].\]
	Let $S := I_\ell - C^\ast C$, the Schur complement of the $(1,1)$ block within $Q$.
	
	Define $T'(n,O)$ from $O$ as in~\eqref{eq:T'}. Applying Proposition~\ref{prop:int-form} and simplifying gives
	\[\frac{\vol_{2kn}T'(n,O)}{v(k,n)} = \s_{n,\ell,k}O_{12}\cdot \int_{O_{22}}(\Det\, Q'_{22})^{n-k}\ d\vol_{k^2}(Q'_{22}).\]
	Combined with~\eqref{eq:det-control}, this re-arranges to give
	\begin{equation}\label{eq:vol-and-sigma}
		e^{-\eps(n-k)}\frac{\vol_{2kn}T'(n,O)}{v(k,n)\vol_{k^2}O_{22}} \le \s_{n,\ell,k}O_{12} \le e^{\eps(n-k)}\frac{\vol_{2kn}T'(n,O)}{v(k,n)\vol_{k^2}O_{22}}.
	\end{equation}
	
	To prove the probability lower bound, note that for any neighbourhood $O_{12}$ and any $\eps > 0$ we may choose $O_{22}$ so that~\eqref{eq:det-control} holds, and then the left-hand side of~\eqref{eq:vol-and-sigma} is estimated by Theorem~\ref{thm:types-2}(a) to give
	\[\s_{n,\ell,k}O_{12} \ge (\det S)^ne^{- \eps}e^{- o(n)}.\]
	On the other hand, if $a > \det S$, then we may choose $O_{11}$, $O_{12}$ and $Q_{22}$ small enough that~\eqref{eq:det-control} holds, and now the right-hand side of~\eqref{eq:vol-and-sigma} and Theorem~\ref{thm:types-2}(b) give
	\[\s_{n,\ell,k}O_{12} \le a^ne^{\eps n}e^{o(n)}.\]
	By the arbitrariness of $\eps$ and $a$, this completes the proof.
\end{proof}

\begin{ex}\label{eq:k=1II}
	When $k = 1$ as in Example~\ref{ex:k=1}, so each $V$ is given by a single random vector $v$, then Theorem~\ref{thm:matrix-LDP1} asserts that the random $k$-dimensional projection $E^\ast v$ obeys a large deviations principle in the closed unit ball $\ol{B}$ of $\bbC^k$ with rate function
	\begin{equation}\label{eq:ball-LDP}
	-\log (1 - |y|^2) \qquad (y \in \ol{B}).
		\end{equation}
	This classical fact can also be read off directly from the formula in Example~\eqref{ex:k=1}.
	
	In this example, notice that $v$ is a uniform random element of the sphere whose \emph{real} dimension is $2n-1$.  This might make it more natural to consider the large deviations rate function for speed $2n$ rather than $n$.  This would be compensated by changing the formula~\eqref{eq:ball-LDP} into
	\[-\frac{1}{2}\log (1 - |y|^2) = -\log \sqrt{1 - |y|^2}.\]
However, by convention we always take the complex dimension $n$ as the speed in this paper.  This is also why the general rate function in Theorem~\ref{thm:matrix-LDP1} involves no square root.  Analogous remarks apply to Theorem~\ref{mainthm:LDP} in full, and to its 1-dimensional predecessor~\cite{GamNagRou16} (see also~\cite[Thm. 4.1]{BreSimZei18}), which is phrased with speed $n$ and no square roots. \qed
\end{ex}

%

The next property of our measures on $\Xi(k,\ell)$ is also needed later.

\begin{lem}\label{lem:a-s-strict}
	If $n \ge k+\ell$ then
	\[\s_{n,\ell,k}(\Xi^\circ(k,\ell)) = 1.\]
\end{lem}

\begin{proof}
	If we fix $X \in \bf{U}(k,n)$ and $Y \in \bf{U}(\ell,n)$, then it is equivalent to show that
	\[m_{\rmU(n)}\{U:\ \|X^\ast U Y\| < 1\} = 1.\]
	The complement of this event is the same as the event that the concatenation $[X,UY]$ is not an injection, or equivalently the event that
	\[\det \left[\begin{array}{cc}I_k & X^\ast UY\\ Y^\ast U^\ast X & I_\ell \end{array}\right] = 0.\]
	Regarded as a function of $U$, this determinant is a polynomial in its entries.  It does not vanish identically, because when $n \ge k +\ell$ we can certainly make some choice of $U$ so that the ranges of $V$ and $UY$ are disjoint subspaces.  Therefore the zero set of this polynomial has zero measure.
\end{proof}

Because of Theorem~\ref{thm:matrix-LDP1} and Lemma~\ref{lem:a-s-strict}, Lemma~\ref{lem:LDP-open-subset} lets us transfer our large deviations principle between $\Xi(k,\ell)$ and $\Xi^\circ(k,\ell)$.

\begin{cor}\label{cor:matrix-LDP1}
	Let $\s^\circ_{n,\ell,k}$ be a Borel probability measure on $\Xi^\circ(k,\ell)$ which agrees with the restriction of $\s_{n,\ell,k}$ whenever $n\ge k+\ell$.  For fixed $k$ and $\ell$, the sequence of distributions $(\s^\circ_{n,\ell,k})_{n\ge 1}$ obeys the LDP on $\Xi^\circ(k,\ell)$ with rate function
	\[-\log\det(I_\ell - C^\ast C) \qquad (C \in \Xi^\circ(k,\ell)).\]
\end{cor}

\begin{proof}
	Theorem~\ref{thm:matrix-LDP1} gives the analogous conclusion for $\Xi(k,\ell)$ rather than $\Xi^\circ(k,\ell)$.  To apply Lemma~\ref{lem:LDP-open-subset}, it remains to observe that $-\log \det(I_\ell - C^\ast C) = \infty$ whenever $C \in \Xi(k,\ell)\setminus \Xi^\circ(k,\ell)$. Indeed, if $C \in \Xi(k,\ell)$ is not a strict contraction, then $I_\ell - C^\ast C$ has a nontrivial kernel, in which case $\det(I_\ell - C^\ast C) = 0$.
\end{proof}

Later we need to apply Theorem~\ref{thm:matrix-LDP1} in a setting of the kind studied in Section~\ref{sec:three-block-Gram}. We next explain how the distributions $\s_{n,\ell,k}$ arise in that setting.  Suppose that $n\ge k+\ell+m$, and as in~\eqref{eq:full-tuple} fix three tuples
\[V = [v_1,\dots,v_k],\quad U = [u_1,\dots,u_\ell], \quad \hbox{and} \quad X_0 = [x_1,\dots,x_m].\]
Let $M$ be $\rm{ran}\,U$, let $G$ be the compact subgroup of $\rmU(n)$ that fixes $M$, let $W$ be a random element of $G$ with distribution $m_G$, and let $X = WX_0$.

Because $W$ is unitary and it fixes $M$, the combined tuple $[U,X] = [U,WX_0]$ always has the same joint Gram matrix as $[WU,WX_0]$, and hence the same as $[U,X_0]$.  Intuitively, $X$ is a tuple drawn `uniformly at random' conditionally on having this joint Gram matrix with $U$.  As a result, the overall joint Gram matrix
\[Q := \left[\begin{array}{ccc} Q_{11} & Q_{12} & R\\ Q_{12}^\ast & Q_{22} & Q_{23} \\ R^\ast & Q_{23}^\ast & I_m\end{array}\right] := \left[\begin{array}{ccc}V^\ast V & V^\ast U & V^\ast X\\ U^\ast V & U^\ast U & U^\ast X \\ X^\ast V & X^\ast U & X^\ast X\end{array}\right] \]
is random only in the blocks $R$ and $R^\ast$.  Thus, the partial matrix $Q^? := \pi^?(Q)$ is deterministic, fixed by our initial choice of $U$, $V$ and $X_0$, and $Q$ is a random element of $\Delta(Q^?)$ (see Section~\ref{sec:three-block-Gram} for notation).

We now add the assumption that $Q^?$ is partially nonsingular.  In this case, the relation~\eqref{eq:param} can be inverted to give a random contraction $C\in \Xi(k,m)$ corresponding to the random completion $Q \in \Delta(Q^?)$.  Having fixed $Q^?$, each of $Q$ and $C$ determines the other uniquely and continuously, by Corollary~\ref{cor:three-block-completion}.

\begin{prop}\label{prop:law-of-contraction}
	Under the assumptions above, $C$ has distribution $\s_{n-\ell,m,k}$.
\end{prop}

\begin{proof}
	Let $P$ be the orthogonal projection onto $M$.  Since every $W \in G$ satisfies $PW = WP = P$, the relevant block of $Q$ is given by
	\begin{align*}
		V^\ast X &= V^\ast WX_0\\
		&= V^\ast P WX_0 + V^\ast P^\perp WX_0 \\
		&= V^\ast P X_0 + V^\ast P^\perp WX_0\\
		&= V^\ast P X_0 + (P^\perp V)^\ast W (P^\perp X_0).
	\end{align*}
	
	The partial Gram matrix $Q^?$ determines the matrices $S_{11}$ and $S_{33}$ in~\eqref{eq:Schurs}. Using these and the terms that appear above, the random contraction $C$ is given by
	\[C = S_{11}^{-1/2}(P^\perp V)^\ast W (P^\perp X_0)S_{33}^{-1/2} = (P^\perp V S_{11}^{-1/2})^\ast W (P^\perp X_0S_{33}^{-1/2}). \]
	
	By~\eqref{eq:proj-Gram}, $S_{11}$ and $S_{33}$ are the Gram matrices of the linearly independent tuples $P^\perp V$ and $P^\perp X_0$, and therefore the tuples $P^\perp VS_{11}^{-1/2}$ and $P^\perp X_0S_{33}^{-1/2}$ are unitary maps from $\bbC^{\oplus k}$ into $M^\perp$ and from $\bbC^{\oplus m}$ into $M^\perp$, respectively.  Since $W$ corresponds to a uniformly random element of $G$, making a final unitary identification of $M^\perp$ with $\bbC^{\oplus (n-\ell)}$ completes the proof.
\end{proof}

\section{Measure concentration}

The groups $\bf{U}(n)$ are examples of the measure concentration phenomenon as $n\to\infty$, and so are fixed Cartesian powers $\bf{U}(n)^r$.  To formulate the result precisely, define the Hilbert--Schmidt metric on $\bf{U}(n)^r$ by
\[d_{\rm{HS}}\big((U_{1,1},\dots,U_{1,r}),(U_{2,1},\dots,U_{2,r})\big) := \sqrt{\sum_{i=1}^r \|U_{1,i}-U_{2,i}\|_{\rm{HS}}^2}.\]

\begin{thm}\label{thm:Un-conc}
If $F:\bf{U}(n)^r\to \bbR$ is $L$-Lipschitz for $d_{\rm{HS}}$, then
	\[m_{\bf{U}(n)}^{\times r}\Big\{\Big|F - \int F\,dm_{\bf{U}(n)}^{\times r}\Big| \ge t\Big\} \le 2e^{-nt^2/24L^2}.\] \qed
	\end{thm}

This generalizes L\'evy's classical concentration inequality for the sphere $\rmS^{2n-1}$, which played a role in the foundations of AP entropy via~\cite[Lem. 6.1]{APE4}.

The analog of Theorem~\ref{thm:Un-conc} for $\bf{SU}(n)$ can be deduced from the lower bound on its Ricci curvature as a Riemannian manifold.  Starting from this curvature lower bound, one can either use the volume-comparison estimates of Gromov to show that these inequalities are implied by their counterparts on the sphere~\cite{GroMil83}, or run the semigroup argument of Bakry and Emery~\cite{BakEme85,Led92} to prove a logarithmic Sobolev inequality.  Theorem~\ref{thm:Un-conc} for $\bf{U}(n)$ can be derived from the result for $\bf{SU}(n)$ by a suitable coupling of the Haar measures on these groups.  See~\cite[Thm. 15]{MecMec13} for a proof of this form that gives the optimal constant.


\subsection*{\emph{Notes and further references}}

See, for instance,~\cite[Exer. 2.1.18 and Prop. 4.1.3]{AndGuiZei--book} (as well as the bibliographical notes to both chapters),~\cite[Sec. 6.2]{WaiHDS}, or~\cite[pp129--131 and Lem. 4.4.7]{HiaPetSL} for several variants and enhancements of Proposition~\ref{prop:int-form} relating to Wishart distributions.

An overview of concentration inequalities and methods for proving them is given in several books, including~\cite{Led95}.  See~\cite[Sec. 5.3]{MecRMTCCG} or~\cite[Subs. 4.4.2]{AndGuiZei--book} for more specialized accounts of concentration for compact matrix groups with thorough selections of references.

\chapter{C*-algebras and their representations}\label{chap:rep-th-op-alg}

We generally assume familiarity with basic functional analysis and operator algebra theory, but this chapter recalls a number of standard results that we need explicitly later.  For definiteness, I cite the standard texts by Dixmier~\cite{Dix--Cstar,Dix--vN}, Folland~\cite{FolAHA}, Paulsen~\cite{PauCB}, or Mackey~\cite{Mac76} wherever possible.

In this book, $\A$ always denotes a \emph{separable} C$\sp\ast$-algebra.  Our main examples are group C$\sp*$-algebras $C^\ast\G$ of countable groups $\G$: see Section~\ref{sec:ctble-groups} for more about these.  In Part~\ref{part:free}, $\G$ is usually a finitely-generated free group.

We adhere to the following convention:
\begin{itemize}
	\item C$^\ast$-algebras may exist \emph{in the abstract};
	\item a von Neumann algebra is a weak-operator closed $\ast$-subalgebra of $\B(H)$ for some particular Hilbert space $H$.
\end{itemize}

In $\A$, the term `ideal' always refers to a closed, two-sided ideal.

Several of the sections below end with a review of more advanced references for the interested reader.  Readers with experience in this area may prefer to skip this chapter and refer back to it.

\section{Representations, positive functionals, and completely positive maps}\label{sec:reps-and-maps}

In this book we always study separable representations, meaning that they act on separable complex Hilbert spaces.  We also always assume that $\A$ is \emph{unital} unless stated otherwise.  We usually denote a representation by a single letter such as $\pi$, and then write its Hilbert space as $H_\pi$ when it is needed.  Part~\ref{part:free} of the book is largely concerned with sequences of finite-dimensional representations together with their limits in various modes of convergence.  All such limits are still separable, so assuming separability from the outset does not entail any real loss of generality.

Once a particular representation $\pi$ is being considered, operator theory provides many auxiliary constructions of operators in $\B(H_\pi)$, for instance via the Borel functional calculus or symmetry considerations.  These often fall outside the operator-norm closure of $\pi(\A)$, which is a C$\sp*$-algebra, but within its bi-commutant $\pi(\A)''$, which agrees with the weak-operator closure of $\pi(\A)$ and is a von Neumann algebra.

If $\A$ is non-unital, or if $\A$ is unital but a representation $\pi$ of it is non-unital, then $\pi$ may be \textbf{degenerate} in the sense that
\begin{equation}\label{eq:non-deg-part}
\ol{\rm{span}}\{\pi(a)v:\ a \in \A,\ v \in H_\pi\}
\end{equation}
is a proper subspace of $H_\pi$.  Equivalently, this holds if $H_\pi$ itself is not cyclic for $\pi$.  Otherwise $\pi$ is called \textbf{non-degenerate}.

If $\pi$ is a representation of $\A$ and $M$ is a closed $\pi$-invariant subspace of $H_\pi$, then we write $\pi^M$ for the associated \textbf{subrepresentation}.  A representation $\pi$ is \textbf{irreducible} if it nonzero but its only closed invariant subspaces are $\{0\}$ and $H_\pi$.  If $\pi$ is irreducible then it is either $0$ or non-degenerate, since the closed span in~\eqref{eq:non-deg-part} is $\pi$-invariant.  We write $\hat{\A}$ for the space of equivalence classes of nonzero irreducible representations of $\A$, and call this the \textbf{spectrum} of $\A$: see~\cite[Chap. 3]{Dix--Cstar}.

We use $\oplus$ for general direct sums in the categories of Hilbert spaces or representations.  For a single representation $\pi$ and any $k \in \{1,2,\dots\} \cup \{\infty\}$, we write $\pi^{\oplus k}$ for the direct sum of $k$ copies of $\pi$, using the indexing set $\bbN$ when $k=\infty$.  We call this the \textbf{$k$-fold inflation} of $\pi$.

If $\pi$ is a representation of $\A$, then a subset $S$ of $H$ is \textbf{cyclic} for $\pi$ if it is not contained in any proper closed invariant subspace of $H_\pi$, or equivalently if
\[\ol{\sum_{v \in S}\pi(\A)v} = H_\pi.\]

If $\pi$ is a representation of $\A$ and $v \in H_\pi$, then the formula
\[\Phi^\pi_v(a) := \langle \pi(a)v,v\rangle \qquad (a \in \A)\]
defines a positive linear functional on $\A$.  All positive linear functionals arise this way because of the GNS construction.  We write $\A^\ast_+$ for the space of positive linear functionals on $\A$, and always endow it with the weak$^\ast$ topology inherited from $\A$.  These standard notions and their relation to representations can be found in~\cite[Chap. 2]{Dix--Cstar}.

More generally, let $v_1$, \dots, $v_k \in H_\pi$, and regard the tuple $V := [v_1,\dots,v_k]$ as a linear map from $\bbC^{\oplus k}$ to $H_\pi$.  Then the formula
\begin{equation}\label{eq:cp-assoc}
\Phi^\pi_V(a) := V^\ast \pi(a) V = [\langle \pi(a)v_j,v_i\rangle]_{i,j=1}^k
\end{equation}
defines a linear map from $\A$ to $\rmM_k$. Notice that the order of the indices in each matrix entry matches the convention for the Gram matrix of a tuple of vectors in~\eqref{eq:Gram}.  We sometimes write $\Phi^\pi_{v_1,\dots,v_k}$ instead of $\Phi^\pi_V$.

The map $\Phi^\pi_V$ lies in the class of \textbf{completely positive} maps. The basic theory of these can be found in~\cite[Chaps. 3--6]{PauCB}, for example.  This particular instance of them is discussed following~\cite[Thm. 4.1]{PauCB}.  We say that the completely positive map in~\eqref{eq:cp-assoc} is \textbf{associated to $\pi$ by $v_1$, \dots, $v_k$}. Alternatively, we adapt a term from information theory by calling $\Phi^\pi_V$ the \textbf{type} of the tuple $V$ in $\pi$ (compare with~\cite[Sec. 11.1]{CovTho06}, for example).  Much of our work later concerns classifying the tuples in a representation space according to their type.

In this work we do not need the definition of complete positivity directly, but only a few of its basic consequences.  First, unital completely positive maps are continuous~\cite[Prop. 3.6]{PauCB}.  Let $\B(\A,\rmM_k)$ be the space of all continuous linear maps from $\A$ to $\rmM_k$, endowed with the pointwise topology (equivalently, the weak$^\ast$ topology applied to each matrix entry), and let $\B(\A,\rmM_k)_+$ be the subset of all completely positive ones.  The definition of complete positivity consists of a family of closed linear inequalities, showing that $\B(\A,\rmM_k)_+$ is a closed cone in $\B(\A,\rmM_k)$ for the pointwise topology.  Finally, a generalization of the GNS construction shows that any $\phi \in \B(\A,\rmM_k)_+$ has the form in~\eqref{eq:cp-assoc} for some representation $\pi$ and tuple $v_1$, \dots, $v_k$. If this tuple is cyclic for $\pi$, then $\pi$ and the tuple are unique up to unitary equivalence; in this case we call them a \textbf{minimal dilation} of $\phi$, and often denote the representation by $\pi_\phi$.   This fact is a special case of Stinespring's theorem: see~\cite{Sti55},~\cite[Thm. 4.1]{PauCB}, or~\cite[Thm. 34.7]{ConOT}.

A completely positive map $\phi:\A \to \rmM_k$ is called \textbf{normalized} if $\tr_k\phi(1) = 1$, and \textbf{unital} if $\phi(1) = I_k$.  We denote the sets of these by $\S_k(\A)$ and $\S_k^\rm{u}(\A)$, respectively.  In particular, $\S_1(\A)$ and $\S_1^\rm{u}(\A)$ are both just the usual state space of $\A$.  For each $k$, the set $\S_k(\A)$ is a compact and convex subset of $\B(\A,\rmM_k)$, and $\S_k^{\rm{u}}(\A)$ is a further compact convex subset of $\S_k(\A)$.


If $\phi \in\B(\A,\rmM_k)_+$ and $Q \in \rmM_k$, then the map defined by
	\begin{equation}\label{eq:psi-Q-phi}
\psi(a) := (Q^\rm{T})^\ast\phi(a)Q^\rm{T} \qquad (a \in \A)
	\end{equation}
is also completely positive.  This may be checked directly from the definition of complete positivity, or by observing the corresponding relationship between associating vectors.  If $\phi$ is associated to a representation $\pi$ by the tuple $v_1$, \dots, $v_k$, then the map $\psi$ in~\eqref{eq:psi-Q-phi} is associated to $\pi$ by the $k$-tuple defined by
\[y_i := \sum_j q_{ij}v_j \qquad (1\le i \le k).\]
Identifying such $k$-tuples with elements of $H_\pi\otimes \bbC^{\oplus k}$, and writing them as column vectors of height $k$ with entries in $H_\pi$, we can alternatively write this new tuple as
\begin{equation}\label{eq:pre-lin-maps-Q}
\left[\begin{array}{c}y_1\\ \vdots \\ y_k\end{array}\right] := (I_{H_\pi}\otimes Q) \left[\begin{array}{c}v_1\\ \vdots \\ v_k\end{array}\right].
\end{equation}
(We put transposes in~\eqref{eq:psi-Q-phi} to avoid them in~\eqref{eq:pre-lin-maps-Q}.) If $\phi(1)$ is nonsingular, then one can often use~\eqref{eq:psi-Q-phi} with $Q = \phi(1)^{-1/2}$ to reduce results about $\phi$ to the unital case.

Many results about positive functionals generalize to matrix-valued completely positive maps fairly mechanically.  Sometimes the scalar-valued versions imply their matrix-valued counterparts via the following standard device.  Given a positive integer $k$, let $\rmM_k(\A)$ be the algebra of $k$-by-$k$ matrices over $\A$, made into a C*-algebra using any faithful representation of $\A$.  Let $\phi:\A\to\rmM_k$ be completely positive, and define a linear functional on $\rmM_k(\A)$ by the pairing
\begin{equation}\label{eq:pairing}
	\langle \phi,a\rangle := \frac{1}{k}\sum_{ij}\phi_{ij}(a_{ij}).
\end{equation}
This pairing defines a canonical isomorphism from $\B(\A,\rmM_k)$ to $\rmM_k(\A)^\ast$, and it identifies the closed cone $\B(\A,\rmM_k)_+$ with $\rmM_k(\A)^\ast_+$.  See~\cite[Thm. 6.1]{PauCB} or~\cite[Subs. 2.6]{APE4}.  The pairing~\ref{eq:pairing} leads to the identity 
\begin{equation}\label{eq:norm-cts}
\|\langle \phi,\cdot\rangle\|_{\rmM_k(\A)^\ast} = \langle \phi,1\otimes I_k\rangle = \tr_k\phi(1),
\end{equation}
which is a consequence of positivity: see~\cite[Cor. 2.8]{PauCB}.

The weak$^\ast$ topology on $\B(\A,\rmM_k)$ is generally not metrizable, or even first countable.  However, since $\A$ is separable, the restriction of that topology is metrizable on any subset of $\B(\A,\rmM_k)_+$ that is bounded in the dual norm obtained from $\rmM_k(\A)_+^\ast$~\cite[Proposition V.5.1]{ConFA}.  Because of~\eqref{eq:norm-cts}, the subsets $\S_k(\A)$ and $\S_k^{\rm{u}}(\A)$ are examples.  These subsets are also weak$^\ast$-closed and convex, and hence compact by the Banach--Alaoglu theorem.  For the whole of $\B(\A,\rmM_k)_+$ we also have the following: see~\cite[Lem. 2.7]{APE4}.

\begin{lem}\label{lem:lcsc}
The restriction of the weak$^\ast$ topology to $\B(\A,\rmM_k)_+$ is locally compact and second countable. \qed
\end{lem}

A linear functional $\tau$ on $\A$ is \textbf{tracial} if
\begin{equation}\label{eq:trace}
	\tau(ab) = \tau(ba) \qquad (a,b \in \A).
\end{equation}
If a tracial positive functional $\tau$ is associated to the representation $\l$ on $H$ by a vector $\xi$, then~\eqref{eq:trace} re-arranges into this identity for $\xi$:
\begin{equation}\label{eq:trace-vec}
	\langle \l(b)\xi,\l(a^\ast)\xi\rangle = \langle \l(a)\xi,\l(b^\ast)\xi\rangle \qquad (a,b \in \A).
\end{equation}
In any representation, a vector satisfying~\eqref{eq:trace-vec} is called \textbf{tracial}.  At this point, we can take weak-operator limits in $a$ and then separately in $b$ to conclude that the positive functional
\begin{equation}\label{eq:tau-on-N}
\t{\tau}(A) := \langle A\xi,\xi\rangle
\end{equation}
is actually tracial on the whole von Neumann algebra $\N := \l(\A)''$.  

\section{Comparing representations or completely positive maps}\label{sec:compare}

If $\pi$ and $\rho$ are two representations of $\A$, then:
\begin{itemize}
	\item $\pi$ is \textbf{equivalent} to $\rho$, written $\pi \simeq \rho$, if there is a unitary operator from $H_\pi$ to $H_\rho$ that intertwines $\pi$ with $\rho$;
	\item $\pi$ is \textbf{contained} in $\rho$, written $\pi \lesssim \rho$, if $\pi$ is equivalent to a subrepresentation of $\rho$;
	\item $\pi$ and $\rho$ are \textbf{disjoint}, written $\pi \spoon \rho$, if no nontrivial subrepresentations of them are equivalent.
\end{itemize}

Another possible relation is that $\pi$ is contained in an inflation of $\rho$. If $H_\pi$ is separable, then we can take this to mean that $\pi \lesssim \rho^{\oplus \infty}$.

The next two propositions give fundamental ways to decompose representations.

\begin{prop}\label{prop:Leb-reps}
	If $\pi$ and $\rho$ are two representations of $\A$, then $\pi$ has a unique invariant subspace $M$ such that $\pi^M \lesssim \rho^{\oplus \infty}$ and $\pi^{M^\perp} \spoon \rho$. The orthogonal projection onto $M$ lies in the centre of $\pi(\A)''$. \qed
\end{prop}

For unitary representations of groups, this is~\cite[Thm. 1.11]{Mac76}; the proof given there handles general C$\sp\ast$-algebras without change.  This result also appeared independently in~\cite{Tak58} with a proof in terms of decompositions of the universal representation of $\A$.

\begin{prop}\label{prop:Leb-reps-2}
If $\pi$ and $\rho$ are two representations of $\A$, then they have decompositions $\pi \simeq \pi_1 \oplus \pi_2$ and $\rho \simeq \rho_1 \oplus \rho_2$ so that $\pi_1 \lesssim \rho_1$, $\pi_2 \gtrsim \rho_2$ and $\rho_1 \spoon \pi_2$. \qed
\end{prop}

See~\cite[Thm. 1.13]{Mac76}, where these decompositions are obtained by comparing against two other carefully chosen representations using Proposition~\ref{prop:Leb-reps}.  Alternatively Proposition~\ref{prop:Leb-reps-2} holds by comparing the corresponding projections in the von Neumann algebra $(\pi\oplus \rho)(\A)'$ via their relations to its centre (see~\cite[Thm. III.1.1]{Dix--vN} or~\cite[Thm. 47.8]{ConOT}, for example).

\begin{cor}\label{cor:Leb-reps}
Let $\pi$ and $\rho$ be decomposed as in Proposition~\ref{prop:Leb-reps-2}.  Then $\rho_1 \oplus \pi_2$ contains both $\pi$ and $\rho$, and is contained in any other representation which contains both $\pi$ and $\rho$.
\end{cor}

\begin{proof}
We have $\rho_1 \oplus \pi_2 \gtrsim \pi_1\oplus \pi_2 \simeq \pi$, and similarly for containing $\rho$.  On the other hand, if $\kappa$ contains both $\pi$ and $\rho$, then $\kappa$ certainly contains both $\rho_1$ and $\pi_2$, and now it must contain their direct sum because they are disjoint.
\end{proof}

Because of the second part of Corollary~\ref{cor:Leb-reps} and the version of the Schr\"oder--Bernstein theorem for *-representations~\cite[Thm. 1.1]{Mac76}, the resulting representation $\rho_1\oplus \pi_2$ is unique up to unitary equivalence.

The relations between representations described above enable various ways of comparing completely positive maps.

Let $\phi:\A\to \rmM_k$ and $\psi:\A\to \rmM_\ell$ be completely positive.  If $k = \ell$, then $\phi \le \psi$ in the \textbf{positive definite} ordering if $\psi - \phi$ is still positive definite.  When $k=\ell=1$,  this relation implies containment of GNS representations by the following basic result.

\begin{prop}\label{prop:dom-contained}
	Let $k=\ell=1$, and assume that $\psi$ is associated to $\pi_\psi$ by $v$. Let $c > 0$. Then $\phi \le c\psi$ if and only if there exists $T \in \pi_\psi(\A)'$ such that (i) $0\le T \le c$ and (ii) $\phi$ is associated to $\pi_\psi$ by $Tv$. If such a $T$ exists, then it is unique.
	
In particular, if $\phi \le c\psi$ for some $c > 0$, then $\phi$ is associated to $\pi_\psi$. \qed
\end{prop}

See~\cite[Prop. 2.5.1]{Dix--Cstar}.  This can be generalized to allow $k=\ell > 1$ using the pairing isomorphism, but we avoid needing that generalization in the sequel.  On the other hand, we do need the following corollary.

\begin{cor}\label{cor:dom-contained}
Let $k=\ell=1$, assume that $\psi \spoon \phi$, and let $a,b > 0$.  Then
\[\pi_{a\phi+b\psi} \simeq \pi_\phi \oplus \pi_\psi.\]
\end{cor}

\begin{proof}
Let $\pi:= \pi_{a\phi+b\psi}$ for brevity. By applying the last part of Proposition~\ref{prop:dom-contained} to $\phi$ (respectively, $\psi)$ and $a\phi + b\psi$, we find that $\phi$ (respectively, $\psi$) is associated to $\pi$, say by the vector $u$ (respectively, $v$).  Then $u$ and $v$ generate copies of $\pi_\phi$ and $\pi_\psi$ inside $\pi$, and these copies must be orthogonal because of the assumed disjointness.  From that orthogonality, it follows that $a \phi + b\psi$ is associated to $\pi$ by the vector $\sqrt{a}u + \sqrt{b}v$.  By the uniqueness of the GNS representation, this vector is cyclic for the whole of $\pi$.  In particular, the sum of those orthogonal copies of $\pi_\phi$ and $\pi_\psi$ equals the whole of $\pi$.
\end{proof}

We also need two other relations involving maps and representations:
\begin{itemize}
\item If $\pi$ is a representation and $\phi$ is a completely positive map, then $\phi$ is \textbf{$\pi$-normal} if $\phi$ is associated to $\pi^{\oplus \infty}$.  This is equivalent to $\phi$ having the form $\t{\phi}\circ \pi$ for some normal positive functional $\t{\phi}$ on $\pi(\A)''$, by the general characterization of normal positive functionals on a von Neumann algebra~\cite[Thm. I.4.1]{Dix--vN}. 
\item On the other hand, $\phi$ is \textbf{$\pi$-singular} if $\pi_\phi \spoon \pi$.  If $\pi = \pi_\psi$, then we call $\phi$ and $\psi$ themselves \textbf{disjoint}, and sometimes write $\phi \perp \psi$.
\end{itemize}

These definitions lead to analogs of the Lebesgue decomposition and Radon--Nikodym theorem.  The first of these is obtained by applying Proposition~\ref{prop:Leb-reps} to $\pi_\phi$ and $\pi$.

\begin{cor}\label{cor:Leb}
	For any $\phi$ and $\pi$ as above, there is a unique decomposition
	\begin{equation}\label{eq:Leb}
	\phi = \phi_{\rm{ac}} + \phi_{\rm{sing}}
	\end{equation}
	into positive summands such that $\phi_{\rm{ac}}$ is $\pi$-normal and $\phi_{\rm{sing}}$ is $\pi$-singular. This decomposition is linear in $\phi$. \qed
\end{cor}

We refer to~\eqref{eq:Leb} as the \textbf{Lebesgue decomposition} of $\phi$ with respect to $\psi$.  As far as I know, for positive functionals the earliest appearance of this result was in~\cite{Tak58}.  The case of $\rmM_k$-valued completely positive maps has a similar proof, or it can be derived via the pairing isomorphism (see~\cite[Subs. 2.7]{APE4}).

The Radon--Nikodym theorem has several non-commutative generalizations.  Indeed, Proposition~\ref{prop:dom-contained} can already be regarded as having this flavour.  Another variant holds when we compare a completely positive map to a \emph{tracial} positive functional.  To explain it, let $\l$ be a representation with a cyclic tracial unit vector $\xi$.

\begin{prop}\label{prop:RadNik}
Let $\phi \in \B(\A,\rmM_k)_+$ be $\l$-normal, and form this $k$-tuple in $H_\l^{\oplus k}$:
	\begin{equation}\label{eq:xi-i}
	\xi_i = [0,\dots,0,\xi,0,\dots,0]^\rm{T} \qquad (i=1,2,\dots,k),
	\end{equation}
	where only the $i^\rm{th}$ coordinate of $\xi_i$ is nonzero. Then there is a unique non-negative operator $T$ affiliated to $\l^{\oplus k}(\A)'$ such that $\xi_i \in \rm{dom}\,T$ for each $i$ and $\phi$ is associated to $\l^{\oplus k}$ by $T\xi_1$, \dots, $T\xi_k$. \qed
\end{prop}

See~\cite[Prop. 2.8]{APE4}.  The case $k=1$ already contains the main ideas, and is due to Dye.

The next definition offers more ways to compare completely positive maps.  Let $k$ and $\ell$ be positive integers, and let
\[K \:= \{1,\dots,k\} \qquad \hbox{and} \qquad L := \{k+1,\dots,k+\ell\}.\]

\begin{dfn}\label{dfn:joining}
Let $\phi:\A \to \rmM_k$ and $\psi:\A \to \rmM_\ell$ be completely positive.  A \textbf{joining} of them is a completely positive map $\theta:\A\to \rmM_{k+\ell}$ such that
\[\theta(a)[K] = \phi(a) \qquad \hbox{and} \qquad \theta(a)[L] = \psi(a) \qquad (a \in \A).\]
In particular, the \textbf{diagonal joining} is defined by
\[\rm{diag}(\phi,\psi)(a) := \left[\begin{array}{cc} \phi(a) & 0\\ 0 & \psi(a)\end{array}\right] \qquad (a \in \A).\]

Joinings of more than two completely positive maps are defined similarly.  If $\phi$ is a positive functional and $k$ is a positive integer, then we write $\phi\otimes I_k$ for the diagonal joining of $k$ copies of $\phi$.
\end{dfn}

The term `joining' is borrowed from Furstenberg's classic paper~\cite{Fur67} in ergodic theory, and was also used this way for representations in~\cite[Def. 2.9]{APE4}.  See~\cite[Sec. 10.(C)]{KecGAEGA} for an example of their use in ergodic theory which is close in spirit to the present work.  Their importance is more limited for representations, because these can always be broken up using orthogonal complements, but they are still a convenient way of capturing various constructions in the sequel.

The joinings of $\phi$ and $\psi$ describe all the ways these two maps can `sit together' inside a single representation.  If $\theta$ is a joining, then $\pi_\theta$ contains copies of $\pi_\phi$ and $\pi_\psi$ that together span the whole of $H_\theta$. The minimal dilation of $\rm{diag}(\phi,\psi)$ is equivalent to $\pi_\phi \oplus \pi_\psi$, and $\phi$ and $\psi$ are disjoint if and only if $\rm{diag}(\phi,\psi)$ is their only joining.

If $\phi$ and $\psi$ are both normalized, then so is any joining $\theta$ of them.  This follows from the simple relation
\begin{equation}\label{eq:still-normalized}
\tr_{k+\ell}\theta(1) = \frac{k}{k+\ell}\tr_k\theta(1)[K] + \frac{\ell}{k+\ell}\tr_\ell\theta(1)[L].
\end{equation}
However, unital maps may have non-unital joinings.

\section{C*-algebras of countable groups}\label{sec:ctble-groups}

Let $\G$ be a countable group.  We write $\bbC[\G]$ for the complex group algebra of $\G$, and regard it concretely as the space of functions from $\G$ to $\bbC$ that have finite support. Given $g \in \G$, we write $\delta_g$ for its canonical image in $\bbC[\G]$, which is the indicator function of the singleton $\{g\}$. We may express a general element of $\bbC[\G]$ as either an indexed family $(a_g:\ g \in \G)$ or a sum like $\sum_g a_g\delta_g$, where in either case $a_g$ is zero for all but finitely many $g$.  If $a$ is this element and $F \subset \G$, then $a$ is \textbf{supported} in $F$ if $a_g = 0$ for all $g \in \G\setminus F$. The vector space operations of $\bbC[\G]$ are pointwise, and its multiplication is convolution:
\[a\ast b := \sum_{g,h} a_g b_h \delta_{gh} \qquad (a,b \in \bbC[\G]).\]
We also always endow $\bbC[\G]$ with the involution
\begin{equation}\label{eq:CG-inv}
	a^\ast := \sum_g \ol{a_g}\delta_{g^{-1}} \qquad (a \in \bbC[\G]),
\end{equation}
making it a $\ast$-algebra with unit $\delta_e$.

The group C*-algebra $C^\ast \G$ is the maximal C*-completion of $\bbC[\G]$.  It is unital and separable.  The main properties of $C^\ast \G$ can be found in~\cite[Sec. 7.1]{FolAHA} or~\cite[Sec. 13.9]{Dix--Cstar}, which both allow the greater generality of locally compact groups.  For each $g \in \G$, we continue to write $\delta_g$ for its image in $C^\ast \G$.  Any unitary representation of $\G$ extends uniquely to a unital representation of $C^\ast\G$, and any unital representation of $C^\ast \G$ arises uniquely in this way.  With this in mind, we pass back and forth between groups and their C$^\ast$-algebras rather freely, for instance by using the same notation for a representation of either.  The space $\hat{\G}$ of irreducible unitary representations of $\G$ is naturally identified with $\hat{C^\ast \G}$, and is called the \textbf{unitary dual} of $\G$: see~\cite[Chap. 18]{Dix--Cstar}.

Any $\phi \in \B(C^\ast \G,\rmM_k)_+$ defines a map from $\G$ itself to $\rmM_k$ by restriction:
\[g\mapsto \phi(\delta_g) \qquad (g\in \G).\]
In the sequel we often write simply $\phi(g)$ instead of $\phi(\delta_g)$ in this situation.  If $\phi = \Phi^\pi_V$ for some representation $\pi$ and $k$-tuple $V$ as in~\eqref{eq:cp-assoc}, then this becomes
\begin{equation}\label{eq:matrix-elements}
\Phi^\pi_V(g) = [\langle \pi(g)v_j,v_i\rangle] \qquad (g \in \G).
\end{equation}
In representation theory, the function on $\G$ given by $\langle \pi(g)v,u\rangle$ is called the \textbf{$(u,v)$-matrix element} of the representation.

If a function $\phi:\G\to \rmM_k$ arises this way, then it is a \textbf{positive definite} map (sometimes called a `map of positive type').  This means that it is bounded and satisfies
\begin{equation}\label{eq:pdf}
	\sum_{1 \le i,j\le k}\phi_{ij}(a_i^\ast a_j) = \sum_{g,h \in \G,\,1 \le i,j \le k}\ol{a_{i,h}}a_{j,g}\phi_{ij}(h^{-1}g) \ge 0
\end{equation}
for any $a_1$, \dots, $a_k \in \bbC[\G]$.  For instance, for the map in~\eqref{eq:matrix-elements}, this holds because the sum in~\eqref{eq:pdf} is just equal to the squared length of the vector $\sum_i\pi(a_i)v_i$.  On the other hand, if $\phi$ is any map on $\G$ satisfying~\eqref{eq:pdf}, then another variant of the GNS construction produces a unitary representation of $\G$ to which $\phi$ is associated as in~\eqref{eq:matrix-elements}.  Using this, $\phi$ then extends to a completely positive map on the whole of $C^\ast\G$.  The case $k=1$ of this result first appeared in work of Gelfand and Raikov~\cite{GelRai43}, and can also be found in~\cite[Sec. 3.3]{FolAHA} or~\cite[Sec. 13.4]{Dix--Cstar}.  It is an analog for groups of applying the GNS construction to a state on a C*-algebra.  Similarly, the generalization to $\rmM_k$-valued positive definite functions on a group is a cousin of Stinespring's theorem, and has an analogous proof.  It can be found as~\cite[Thm 4.8]{PauCB}, where the key ideas are traced back to Naimark. The relationship between positive definite functions on a group and completely positive maps on the group C$\sp*$-algebra is discussed further at the end of~\cite[Chap. 4]{PauCB}. However, beware that the versions for groups are not quite special cases of the versions for C$\sp*$-algebras, because the notions of `positivity' in their hypotheses are not known to match beforehand. See~\cite{Oza13,Helt02} for related discussion and the connection to generalizations of Hilbert's 17th problem. 

These results establish a bijection between positive definite maps from $\G$ to $\rmM_k$ and the space $\B(C^\ast \G,\rmM_k)_+$.  Under this bijection, the weak$^\ast$ topology on $\B(C^\ast \G,\rmM_k)_+$ corresponds to the usual weak$^\ast$ topology of $\ell^\infty(\G;\rmM_k)$, and when restricted to bounded subsets it coincides with the topology of pointwise convergence for functions on $\G$.  We carry all of our previous notation and terminology for completely positive maps such as `assocation', `dilations', and `types' over to this setting in the obvious way.

When working with a countable group $\G$, we often write $\S_k(\G)$ instead of $\S_k(C^\ast \G)$ and regard it as a space of maps on $\G$ itself, and similarly with $\S^\rm{u}_k(\G)$.  Both of these are convex subsets of $\ell^\infty(\G;\rmM_k)$, and compact by the Banach--Alaoglu theorem.

A state on $C^\ast \G$ is tracial if and only if it arises from a \textbf{character} of $\G$, meaning a normalized positive definite function $\chi$ that satisfies
\begin{equation}\label{eq:char}
	\chi(g^{-1}hg) = \chi(h) \qquad (g,h \in \G).
\end{equation}

\begin{ex}\label{ex:regular}
	On any countable group $\G$, the function $1_{\{e\}}$ is positive definite. It is associated to the left regular representation on $\ell^2(\G)$ by the vector $\delta_e$.  It is called the \textbf{regular character} of $\G$.  If $\chi$ is the regular character, then $\chi \otimes I_k$ is the \textbf{regular} $\rmM_k$-valued positive definite function.  
\end{ex}

\begin{ex}\label{ex:subgroup}
	If $H$ is a subgroup of $\G$, then $1_H$ is positive definite.  It is associated to the quasi-regular representation of $\G$ on $\ell^2(\G/H)$ by the function $\delta_{eH}$.  It is a character if and only if $N$ is normal in $\G$.  These examples are called \textbf{quasi-regular} characters. \qed 
\end{ex}

For a free group, we discuss another simple family of positive definite functions with more diverse properties in Section~\ref{sec:Haa} below.

\subsection*{\emph{Notes and further references}}

Thorough accounts of unitary group representations and their relation to group C*-algebras can be found in~\cite[Apps. A--C]{BekdelaHarValKaz},~\cite[Chaps. 1 and 2]{BekdelaHarUDC}, or~\cite[Chap. VI]{DorFelRepI}.

Group C$\sp*$-algebras offer a diverse source of examples of C$\sp*$-algebras with many different properties.  Several of these can be found in~\cite{BroOza08},~\cite{Oza13}, or~\cite{PisOST}, for example.

Besides the regular and quasi-regular representations, some research has investigated in general which representations generate a von Neumann algebra of operators that admits a normal tracial functional.  Some of the results are recounted in~\cite[Chaps. 10--12]{BekdelaHarUDC}. It would be interesting to know whether these representations occupy any special place in the theory of AP entropy that we develop below.

\chapter{Approximate equivalence and the strong-quotient topology}\label{chap:approx}

In this chapter we recall the notions of approximate association for maps and approximate equivalence of representations, and study natural topologies on the space of approximate equivalence classes.  Throughout this chapter, unless specified otherwise, $\A$ is a separable unital C*-algebra, and all representations of $\A$ are also separable.

\section{Approximate association and typical tuples}

We can classify tuples of vectors in a representation according to their types.

\begin{dfn}\label{dfn:typical}
For any positive integer $k$ and subset $O$ of $\B(\A,\rmM_k)$, let
\[ \X(\pi,O) := \big\{[v_1,\dots,v_k]^\rm{T} \in H_\pi^{\oplus k} :\ \Phi^\pi_{v_1,\dots,v_k} \in O \big\}. \]
The elements of $\X(\pi,O)$ are the \textbf{$O$-typical} tuples of the representation $\pi$.
\end{dfn}

We often use Definition~\ref{dfn:typical} when $O$ is a small neighbourhood around a given `target' completely positive map $\phi$.  In this case we may write informally about elements of $\X(\pi,O)$ as `approximately $\phi$-typical vectors'.  This resembles the use of terms such as `microstate' or `good model' in free probability and the study of sofic entropy in ergodic theory (compare~\cite[Sec. 2.3]{Bowen--survey}, for example).

For a group C$\sp*$-algebra $C^\ast\G$, completely positive maps to $\rmM_k$ can be identified with positive definite maps from $\G$ itself to $\rmM_k$, as explained in~\ref{sec:ctble-groups}.  We modify our notation accordingly by writing $\X(\pi,O)$ when $\pi$ is any representation of $\G$ and $O$ is any set of functions from $\G$ to $\rmM_k$.  All the general properties established in the rest of this section carry over without change.

As a vector space topology, the weak$^\ast$ topology on $\B(\A,\rmM_k)$ also defines a canonical uniform structure on this space (see, for instance,~\cite[Sec. 8.1]{Eng89} for the basics of uniform structures).  Restricted to any bounded subset of $\B(\A,\rmM_k)$, this uniform structure is generated by any translation-invariant metric compatible with the topology.  We need this uniform structure at a few points through the following lemma:

\begin{lem}\label{lem:unif-cts}
	For any $\pi$ and $k$, the type map
	\[H_\pi^{\oplus k} \mapsto \B(\A,\rmM_k):[v_1,\dots,v_k]^\rm{T}\mapsto \Phi^\pi_{v_1,\dots,v_k}\]
	is continuous, and uniformly continuous on any bounded subset of $H_\pi^{\oplus k}$.
\end{lem}

\begin{proof}
These properties are elementary for the inner product map $H_\pi\times H_\pi\to \bbC$.  They follow for types by arguing pointwise for each $i$, $j$, and $a \in \A$.
\end{proof}

We also need a notation for the images of the maps in Lemma~\ref{lem:unif-cts} when we restrict to orthonormal tuples.

\begin{dfn}\label{dfn:type-set}
For any representation $\pi$, we set
\[\S_k(\pi) := \big\{\Phi^\pi_{v_1,\dots,v_k}:\ v_1,\dots,v_k \in H_\pi\ \hbox{and}\ \|v_1\|^2 + \cdots + \|v_k\|^2 = k\big\}\]
and
\[\S_k^{\rm{u}}(\pi) := \big\{\Phi^\pi_{v_1,\dots,v_k}:\ v_1,\dots,v_k \in H_\pi\ \hbox{are orthonormal}\big\}.\]
These are the subsets of elements of $\S_k(\A)$ and $\S_k^{\rm{u}}(\A)$ that are associated to $\pi$.
\end{dfn}

For any representation $\pi$ and subset $O$ of $\S_k(\A)$, observe that
$O$ meets $\S_k(\pi)$ if and only if $\X(\pi,O) \ne \emptyset$.

Now let $\pi$ be a separable representation and $\phi \in \B(\A,\rmM_k)_+$. We say that $\phi$ is \textbf{approximately associated} to $\pi$ if $\phi$ lies in the closure in $\B(\A,\rmM_k)_+$ of the set of maps associated to $\pi$.  By Lemma~\ref{lem:lcsc} and a diagonal argument, this holds if and only if there is a sequence of $k$-tuples $v_{n,1}$, \dots, $v_{n,k}$ in $H_\pi$ such that
\[\Phi^\pi_{v_{n,1},\dots,v_{n,k}}(a) \to \phi(a) \qquad \hbox{as}\ n\to\infty \qquad \hbox{for every}\ a\in \A.\]

If $\phi$ is normalized (respectively, unital), and the sequence of tuples $v_{n,1}$, \dots, $v_{n,k}$ witnesses the convergence above, then we can always perturb those tuples slightly so that they are normalized (respectively, orthonormal).  Compare, for instance, the remarks that precede~\cite[Thm. 4]{Arv77} or parts (iv) and (v) of~\cite[Prop. 2.2]{Kec05}.  It follows that an element of $\S_k(\A)$ (respectively, $\S_k^{\rm{u}}(\A)$) is approximately associated to $\pi$ if and only if it lies in the closure $\ol{\S_k(\pi)}$ (respectively, $\ol{\S_k^{\rm{u}}(\pi)}$).

The term `approximately associated' is not standard.  I have chosen it to make contact with `approximate equivalence' in the next section.  See also Remark~\ref{rmk:approx-contain} below.

Alternatively, a unital completely positive map $\phi \in\S_k(\A)$ is \textbf{weakly associated} to a representation $\pi$ if it can be approximated by convex combinations of elements of $\S_k(\pi)$: that is, if $\phi \in \ol{\rm{conv}}\,\S_k(\pi)$.  This definition is conventional only in the case $k=1$, but no new ideas are needed if we allow larger values of $k$.  Weak association is the basis for Godemont and Fell's relation of weak containment for representations: see~\cite{God48},~\cite{Fel60} or~\cite[Sec. 3.4]{Dix--Cstar}.  As a result, weak association has a somewhat bigger place in the literature than approximate association

We can relate approximate and weak association to each other in terms of inflations.

\begin{lem}\label{lem:approx-and-weak-dilation}
For any positive integer $k$ and representation $\pi$ we have
	\[\ol{\rm{conv}}\,\S_k(\pi) = \ol{\bigcup_{n\ge 1}\S_k(\pi^{\oplus n})} = \ol{\S_k(\pi^{\oplus \infty})}.\]
	In particular, a map $\phi \in \S_k(\A)$ is weakly associated to $\pi$ if and only if it is approximately associated to $\pi^{\oplus \infty}$.
	\end{lem}

See, for instance,~\cite[Prop. 2.2]{Kec05}, where this result is formulated for unitary group representations, but the proof given works with only cosmetic changes in the general case.  We sketch it for completeness. 

\begin{proof}
	For any positive integer $n$, the decomposition of $\pi^{\oplus n}$ into summands expresses any element of $\S_k(\pi^{\oplus n})$ as a convex combination of at most $n$ elements of $\S_k(\pi)$.  Letting $n$ increase and then taking closures, this gives the left-hand equality.  The right-hand equality holds by approximating vectors in $H_\pi^{\oplus \infty}$ by vectors with only finitely many nonzero coordinates.
	\end{proof}

If $\phi$ is an extreme point of $\S_k(\A)$, so its GNS representation is irreducible, then allowing a convex hull makes no difference.  Such a $\phi$ is weakly associated to another representation $\pi$ if and only if it is approximately associated; see, for instance, the proof of~\cite[Thm. 3.4.10]{Dix--Cstar}.

If a map $\phi$ is associated to a representation $\pi$, then it is clearly also approximately associated.  The next lemma is a robust version of this fact.

\begin{lem}\label{lem:typ-trans}
Let $\phi \in \B(\A,\rmM_k)_+$ and $\psi \in \B(\A,\rmM_\ell)_+$, and assume that $\psi$ is approximately associated to $\pi_\phi$.  Then for any neighbourhood $U$ of $\psi$ there is a neighbourhood $V$ of $\phi$ such that
\[\X(\pi,V)\ne \emptyset \qquad \Rightarrow \qquad \X(\pi,U) \ne \emptyset\]
	for any representation $\pi$. \qed
\end{lem}

%

See~\cite[Cor. 5.4]{APE4}.

On the other hand, if $\phi$ and $\psi$ are normalized and disjoint, then we have a robust version of the fact that $\rm{diag}(\phi,\psi)$ is their only joining.

\begin{lem}\label{lem:near-diag}
Let $\phi \in \S_k(\A)$ and $\psi \in \S_\ell(\A)$, and assume that they are disjoint.  Then for any neighbourhood $W$ of $\rm{diag}(\phi,\psi)$ there are neighbourhoods $U$ of $\phi$ in $\S_k(\A)$ and $V$ of $\psi$ in $\S_\ell(\A)$ such that
\[[\ \X(\pi,U)\ne \emptyset \quad \hbox{and}\quad \X(\pi,V)\ne \emptyset \ ]\qquad \Rightarrow \qquad \X(\pi,W) \ne \emptyset\]
	for any representation $\pi$. \qed
\end{lem}

See~\cite[Lem. 5.5]{APE4}. 

If $\phi:\A\to \rmM_k$ is completely positive, $Q \in \rmM_k$, and we define $\psi:\A\to \rmM_k$ as in equation~\eqref{eq:psi-Q-phi}, then we obtain the following more explicit relationship between the sets of approximately typical vectors for $\phi$ and for $\psi$.

\begin{lem}\label{lem:lin-maps}
Suppose that $\phi \in \B(\A,\rmM_k)_+$, and let
\[\psi(a) := (Q^\rm{T})^\ast \phi(a) Q^\rm{T} \qquad (a \in \A)\]
for some $Q \in \rmM_{\ell,k}$.  For any neighbourhood $O$ of $\psi$ there is a neighbourhood $U$ of $\phi$ such that
\begin{equation}\label{eq:lin-maps-Q}
(I_{H_\pi}\otimes Q)[\X(\pi,U)] \subset \X(\pi,O)
\end{equation}
for any representation $\pi$, recalling the notation from~\eqref{eq:pre-lin-maps-Q}. \qed
\end{lem}

See~\cite[Lem. 5.3]{APE4}.


\section{Approximate equivalence and approximate containment}\label{sec:approx-equiv}

\begin{dfn}\label{dfn:approx-equiv}
Two separable representations $\pi$ and $\rho$ are \textbf{approximately equivalent}, written $\pi \simeq_\rm{a} \rho$, if there are unitary isomorphisms $U_1,U_2,\dots:H_\pi \to H_\rho$ such that
\[\|\pi(a) - U_n^\ast \rho(a) U_n\| \to 1 \qquad \hbox{as} \ n\to\infty \qquad \hbox{for every}\ a \in \A.\]

Similarly, $\pi$ is \textbf{approximately contained} in $\rho$, written $\pi \lesssim_{\rm{a}} \rho$, if there are unitary embeddings $V_1,V_2,\dots:H_\pi \to H_\rho$ such that
\[\|\pi(a) - V_n^\ast \rho(a) V_n\| \to 1 \qquad \hbox{as} \ n\to\infty \qquad \hbox{for every}\ a \in \A.\]
\end{dfn}

The study of approximate equivalence originated in Voiculescu's paper~\cite{Voi76}, inspired by older questions about single operators and building on parts of Glimm's work in~\cite{Gli60}.  Arveson offered a revealing alternative approach in~\cite{Arv77}. Hadwin extended the theory in several directions, including to approximate containment in~\cite{Had81} and also to `approximate compressions' of a representation in~\cite{Had87}.  Here we state only the unital cases of the facts we need, which avoid some complications.


\begin{thm}\label{thm:get-sum}
	If $\rho$ and $\pi$ are separable unital representations such that $\pi(a)$ is zero whenever $\rho(a)$ is compact, then $\rho\oplus \pi \simeq_{\rm{a}} \rho$. \qed
\end{thm}
See~\cite[Thm. 1.3]{Voi76},~\cite[Cor. 2 of Thm. 4]{Arv77}, or~\cite[Cor. 41.13]{ConOT}.  Building on Theorem~\ref{thm:get-sum}, the main result in~\cite{Voi76} is a list of apparently quite different conditions that are all actually the same as approximate equivalence.  The next theorem records a few of these that we need.  (Proofs of the theorem usually also go via other conditions not listed here as well.) 

\begin{thm}\label{thm:Voi-main}
For separable unital representations $\pi$ and $\rho$, the following are equivalent:
\begin{itemize}
\item[i.] $\pi \simeq_{\rm{a}}\rho$;
\item[ii.] we have $\ker \pi = \ker \rho$, and $\pi^{-1}[\mathfrak{K}(H_\pi)] = \rho^{-1}[\mathfrak{K}(H_\rho)]$, and the restrictions of $\pi$ and $\rho$ to this latter ideal are unitarily equivalent;
\item[iii.] we have $\rm{rank}\,\pi(a) = \rm{rank}\,\rho(a)$ for every $a\in \A$ (regarding ranks as values in $\{0,1,2,\dots\}\cup \{\infty\}$);
	\item[iv.] for some subset $S$ of $\A$ which generates a dense $\ast$-subalgebra, there is a sequence of unitary operators $U_n:H_\pi \to H_\rho$ such that
	\[\|\pi(a) - U_n^\ast \rho(a) U_n\| \to 0 \quad \hbox{for every}\ a\in S.\] \qed
	\end{itemize}
\end{thm}

According to Theorem~\ref{thm:Voi-main}(ii), the relation $\pi \simeq_{\rm{a}} \rho$ is stronger than weak equivalence, but only through the representations of the common ideal $\pi^{-1}[\mathfrak{K}(H_\pi)]$, which consist of compact operators.

The equivalence of (i) and (ii) in Theorem~\ref{thm:Voi-main} is covered by~\cite[Thm. 1.5]{Voi76} or~\cite[Thm. 5]{Arv77}, or see~\cite[Thm. 41.12(b,d,e)]{ConOT} for a textbook account.  The equivalence with condition (iii) is~\cite[Thm. 2.5]{Had81}.  Finally, condition (iv) is equivalent to (i) because, for any unitary operators $U_n:H_\pi \to H_\rho$, the set
\[\big\{a \in \A:\ \|\pi(a) - U_n^\ast \rho(a) U_n\| \to 0\big\}\]
is a closed $\ast$-subalgebra of $\A$.

\begin{ex}\label{ex:one-op1}
Let $L > 0$, let $\A := C[-L,L]$, and let $a$ be the identity function on $[-L,L]$ regarded as an element of $\A$.  Then any unital representation $\pi$ of $\A$ is specified uniquely by $\pi(A)$, which is an arbitrary self-adjoint operator $\pi(A)$ of norm at most $1$.  In this example, approximate equivalence of representations reduces to approximate equivalence of such operators, which was studied earlier in connection with the Weyl--von Neumann--Berg theorem.  According to the resulting special case of Theorem~\ref{thm:Voi-main}, two such operators are approximately equivalent if (i) they have the same spectra, and (ii) each of their isolated eigenvalues has the same multiplicity. See, for example,~\cite[Thm 39.8 and Prop 39.10]{ConOT}. \fin
\end{ex}

For some applications, another characterization of approximate equivalence is important: if $\pi \simeq_{\rm{a}} \rho$ and both are unital, then one can also choose the unitary isomorphisms in Definition~\ref{dfn:approx-equiv} so that $\pi(a) - U_n^\ast\rho(a)U_n$ is compact for every $a \in \A$.  This is also covered in~\cite{Voi76} or~\cite{Arv77}.   We do not need this fact directly in the present book, but some similar conclusions appear naturally in Section~\ref{sec:tempered} below.  

Hadwin introduced approximate containment in~\cite[Sec. 5]{Had81} under the term `approximate subrepresentations'.  Like approximate equivalence, approximate containment for unital representations has many alternative characterizations (indeed, our initial choice in Definition~\ref{dfn:approx-equiv} is different from Hadwin's).  From these we extract the following.

\begin{thm}\label{thm:Had-main}
For separable unital representations $\pi$ and $\rho$, the following are equivalent:
\begin{itemize}
\item[i.] $\pi \lesssim_{\rm{a}}\rho$;
\item[ii.] we have $\rm{rank}\,\pi(a) \le \rm{rank}\,\rho(a)$ for every $a\in \A$;
\item[iii.] there is a representation $\pi_1$ such that $\pi \oplus \pi_1 \simeq_{\rm{a}} \rho$.
\end{itemize}
\end{thm}

This is covered by~\cite[Thm. 5.1]{Had81}.  The implication from property (ii) to property (iii) takes the most work, and Hadwin explains it in terms of the `approximate multiplicities' of irreducible sub-representations, which we do not define here.  Actually, the proof of property (iii) in~\cite{Had81} really gives representations $\pi' \simeq_{\rm{a}} \pi$ and $\rho' \simeq_{\rm{a}} \rho$ such that $\pi' \lesssim \rho'$.  However, having found these, we can identify $\rho'$ with $\pi' \oplus \pi_1$ and then swap out $\pi'$ for $\pi$ in this direct sum to conclude that $\pi \oplus \pi_1 \simeq_{\rm{a}} \rho' \simeq_{\rm{a}} \rho$. Alternatively, Theorem~\ref{thm:Had-main}(iii) can be proved along the lines of~\cite[Thm. 4 and Prop. 6]{AbeEle11}; we discuss this reference further in the next section.

The relations of approximate equivalence and approximate containment are both transitive.  Also, two separable unital representations $\pi$ and $\rho$ are approximately equivalent if and only if each is approximately contained in the other, for example by Theorem~\ref{thm:Voi-main}(ii) and Theorem~\ref{thm:Had-main}(ii).  Therefore the truth of the assertion ``$\pi \lesssim_{\rm{a}} \rho$'' depends only on the approximate equivalence classes of $\pi$ and $\rho$.  We henceforth regard ``$\simeq_{\rm{a}}$''  or ``$\lesssim_{\rm{a}}$'' as relations between approximate equivalence classes rather than individual separable unital representations whenever this is convenient.

In addition to Theorems~\ref{thm:Voi-main} and~\ref{thm:Had-main}, we can also characterize approximate containment of representations by comparing their sets of approximately associated completely positive maps.

\begin{lem}\label{lem:approx-contain}
Two separable unital representations satisfy $\pi \lesssim_\rm{a} \rho$ if and only if $\ol{\S_k(\pi)}\subset \ol{\S_k(\rho)}$ for all $k$, and $\pi \simeq_\rm{a} \rho$ if and only if $\ol{\S_k(\pi)} =\ol{\S_k(\rho)}$ for all $k$.
	\end{lem}

\begin{proof}
First, let $\pi \lesssim_\rm{a} \rho$, let $(U_n)_{n\ge 1}$ be a sequence of unitary embeddings as in the second part of Definition~\ref{dfn:approx-equiv}, and let $v_1$, \dots $v_k$ in $H_\pi$.  Then
\[\langle \rho(a)U_nv_j,U_nv_i\rangle \to \langle \pi(a)v_j,v_i\rangle \qquad (a \in \A).\]
This shows that $\S_k(\pi)\subset \ol{\S_k(\rho)}$, so taking closures gives the forward implication.

Now assume that $\ol{\S_k(\pi)}\subset \ol{\S_k(\rho)}$ for every $k$.  For any positive integer $k$ and $a \in \A$, we have $\rm{rank}\,\rho(a) \le k$ if and only if $\Det\,\phi(a^\ast a) = 0$ for every $\phi \in \bigcup_{\ell \ge k+1}\ol{\S_\ell(\rho)}$, and similarly for $\pi$.  Therefore $\rm{rank}\,\pi(a) \le \rm{rank}\,\rho(a)$ for every $a\in \A$, and so $\pi \lesssim_{\rm{a}} \rho$ by Theorem~\ref{thm:Had-main}.

This proves the characterization of approximate containment.  The characterization of approximate equivalence follows immediately.
	\end{proof}

For weak containment of separable unital representations, the analog of Lemma~\ref{lem:approx-contain} asserts that $\pi$ is weakly contained in another one $\rho$ if and only if $\ol{\rm{conv}}\,\S_1(\pi) \subset \ol{\rm{conv}}\,\S_1(\rho)$. This leads to the standard characterization of weak containment in terms of inclusion of supports: see~\cite[Thm. 3.4.4]{Dix--Cstar}.  We can also write the support of a general representation $\pi$ as
\begin{align*}
\rm{spt}\,\pi &= \{[\rho] \in \hat{\A}:\ \rho\ \hbox{weakly contained in }\ \pi\}\\
&= \{[\rho] \in \hat{\A}:\ \rho\ \hbox{approximately contained in }\ \pi\}.
\end{align*}
Both formulas follow from the equivalences proved in~\cite[Sec. 3.4]{Dix--Cstar}.  Comparing this with Theorem~\ref{thm:Voi-main} sheds some light on the possible failure of the sets $\ol{\S_k(\rho)}$ to be convex: this can happen, but only if $\rho$ contains some non-zero compact operators, whose ranks must then be respected under approximate equivalence.

\begin{rmk}\label{rmk:approx-contain}
A more recent convention refers to approximate containment as `weak containment in the sense of Zimmer': see, for instance,~\cite{BekdelaHarValKaz,KecGAEGA,AbeEle11}.  Zimmer does indeed use the term `weak containment' in~\cite[Dfn. 7.3.5]{Zim84} for what we call `approximate containment'.  However, in~\cite{Zim84} Zimmer himself attributes this definition to Fell's original papers.  I have not found it there explicitly, but it is certainly suggested by Fell's manipulations with the quotient topology: see, for example,~\cite[Lem. 1.3]{Fel62}.  The book~\cite{DorFelRepI} presents that topology again (calling it `regional' rather than `quotient') in a way that makes the connection clearer, and the preface to that book suggests that it was decades in preparation.  These references effectively define approximate containment directly through the criterion in Lemma~\ref{lem:approx-contain}, and its equivalence with Hadwin's notion does not seem to be so widely known. \fin 
\end{rmk}

\section{Background: hyperspaces and their topologies}\label{sec:Vietoris}

If $T$ is a topological space, then we write $\calK(T)$ for its \textbf{hyperspace}: the space of all nonempty compact subsets of $T$.  Unless specified otherwise, this is endowed with the \textbf{Vietoris topology}.  This is generated by the \textbf{Vietoris basic sets}: these are the sets
\begin{multline*}
	\V(U_1,\dots,U_k) := \{K \in \calK(T):\ K\subset U_1\cup\cdots \cup U_k\\ \hbox{and}\ K\cap U_i \ne \emptyset\ \hbox{for}\ i=1,2,\dots,k\}
\end{multline*}
as $U_1$, \dots, $U_k$ range over all open subsets of $T$.

The basic properties of the Vietoris topology can be found in~\cite[Ex.s 2.7.20 and 3.12.27]{Eng89} or~\cite[Subs. I.4.F]{Kec95}.  Here we recall a few that we need without proof.

If $T$ is compact and metrizable, then so is $\calK(T)$.  For example, starting from a compact metric for $T$, a compact metric for $\calK(T)$ is given by the resulting Hausdorff metric~\cite[Ex. 4.5.23]{Eng89}.  The examples in our applications below fall into this class, but often without any particular canonical choice of metric, so we generally use the more abstract description in terms of Vietoris basic sets.  We restrict our attention to this case in the rest of this section.

Because of metrizability, the Vietoris topology on $\calK(T)$ may be described in terms of limits.  In hyperspaces we can describe this in the following explicit way.  Given a sequence $(K_n)_{n\ge 1}$ in $\calK(T)$, its \textbf{topological upper} and \textbf{lower limits} are the sets
\[\rm{T}\limsup_{n\to\infty} K_n := \{t \in T:\ \hbox{every nbhd of $t$ meets $K_n$ for infinitely many $n$}\}\]
and
\begin{align*}
\rm{T}\liminf_{n\to\infty} K_n  &:= \{t \in T:\ \hbox{every nbhd of $t$ meets $K_n$ for all sufficiently large $n$}\}\\
&= \{t \in T:\ t = \lim_n t_n\ \hbox{for some convergent sequence with $t_n \in K_n$}\}.
\end{align*}
These are both closed.  If $t_n \in K_n$ for each $n$, then any subsequential limit of $(t_n)_{n\ge 1}$ lies in $\rm{T}\limsup_n K_n$, so compactness implies that this set is nonempty.  Finally, the sequence $(K_n)_{n\ge 1}$ converges in the Vietoris topology if and only if
\begin{equation}\label{eq:Vietoris-conv}
\rm{T}\limsup_n K_n = \rm{T}\liminf_n K_n,
\end{equation}
and then this common set is the limit of the sequence.

An alternative topology on $\calK(T)$ is the \textbf{lower Vietoris topology}: see, for instance,~\cite[App.]{Mic51}. This is generated by the basic open sets of the form
\[	\V'(U_1,\dots,U_k) := \{K \in \calK(T):\ K\cap U_i \ne \emptyset\ \hbox{for}\ i=1,2,\dots,k\}\]
for some open subsets $U_1$, \dots, $U_k$ of $T$. In the previous notation, this is simply equal to $\V(U_1,\dots,U_k,T)$, so the lower Vietoris topology is weaker than the Vietoris topology.  Like the Vietoris topology, it is second countable in case $T$ is compact and metrizable.

However, in all but degenerate cases, the lower Vietoris topology is much weaker than the Vietoris topology, and also less well-behaved.  Indeed, if $T$ is a Hausdorff space with at least two points, then the lower Vietoris topology is not T$_1$, because for any $K \in \calK(T)$ the lower-Vietoris closure of the singleton $\{K\}$ is the whole of $\calK(K)$, regarded as a subset of $\calK(T)$.  On the other hand, since the lower Vietoris topology is weaker than the Vietoris topology, any subset of $\calK(T)$ that is Vietoris-closed is still quasi-compact in the lower Vietoris topology (meaning that every open cover has a finite sub-cover, but the space fails to be Hausdorff~\cite[p132]{Eng89}).

\section{Topologies on approximate equivalence classes}\label{sec:q-and-sq}

\subsection{The space of approximate equivalence classes}

Choose a list of all separable Hilbert spaces up to isomorphism, say $\bbC^{\oplus n}$ for $n\ge 0$ and $\ell^2(\bbN)$.  Let $\Rep(\A)$ be the set of all \emph{unital} representations of $\A$ on any of these spaces, let $\Rep^\sim(\A)$ be the resulting set of unitary equivalence classes, and $\Rep^\sim_{\rm{a}}(\A)$ be the set of approximate equivalence classes.  Similarly, if $\G$ is a countable group, then we write $\Rep^\sim_{\rm{a}}(\G)$ for the space of approximate equivalence classes of unitary representations of $\G$.

%
%

If $\pi \in \Rep(\A)$, then by Lemma~\ref{lem:approx-contain} each of the sets $\ol{\S_k(\pi)}$ depends only on the class of $\pi$ in $\Rep^\sim_{\rm{a}}(\A)$.  As in~\cite{APE4}, we call $\ol{\S_k(\pi)}$ the \textbf{$k$-summary} of $\pi$, and we sometimes regard it instead as a function of a class in $\Rep^\sim_{\rm{a}}(\A)$.  Lemma~\ref{lem:approx-contain} tells us that the equivalence class $\t{\pi} \in \Rep^\sim_{\rm{a}}(\A)$ of a separable unital representation $\pi$ is uniquely parameterized by the sequence
\[\ol{\S_\bullet(\pi)} := \big(\ol{\S_1(\pi)},\ol{\S_2(\pi)},\dots\big) \in \prod_{k\ge 1} \calK(\S_k(\A)).\]
We call this sequence the \textbf{summary} of $\pi$.  Since it depends only $\t{\pi}$, we may also write it as $\ol{\S_\bullet(\t{\pi})}$.

We next identify the subset of all summaries in this product space more explicitly.

\begin{prop}\label{prop:approx-equiv}
A sequence $Z_\bullet = (Z_1,Z_2,\dots)$ in $\prod_{k\ge 1}\calK(\S_k(\A))$ is the summary of some separable unital representation $\pi$ if and only if it has these two closure properties:
	\begin{itemize}
		\item[i.] if $\phi \in Z_k$ then $\S_\ell(\pi_\phi) \subset Z_\ell$ for every $\ell$;
		\item[ii.] if $\phi \in Z_k$ and $\psi \in Z_\ell$ then some joining of them lies in $Z_{k+\ell}$.
	\end{itemize}
\end{prop}

\begin{rmk}
Using Proposition~\ref{prop:Leb-reps-2}, one can prove that, in the presence of condition (i) above, condition (ii) is implied by this weaker version:
\begin{itemize}
\item[ii$^\prime$.] if $\phi \in Z_k$ and $\psi \in Z_\ell$ are mutually singular, then $\rm{diag}(\phi,\psi) \in Z_{k+\ell}$.
\end{itemize}
However, the slightly more general formulation in (ii) seems to be more convenient for our work later. \fin
\end{rmk}

%
%

\begin{proof}
\emph{($\Rightarrow$).}\quad Let $Z_\bullet := \ol{\S_\bullet(\pi)}$.  It satisfies property (i) by the transitivity of approximation association (see Lemma~\ref{lem:typ-trans}).  To check property (ii), suppose that $\phi \in \ol{\S_k(\pi)}$ and $\psi \in \ol{\S_\ell(\pi)}$. Then there are sequences of $k$-tuples $X_n$ and $\ell$-tuples $Y_n$ in $H_\pi$ such that
\[\Phi^\pi_{X_n} \to \phi \qquad \hbox{and} \qquad \Phi^\pi_{Y_n} \to \psi \qquad \hbox{as}\ n\to\infty.\]
The sequence of combined types $\Phi^\pi_{[X_n,Y_n]}$ is normalized by~\eqref{eq:still-normalized}, and hence norm-bounded by~\eqref{eq:norm-cts}. It therefore has subsequential limits by the Banach--Alaoglu theorem and the metrizability of $\S_{k+\ell}(\A)$.  Any of these is a joining of $\phi$ and $\psi$ that lies in $\ol{\S_{k+\ell}(\pi)}$.

\vspace{7pt}

\emph{($\Leftarrow$).}\quad Given $Z_\bullet$, we must construct a suitable representation $\pi$.

Let $S_k$ be a countable dense subset of $Z_k$ for each $k$, and let $S$ be their union.  Enumerate $S$ as $\{\phi_1,\phi_2,\dots\}$, and let $\rho_i = \pi_{\phi_i}$ for each $i$.  As minimal dilations, these are all unital.

For notational purposes, let $\pi_0$ be a trivial representation. We now define by recursion a sequence of representations $\pi_i$, $i\ge 1$, that has the following properties:
\begin{itemize}
\item $\pi_{i+1}$ contains both $\pi_i$ and $\rho_{i+1}$ for each $i\ge 0$;
\item each $\pi_i$ has a finite cyclic tuple;
\item $Z_\bullet$ contains $\ol{\S_\bullet(\pi_i)}$ for each $i$.
\end{itemize}
To begin the recursion, let $\pi_1 := \rho_1$.  Now suppose that $\pi_1$, \dots, $\pi_i$ have already been constructed.  Since $\pi_i$ and $\rho_{i+1}$ both have finite cyclic tuples, and since \[Z_\bullet \supset \ol{\S_\bullet(\pi_i)} \cup \ol{\S_\bullet(\rho_{i+1})},\]
applying property (ii) and then property (i) gives a new unital representation $\pi_{i+1}$ with a finite cyclic tuple that contains both $\pi_i$ and $\rho_{i+1}$ and satisfies $Z_\bullet \supset \ol{\S_\bullet(\pi_{i+1})}$.  This continues the recursion.

To finish, let $\pi^{M_i}_{i+1}$ be a subrepresentation of $\pi_{i+1}$ which is equivalent to $\pi_i$, and define
\[\pi:= \pi_1 \oplus \pi_2^{M_1^\perp}\oplus \pi_3^{M_2^\perp}\oplus \cdots.\]
This contains an increasing sequence of subrepresentations equivalent to $\pi_1$, $\pi_2$, \dots, and their union is dense in $H_\pi$.  It follows by Lemma~\ref{lem:unif-cts} that
\[\ol{\S_k(\pi)} = \ol{\bigcup_{i\ge 1}\S_k(\pi_i)} \subset Z_k \qquad \hbox{for each}\ k.\]
On the other hand, for every $k$, we also have $\ol{\S_k(\pi)} \supset \ol{\S_k(\rho_i)}$ for each $i$, and hence $\ol{\S_k(\pi)} \supset Z_k$ by the density of $S_k$.
\end{proof}

\subsection{The lower Vietoris and quotient topologies}

We can endow the product space $\prod_{k\ge 1}\calK(\S_k(\A))$ with the product of either the Vietoris topologies or the lower Vietoris topologies.  By a slight abuse of nomenclature, we also refer to this product topology itself as `the' lower Vietoris or Vietoris topology on the product space.  Now we pull this topology back to $\Rep^\sim_{\rm{a}}$ through the map
\begin{equation}\label{eq:lower-Vietoris-to-Fell}
\t{\pi} \mapsto \ol{\S_\bullet(\pi)},
\end{equation}
where $\t{\pi}$ is the approximate equivalence class of a separable unital representation $\pi$.

We consider the pull-back of the lower Vietoris topology in the present subsection, and the pull-back of the Vietoris topology in the next subsection.  We call this pulled-back topology the \textbf{quotient topology} on $\Rep^\sim_{\rm{a}}(\A)$. Like the lower Vietoris topology itself, the quotient topology is second countable, but it can have poor separation properties.

A similar pull-back of the lower Vietoris topology is one standard way of defining the Fell topology on $\hat{\A}$ from~\cite{Fel62}, which generalized Godemont's topology on $\hat{\G}$ for a locally compact group $\G$ from~\cite{God48}.  It has many other equivalent characterizations: see~\cite[Chap. 3]{Dix--Cstar} (or~\cite[sec. 7.2]{FolAHA} for a quicker overview in the case of a group C$\sp*$-algebra).

However, these characterization can diverge outside the class of irreducible representations.  Fell studied two of these possibilities in~\cite[Sec. 3]{Fel60b} and~\cite{Fel62}, calling them the `quotient' and `inner hull-kernel' topologies.  Modern practice often assigns Fell's name to the second of these, but here we need the first one.  Fell defines it as the quotient topology on $\Rep^\sim_{\rm{a}}(\A)$ of the pointwise strong operator topology on $\Rep(\A)$.  It coincides with our pull-back of the lower Vietoris topology, so we have adopted Fell's name for it.  We do not work directly with Fell's original definition in this book, but it can be shown to agree with ours by slightly adapting the arguments in~\cite[Sec. 1]{Fel62} (or with a little more work, those in~\cite[Sec. 3.5]{Dix--Cstar}).  We often indicate this topology by the single letter `q' in a phrase such as `q-open'.  One could also study Fell's inner hull-kernel topology in an analogous way: one would simply replace the sets $\ol{\S_k(\pi)}$ in~\eqref{eq:lower-Vietoris-to-Fell} with their convex hulls.  See~\cite[Thm 2.1]{Fel62}, and also compare with Lemma~\ref{lem:approx-and-weak-dilation}.

Since the quotient topology on $\Rep^\sim_{\rm{a}}(\A)$ is second countable, it can be characterized in terms of convergence of sequences.  Unpacking the definition of the lower Vietoris topology, we find that a sequence $(\t{\pi_n})_{n \ge 1}$ of approximate equivalence classes in $\Rep^\sim_{\rm{a}}(\A)$ q-converges to another class $\t{\pi}$ if and only if
\[\rm{T}\liminf_n \ol{\S_k(\pi_n)}\supset \ol{\S_k(\pi)} \qquad \hbox{for every}\ k=1,2,\dots.\]
So the class in $\Rep^\sim_{\rm{a}}(\A)$ of any representation whose $k$-summaries are all sufficiently small is a q-limit of $(\t{\pi_n})_{n\ge 1}$, illustrating how this topology often fails to be T$_1$.

We can also pull back the lower Vietoris basic sets from Section~\ref{sec:Vietoris} to obtain a base for the quotient topology on $\Rep^\sim_{\rm{a}}(\A)$.  It then turns out that the properties from Proposition~\ref{prop:approx-equiv} enable a simplication to this base.

\begin{lem}\label{lem:lower-Vietoris-simplify}
	For any positive integer $k$ and open subset $U$ of $\S_k(\A)$, let
	\[\cal{W}(k,U) := \big\{\t{\pi} \in \Rep^\sim_{\rm{a}}(\A):\ \S_k(\pi)\ \hbox{meets}\ U\big\}.\]
The collection of all subsets of $\Rep^\sim_{\rm{a}}(\A)$ of this form is a base for the quotient topology.
\end{lem}


\begin{proof}
As we vary $k$ and $U$, the definitions that the collection of sets of this form is a sub-base.  To show that it is actually a base, consider two such sets, say $\cal{W}(k,U)$ and $\cal{W}(\ell,V)$, and let $\pi$ be a separable unital representation whose approximate equivalence class lies in $\cal{W}(k,U) \cap \cal{W}(\ell,V)$.  Choose tuples $X$ and $Y$ in $H_\pi$ so that $\Phi^\pi_X \in U$ and $\Phi^\pi_Y \in V$.  Now Lemma~\ref{lem:typ-trans} gives a neighbourhood $W$ of $\Phi^\pi_{[X,Y]}$ such that
\[\cal{W}(k+\ell,W)\subset \cal{W}(k,U) \cap \cal{W}(\ell,V),\]
and the class of $\pi$ still lies in the set on the left.
\end{proof}

The next result is a kind of `partial reverse' to Lemma~\ref{lem:lower-Vietoris-simplify}.  We need it later during the proof of Theorem~\ref{mainthm:sq-LDP}.

\begin{lem}\label{lem:two-q-nbhds}
Let $\pi$ and $\rho$ be separable unital representations, and let $\pi' \lesssim \pi$ and $\rho' \lesssim \rho$ be sub-representations produced by Proposition~\ref{prop:Leb-reps-2}.  If $W$ is any q-neighbourhood of $(\pi'\oplus \rho')^\sim$, then there are q-neighbourhoods $O$ of $\t{\pi}$ and $U$ of $\t{\rho}$ such that $O\cap U \subset W$.
\end{lem}

\begin{proof}
Since $\pi' \lesssim \pi$, any q-neighbourhood of $\t{\pi'}$ is also a q-neighbourhood of $\t{\pi}$, and similarly for $\rho'$ and $\rho$.  It therefore suffices to find q-neighbourhoods $O$ of $\t{\pi'}$ and $U$ of $\t{\rho'}$.  Having observed this, we may henceforth assume that $\pi \spoon \rho$, $\pi' = \pi$, and $\rho' = \rho$.

Next, by Lemma~\ref{lem:lower-Vietoris-simplify}, we may assume that $W = \cal{W}(k,W_0)$, where $k$ is a positive integer and $W_0$ is an open subset of $\S_k(\A)$ which meets $\S_k(\pi \oplus \rho)$.  This means that $W_0$ contains $\Phi^{\pi\oplus \rho}_{x_1,\dots,x_k}$ for some vectors $x_i = (y_i,z_i) \in H_\pi \oplus H_\rho$, $i=1,2,\dots,k$, such that the values
\[a^2 := \sum_{i=1}^k\|y_i\|^2 \qquad \hbox{and} \qquad b^2 := \sum_{i=1}^k \|z_i\|^2\]
satisfy $a^2 + b^2 = k^2$.

Assume first that $a$ and $b$ are strictly positive, define $y'_i:= ky_i/a$ and $z'_i := kz_i/a$ for each $i$, and let $\phi := \Phi^\pi_{y'_1,\dots,y'_k}$ and $\psi := \Phi^\rho_{z'_1,\dots,z'_k}$. Then $\Phi^{\pi \oplus \rho}_{x_1,\dots,x_k}$ is associated to the GNS representation of $\rm{diag}(\phi,\psi)$, and $\phi$ and $\psi$ are disjoint.  Therefore Lemma~\ref{lem:near-diag} gives neighbourhoods $O_0$ of $\phi$ and $U_0$ of $\psi$ in $\S_k(\A)$ such that any $\theta \in \S_{2k}(\A)$ which satisfies
\[\theta[\{1,\dots,k\}] \in O_0 \qquad \hbox{and} \qquad \theta[\{k+1,\dots,2k\}] \in U_0\]
must also lie in $W_0$.  Finally, letting $O := \cal{W}(k,O_0)$ and $U := \cal{W}(k,U_0)$ as in Lemma~\ref{lem:lower-Vietoris-simplify}, this completes the proof.

The arguments when $a = 0$ or $b=0$ are degenerate versions of that above.
\end{proof}

\subsection{The Vietoris and strong-quotient topologies}

We now consider the topology on $\Rep^\sim_{\rm{a}}(\A)$ pulled back from the product of Vietoris topologies through the map in~\eqref{eq:lower-Vietoris-to-Fell}.  As far as I know, this topology on approximate equivalence classes of representations was first studied by Ab\'ert and Elek in~\cite{AbeEle11}; see also~\cite{TucDro15} and~\cite{BurKec20}.  Those papers all focus on an analogous construction for measure-preserving systems, but Ab\'ert and Elek indicate the variants for unitary representations as well.  They phrase their work in terms of Hausdorff metrics rather than Vietoris topologies.  Many of the terms and conventions we use here appeared for the first time in~\cite[Subs. 5.3]{APE4}.

Since $A$ is separable and unital, each $\S_k(\A)$ is compact and metrizable (see~\cite[Prop. V.5.1]{ConFA}, for example).  Therefore the product topology on $\prod_{k\ge 1}\calK(\S_k(\A))$ is also compact and metrizable.  By Lemma~\ref{lem:approx-contain}, the map in~\eqref{eq:lower-Vietoris-to-Fell} is a bijection from $\Rep^\sim_{\rm{a}}(\A)$ to its image, so the pull-back of the Vietoris topology to $\Rep^\sim_{\rm{a}}(\A)$ is still metrizable. Following~\cite{APE4}, we call this pull-back the \textbf{strong-quotient topology} on $\Rep^\sim_{\rm{a}}(\A)$.   This name locates this topology relative to some other standard notions for representations: see Subsection~\ref{subs:other-tops} below.  We often indicate it by the prefix `sq-', as in `sq-open set'.  Since the strong-quotient topology is metrizable, it may be described in terms of sequences: in practice, this means we apply the criterion in~\eqref{eq:Vietoris-conv} to sequences of $k$-summaries for each $k$.


The next property is less immediate.

\begin{lem}\label{lem:catalog-closed}
The strong-quotient topology on $\Rep^\sim_{\rm{a}}(\A)$ is compact.
\end{lem}

This can be proved in the same way as~\cite[Thm. 3 and Prop. 7]{AbeEle11} and~\cite[Thm. 5.1]{TucDro15}.  Here we give an alternative proof which avoids their use of ultralimits and depends on Proposition~\ref{prop:approx-equiv}.

\begin{proof}
Let $\cal{Z}$ be the image of $\Rep^\sim_{\rm{a}}(A)$ in $\prod_{k\ge 1}\calK(\S_k(\A))$ under the map in~\eqref{eq:lower-Vietoris-to-Fell}.  This image is characterized by Proposition~\ref{prop:approx-equiv}.  Since that map is a bijection and the Vietoris topologies are compact, it suffices to show that $\cal{Z}$ is a closed subset of that product space.  So now let $(Z_{n,\bullet})_{n\ge 1}$ be a sequence in $\cal{Z}$ that converges to a limit $Z_\bullet$ in the Vietoris topology.  We must show that $Z_\bullet$ still satisfies properties (i) and (ii) from Proposition~\ref{prop:approx-equiv}.

\vspace{7pt}

\emph{Property (i).}\quad Suppose that $\phi \in Z_k$, $\psi \in \S_\ell(\pi_\phi)$, and $U$ is a neighbourhood of $\psi$.  Then Lemma~\ref{lem:typ-trans} gives a neighbourhood $V$ of $\phi$ such that
\begin{equation}\label{eq:U-V-nonempty}
\X(\pi,V)\ne \emptyset \qquad \Rightarrow \qquad \X(\pi,U) \ne \emptyset
\end{equation}
for any representation $\pi$.

Since $\phi \in Z_k = \rm{T}\lim_n Z_{n,k}$, there are elements $\phi_n \in Z_{n,k}$ that converge to $\phi$.  These elements witness that $\X(\pi_{\phi_n},V)$ is nonempty for all sufficiently large $n$, and so $\X(\pi_{\phi_n},U)$ is also nonempty for all sufficiently large $n$, by~\eqref{eq:U-V-nonempty}. Therefore $U$ meets $Z_{n,\ell}$ for all sufficiently large $n$.  Since $U$ is an arbitrary neighbourhood of $\psi$, this shows that $\psi \in Z_\ell$.

\vspace{7pt}

\emph{Property (ii).}\quad Let $\phi \in Z_k$ and $\psi \in Z_\ell$.  Choose sequences $\phi_n \in Z_{n,k}$ converging to $\phi$ and $\psi_n \in Z_{n,\ell}$ converging to $\psi$. For each $n$, property (ii) for $Z_{n,\bullet}$ gives a joining $\theta_n$ of $\phi_n$ and $\psi_n$ that lies in $Z_{n,k+\ell}$. By compactness, the sequence $(\theta_n)_{n\ge 1}$ has a subsequential limit in $\S_{k+\ell}(\A)$.  This limit is a joining of $\phi$ and $\psi$ that lies in $Z_{k+\ell}$.
\end{proof}

The next lemma is a cousin of Lemma~\ref{lem:lower-Vietoris-simplify} for the strong-quotient topology.

\begin{lem}\label{lem:Vietoris-simplify}
Consider the subsets of $\Rep^\sim_{\rm{a}}(\A)$ that have the form $U\cap \cal{W}'(\ell,O)$, where $U$ is a q-open subset of $\Rep^\sim_{\rm{a}}(\A)$ and
		\[\cal{W}'(\ell,O) := \big\{\t{\kappa} \in \Rep^\sim_{\rm{a}}(\A):\ \ol{\S_\ell(\kappa)} \subset O\big\}\]
for some positive integer $\ell$, and some open subset $O$ of $\S_\ell(\A)$.  The collection of all such subsets is a base for the strong-quotient topology of $\Rep^\sim_{\rm{a}}(\A)$.
	\end{lem}

\begin{proof}
As in the proof of Lemma~\ref{lem:lower-Vietoris-simplify}, this collection is a sub-base for the strong-quotient topology simply by transplanting the usual sub-base for the Vietoris topology on $\prod_{k\ge 1}\calK(\S_k(\A))$.  To show that it is a base, consider two such sets, say $U_1\cap \cal{W}'(\ell_1,O_1)$ and $U_2\cap \cal{W}'(\ell_2,O_2)$, and let $\pi$ be a separable unital representation whose approximate equivalence class lies in $U_1\cap \cal{W}'(\ell_1,O_1)\cap U_2\cap \cal{W}'(\ell_2,O_2)$.  Now let
	\[L_1 = \{1,\dots,\ell_1\}, \qquad L_2 := \{\ell_1+1,\dots,\ell_1 + \ell_2\},\]
and
	\[O := \{\phi \in \S_{\ell_1 + \ell_2}(\A):\ \phi[L_1] \in O_1\ \hbox{and}\ \phi[L_2] \in O_2\}.\]
Then property (i) from Proposition~\ref{prop:approx-equiv} gives $\ol{\S_{\ell_1 + \ell_2}(\pi)} \subset O$, and therefore the class of $\pi$ lies in
\[(U_1\cap U_2)\cap \cal{W}'(\ell_1 + \ell_2,O) \subset U_1\cap \cal{W}'(\ell_1,O_1)\cap U_2\cap \cal{W}'(\ell_2,O_2).\]
	\end{proof}

\begin{ex}[\emph{Continuation of Example~\ref{ex:one-op1}}]\label{ex:one-op2}
Let $\A = C[-L,L]$, and let $a$ be the identity function on $[-L,L]$ regarded as an element of $a$.  Let $A_1,A_2,\dots$ and $A$ be self-adjoint operators of norm at most $L$, say on separable Hilbert spaces $H_1,H_2,\dots$ and $H$.  These define unital representations of $\A$ by requiring that $\pi_n(a) = A_n$ and $\pi(a) = A$.  In addition, let $E_n$ be the spectral resolution of $A_n$ for each $n$, and let $E$ be the spectral resolution of $A$. Building on Example~\ref{ex:one-op1}, one can check that $\pi_n \to \pi$ in the strong-quotient topology if and only if we have
\[\rm{rank}\,E(O) \le \liminf_n \rm{rank}\,E_n(O) \qquad \hbox{and} \qquad \rm{rank}\,E(C) \ge \limsup_n \rm{rank}\,E_n(C)\]
whenever $O$ and $C$ are open and closed subsets of $[-L,L]$, respectively.  We omit the proof for brevity. It is also straightforward to adapt this example to unitary operators, normal operators, or larger tuples of commuting self-adjoint operators. \fin 
\end{ex}

Related to sq-convergence, the next definition is also used at many points below.  Let $(\pi_n)_{n\ge 1}$ be a sequence of separable unital representations and let $\t{\pi_n}$ for $n\ge 1$ be their approximate equivalence classes.

\begin{dfn}\label{dfn:asymp-assoc}
	A completely positive map $\phi\in \B(\A,\rmM_k)_+$ is \textbf{asymptotically associated} to $(\pi_n)_{n\ge 1}$ (or to $(\t{\pi_n})_{n\ge 1}$) if, for every neighbourhood $U$ of $\phi$ in $\B(\A,\rmM_k)_+$, there are $k$-tuples $V_n$ in $H_{\pi_n}$ such that $\Phi^{\pi_n}_{V_n} \in U$ for infinitely many $n$.
	\end{dfn}

By Lemma~\ref{lem:lcsc} and a diagonal argument, $\phi$ is asymptotically associated to $(\pi_n)_{n\ge 1}$ if and only if there are a subsequence $n_1 < n_2 < \dots$ and $k$-tuples $V_i$ in $H_{\pi_{n_i}}$ such that $\Phi^{\pi_{n_i}}_{V_i} \to \phi$.  If $\phi$ is normalized, then we may normalize these tuples as well, and conclude that $\phi$ is asymptotically associated to $(\pi_n)_{n\ge 1}$ if and only if $\phi$ belongs to $\rm{T}\limsup_n \ol{\S_k(\pi_n)}$. The choice of $\rm{T}\limsup$ rather $\rm{T}\liminf$ here is a matter of convention.  If $(\pi_n)_{n\ge 1}$ strong-quotient converges then this choice makes no difference: in that case, the limit of $\ol{\S_k(\pi_n)}$ is precisely the set of maps in $\S_k(\A)$ that are asymptotically associated to $(\pi_n)_{n\ge 1}$.

\subsection{Other topologies and their relations}\label{subs:other-tops}

Both Fell's quotient and inner hull-kernel topologies can be defined by comparing representations through their $k$-summaries for all $k$.  From this point of view, the difference is that we take convex combinations of the $k$-summaries before defining the inner hull-kernel topology.  As a consequence, the Fell topology is insensitive to multiplicities of subrepresentations, whereas the quotient topology does remember them to some extent: specifically, through the ranks of all compact operators in the representations, as a consequence of Theorem~\ref{thm:Voi-main}.

On the other hand, neither the quotient topology nor the Fell topology controls the values of operator norms $\|\pi(a)\|$ for a fixed element $a$ of $\A$ and a variable choice of the representation $\pi$.  In general these norms are only lower semicontinuous in those topologies, not continuous; see~\cite[Chap. II]{Fel60} or~\cite[Sec. 3.3]{Dix--Cstar}.  Recent years have seen a swell of interest in a mode of convergence for representations called `strong convergence', which is defined precisely to control these operator norms; see, for intance, the recent surveys~\cite{Magee--survey,vanHan--strong-survey}.  On the other hand, control of those norms does not require controlling multiplicities, or tuples of more than one vector, and so strong convergence is again insensitive to multiplicities.  In fact, the strong convergence of an AP sequence $(\pi_n)_{n\ge 1}$ is equivalent to Vietoris convergence of the sequence of closed convex hulls $\ol{\rm{conv}}\,\S_1(\pi_n)$.

For example, if $\A = C(\bbT)$, then a non-degenerate representation is generated by a single unitary operator.  For a sequence of unitaries $(U_n)_{n\ge 1}$ and another one $U$, the resulting representations converge strongly if and only if their spectra satisfy $\s(U_n) \to \s(U)$ in the Hausdorff distance, while strong-quotient convergence additionally keeps track of the multiplicities of isolated eigenvalues.

I have chosen the term `strong-quotient topology' because strong-quotient convergence is a simultaneous strengthening of (i) quotient convergence and (ii) strong convergence.  The previous subsubsections show (i), and the next lemma gives (ii).

\begin{lem}\label{lem:sq-to-norm}
Fix $a \in \A$.
\begin{itemize}
\item[a.] If $a$ is self-adjoint and $\pi \in \Rep(\A)$, let $\s_{\max}(\pi(a))$ be the largest value in the spectrum of the operator $\pi(a)$.  This value depends only on $\t{\pi}$, and it defines a continuous function of this class in the strong quotient topology of $\Rep^\sim_{\rm{a}}(\A)$.
\item[b.] For any $a$, the operator norm $\|\pi(a)\|$ depends only on $\t{\pi}$, and it defines a continuous function of this class in the strong quotient topology of $\Rep^\sim_{\rm{a}}(\A)$.
\end{itemize}
\end{lem}

\begin{proof}
If $a$ is self-adjoint and $\pi \in \Rep(\A)$, then
\[\s_{\max}(\pi(a)) = \sup\{\langle \pi(a)v,v\rangle:\ v \in H_\pi,\ \|v\|=1\} = \sup\{\phi(a):\ \phi \in \S_1(\pi)\}.\]
Replacing $\S_1(\pi)$ with $\ol{\S_1(\pi)}$ does not affect this last supremum. This proves (a) because the set $\ol{\S_1(\pi)}$ is a Vietoris-continuous function of $\t{\pi} \in \Rep^\sim_{\rm{a}}(\A)$ by construction.  Part (b) follows by applying part (a) to $\s_{\max}(a^\ast a)$.
\end{proof}

The next table summarizes this point of view on the Fell, quotient, strong, and strong-quotient topologies.

\begin{center}
\begin{tabular}{ccc}
This mode...  & ... means convergence of ... & ... in this topology\\
Fell & $\ol{\rm{conv}}\,\S_1(\pi_n)$ & lower Vietoris\\
quotient & $\ol{\S_k(\pi_n)}$ for each $k$ & lower Vietoris\\
strong & $\ol{\rm{conv}}\,\S_1(\pi_n)$ & Vietoris\\
strong-quotient & $\ol{\S_k(\pi_n)}$ for each $k$ & Vietoris\\
\end{tabular}
\end{center}

Although these four topologies differ in general, sometimes special features of a representation $\pi$ imply that convergence of a sequence to $\pi$ in a weaker topology actually implies it in a stronger topology.


\begin{prop}\label{prop:perfect}
Let $\rho$ be a separable unital representation whose support is a perfect subset of $\hat{\A}$ for the Fell topology: that is, it has no isolated points.
\begin{enumerate}
\item[a.] If another separable unital representation $\pi$ weakly contains $\rho$, then $\pi \gtrsim_\rm{a} \rho$.
\item[b.] If another separable unital representation $\pi$ is weakly equivalent to $\rho$, then ${\pi \simeq_\rm{a} \rho}$.
\item[c.] If a sequence $(\pi_n)_{n\ge 1}$ converges to $\rho$ strongly, then $\t{\pi_n}$ sq-converges to $\t{\rho}$.
\end{enumerate}
\end{prop}

\begin{proof}
\emph{Part (a).}\quad  Consider the ideals
\[\I_0 := \ker \pi \qquad \hbox{and} \qquad \I_1 := \pi^{-1}[\mathfrak{K}(H_\pi)].\]
We have $\I_0 \subset \ker \rho$ by assumption.  To complete the proof we show that in fact $\I_1 \subset \ker \rho$, for then Theorem~\ref{thm:get-sum} gives $\pi \simeq_{\rm{a}}\pi \oplus \rho$.

Both $\pi$ and $\rho$ give well-defined quotient representation of $\A/\I_0$, say $\pi'$ and $\rho'$.  According to~\cite[Prop 3.2.1]{Dix--Cstar}, there is a canonical homeomorphism
\[\rm{spt}\,\pi = \{\kappa \in \hat{\A}:\ \kappa|\I_0 = 0\} \to \hat{\A/\I_0}.\]
It carries the subset $\rm{spt}\,\rho$ to $\rm{spt}\,\rho'$, so $\rm{spt}\,\rho'$ is also perfect.  Therefore, after replacing $\pi$ with $\pi'$ and $\rho$ with $\rho'$ if necessary, we may assume that $\I_0 = 0$ and $\rm{spt}\,\pi = \hat{\A}$.

Having done so, it follows that $\pi|\I_1$ is an injective representation into $\mathfrak{K}(H_\pi)$.  As a result, the structure theory for C*-algebras of compact operators shows that $\I_1$ is isomorphic to the restricted direct sum of $\mathfrak{K}(H_j)$ for some family $(H_j:\ j \in J)$ of pairwise-orthogonal subspaces of $H_\pi$ (see~\cite[Thm. 16.18]{ConOT}, for example).  Therefore $\hat{\I_1}$ consists of a collection of isolated points indexed by $J$. 

Finally, according to~\cite[Prop. 3.2.1]{Dix--Cstar} again, the subset
\[\hat{\A}^{\I_1} = \{\kappa \in \hat{\A}:\ \kappa|\I_1 \ne 0\}\]
is open in $\hat{\A}$, and restriction of representations defines a homeomorphism from $\hat{\A}^{\I_1}$ to $\hat{\I_1}$.  Therefore $\hat{\A}^{\I_1}$ is a open subset of $\hat{\A}$ that also consists of a family of isolated points.  Since $\rm{spt}\,\rho$ is perfect, this requires that $\hat{\A}^{\I_1}$ and $\rm{spt}\,\rho$ are disjoint, which implies that $\rho|\I_1 = 0$, as required.

\vspace{7pt}

\emph{Part (b).}\quad Part (a) gives $\pi \gtrsim_{\rm{a}} \rho$.  On the other hand, applying part (a) to $\rho$ (in place of $\pi$) and $\rho^{\oplus \infty}$ (in place of $\rho$), it gives that $\rho \gtrsim_{\rm{a}} \rho^{\oplus \infty}$.  Therefore $\rho$ contains no nonzero compact operators, and $\ker \rho \subset \ker \pi$ by assumption, so now Theorem~\ref{thm:get-sum} gives $\pi \lesssim \rho\oplus \pi \simeq_{\rm{a}} \rho$.

\vspace{7pt}

\emph{Part (c).}\quad Since $\Rep^\sim_{\rm{a}}(\A)$ is compact and metrizable, the sequence $(\t{\pi_n})_{n\ge 1}$ has a subsequential limit in the strong-quotient topology, say $\t{\pi}$.  It suffices to prove that $\pi \simeq_\rm{a} \rho$.  Since sq-convergence implies strong convergence, this limit must satisfy
\[\|\pi(a)\| = \lim_n \|\pi_n(a)\| = \|\rho(a)\| \qquad (a \in \A).\]
So $\pi$ is weakly equivalent to $\rho$, and now the result follows from part (b).
\end{proof}

\begin{rmk}\label{rmk:perfect-disambig}
Proposition~\ref{prop:perfect} refers to the support of $\rho$ as a subset of $\hat{\A}$, not to its image in the quotient space $\hat{\A}_\rm{a}$ of approximate equivalence classes.  This distinction is important, because even if the support of $\rho$ is a perfect subset of $\hat{\A}$ for the Fell topology, it may collapse to a single point (and hence not be perfect) in $\hat{\A}_\rm{a}$.  See Remark~\ref{rmk:KK} below for some examples. \fin
\end{rmk}

%

\begin{rmk}\label{rmk:local-global}
Let $G_n = (V_n,E_n)$ for $n=1,2,\dots$ be finite graphs whose degrees are uniformly bounded.  Combinatorists have defined several modes of convergence that capture different asymptotic properties of such a sequence.  Among these is `local-global convergence', which originates in~\cite{BolRio11,HatamiLovSze14}.  If $A_k$ is the finite set $\{1,2,\dots,k\}$, and $\bf{x}_n$ is an $A_k$-colouring of $V_n$ for each $n$, then one can define an `empirical distribution' $P_{G_n,\bf{x}_n}$ that describes the local statistics of the colouring relative to the local statistics of the underlying graph.  Very roughly, the sequence $(G_n)_{n\ge 1}$ converges `locally-globally' if, for each $k$, the set of distributions that one can approximate using this $G_n$ and different choices of $\bf{x}_n$ converges in the Vietoris topology of a suitable space of possible `limit distributions'.  This is a graph-theoretic analog of strong-quotient convergence, and is cited by the authors of~\cite{AbeEle11} as one of their inspirations. \fin
\end{rmk}



\section{Regular representations of countable groups}\label{sec:regular}

When $\A = C^\ast \G$, the left regular representation $\l$ plays a distinguished role in the representation theory of $\A$.  The image $C^\ast_{\rm{r}}\G := \l(C^\ast \G)$ is the \textbf{reduced group C*-algebra} of $\l$.

\begin{lem}\label{lem:mix}
If $\G$ is infinite then the C*-algebra $C_{\rm{r}}^\ast \G$ contains no nonzero compact operators, and we have $\l \simeq_{\rm{a}} \l^{\infty}$.
\end{lem}

\begin{proof}
The right regular representation of $\G$ on $\ell^2(\G)$ commutes with $\l$.  Also, like $\l$, the right regular representation is mixing~\cite[Sec. 11]{KecGAEGA}, and therefore it has no finite-dimensional subspaces.  As a result, the right regular representation cannot commute with any nonzero finite-rank operator on $\ell^2(\G)$, and hence also not with any nonzero compact operator.  Now Theorem~\ref{thm:get-sum} gives the last conclusion.
\end{proof}

The \textbf{reduced dual} $\hat{\G}_\rm{r}$ of $\G$ is the support of $\l$ in the unitary dual $\hat{\G}$.  The reduced dual is canonically homeomorphic to $\hat{C^\ast_{\rm{r}}\G}$~\cite[Prop. 18.2.3]{Dix--Cstar}.

A positive definite function $\phi$ on $\G$ is called \textbf{tempered} if it is approximately associated to $\l$.  Temperedness is more often defined with `weakly associated' in place of `approximately associated', but in the case of $\l$ these coincide by a result of Takenouchi: see~\cite[Prop. 18.3.5(a)]{Dix--Cstar}. This also follows by combining Lemmas~\ref{lem:mix} and~\ref{lem:approx-and-weak-dilation}. 
Tempered positive definite functions on free groups are the key objects in Theorem~\ref{mainthm:tempered}, and they play an important role in much of Part~\ref{part:free}.

With a little more care, we can enhance Lemma~\ref{lem:mix} to the next proposition and corollary.  They are well-known, but we include a proof for completeness.  We do not need them in the proofs of our main theorems later, but at some points they help clarify their context.

\begin{prop}\label{prop:Mau}
If $\G$ is a finitely generated infinite discrete group, then the reduced dual $\hat{\G}_\rm{r}$ is a perfect subset of $\hat{\G}$ for the Fell topology.
\end{prop}

\begin{proof}
Any isolated points in $\hat{\G}_\rm{r}$ would give rise to irreducible direct summands in $\l$ (by applying~\cite[Thm. 1.7]{Wang75} to singleton subsets of $\hat{C^\ast_{\rm{r}}\G}$, for example).  So it suffices to prove that $\l$ has no irreducible direct summands.

If $\G$ is virtually Abelian, then it is amenable, so $\hat{\G}_\rm{r} = \hat{\G}$.  From this point the result reduces to commutative Pontrjagin duality for a finite-index Abelian subgroup.

Now suppose that $\G$ is finitely generated and not virtually Abelian.  Let $\G^{\rm{fc}}$ be the union of all finite conjugacy classes in $\G$.  Then $\G^{\rm{fc}}$ is a normal subgroup of $\G$, and if $\G$ is finitely generated but not virtually Abelian then this subgroup must have infinite index.  At this point a classic theorem of Mautner from~\cite{Mau50} asserts that the von Neumann algebra $\l(\G)''$ is of Type II.  This rules out irreducible direct summands in $\l$, because the corresponding isotypic summand in $\l$ would generate a Type I summand in $\l(\G)''$.
\end{proof}

Using some of Thoma's arguments from~\cite{Thoma64,Thoma68}, one can extend Proposition~\ref{prop:Mau} to some discrete groups that are not finitely generated, but we leave these aside here.

\begin{rmk}\label{rmk:KK}
If $\G$ is C*-simple, meaning that $C^\ast_{\rm{r}}\G$ has no nontrivial closed ideals, then the restriction of the Fell topology to $\hat{\G}_{\rm{r}}$ is actually the trivial topology.  This is stronger than the conclusion of Proposition~\ref{prop:Mau} whenever $\hat{\G}_{\rm{r}}$ contains at least two points.  Through another characterization, if $\G$ is C*-simple then $\hat{\G}_\rm{r}$ contains only a single approximate equivalence class, illustrating the issue discussed in Remark~\ref{rmk:perfect-disambig}.

Powers showed that non-Abelian free groups are C*-simple in~\cite{Pow75}. Concrete characterizations of C*-simplicity for general countable groups were obtained in~\cite{KalKen17} and~\cite{BreKalKenOza17}, with the consequence that many groups of interest have this property. \fin
\end{rmk}

\begin{cor}\label{cor:perfect}
Parts (a)--(c) of Proposition~\ref{prop:perfect} apply when $\rho = \l$. \qed
\end{cor}

%
%

\chapter{Almost periodic sequences and almost periodic entropy}\label{chap:AP}

In this chapter we recall the basic theory of a new quantity recently introduced in~\cite{APE4}. This is `almost periodic' (or `AP') entropy.  It depends on a completely positive map and a fixed sequence of finite-dimensional representations called an `almost periodic sequence'. It is a rough analog of sofic entropy for representations of C*-algebras rather than measure-preserving group actions.  It is the deterministic predecessor of the two new notions of entropy that we introduce in Part~\ref{part:general}: `annealed' and `zeroth-order' AP entropy.

\section{Almost periodic sequences}\label{sec:FD-conv}

\begin{dfn}[Almost periodic sequence]\label{dfn:AP}
	An \textbf{almost periodic} (`\textbf{AP}') \textbf{sequence} for $\A$ is a sequence of finite-dimensional representations of $\A$ whose dimensions tend to $\infty$.
\end{dfn}

The quotient and strong-quotient topologies from Section~\ref{sec:q-and-sq} offer two possible modes of convergence for an AP sequence.  Neither of these depends on the representations being finite-dimensional.

We sometimes need another mode of convergence which is specific to finite-dimensional representations.  For an AP sequence $\bspi$, this is the convergence of the pulled-back traces on $\A$: that is, $\tr_{d_n}\circ \pi_n \to \tau$ in the weak$^\ast$ topology for some limit tracial state $\tau$ on $\A$.  In the terminology of free probability theory, this asserts that, for any finite subset $F$ of $\A$, the tuples $(\pi_n(a):\ a \in F)$ form a sequence of `microstates' for $F$ in the `non-commutative probability space' $(\A,\tau)$ (see~\cite{Voi02--survey}, for example)  This mode of convergence is rather different from the previous two modes, but the phenomenon of measure concentration does give rise to a connection: see Proposition~\ref{prop:APent-properties-trace}(b) below.

All of these modes of convergence appear naturally in explorations of how the finite-dimensional representations are distributed among all representations in either topology.  We next recall a few aspects of this long-running story.

A group $\G$ is called \textbf{maximally almost periodic} if it has an AP sequence $(\pi_i)_{i\ge 1}$ that separates the elements of the group.  For finitely generated groups, this property turns out to be equivalent to residual finiteness.  Maximally almost periodic groups were introduced by von Neumann~\cite{vonNeu34} and treated by Weil in~\cite[Chap. VII]{WeilIGT} and Dixmier in~\cite[Secs. 16.4--5]{Dix--Cstar}.

If $\A = C^\ast \G$ where $\G$ is countable and discrete, and $\tau$ is the regular tracial state on $\A$, then an AP sequence whose tracial states converge to $\tau$ is called \textbf{asymptotically regular}.  Such an AP sequence is necessarily separating, so a group must be maximally almost periodic to have such a sequence.  On the other hand, starting from any AP sequence that separates the elements of $\G$, a standard construction of running tensor products produces a new such sequence whose tracial states converge to $\tau$.  So a finitely generated group has an asymptotically regular AP sequence if and only if it is residually finite, by the result recalled above.

A general C$\sp*$-algebra $\A$ is called \textbf{residually finite dimensional} if it has a faithful family of finite-dimensional representations.  This is equivalent to the subset $\hat{\A}_\rm{fd}$ of finite-dimensional irreducible representation classes being dense in $\hat{\A}$ for the Fell topology.  For a group C$\sp*$-algebra, this is strictly stronger than maximal almost periodicity in general, and it is not known to hold for many groups.  One important family of examples is free groups, for which this property follows from work of Lubotzky and Shalom~\cite{LubotSha04}.  In fact, they prove the even stronger result that representations induced by finite permutation actions are dense in the Fell topology.  By an extension of this argument one can actually approximately `count' the finite-permutation actions that weakly approximate a given measure-preserving free-group action.  This discovery emerges independently from Bowen's works~\cite{Bowen10free,Bowen10c,Bowen10c} on annealed sofic entropy.  The survey~\cite{BurKec20} goes over this story carefully.

However, AP sequences and AP entropy can also be used to study groups that are not maximally almost periodic.  We do this by requiring that the group law hold only `in the limit' along an AP sequence.  To be specific, let $\G$ be any countable group, and write it as $F/N$ for some free group $F$ and normal subgroup $N$. Now we can look for finite-dimensional representations of $F$ whose pulled-back tracial states converge to $1_N$, rather than finite-dimensional representations of $\G$ whose tracial states converge to $1_{\{e\}}$.  This allows considerable extra flexibility, because those finite-dimensional representations of $F$ need not have trivial restriction to $N$ until we take their limit.  Allowing convergence in this sense, the availability of finite-dimensional approximants to the regular representation of $\G$ is equivalent to $\G$ being what R\u{a}dulescu called `hyperlinear': see, for instance, the original paper~\cite{Rad08} or the introductions in~\cite{KwiPes13} or~\cite{CapLup15}.

In this way, free groups are essentially universal among all countable groups for the purposes of the present work, since the desired convergence for other groups can always be captured by choosing the right limit character over the free group.  For this reason, results that are formulated for free groups, including the results of Part~\ref{part:free} below, could still lead to new understanding of other classes of groups or other C$\sp*$-algebras entirely.

Focusing on free groups in particular, we arrive at Voiculescu's theory of `free probability' and the use of probabilistic techniques.  For example, if $\G$ is freely generated by $\{s_1,\dots,s_r\}$, and $\bspi = (\pi_n)_{n\ge 1}$ is an AP sequence for it, then $\bspi$ is asymptotically regular if and only if the tuples of generators
\[\pi_n(s_1), \dots \pi_n(s_r)\in \bf{U}(d_n) \qquad (n\ge 1)\]
are \textbf{asymptotically free} in Voiculescu's sense: see the survey~\cite{Voi02--survey} or textbooks such as~\cite{VoiDykNic92,HiaPetSL} or~\cite[Chap. 5]{AndGuiZei--book}.

Free groups have many such sequences, but they may not be easy to construct explicitly.  Instead, one of the basic results of free probability theory asserts that `randomly chosen' $n$-dimensional representations are asymptotically free with high probability as $n\to\infty$.  See Theorems~\ref{thm:asymp-free1} and~\ref{thm:asymp-free2} for asymptotic freeness itself.

Beyond this, essentially following~\cite{CapDon-Mar07,ColMal14}, an AP sequence $\bspi = (\pi_n)_{n\ge 1}$ for a free group $\G$ is called \textbf{strongly asymptotically free} if it converges strongly to the regular representation.  Here, too, the known proofs that such sequences exist for free groups are probabilistic.  They include the result of Collins and Male~\cite{ColMal14} for uniform random representations (which is based on an analogous result for the Gaussian unitary ensemble by Haagerup and Thorbj\o rnsen~\cite{HaaTho05}) and the result of Bordenave and Collins~\cite{BordCol19} for random permutations; we recall these formally as Theorems~\ref{thm:ColMal} and~\ref{thm:BordCol} below.  See also the surveys~\cite{Magee--survey,vanHan--strong-survey}.  Note that, while those references discuss strong convergence, this immediately implies sq-convergence by part (c) of Corollary~\ref{cor:perfect} as well, since the limit is a regular representation.

For any countable group $\G$, if an AP sequence $\bspi$ converges strongly to the left regular representation $\l$, then any subsequential limit of the sequence $\tr_{d_n}\circ \pi_n$ in the compact space $\S_1(\G)$ extends to a tracial state on the reduced C$\sp*$-algebra $C_{\rm{r}}^\ast \G := \l(C^\ast\G)$.  In case $C_{\rm{r}}^\ast \G$ has only the regular tracial state, it follows that $\tr_{d_n}\circ \pi_n$ must also converge to that tracial state.  This property is quite widespread. For free groups it is a classical result of Powers~\cite{Pow75}, so strong asymptotic freeness does indeed imply asymptotic freeness.  For other groups, the definitive recent work~\cite{BreKalKenOza17} explains the story in full.

\begin{rmk}
	This discussion above also has much in common with the study of various modes of convergence for sparse sequences of large finite graphs.  This link consists of instructive analogies more than formal results, but offers considerable guidance and motivation for our study.  It is discussed more fully in~\cite{AbeEle11}. \fin
	\end{rmk}

%
%
%
%
%

\section{Almost periodic entropy}

Define the function
\begin{equation}\label{eq:ball-vol}
v(d) := \frac{\pi^d}{d!}\qquad (d=1,2,\dots).
\end{equation}
Then $v(d)$ is the volume of the unit ball in $\bbC^d$~\cite[Subsection 1.4.9]{RudinFTUB}.

Now let $\A$ be a separable unital C$\sp*$-algebra.  Fix an AP sequence $\bspi = (\pi_n)_{n\ge 1}$, a positive integer $k$, and an element $\phi$ of $\B(\A,\rmM_k)_+$.  Let $d_n$ be the dimension of $\pi_n$ for each $n$.

\begin{dfn}\label{dfn:APent}
The \textbf{almost periodic} (`\textbf{AP}') \textbf{entropy of $\phi$ along $\bs{\pi}$} is the quantity
\begin{equation}\label{eq:APent}
 \rmh_{\bspi}(\phi) := \inf_O \limsup_{n \to\infty} \frac{1}{d_n}\log \frac{\vol_{2kd_n}\X(\pi_n,O)}{v(d_n)^k},
\end{equation}
where the infimum runs over all neighbourhoods of $\phi$ in $\B(\A,\rmM_k)_+$.
\end{dfn}

Recall Definition~\ref{dfn:typical} for the set $\X(\pi_n,O)$ of $O$-typical vectors.  By the monotonicity of these sets in $O$, we can restrict $O$ to any base of neighbourhoods at $\phi$ in formula~\ref{eq:APent} without changing the resulting value.

The next results are some of the basic properties of AP entropy established in~\cite{APE4}.  The first is an immediate consequence of its definition and Lemma~\ref{lem:upper-semicts}.

\begin{lem}\label{lem:APent-usc}
For any $\bspi$ and $k$, the function $\rmh_{\bspi}$ is upper semicontinuous on $\B(\A,\rmM_k)_+$. \qed
\end{lem}

\begin{prop}\label{prop:APent-properties}
Fix a positive integer $k$ and a map $\phi\in\B(\A,\rmM_k)_+$.
\begin{itemize}
\item[a.] We have
\[\rmh_{\bspi}(\phi) \le \log\det \phi(1).\]
\item[b.] If $Q$ is an invertible $k$-by-$k$ matrix, and we define
\begin{equation}
\psi(a) := Q^\ast \phi(a)Q \qquad (a \in \A),
\end{equation}
then
\[\rmh_{\bspi}(\psi) = 2\log|\det Q| + \rmh_{\bspi}(\phi).\] \qed
\end{itemize}
	\end{prop}

See~\cite[Lems. 6.6 and 6.7 and Prop. 6.10(a)]{APE4}, respectively.  We extend these results to annealed AP entropy in Chapter~\ref{chap:random-AP}.

Some proofs about AP entropy are easier to digest when $k=1$.  Several properties are proved in~\cite{APE4} by starting with that case and then using pairings as in~\eqref{eq:pairing} and the general identity
\begin{equation}\label{eq:h-pi-k}
\rmh_{\bspi^{(k)}}(\langle\phi,\cdot\rangle) = \frac{1}{k}\rmh_{\bspi}(\phi) \qquad (\phi \in \B(\A,\rmM_k)_+).
\end{equation}
See~\cite[Lem. 6.8]{APE4}. We do not use identity~\eqref{eq:h-pi-k} directly in the sequel, but we do cite some properties of AP entropy that are proved this way in~\cite{APE4}. Some of those proofs use~\eqref{eq:h-pi-k} in combination with the following alternative formula for the AP entropy of a state, which is~\cite[Lem. 6.12]{APE4}.

\begin{lem}\label{lem:use-of-spherical}
Let $\phi \in \S(\A)$, and let $\cal{O}$ be a neighbourhood base at $\phi$ in $\S(\A)$.  Then
\[\rmh_{\bspi}(\phi) = \inf_{O \in \cal{O}}\limsup_{n\to\infty}\frac{1}{d_n}\log\s_{2d_n-1}\X(\pi_n,O).\]
\qed
\end{lem}

\section{Convergence of traces}\label{sec:trace-conv}

Several further properties of AP entropy can be deduced if the normalized traces $\tr_{d_n}\circ \pi_n$ are known to converge to a limiting trace in $\A^\ast_+$.

\begin{prop}\label{prop:APent-properties-trace}
Assume that $\tr_{d_n}\circ \pi_n\to \tau$, and let $\l$ be the GNS representation of $\tau$. Let $k$ be a positive integer and let $\phi \in \B(\A,\rmM_k)_+$.
\begin{enumerate}
\item[a.] If $\phi$ is approximately associated to $\l^{\oplus \infty}$, then $\phi$ is asymptotically associated to $\bspi$.
\item[b.] If $\phi_{\rm{sing}}$ is asymptotically associated to $\bspi$, then so is the whole of $\phi$.
\item[c.] If $\phi(1)$ is nonsingular, then equality holds in Proposition~\ref{prop:APent-properties}(a) if and only if $\phi = \tau \otimes \phi(1)$.
\end{enumerate} \qed
\end{prop}

See~\cite[Cors. 6.14 and 6.15 and Prop. 6.16]{APE4}. Both parts stem from an application of measure concentration for the high-dimensional spherical measures $\s_{2d-1}$. Since the trace of a matrix is invariant under unitary conjugation, if $\pi$ is a $d$-dimensional representation, then averaging over all of $\bf{U}(d)$ gives
\begin{equation}\label{eq:trace-integral}
\tr_d\circ \pi = \int_{\rmS^{2d-1}}\Phi^\pi_x\ d\sigma_{2d-1}(x),
\end{equation}
where $\s_{2d-1} = m_{\bf{U}(1,n)}$ is the uniform distribution on $\rmS^{2d-1}$.  When $d$ is large, an appeal to measure concentration improves on~\eqref{eq:trace-integral} considerably: $\Phi^\pi_x$ is actually close to $\tr_d\circ \pi$ for `most' $x \in \rmS^{2d-1}$ individually.  Now one can combine this conclusion with Lemma~\ref{lem:typ-trans}.

In the sequel, we need an approximate version of Proposition~\ref{prop:APent-properties-trace} that applies to a single high-dimensional representation, rather than to a whole AP sequence.

\begin{cor}\label{cor:APent-properties-trace}
Let $\tau$ be a tracial state of $\A$, let $\l$ be its GNS representation, and let $\phi \in \B(\A,\rmM_k)_+$.
\begin{enumerate}
\item[a.] Assume that $\phi$ is approximately associated to $\l^{\oplus \infty}$. Then, for every neighbourhood $O$ of $\phi$, there are a positive integer $d_0$ and a neighbourhood $V$ of $\tau$ such that the following holds:
\begin{quote}
If $\pi$ is a $d$-dimensional representation satisfying
\[d \ge d_0 \qquad \hbox{and} \qquad \tr_d\circ \pi \in V\]
then $\X(\pi,O) \ne \emptyset$.
\end{quote}
\item[b.] For every neighbourhood $O$ of $\phi$ there are a positive integer $d_0$ and neighbourhoods $U$ of $\phi_{\rm{sing}}$ and $V$ of $\tau$ such that the following holds:
\begin{quote}
If $\pi$ is a $d$-dimensional representation satisfying
\[d\ge d_0, \qquad \tr_d\circ \pi \in V, \qquad \hbox{and} \qquad \X(\pi,U) \ne \emptyset\]
then $\X(\pi,O) \ne \emptyset$.
\end{quote}
\end{enumerate}
\end{cor}

\begin{proof}
We show how to deduce part (a) from Proposition~\ref{prop:APent-properties-trace}(a).  The proof of (b) follows in the same way from Proposition~\ref{prop:APent-properties-trace}(b), so we leave this to the reader.

Suppose that the result is false, and let $O$ be a neighbourhood of $\phi$ for which it fails.  Because of Lemma
\ref{lem:lcsc}, we can choose a countable neighbourhood base $V_1\supset V_2 \supset \dots$ at $\tau$.  Then a diagonal argument gives a sequence $(\pi_n)_{n\ge 1}$ of finite-dimensional representations such that
\[d_n\ge n \qquad \hbox{and} \qquad \tr_{d_n}\circ \pi_n \in V_n\]
but $\X(\pi_n,O) = \emptyset$, meaning that $O$ does not meet $\S_k(\pi_n)$.  Together, these properties of $(\pi_n)_{n \ge 1}$ show that it is an AP sequence which contradicts Proposition~\ref{prop:APent-properties-trace}(a).
\end{proof}

\section{Fuglede--Kadison determinants}\label{sec:FK}

In the next section we recall one of the most substantial results from~\cite{APE4}: a formula for AP entropy as a Fuglede--Kadison determinant.  The present section we recalls some necessary background about these determinants.

Let $\A$ be a separable C*-algebra, $\tau$ a tracial state on it, and $\l$ the GNS representation of $\tau$ with canonical cyclic tracial vector $\xi$.  Write $\t{\tau}$ for the extension of $\tau$ to a normal tracial state on $\l(\A)''$ defined in~\eqref{eq:tau-on-N}, and also for its counterpart on $\l(\A)'$ defined by the same formula. Any positive invertible operator $A$ in $\l^{\oplus k}(\A)'$ has a Fuglede--Kadison determinant defined by
\[\Delta_{\tau\otimes \tr_k} A:= \exp((\t{\tau} \otimes \tr_k)(\log A)).\]
Alternatively, this can be written in terms of the spectral resolution of $A$.  In this setting its basic properties are covered in~\cite[Sec. I.6.11]{Dix--vN}.  From this starting point, it can be extended in several directions, for example to a large class of non-negative affiliated operators: see~\cite[Sec. 2]{HaaSch07} for a more complete account.

In~\cite[Subs. 2.8]{APE4}, a further extension defines $\Delta_\tau \phi$ for $\phi \in \B(\A,\rmM_k)_+$.  If $\phi$ is $\l$-normal, then we let $T$ be the non-negative operator affiliated to $\l^{\oplus k}(\A)'$ from Proposition~\ref{prop:RadNik}, and set
\[\Delta_\tau \phi:= (\Delta_{\tau\otimes \Tr_k}T)^2 = (\Delta_{\tau\otimes \tr_k}T)^{2k}.\]
In general, we simply set $\Delta_\tau \phi := \Delta_\tau \phi_{\rm{ac}}$, where $\phi_{\rm{ac}}$ is the absolutely continuous part of $\phi$ with respect to $\l$ (see Corollary~\ref{cor:Leb}).  We often abbreviate $\Delta_\tau$ to $\Delta$ when $\tau$ is clear from the context.

The definition above interacts with pairings in this simple way:
\begin{equation}\label{eq:FK-det-pairing}
\Delta_\tau \phi = (\Delta_{\tau \otimes \tr_k}\langle \phi,\cdot\rangle)^k \qquad (\phi \in \B(\A,\rmM_k)_+).
\end{equation}

The quantity $\Delta_\tau \phi$ can also expressed through a variational principle~\cite[Prop. 2.14]{APE4}.  The idea for this originates with Arveson's work on subdiagonal subalgebras in~\cite[Dfn. 4.3.7]{Arv67}, where he takes a special case of this principle as his definition of `determinant'.  This and other basic theory for $\Delta_\tau \phi$ are covered fully in~\cite[Subs. 2.8]{APE4}.

A Fuglede--Kadison determinant of a completely positive map appears in a key formula for AP entropy.  While using this formula later, we need the following general properties of Fuglede--Kadison determinants.

\begin{prop}\label{prop:FK-det-properties}
Fix $\A$, $\tau$, and $\l$ as above, let $\Delta := \Delta_\tau$, and consider the restrictions $\Delta|\B(\A,\rmM_k)_+$
for each positive integer $k$.  They have the following properties.
\begin{enumerate}
\item[a.] The restriction $\Delta|\B(\A,\rmM_k)_+$ is upper semicontinuous for every $k$.
\item[b.] For every $k$ and $\phi,\psi \in \B(\A,\rmM_k)_+$, we have $\Delta \phi \le (\tr_k\phi(1))^k$, and $\Delta \phi \le \Delta \psi$ in case $\phi \le \psi$. 
\item[c.] If $\phi,\psi \in \B(\A,\rmM_k)_+$ and $0 \le t \le 1$, then $\Delta(t\phi) = t^k\Delta \phi$ for $t \ge 0$ and 
\[\log \Delta(t\phi + (1-t)\psi) \ge t\log \Delta \phi + (1-t)\log \Delta \psi \qquad (0\le t \le 1).\]
\item[d.] For every $\phi \in \B(\A,\rmM_k)_+$ and $\psi \in \B(\A,\rmM_\ell)_+$, we have
\[\Delta (\rm{diag}(\phi,\psi)) = \Delta \phi \cdot \Delta \psi.\]
\item[e.] If $\phi \in \A^\ast_+$, $a \in \A$, and we define $\psi := \phi(a^\ast(\cdot)a)$, then
\[\Delta \psi = (\Delta_\tau |a|)^2\cdot\Delta \phi.\]
\qed
\end{enumerate}
\end{prop}

These results all have fairly short proofs from either the definition of $\Delta$ or the variational principle: see~\cite[Props. 2.14 and 2.15]{APE4} and the subsequent remarks.  Those results are slightly more general than Proposition~\ref{prop:FK-det-properties} in places.

\section{The determinantal formula}

Given a tracial state $\tau$ and a completely positive map $\phi:\A\to\rmM_k$, Corollary~\ref{cor:Leb} decomposes $\phi$ into a summand $\phi_{\rm{ac}}$ that is absolutely continuous with respect to $\tau$ and another summand $\phi_{\rm{sing}}$ that is singular with respect to $\tau$.  The most substantial theorem about AP entropy from~\cite{APE4} is~\cite[Thm. C]{APE4}:

\begin{thm}\label{thm:det-form}
Assume that $\tr_{d_n}\circ \pi_n\to \tau$, and let $\Delta$ be the Fuglede--Kadison determinant associated to $\tau$. If $\phi \in \B(\A,\rmM_k)_+$ is asymptotically associated to $\bspi$, then
\[\rmh_{\bspi}(\phi) = \log \Delta \phi_{\rm{ac}}.\]
\end{thm}

Theorem~\ref{thm:det-form} can be regarded as a non-commutative generalization of Szeg\H{o}'s classic limit theorem for positive definite Toeplitz determinants (see~\cite{SimOPUCI,SimOPUCII}).  It is discussed more fully from this point of view in~\cite{APE4}, which includes other theorems that also have this flavour.  The most basic connection here is that the limit in Szeg\H{o}'s theorem can also be interpreted as the logarithm of a Fuglede--Kadison determinant applied to the absolutely continuous part of a positive functional.

Szeg\H{o}'s original theorem connects to several finer aspects of the study of orthogonal polynomials on the unit circle, and I am not aware of generalizations for most of these to the setting of Theorem~\ref{thm:det-form}.  However, in the special case of uniformly random AP sequences for free groups, we do find a tighter connection with Szeg\H{o}'s theorem in which many of these additional details reappear.  This begins with the definition of `generalized Verblunsky coefficients' in Section~\ref{sec:restrict-extend} below, and continues in the discussion of Subsection~\ref{subs:harm-an}.

The next example describes a different commutative special case of Theorem~\ref{thm:det-form} along the lines of Example~\ref{ex:one-op1}.

\begin{ex}[\emph{Continuation of Example~\ref{ex:one-op2}}]\label{ex:one-op3}
Let $\A = C[-L,L]$ and let $a$ be the identity function on $[-L,L]$ regarded as an element of $\A$.  An AP sequence $(\pi_n)_{n\ge 1}$ for $\A$ is generated by the sequence of self-adjoint matrices $A_n := \pi_n(a)$, all having operator norm at most $L$.  Let $C_n$ be the spectrum of $A_n$ for each $n$, and for each $s\in C_n$ let $P_{n,s}$ be the orthogonal projection from $\bbC^{\oplus d_n}$ to $\ker(A_n - sI_{d_n})$.

By the Riesz--Markov theorem, positive functionals on $\A$ are represented by finite Borel measures on $[-L,L]$.  Let $\mu$ be a Borel probability measure on this interval, regarded as a tracial state on $\A$.  Then the convergence $\tr_{d_n}\circ \pi_n\to \mu$ holds if and only if the empirical spectral distributions of our matrices satisfy
\[\frac{1}{d_n}\sum_{s\in C_n} \rm{rank}\,P_{n,s}\cdot \delta_s \to \mu \qquad \hbox{weakly* as}\ n\to\infty.\]
By passing to a subsequence for simplicity, let us also assume that the closed sets $C_n$ converge to a limit $C$ in the Vietoris topology of $\calK([-L,L])$.

Now let $\nu$ be any other finite Borel measure on $[-L,L]$, regarded as a positive functional on $\A$.  For a typical basic weak* neighbourhood $O$ of $\nu$, the resulting set $\X(\pi_n,O)$ has the form
\[\bigcap_{i=1}^m\Big\{v \in \bbC^{\oplus d_n}:\ \int g_i\,d\nu- \eps < \sum_{s \in C_n}\|P_{n,s}v\|^2g_i(s) < \int g_i\,d\nu + \eps\Big\}\]
for some $g_1,\dots,g_m \in C[-L,L]$ and some $\eps > 0$. 
If $\nu$ is not supported on $C$, then we can find such a neighbourhood $O$ for which $\X(\pi_n,O)$ is empty for all sufficiently large $n$.  If $\nu$ is supported on $C$, then we can check by hand that $\nu$ is asymptotically associated to $(\pi_n)_{n\ge 1}$, and then Theorem~\ref{thm:det-form} becomes
\[\inf_O \limsup_n \frac{1}{d_n}\log \vol_{2d_n}\X(\pi_n,O) = \int_{[-L,L]}\log \frac{d\nu_{\rm{ac}}}{d\mu}\ d\mu,\]
where the infimum runs over weak* neighbourhoods of $\nu$, and $\nu_{\rm{ac}}$ is the part of $\nu$ absolutely continuous with respect to $\mu$. This special case of Theorem~\ref{thm:det-form} lies close to many well-used estimates in random matrix theory, but I have not found a reference that contains it explicitly. \fin
\end{ex}

Even if $\phi$ is asymptotically associated to $\bspi$, it may happen that $\Delta \phi_{\rm{ac}} = 0$, in which case its logarithm is still equal to $-\infty$.  However, very importantly, there is at most one possible finite value for $\rmh_{\bspi}(\phi)$ once we know that $\tr_{d_n}\circ \pi_n\to \tau$.

In the sequel, we need to cite a reformulation of Theorem~\ref{thm:det-form} that gives an approximation for a single high-dimensional representation, rather than referring to an AP sequence.  This is the next theorem.

\begin{thm}\label{thm:APE4}
Fix $\tau$ and $\phi$ as above.
\begin{enumerate}
\item[a.] For any $h > \log \Delta \phi_{\rm{ac}}$, there are a positive integer $d_0$, a neighbourhood $V$ of $\tau$, and a neighbourhood $O$ of $\phi$ such that the following holds:
\begin{quote}
If $\pi$ is a $d$-dimensional representation of $\A$ satisfying
\[d\ge d_0 \quad \hbox{and} \quad \tr_d\circ \pi \in V,\]
then
\[\frac{\vol_{2kd}\X(\pi,O)}{v(d)^k} \le e^{hd}.\]
\end{quote}
\item[b.] If $\log \Delta \phi_{\rm{ac}} > h > \infty$, and if $O$ is any neighbourhood of $\phi$, then there are a positive integer $d_0$, a neighbourhood $V$ of $\tau$, and another neighbourhood $O'$ of $\phi$ such that the following holds:
\begin{quote}
If $\pi$ is a $d$-dimensional representation of $\A$ satisfying
\[d\ge d_0, \quad  \tr_d\circ \pi \in V, \quad \hbox{and} \quad \X(\pi,O')\ne \emptyset,\]
then
\[\frac{\vol_{2kd}\X(\pi,O)}{v(d)^k} \ge e^{hd}.\]
\end{quote}
\end{enumerate}
\end{thm}

\begin{proof}
Both parts can be proved from Theorem~\ref{thm:det-form} by a diagonal argument and a contradiction, similarly to the proof of Corollary~\ref{cor:APent-properties-trace}.  We give the details for part (a) and leave part (b) to the reader.

Suppose that part (a) is false for some $\tau$ and $\phi$. Then it fails for some $h > \log\Delta\phi$.  Because of Lemma~\ref{lem:lcsc}, we can choose countable neighbourhood bases $V_1\supset V_2\supset \cdots$ at $\tau$ and $O_1 \supset O_2\supset \cdots$ at $\phi$.  Now a diagonal argument gives a sequence $(\pi_n)_{n\ge 1}$ of finite-dimensional representations such that $d_n \ge n$ and $\tr_{d_n}\circ \pi_n \in V_n$ but
\[\frac{\vol_{2kd_n}\X(\pi_n,O_n)}{v(d_n)^k} \ge e^{hd_n}.\]
Letting $O$ be any other neighbourhood of $\phi$, we must have $O\supset O_n$ for all sufficiently large $n$, and at that point $\vol_{2kd_n}\X(\pi_n,O)$ is bounded below by the last inequality above.  Substituting this lower bound into Definition~\ref{dfn:annAPent} and taking the infimum over $O$, we arrive at the inequality $\rmh_{\bspi}(\phi)\ge h$, which contradicts Theorem~\ref{thm:det-form}.
\end{proof}

We can use Theorem~\ref{thm:det-form} to relate the limits of the sequences $(\S_k(\pi_n))_{n\ge 1}$, $k=1,2,\dots$, to properties of the functional $\rmh_{\bspi}$.  To do this, we study a general element $\phi$ of $\B(\A,\rmM_k)_+$ by using it to form the perturbations $\phi_t := \tau\otimes I_k + t\phi$ for $t\ge 0$.

\begin{cor}\label{cor:h-pi-and-strong-quot}
If $\tr_{d_n}\circ \pi \to\tau$ and $k$ is a positive integer, then
\[\rm{T}\limsup_{n\to\infty} \S_k(\pi_n)  = \big\{\phi \in \S_k(\A):\ \rmh_{\bspi}(\phi_t) \to 0\ \hbox{as}\ t\downarrow 0\big\}.\]
\qed
\end{cor}

See~\cite[Cor. 6.22]{APE4}.  If we replace `$\limsup$' with `$\liminf$' in the definition of $\rmh_{\bspi}$, then the same reasoning leads to a formula for $\rm{T}\liminf_{n\to\infty} \S_k(\pi_n)$.

In case the sequence $\S_k(\pi_n)$ converges to a limit in the Vietoris topology of $\calK(\S_k(\A))$, Theorem~\ref{thm:det-form} shows that the quantity $\rmh_{\bspi}$ is unchanged if we replace `$\limsup$' by `$\liminf$' in Definition~\ref{dfn:APent}, or if we pass from $\bspi$ to a subsequence.  Consequently, under this Vietoris-convergence assumption, we can re-cast the conclusion of Theorem~\ref{thm:det-form} as a large deviations principle.  For simplicity we explain this only when $k=1$, in which case it follows directly from Lemma~\ref{lem:use-of-spherical}. For each $n$, let $v_n$ be a uniformly random element of unit vector in $\bbC^{d_n}$.

\begin{cor}\label{cor:deterministic-LDP}
If $\tr_{d_n}\circ \pi_n \to \tau$ and the sequence $\S_1(\pi_n)$ converges in $\calK(\S_1(\pi_n))$, then the random functionals $\Phi^{\pi_n}_{v_n}$ obey the large deviations principle in $\S_1(\A)$ with speed $d_n$ and rate function $-\rmh_{\bspi}$. \qed
\end{cor}

By the contraction principle (see~\eqref{eq:contraction}), Corollary~\ref{cor:deterministic-LDP} also implies a large deviations principle for the scalar random variables $\langle \pi_n(a)v_n,v_n\rangle$ for any fixed $a \in \A$.

We can regard Theorem~\ref{mainthm:LDP} as an annealed generalization of Corollary~\ref{cor:deterministic-LDP} in a case where the representation $\pi_n$ is also random.

Another corollary of Theorem~\ref{thm:det-form} gives conditions for AP entropy to be additive or concave.

\begin{cor}\label{cor:AP-additive-concave}
Assume that $\tr_{d_n}\circ \pi_n \to\tau$, and let $\phi \in \B(\A,\rmM_k)_+$ and $\psi \in \B(\A,\rmM_\ell)_+$.
\begin{enumerate}
\item[a.] If $\rm{diag}(\phi,\psi)$ is asymptotically associated to $\bspi$, then
\[\rmh_{\bspi}(\rm{diag}(\phi,\psi)) = \rmh_{\bspi}(\phi) + \rmh_{\bspi}(\psi).\]
\item[b.] If $k = \ell$, $t \in [0,1]$, and $t\phi + (1-t)\psi$ is asymptotically associated to $\bspi$, then
\[\rmh_{\bspi}(t\phi + (1-t)\psi) \ge t\rmh_{\bspi}(\phi) + (1-t)\rmh_{\bspi}(\psi).\]
\end{enumerate}
\qed
\end{cor}

See~\cite[Cor. 6.23]{APE4} and~\cite[Cor. 6.25]{APE4} for parts (a) and (b), respectively.  (Those original corollaries are slightly more detailed.)

Both parts of Corollary~\ref{cor:AP-additive-concave} include an explicit assumption of asymptotic association: for $\rm{diag}(\phi,\psi)$ in part (a), and for $t\phi + (1-t)\psi$ in part (b).  If these assumptions fail, then it can happen that the left-hand side of the desired equality is $-\infty$, even though both terms on the right are finite: see~\cite[Ex. 6.26]{APE4}. In the case of part (b), this shows that the functional $\rmh_{\bspi}$ can fail to be concave, but only in this one specific way.

\section{Legendre transforms}\label{sec:AP-Legendre}


Consider the restriction of the function $-\log \Delta$ to the compact set $\S_k(\A)$.  This restriction is upper semicontinuous and concave by Proposition~\ref{prop:FK-det-properties}.  Its \textbf{Legendre transform} is defined on the subspace $\rmM_k(\A)_{\rm{sa}}$ of self-adjoint elements by
\begin{equation}\label{eq:Leg-transf}
p(a) := \sup\{\langle \phi,a\rangle + \log \Delta \phi:\ \phi \in \S_k(\A)\}.
\end{equation}
A general introduction to the Legendre transform in infinite dimensions can be found in~\cite[Sec. I.6]{SimSMLG}.  In the terminology of that reference, $(-\log \Delta,\S_k(\A))$ is a Fenchel pair, and so Fenchel's theorem shows that $-\log \Delta$ is recovered from $p$ by taking the Legendre transform again~\cite[Thm. I.6.4]{SimSMLG}.  A map $\phi \in \S_k(\A)$ attains the supremum in~\eqref{eq:Leg-transf} if and only if $\phi$ is a tangent to the function $p$ at $a$~\cite[Thm. I.6.6]{SimSMLG}, and the set of such tangents is a nonempty, compact, convex subset of $\S_k(\A)$.  We henceforth refer to these as the \textbf{$p$-tangent maps} at $a$, or `$p$-tangent states' if $k=1$.  We use the term `tangent' in the same way for some other related convex functions later.

\begin{rmks}$\phantom{i}$
\begin{enumerate}
\item Given that $\log \Delta$ is upper semicontinuous and concave, Fenchel's theorem actually implies the variational principle for Fuglede--Kadison determinants of completely positive maps from~\cite[Prop. 2.14]{APE4}.  However, in that paper, that principle is proved first and then used to prove upper semicontinuity and concavity.

\item Since $\Delta|\S_k(\A)$ is also upper semicontinuous and concave, we could take its Legendre transform rather than that of its logarithm.  This changes the specific shape of various auxiliary functions below, but leads to the same family of tangent states and a broadly similar analysis.  We do not discuss it further here. \fin
\end{enumerate}
\end{rmks}

We can largely reduce the study of $p$ to the case $k=1$ by replacing $\A$ with $\rmM_k(\A)$ and using the logarithm of equation~\eqref{eq:FK-det-pairing}.  We mostly restrict ourselves to this case for the rest of this section in order to lighten notation.

Theorem~\ref{thm:AP-equilibrium} below is the first main result of this section.  It determines $p(a)$ and the $p$-tangent states at $a$ in terms of other features of $a$.  Let $\s(a)$ denote the spectrum of $a$, and let $\s_{\max}(a)$ denote the maximum value in this set.

We eliminate one special case separately first.  If $\l(a)$ is a multiple of the identity, then $\phi(a) = \tau(a)$ whenever $\phi$ is a $\l$-normal state, and direct simplifications lead to
\begin{align*}
p(a) &= \sup\{t\tau(a) + (1-t)\s_{\max}(a) + \log t:\ 0 \le t \le 1\} \\ &= \left\{\begin{array}{ll} \tau(a) &\quad\hbox{if}\ \tau(a) \ge \s_{\max}(a)-1\\ \s_{\max}(a)-1 - \log(\s_{\max}(a)-\tau(a)) &\quad \hbox{if}\ \tau(a) < \s_{\max}(a)-1.\end{array}\right.
\end{align*}
Theorem~\ref{thm:AP-equilibrium} contains generalizations of the two cases in the last formula above, but they are easier to derive if we now fix this assumption:
\begin{equation}\label{eq:non-mult}
\l(a)\ \hbox{is not a multiple of the identity}.
\end{equation}

Our next step determines the cases of equality in the variational principle for $\Delta$.

\begin{prop}\label{prop:variational-equality}
Let $\phi \in \A_+^\ast$, and let $a \in \A$ satisfy $a \ge 0$ and $\Delta a > 0$.  Then
\begin{equation}\label{eq:var-again}
\Delta \phi \le \frac{\phi(a)}{\Delta a},
\end{equation}
and the following are equivalent:
\begin{enumerate}
\item[i.] equality holds in~\eqref{eq:var-again};
\item[ii.] we have $\phi(\sqrt{a}(\cdot)\sqrt{a}) = \phi(a)\cdot \tau$;
\item[iii.] the operator $\l(a)$ has trivial kernel, the operator $\l(a)^{-1/2}$ is affiliated to $\l(\A)''$ and square integrable with respect to $\tau$, and we have
\[\phi_{\rm{ac}} = \phi(a)\Phi^\l_{\l(a)^{-1/2}\xi} \qquad \hbox{and} \qquad \phi_{\rm{sing}}(a) = 0.\]
If $\phi_{\rm{sing}}$ is associated to $\pi_{\rm{sing}}$ by the cyclic vector $v$, then $v \in \ker \pi_{\rm{sing}}(a)$.
\item[iv.] we have $\phi(ab) = \phi(a)\tau(b)$ for all $b\in \A$.
\end{enumerate}
\end{prop}

\begin{proof}
Inequality~\eqref{eq:var-again} holds by the variational principle for $\Delta$, so we need only prove the conditions for equality.

\vspace{7pt}

\emph{(i) $\Rightarrow$ (ii).}\quad Defining $\psi:= \phi(\sqrt{a}(\cdot)\sqrt{a})$, Proposition~\ref{prop:FK-det-properties}(e) turns property (i) into $\Delta \psi = \psi(1)$.

Next, $\Delta\psi = \Delta\psi_{\rm{ac}}$, and Proposition~\ref{prop:FK-det-properties}(b) also holds for $\psi_{\rm{ac}}$ by itself, so we must have $\psi_{\rm{sing}}(1) = 0$ and hence $\psi = \psi_{\rm{ac}}$.

Now Proposition~\ref{prop:RadNik} gives $\psi_{\rm{ac}} = \Phi^\l_{T\xi}$ for some non-negative operator $T$ affiliated to $\l(\A)''$. If $E$ is the spectral resolution of $T$ on $[0,\infty)$, then property (i) turns into
\[\int_{[0,\infty)} \log t^2\ \langle E(dt)\xi,\xi\rangle = \log ((\Delta T)^2) = \log \|T\xi\|^2 = \log \int_{[0,\infty)} t^2\ \langle E(dt)\xi,\xi\rangle.\]
This means that Jensen's inequality is saturated for the strictly concave function $\log$ and the measure $\langle E(\cdot)\xi,\xi\rangle$, so that measure must be a point mass, implying that $T$ itself is a multiple of the identity.  This gives that $\psi = c\tau$ for some $c \ge 0$, and finally $c = \psi(1) = \phi(a)$.

\vspace{7pt}

\emph{(ii) $\Rightarrow$ (iii).}\quad Property (ii) and the Lebesgue decomposition give
\[\phi(a)\tau = \phi(\sqrt{a}(\cdot)\sqrt{a}) = \phi_{\rm{ac}}(\sqrt{a}(\cdot)\sqrt{a}) + \phi_{\rm{sing}}(\sqrt{a}(\cdot)\sqrt{a}).\]

The second term is still $\l$-singular, so it must vanish, giving
\[\langle \pi_{\rm{sing}}(a)v,v\rangle = \phi_{\rm{sing}}(a) = \phi_{\rm{sing}}(\sqrt{a}\cdot 1 \cdot \sqrt{a}) = 0.\]
Since $a\ge 0$, this also implies that $\pi_{\rm{sing}}(a)v=v$.

On the other hand, Proposition~\ref{prop:RadNik} lets us write $\phi_{\rm{ac}} = \Phi^\l_v$ for some vector $v \in H_\l$, and then $\phi_{\rm{ac}}(\sqrt{a}(\cdot)\sqrt{a}) = \Phi^\l_{\l(a)^{1/2}v}$.  Since this equals $\phi(a)\tau$, the uniqueness of associating vectors up to isomorphism implies that $\l(a)^{1/2}v = \sqrt{\phi(a)}U\xi$ for some unitary $U$ in $\l(\A)'$.  Since $U$ commutes with $\l$, this re-arranges to $\l(a)^{1/2}(U^\ast v) = \sqrt{\phi(a)}\xi$. Using $U^\ast v$ to associate $\phi_{\rm{ac}}$ to $\l$, this becomes the desired formula.

\vspace{7pt}

\emph{(iii) $\Rightarrow$ (iv).}\quad Substitute $ab$ into the description of $\phi$ given by property (iii). 

\vspace{7pt}

\emph{(iv) $\Rightarrow$ (i).}\quad  Applying property (iv) with $ba$ in place of $b$, and using the tracial identity, we obtain
\[\phi(aba) = \phi(a)\tau(ba) = \phi(a)\tau(\sqrt{a}b\sqrt{a}) \qquad (b \in \A),\]
and hence $\phi(a(\cdot)a) = \phi(a)\tau(\sqrt{a}(\cdot)\sqrt{a})$.  Therefore
\[(\Delta a)^2\Delta \phi = \Delta(\phi(a(\cdot)a)) = \phi(a)\Delta(\tau(\sqrt{a}(\cdot)\sqrt{a})) = \phi(a) \Delta a.\]
Since $\Delta a > 0$, this implies property (i).
\end{proof}

\begin{rmk}
If $a$ is invertible, then we can pass between properties (ii) and (iv) above directly using the tracial identity for $\tau$, but the general case needs a little more care. \fin
\end{rmk}

Consider again $a\in \A_{\rm{sa}}$. Theorem~\ref{thm:AP-equilibrium} refers to some other objects constructed from $a$. First, the state $\tau$ defines a spectral measure $\mu_a$ for $a$ supported on $\s(\l(a))$, and $\mu_a$ is not a point mass because of~\eqref{eq:non-mult}. Now define the functions
\begin{equation}\label{eq:Bor-Sti}
F_a(z) := \int \frac{1}{z-x}\,d\mu_a(x) \qquad (z > \s_{\max}(\l(a)))
\end{equation}
and
\begin{equation}\label{eq:log-pot}
U_a(z) := \int\log \frac{1}{z-x}\,d\mu_a(x) \qquad (z > \s_{\max}(\l(a))).
\end{equation}
Both $F_a$ and $U_a$ are smooth and strictly decreasing on $(\s_{\max}(\l(a)),\infty)$, so we may extend both of them to $z = \s_{\max}(\l(a))$ by taking their right-hand limits at this point, allowing the value $+\infty$ if necessary.  Since $\log$ is sub-linear at large arguments, we have $U_a(\s_{\max}(\l(a))) < \infty$ if $F_a(\s_{\max}(\l(a))) < \infty$. At the other extreme, we have
\[F_a(z) \to 0 \qquad \hbox{and} \qquad U_a(z) \to -\infty \qquad \hbox{as}\ z\to\infty.\]
In particular, $F_a$ is a decreasing homeomorphism
\[\hbox{from}\quad (\s_{\max}(\l(a)),\infty) \quad \hbox{to} \quad \big(0,F\big(\s_{\max}(\l(a))\big)\big),\]
with an extension by continuity to $\s_{\max}(\l(a))$ if $F_a(\s_{\max}(\l(a))) < \infty$.

The formulas for $F_a$ and $U_a$ can be extended to non-real arguments $z$, although we do not need these in the sequel.  These extensions define the negative of the Borel--Stieltjes transform of $\mu_a$~\cite[eqn. (1.2.6)]{SimOPUCI} and the log-potential of $\mu_a$~\cite[Sec. 8.1]{SimOPUCI}, respectively.  These transforms of $\mu_a$ appear widely in free probability and random matrix theory: see, for instance, the discussion and references in~\cite{GuiMai05}.


If $z > \s_{\max}(\l(a))$, or if $z = \s_{\max}(\l(a))$ and $F_a(\s_{\max}(\l(a))) < \infty$, then we also define a new state on $\A$ by
\begin{equation}\label{eq:gamma-c}
\g_{a,z} := \frac{1}{F_a(z)}\Phi^\l_{(z - \l(a))^{-1/2}\xi}.
\end{equation}
This satisfies
\[\g_{a,z}(a) = \frac{1}{F_a(z)}\int \frac{x}{z-x}\,d\mu_a(x) = \frac{1}{F_a(z)}\int \Big(\frac{z}{z-x} - 1\Big)\,d\mu_a(x) = z - \frac{1}{F_a(z)}\]
and
\[\log\Delta \g_{a,z} = \int \log\frac{1}{F_a(z)(z-x)}\,d\mu_a(x) = U_a(z) - \log F_a(z).\]
When $z > \s_{\max}(a)$, an alternative formula is
\[\g_{a,z} = \frac{1}{F_a(z)}\tau\big((z-a)^{-1/2}(\cdot)(z-a)^{-1/2}\big).\]

When $z > \s_{\max}(\l(a))$, differentiating under the integral sign gives
\begin{equation}\label{eq:grad-derivative}
\frac{d\phantom{z}}{d z}\Big(\frac{1}{F_a(z)}\Big) = - \frac{F_a'(z)}{F_a(z)^2} = \int \frac{1}{(z-x)^2}\,d\mu_a(x) \Big/ \Big(\int \frac{1}{z-x}\,d\mu_a(x)\Big)^2.
\end{equation}
Since $\mu$ is a probability measure but not a point mass, Jensen's inequality shows that this expression is strictly greater than $1$.  Combined with the calculation above, it follows that $\g_{a,z}(a)$ is strictly decreasing with $z$.  This function of $z$ extends to $\s_{\max}(\l(a))$, and it is a decreasing homeomorphism
\begin{equation}\label{eq:gamma-param}
\big[\s_{\max}(\l(a)),\infty\big)  \to \Big(\tau(a),\s_{\max}(\l(a)) - \frac{1}{F_a(\s_{\max}(\l(a))}\Big].
\end{equation}

The states $\g_{a,z}$ are basic ingredients for describing the $p$-tangent states at $a$.  They are roughly analogs of exponential families of distributions in elementary information theory (see~\cite[Sec. 12.1]{CovTho06}, for example) or of Gibbs measures in statistical physics (see~\cite[Sec. III.2--4]{SimSMLG} or~\cite[Sec. 7.2]{Gri--graph-book}).  However, as the next theorem shows, some $p$-tangent states also have a $\l$-singular component.

Throughout the remainder of section, we write $\pi_\phi$ for the GNS representation of any state $\phi$.

\begin{thm}\label{thm:AP-equilibrium}
Assume~\eqref{eq:non-mult} and let $m:= \s_{\max}(a)$.
\begin{enumerate}
\item[a.] If $F_a(m) \ge 1$, and $z \in [m,\infty)$ is the unique value satisfying $F_a(z) = 1$, then $\g_{a,z}$ is the unique $p$-tangent state at $a$ and $p(a) = z - 1 + U_a(z)$.
\item[b.] If $F_a(m) < 1$, then $\phi$ is a $p$-tangent state at $a$ if and only if
\[\phi = F_a(m)\gamma_{a,m} + (1-F_a(m))\phi_0\]
for some $\l$-singular state $\phi_0$ satisfying $\phi_0(a) = m$, and $p(a) = m - 1 + U_a(m)$. In this case, $\s_{\max}(\pi_\phi(a)) = m$, $\phi$ is unique if and only if $\phi_0$ is unique, and if this holds then the GNS representation of $\phi_0$ is irreducible.
\end{enumerate}
\end{thm}

\begin{proof}
Within this proof we drop the subscript `$a$' from $F_a$, $U_a$, and $\g_{a,z}$.

Consider any $z \in [m,\infty)$ such that $F(z)$ and $U(z)$ are finite. Then $z-a \ge 0$ and we have $\Delta(z-a) = e^{-U(z)} > 0$ and $\phi(z-a) = z-\phi(a)$. We may therefore apply Proposition~\ref{prop:variational-equality} using the element $z-a$ of $\A$ to obtain
\begin{equation}\label{eq:pressure-upper}
\phi(a) + \log \Delta \phi \le \phi(a) + \log (z-\phi(a)) + U(z),
\end{equation}
with equality if and only if $\phi_{\rm{ac}} = (z-\phi(a))\cdot F(z)\cdot \g_z$ and $\phi_{\rm{sing}}(z-a) = 0$. Moreover, by comparing $\log$ to its supporting tangent at $1$, the right-hand side of~\eqref{eq:pressure-upper} is at most
\[\phi(a) + (z-\phi(a)) - 1 + U(z)= z-1+U(z),\]
with equality if and only if $\phi(a) = z-1$.  Combining the conditions for equality, these two upper bounds are saturated simultaneously if and only if
\begin{equation}\label{eq:equal-together}
\phi_{\rm{ac}} = F(z)\cdot \g_z, \qquad \phi(a) = z-1, \qquad\hbox{and} \qquad \phi_{\rm{sing}}(z-a) = 0.
\end{equation}

If $F(m) \ge 1$, then we apply this reasoning with $z$ chosen so that $F(z) =1$.  The first condition in~\eqref{eq:equal-together} reduces to $\phi = \g_z$, which forces $\phi_{\rm{sing}}$ to vanish and satisfies $\phi(a) = z-1/F(z) = z-1$ as required. 

If $F(m) < 1$, then we apply the reasoning above with $z=m$.  The first condition in~\eqref{eq:equal-together} becomes $\phi_{\rm{ac}} = F(m)\g_m$. This implies that $\phi_{\rm{sing}} = (1-F(m))\phi_0$ for some $\l$-singular state $\phi_0$, and now the third condition in~\eqref{eq:equal-together} gives $\phi_0(a) = m$.  Finally, the state $\phi = F(m)\g_m + (1-F(m))\phi_0$ does satisfy all three conditions in~\eqref{eq:equal-together} because
\[F(m)\g_m(a) + (1-F(m))\phi_0(a) = F(m)\Big(m-\frac{1}{F(m)}\Big) + (1-F(m))m= m-1.\]
The explicit form for $\phi$ also gives $\s_{\max}(\pi_\phi(a)) = m$.

Whenever $F(m) \ge 1$, our formula for $\phi$ specifies it uniquely.  If $F(m) < 1$, then the formula shows that $\phi$ is unique if and only if $\phi_0$ is unique.  If uniqueness holds in the latter case, then it remains to show that the GNS representation of $\phi_0$ is irreducible, or equivalently that $\phi_0$ is an extreme poinit of $\S_1(\A)$ (see~\cite[Prop. 2.5.4]{Dix--Cstar}).  So suppose that $\phi_0 = r\phi_1 + (1-r)\phi_2$ for some $r \in (0,1)$ and states $\phi_1$ and $\phi_2$.  The optimality of $\phi_0$ requires that $\phi_1(a) = \phi_2(a) = \phi_0(a)$, and now the uniqueness of this optimizer gives $\phi_1 = \phi_2 = \phi_0$, as required.
\end{proof}

From Theorem~\ref{thm:AP-equilibrium}, we see that $p(a)$ depends on only $\mu_a$ and $\s_{\max}(a)$.  For example, it does not depend on any other points in the set difference $\s(a)\setminus \s(\l(a))$.

The commutative case of Theorem~\ref{thm:AP-equilibrium} is covered by known results: see in particular~\cite[Thms. 2.10 and 3.5]{BorLew93},~\cite{GamGas94} and~\cite[Thm. 2.1]{GamGas97}, as well as the further references discussed in those.  If $\A$ is commutative, then we may identify it with $C(X)$ for some compact metrizable space $X$, identify $\tau$ with a Borel probability measure on $X$ by the Riesz--Markov theorem, and let $a$ be a real-valued continuous function on $X$.  This case of Theorem~\ref{thm:AP-equilibrium} identifies which Borel probability measures $\nu$ on $X$ maximize the expression
\begin{equation}\label{eq:Burg-tilt}
\int a\,d\nu + \int \log \frac{d\nu_{\rm{ac}}}{d\tau}\ d\tau,
\end{equation}
where $\nu_{\rm{ac}}$ is the part of $\nu$ which is absolutely continuous with respect to $\tau$.  The maximizers are convex combinations of the form
\[\frac{t}{F_a(z)(z-a(x))}d\tau(x) + (1-t)d\nu_0(x) \qquad (x\in X)\]
for some $z \ge m$ and some $t \in (0,1]$, where $\nu_0$ is supported on the subset of $X$ where $a$ attains its maximum.  Parts (a) and (b) of Theorem~\ref{thm:AP-equilibrium} tell us how to determine these parameters, and in particular whether the optimal $\nu$ is absolutely continuous with respect to $\tau$ or not.

More recently, the function of $\nu$ in~\eqref{eq:Burg-tilt} has been related to the rate function in a large deviations principle for a certain model of random probability measures, and the form of the maximizing measures has been interpreted in that light~\cite[Rem. 2.1]{GamLozCha04}.  

Given a general separable unital C*-algebra $\A$ and a self-adjoint element $a$ which satisfies~\eqref{eq:non-mult}, we can let $\A_0$ be the commutative C*-subalgebra generated by $\{1,a\}$, restrict all of our positive functionals to $\A_0$, and observe that $p(a)$ remains the same.  So allowing a non-commutative C*-algebra does not really increase the generality of $p$.  However, Theorem~\ref{thm:AP-equilibrium} also identifies the $p$-tangent states on the whole of $\A$, even when this algebra is not commutative, and we need this extra strength for some applications in Part~\ref{part:free}.

Now consider the general case of Theorem~\ref{thm:AP-equilibrium} again.  It lets us classify a self-adjoint element $a$ of $\A$ according to the following phases:
\begin{itemize}
\item It is \textbf{$\Delta$-subcritical} if $F_a(\s_{\max}(a)) > 1$.  In this case the unique $p$-tangent state at $a$ has the form $\g_{a,z}$ given by Theorem~\ref{thm:AP-equilibrium}(a), and so is $\l$-normal.
\item It is \textbf{$\Delta$-critical} if $F_a(\s_{\max}(a)) = 1$.  In this case the unique $p$-tangent state equals $\g_{a,\s_{\max}(a)}$, and is still $\l$-normal.
\item It is \textbf{$\Delta$-supercritical} if $F_a(\s_{\max}(a)) < 1$.  In this case the $p$-tangent states are proper mixtures of $\l$-normal and $\l$-singular parts, with the structure and proportions of those parts described by Theorem~\ref{thm:AP-equilibrium}(b).
\end{itemize}

The subset of $\Delta$-subcritical elements of $\A_{\rm{sa}}$ is nonempty and star-shaped around $0$.  To see the phase transition more clearly, let us see how the $p$-tangent states identified by Theorem~\ref{thm:AP-equilibrium} vary along the one-parameter family $(\b a:\ \b > 0)$.  The spectral measure $\mu_{\b a}$ is the pushforward of $\mu_a$ under dilation by $\b$, and performing this change of variables in~\eqref{eq:Bor-Sti} and~\eqref{eq:log-pot} gives
\begin{equation}\label{eq:F-U-dilate}
F_{\b a}(z) = \frac{1}{\b}F_a\Big(\frac{z}{\b}\Big) \qquad \hbox{and} \qquad U_{\b a}(z) = U_a\Big(\frac{z}{\b}\Big) - \log \b \qquad (z \ge \b \s_{\max}(\l(a))).
\end{equation}
Now these identities lead to
\begin{equation}\label{eq:gamma-dilate}
\g_{\b a,z} = \frac{1}{F_{\b a}(z)}\Phi^\l_{(z-\b\l(a))^{-1/2}\xi} = \frac{\b}{F_a(z/\b)}\Phi^\l_{\b^{-1/2}((z/\b)-\l(a))^{-1/2}\xi} = \g_{a,z/\b}.
\end{equation}
Using these relations, Theorem~\ref{thm:AP-equilibrium} translates into the following corollary.  To lighten notation, we also adjust the parameter $z$ by a factor of $\b$ in its statement.

\begin{cor}\label{cor:AP-equilibrium}
Assume~\eqref{eq:non-mult}, let $\b > 0$, and let $m := \s_{\max}(a)$.  Let $F_a^{-1}$ denote the inverse of $F_a$ regarded as a homeomorphism from $(0,F(\s_{\max}(\l(a))))$ to $(\s_{\max}(\l(a)),\infty)$, and extended to the end-point $F(\s_{\max}(\l(a)))$ if this is finite.
\begin{enumerate}
\item[a.] If $\b \le F_a(m)$, and $z:= F_a^{-1}(\b)$, then $\g_{a,z}$ is the unique $p$-tangent state at $\b a$ and
\[p(\b a) = \b z - 1 + U_a(z) - \log \b.\]
\item[b.] If $\b > F_a(m)$, then $\phi$ is a $p$-tangent state at $\b a$ if and only if
\[\phi = \frac{F_a(m)}{\b}\gamma_{a,m} + \Big(1-\frac{F_a(m)}{\b}\Big)\phi_0\]
for some $\l$-singular state $\phi_0$ satisfying $\phi_0(a) = m$, and
\[p(\b a) = \b m - 1 + U_a(m) - \log \b.\]
In this case, $\phi$ is unique if and only if $\phi_0$ is unique, and if this holds then the GNS representation of $\phi_0$ is irreducible. \qed
\end{enumerate}
\end{cor}

In the context of Corollary~\ref{cor:AP-equilibrium}, let us call $\b_{\rm{c}} := F_a(m)$ the \textbf{critical value} for $a$ relative to $\tau$.  In case $\b_{\rm{c}} = \infty$, only part (a) of Corollary~\ref{cor:AP-equilibrium} is needed, and all scalar multiples of $a$ are $\Delta$-subcritical.  However, if $\b_{\rm{c}} < \infty$, then $\b a$ undergoes a transition from $\Delta$-subcritical to $\Delta$-supercritical as $\b$ passes $\b_{\rm{c}}$.  When $\b \le \b_{\rm{c}}$, the unique $p$-tangent state is still the $\l$-normal state $\g_{a,z}$ for $z = F_a^{-1}(m)$.  When $\b > \b_{\rm{c}}$, it is preferable to mix $\g_{a,m}$ with a singular component in proportions determined by $\b$. 

If $\b_{\rm{c}} < \infty$, then a calculation from Corollary~\ref{cor:AP-equilibrium} using the chain rule yields
\begin{equation}\label{eq:Legendre-gradient}
\frac{d\phantom{\b}}{d\b}p(\b a) = \left\{\begin{array}{ll}F_a^{-1}(\b) - 1/\b &\quad \hbox{if}\ \b < \b_{\rm{c}}\\ m - 1/\b &\quad \hbox{if}\ \b > \b_{\rm{c}}.\end{array}\right.
\end{equation}
For $z > m$, the value $\g_{a,z}(a)$ extends to a decreasing homeomorphism as in~\eqref{eq:gamma-param}, as explained following the calcluation~\eqref{eq:grad-derivative}. Composing this map with $F_a^{-1}$ shows that the derivative in~\eqref{eq:Legendre-gradient} is continuous and strictly increasing on both $(0,\b_{\rm{c}})$ and $(\b_{\rm{c}},\infty)$.  In addition, both of the formulas in~\eqref{eq:Legendre-gradient} give the same value $m - 1/\b_{\rm{c}}$ at $\b = \b_{\rm{c}}$, so the derivative exists and it continuous at the critical point also.  In the terminology of statistical physics, the phase transition at $\b_{\rm{c}}$ is of `higher order': see~\cite[eqn. (I.1.8)]{SimSMLG} and the ensuing discussion.  In general, the graph of $p(\b a)$ as a function of $\b$ looks something like the purple-red or purple-blue curves in Figure~\ref{fig:phase-trans} .

Tangent states at $a$ for the function $p$ have the property that they that maximize $\Delta$ subject to a linear constraint given by $a$.  This amounts to an infinite-dimensional application of Lagrange multipliers. As a function of $z$, the value $\g_{a,z}(a)$ extends to a decreasing homeomorphism as in~\eqref{eq:gamma-param}.  Therefore, for any $r \in (\tau(a),m)$, we may tune the parameter $\b$ in Corollary~\ref{cor:AP-equilibrium} to find the states $\phi$ that maximize $\Delta \phi$ subject to the constraint that $\phi(a) = r$.  They are described by part (a) or (b) of that corollary according as $r \le m-1/F_a(m)$ or $r > m-1/F_a(m)$.

In the range $\b < \b_{\rm{c}}$, the function of $\b$ in~\eqref{eq:Legendre-gradient} is called the \textbf{R-transform} of $\mu_a$.  It is important in free probability because it determines $\mu_a$ uniquely and is additive under free convolution~\cite{Voi86}.  It could be interesting to look for further connections between its role in that theory and in the work of the present section. 

We next consider how the function $p$ changes if we restrict the supremum in~\eqref{eq:Leg-transf} to states that are lifted from a quotient C*-algebra of $\A$.  Suppose that $\frak{I}$ is a closed ideal of $\A$ which is contained in $\ker \l$, and let $\ol{\A}:= \A/\frak{I}$.  Then there are quotient homomorphisms
\[\A \to \ol{\A} \to \l(\A),\]
and $\tau$ and $\l$ factorize through a tracial state $\ol{\tau}$ on $\ol{\A}$ and its GNS representation $\ol{\l}$.  Let $\ol{\Delta}$ be the Fuglede--Kadison determinant defined by $\ol{\tau}$. Given $a \in \A_{\rm{sa}}$, we write either $a + \frak{I}$ or $\ol{a}$ for its image in $\ol{\A}_{\rm{sa}}$. Let $m:= \s_{\max}(a)$ and $\ol{m} := \s_{\max}(\ol{a})$.  We have $\mu_{\ol{a}} = \mu_a$ and hence also $F_{\ol{a}} = F_a$ and $U_{\ol{a}} = U_a$, but it may happen that $\ol{m} < m$.

Now consider the quantity
\begin{equation}\label{eq:Leg-transf-2}
p_{\frak{I}}(a) := \sup\{\phi(a) + \log \Delta\phi:\ \phi \in \S_1(\A),\ \ker \pi_\phi \supset \frak{I}\}.
\end{equation}
A state $\phi$ on $\A$ satisfies $\ker \pi_\phi \supset \frak{I}$ if and only if $\phi$ is lifted from a state on $\ol{\A}$.  Therefore $p_{\frak{I}}(a)$ is equal to the quantity $\ol{p}(\ol{a})$ that we obtain directly from~\eqref{eq:Leg-transf} if we replace $\A$, $a$, and $\tau$ with $\ol{\A}$, $\ol{a}$, and $\ol{\tau}$.  The same reasoning about lifting states from $\ol{\A}$ also gives
\[\ol{m} := \sup\{\phi(a):\ \phi \in \S_1(\A),\ \ker \pi_\phi \supset \frak{I}\}.\]
Just as for $p$ itself, $p_{\frak{I}}$ is convex and lower semicontinuous, and the states $\phi$ that achieve equality in the supremum in~\eqref{eq:Leg-transf} are the tangent states of $p_{\frak{I}}$ at $a$.

Because $p_{\frak{I}}(a) = \ol{p}(\ol{a})$, we can apply Corollary~\ref{cor:AP-equilibrium} to describe $p_{\frak{I}}(\b a)$ as $\b$ varies.  The values $p(\b a)$ and $p_{\frak{I}}(\b a)$ agree for small enough $\b$, but for $p_{\frak{I}}(\b a)$ the critical value of $\b$ is $F_a(\ol{m})$, which is larger than $F_a(m)$ in case $\ol{m} < m$.  More generally, we can make the following comparison.

\begin{lem}\label{lem:diff-of-ps}
Let $\frak{I}_1$ and $\frak{I}_2$ be closed ideals contained in $\ker \l$, let ${m_i := \s_{\max}(a+\frak{I}_i)}$ for $i=1,2$, and assume that $m_1 < m_2$. Then
\[p_{\frak{I}_2}(\b a) - p_{\frak{I}_1}(\b a) \quad \hbox{is} \quad \left\{\begin{array}{l}\hbox{zero for}\ \b \le F_a(m_2)\\ \hbox{strictly increasing with $\b$ for $\b \ge F_a(m_2)$}\end{array}\right.\]
\end{lem}

\begin{proof}
Since $F_a$ is strictly decreasing, our assumptions give $F_a(m_1) > F_a(m_2)$.  When $\b \le F_a(m_2)$, Corollary~\ref{cor:AP-equilibrium}(a) gives the same value for both $p_{\frak{I}_2}(\b a)$ and $p_{\frak{I}_1}(\b a)$.  Beyond that, the two analogs of~\eqref{eq:Legendre-gradient} for $p_{\frak{I}_1}$ and $p_{\frak{I}_2}$ give
\[\frac{d\phantom{b}}{d\b}\big(p_{\frak{I}_2}(\b a) - p_{\frak{I}_1}(\b a)\big) = \left\{\begin{array}{ll}m_2 - F_a^{-1}(\b) &\quad \hbox{if}\ F_a(m_2) < \b \le F_a(m_1)\\ m_2 - m_1 &\quad \hbox{if}\ \b > F_a(m_1)\end{array}\right.\]
Since $F_a^{-1}$ is strictly decreasing, this is strictly positive whenever $\b > F_a(m_2)$.
\end{proof}

Figure~\ref{fig:phase-trans} shows an example of the situation in Lemma~\ref{lem:diff-of-ps}. The graphs of $p_{\frak{I}_1}(\b a)$ and $p_{\frak{I}_2}(\b a)$ are labeled `I' and `II', respectively.

\begin{figure}[h!tbp]
\centering
\includegraphics[width=16cm]{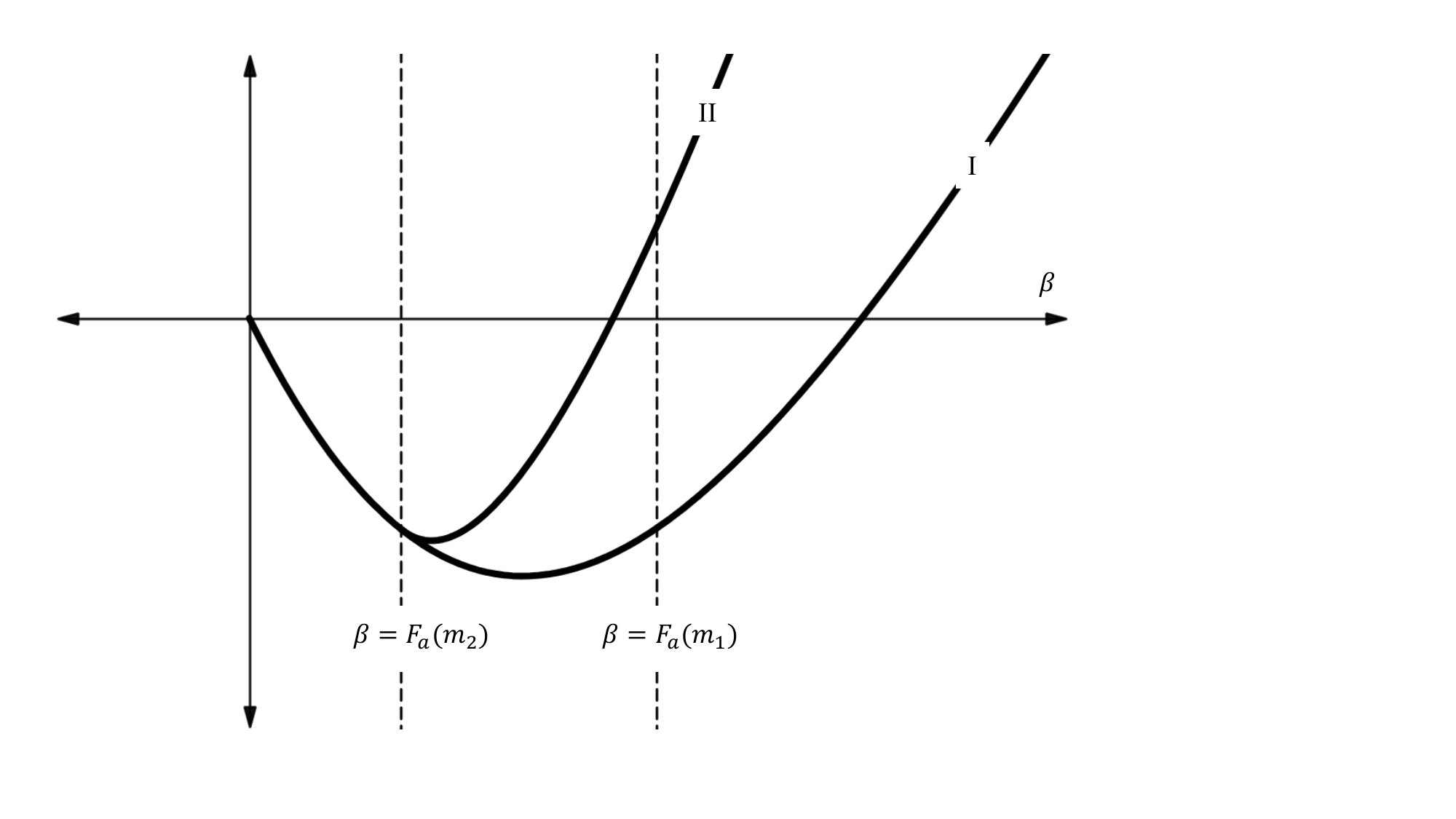}
\caption{\label{fig:phase-trans} Examples of $p_{\frak{I}_1}(\b a)$ (curve I) and $p_{\frak{I}_2}(\b a)$ (curve II).}
\end{figure}

\begin{ex}\label{ex:Green}
Let $\G$ be a countable group, let $\tau$ be the regular tracial state, and let $\frak{I} := \ker \l$. Let $a$ be a self-adjoint element of $C^\ast\G$, and let $m := \s_{\max}(a)$ and $\ol{m} := \s_{\max}(\ol{a}) = \s_{\max}(\l(a))$ as before.

If $\b > 0$ and $\phi$ is a $p$-tangent state at $\b a$, then $\phi$ satisfies condition (iv) from Proposition~\ref{prop:variational-equality} with $z-a$ in place of $a$, where $z$ is either given by Corollary~\ref{cor:AP-equilibrium}(a) or equal to $m$ in the case of Corollary~\ref{cor:AP-equilibrium}(b).  Written out in terms of the convolution multiplication of $\bbC[\G]$, this implies the resolvent condition that $(a \ast \phi)(g) = z\phi(g)$ for all $g \in \G\setminus \{e\}$.

The \textbf{Green's function of $a$} is defined by
\begin{equation}\label{eq:Green}
G_a(g,h,z) := \langle \l(h)(z-\l(a))^{-1}\xi,\l(g)\xi\rangle \qquad (g,h \in \G,\ z > \ol{m})
\end{equation}
(compare, for instance,~\cite[Sec. 1.4]{SimOPUCI} or~\cite[eqns. (10.9.8) and (11.12.17)]{SimOPUCII}). This satisfies $G_a(g,h,z) = G_a(e,g^{-1}h,z)$ for any $g$, $h$ and $z$.  Comparing with~\eqref{eq:gamma-c} we find that
\[\g_{a,z} = \frac{1}{F_a(z)}G_a(e,\cdot,z) \qquad (z > \ol{m}),\]
regarding both of these as positive definite functions on $\G$.  The resolvent condition satisfied by $\g_{a,z}$ as a $p$-tangent state is the standard characterization of $G_a(e,\cdot,z)$ as a fundamental solution to that convolutional resolvent equation.

Now let $\b > 0$ and let $\phi$ be a $p_{\frak{I}}$-tangent state at $\b a$, so $\phi$ is described by Corollary~\ref{cor:AP-equilibrium}.  If $\b < F_a(\ol{m})$, then $\phi$ is a multiple of $G_a(e,\cdot,F_a^{-1}(\b))$.  On the other hand, if $F_a(\ol{m}) < \b < \infty$, then $\phi$ is a linear combination of a limiting Green's function at $z = \ol{m}$ and another positive definite function which is tempered but $\l$-singular.  In Part~\ref{part:free} we see that non-Abelian free groups provide examples in which this phase is nonempty.

Alternatively, we can ignore $\frak{I}$ and consider the function $p$ on $C^\ast \G$.  The $p$-tangent states for $\b a$ are still multiples of Green's functions for sufficiently small $\b$, but their phase transition occurs sooner if it happens that $m > \ol{m}$.  Then the difference between $p$ and $p_{\frak{I}}$ is described by a special case of Lemma~\ref{lem:diff-of-ps}.  The specific curves in Figure~\ref{fig:phase-trans} are obtained in this way from the free group $\G$ on three generators and the group-algebra element $a - 2\delta_e$, where $a$ is the `adjacency element' (see Example~\ref{ex:adj}).  Methods for studying these particular examples are developed in Section~\ref{sec:Haa-as-equil} below. \fin

\end{ex}

In the sequel, we also need a slightly more flexible construction than passage to a quotient C*-algebra.  Assume again that~\eqref{eq:non-mult} holds. Suppose that $K$ is a nonempty closed subset of $\S_1(\A)$ which may be written in the form
\begin{equation}\label{eq:K-is-a-union}
K = \bigcup_{i \in I}\ol{\S_1(\rho_i)},
\end{equation}
where $(\rho_i:\ i \in I)$ is a family of separable unital representations which all satisfy $\rho_i \gtrsim \l$, and let $\frak{I} := \bigcap_{i \in I}\ker \rho_i$.

\begin{lem}\label{lem:K-and-I}
In the situation above, we have
\begin{equation}\label{eq:Leg-transf-I-K}
p_{\frak{I}}(a) = \sup\{\phi(a) + \log \Delta \phi:\ \phi \in K\} \qquad (a \in \A_{\rm{sa}}).
\end{equation}
If $T$ is the set of $p_{\frak{I}}$-tangent states at $a$, then $T$ is the closed convex hull of $T\cap K$.
\end{lem}

\begin{proof}
By our assumption on the representations $\rho_i$, we have $\frak{I} \subset \ker \l$, so we can define $p_{\frak{I}}$ as in~\eqref{eq:Leg-transf-2}.  Any element of $K$ factorizes through the quotient homomorphism $\A \to \ol{\A}$, so the inequality ``$\ge$'' in~\eqref{eq:Leg-transf-I-K} is immediate.  On the other hand, Theorem~\ref{thm:AP-equilibrium} describes the $p_{\frak{I}}$-tangent states explicitly, and now we can argue about these as follows:
\begin{itemize}
\item[a.] If $F_a(\ol{m}) \ge 1$, then the unique $p_{\frak{I}}$-tangent state at $a$ has the form $\g_{a,z}$ for some $z \ge \ol{m}$.  Since this state is associated to $\l$ and hence also to every $\rho_i$, it lies in $K$.
\item[b.] If $F_a(\ol{m}) < 1$, then the $p_{\frak{I}}$-tangent states at $a$ are the states of the form
\begin{equation}\label{eq:T0T}
\phi = F_a(\ol{m})\g_{a,\ol{m}} + (1-F_a(\ol{m}))\phi_0
\end{equation}
for some $\phi_0 \in T_0$, where
\[T_0 := \{\psi \in \S_1(\A):\ \ker \pi_\psi \supset \frak{I}\ \hbox{and}\ \psi(a) = \ol{m}\}.\]
If $\phi \in T_0$, then its GNS representation $\pi_{\phi_0}$ is weakly contained in the collection $(\rho_i:\ i \in I)$, and therefore $\phi_0$ lies in the closed convex hull of $K$~\cite[Thm. 3.4.4]{Dix--Cstar}.  Since $K$ is compact, it follows that $T_0$ is the closed convex hull of $T_0\cap K$.  Forming convex combinations with $\g_{a,\ol{m}}$ as in~\eqref{eq:T0T}, this implies that $T$ is the closed convex hull of $T \cap K$, and in particular that $T \cap K$ is nonempty.
\end{itemize}
In both cases, we have found a $p_{\frak{I}}$-tangent state at $a$ which belongs to $K$, and this state witnesses the equality in~\eqref{eq:Leg-transf-I-K}.
\end{proof}

\begin{rmk}\label{rmk:Paschke} 
Because of~\eqref{eq:T0T}, if $F_a(\ol{m}) < 1$, then the closed convex set of $p_{\frak{I}}$-tangent states at $a$ is an affine image of the set of states $\psi$ which annihilate $|\ol{m}-a|^2$.  The structure of the latter set has been investigated in various special cases.  For example, suppose that $\G$ is a free group, $\A = \ol{\A} = C^\ast_{\rm{r}}\G$, and $a$ is an element of $\bbC[\G]$.  In this case Paschke conjectured in~\cite{Pasch01} that this set of states should be the convex hull of only finitely many extreme points.  This is verified in some special cases in his paper, and in some others in~\cite{Pasch02,KuhSte03}.  When something like this can be proved, we may also obtain a more complete description of the $p$-tangent states at $\b a$ for each $\b > 0$. \fin
\end{rmk}

To finish this section, let us derive some consequences of the results above for AP entropy.  We return to the generality of $\rmM_k$-valued maps in order to make the following definition.

\begin{dfn}\label{dfn:APpressure}
Let $a \in \rmM_k(\A)_{\rm{sa}}$.  The \textbf{almost periodic} (`\textbf{AP}') \textbf{pressure of $a$ along $\bspi$} is the quantity
\[\rmp_{\bspi}(a) := \sup\big\{\langle \phi,a\rangle + \rmh_{\bspi}(\phi):\ \phi \in \S_k(\A)\big\}.\]
An element $\phi$ of $\S_k(\A)$ is an \textbf{$\rmh_{\bspi}$-equilibrium map for $a$} if it achieves the supremum on the right-hand side above.  When $k=1$ we write \textbf{equilibrium state} rather than `equilibrium map'.
\end{dfn}

Recall that, unlike $\log \Delta$, the function $\rmh_{\bspi}$ need not be concave, but this property can fail only in the special way described in Corollary~\ref{cor:AP-additive-concave}.  So we cannot identify $\rmh_{\bspi}$-equilibrium states directly with tangent states to $\rmp_{\bspi}$ as we did for the function $p$ defined in~\eqref{eq:Leg-transf}.

The terminology in Definition~\ref{dfn:APpressure} is analogous to standard usage in ergodic theory and statistical mechanics: see~\cite[Secs. 9.3--5]{Walters--book} or~\cite[Sec. 3.4]{SimSMLG}, respectively.  A more recent precedent is sofic pressure in ergodic theory: see~\cite[Sec. 9]{Bowen--survey} and the references given there.  In ergodic theory, equilibrium measures for topological pressure serve as distinguished elements of the set of all invariant probability measures, and for certain choices of potential function they capture other natural properties of a dynamical system: see the survey in~\cite{ChazKel12}, for example.  In the present book, $\rmh_{\bspi}$-equilibrium states and their analogs for annealed AP entropy play a key role in certain large deviations principles that we derive in Section~\ref{sec:op-norm-LDP} below. 

From the identification of $\S_k(\A)$ with $\S_1(\rmM_k(\A))$ and equation~\eqref{eq:h-pi-k}, we obtain
\[\rmp_{\bspi}(a) = k\rmp_{\bspi^{(k)}}(a/k) \qquad (a \in \rmM_k(\A)_{\rm{sa}}).\]
This identity lets us reduce basic results to the case $k=1$. We now focus again on that case for the sake of lighter notation.


If $\bspi$ is an AP sequence such that $\tr_{d_n}\circ \pi_n \to \tau$, then the set $K := \rm{T}\limsup_n \S_1(\pi_n)$ is closed, and also it has the form in~\eqref{eq:K-is-a-union}.  To see this, let $(\phi_i:\ i \in I)$ be the set of all positive functionals that are approximately associated to $\bspi$, indexed arbitrarily, and then let
\[\rho_i := (\pi_{\phi_i})_{\rm{sing}} \oplus \l \qquad (i \in I).\]
The set $K$ satisfies~\eqref{eq:K-is-a-union} for this family of representations because of Lemma~\ref{lem:typ-trans}.  As a result, Theorem~\ref{thm:det-form} gives $\rmh_{\bspi}(\phi) = \log \Delta \phi$ for $\phi \in K$ and $\rmh_{\bspi}(\phi) = -\infty$ for $\phi \in \S_1(\A)\setminus K$, and so Lemma~\ref{lem:K-and-I} gives $\rmp_{\bspi} = p_{\frak{I}}$ for $\frak{I} := \bigcap_i \ker \rho_i$.  Now the combination of Theorem~\ref{thm:AP-equilibrium} and Lemma~\ref{lem:K-and-I} evaluates the AP pressure along $\bspi$ and determines the $\rmh_{\bspi}$-equilibrium states.

\begin{ex}[\emph{Continuation of Example~\ref{ex:Green}}]\label{ex:Green2}
Let $\G$ be a countable group, and let $\bspi$ be an AP sequence for $\G$ that converges strongly to $\l$ (see~\cite{Magee--survey} for various examples).  Then the states constructed from Green's functions in Example~\ref{ex:Green} appear as the absolutely continuous parts of $\rmh_{\bspi}$-equilibrium states. \fin
\end{ex}

If $\bspi$ is a sequence satisfying $\tr_{d_n}\circ \pi_n\to \tau$ and also $\S_1(\pi_n)\to K$, then Corollary~\ref{cor:deterministic-LDP} identifies $-\rmh_{\bspi}$ as the rate function in a certain large deviations principle. Consequently, if $a \in \A_{\rm{sa}}$, then the Legendre transform $\rmp_{\bspi}(a)$ appears through Varadhan's asymptotic formula~\eqref{eq:Var} like this:
\begin{equation}\label{eq:spherical-1}
\frac{1}{d_n}\log \int e^{d_n\langle \pi_n(a)v,v\rangle}\ d\s_{2d_n-1}(v) \to \rmp_{\bspi}(a) \qquad \hbox{as}\ n\to\infty.
\end{equation}
At this point, Theorem~\ref{thm:AP-equilibrium} and Lemma~\ref{lem:K-and-I} determine the limiting exponent $\rmp_{\bspi}(a)$ in terms of properties of $a$ relative to $\tau$ and $K$.  Alternatively, we can identify $\rmp_{\bspi}$-equilibrium states as the states that maximize $\rmh_{\bspi}$ subject to a linear constraint defined by $a$, similarly to our previous discussion for Fuglede--Kadison determinants.  If that $\rmp_{\bspi}$-equilibrium state is unique, then Theorem~\ref{thm:AP-equilibrium} and Lemma~\ref{lem:K-and-I} can be used to derive either (i) an approximate description of $\Phi^{\pi_n}_v$ for those unit vectors $v$ that carry most of the `weight' in the integral in~\eqref{eq:spherical-1}, or (ii) a conditional limit theorem to accompany the large deviations principle in Corollary~\ref{cor:deterministic-LDP}, obtained from that corollary via Lemma~\ref{lem:cond-lim}.  In both cases one expects roughly an example of a `BBP transition', named after a transition from null to outlying behaviour in the largest eigenvalues of certain deformed models of Wishart matrices studied in~\cite{BaiBAPec05}. We do not present these more fully here. 

The integrals on the left-hand side of~\eqref{eq:spherical-1} are examples of `spherical integrals', and have been widely studied in random matrix theory.  The asymptotic behaviour of these particular examples was derived previously in~\cite[Thms. 2 and 6]{GuiMai05}.  Since these integrals depend only on the self-adjoint matrices $A_n := \pi_n(a)$, we may reduce their study to the case when $\A = C[-L,L]$ for some $L > 0$ and $a$ is the identity function on $[-L,L]$, regarded as an element of $\A$ (compare with Example~\ref{ex:one-op3}).  This is closer to the formulation in~\cite{GuiMai05}.  Several steps in that paper are related to ours, particularly in the proofs of~\cite[Lems. 18--21]{GuiMai05}.  On the other hand,~\cite{GuiMai05} also provides more quantitative estimates and a number of other variations, and as a result those authors can also handle some other regimes such as slowly-growing operator norms $\|A_n\|$.

Recently, spherical integrals such as those in~\eqref{eq:spherical-1} have been central to some new large deviations principles for largest eigenvalues of various random matrix models.  We return to this connection in Section~\ref{sec:op-norm-LDP} below.

In~\cite[Prop. 1]{GuiHus22}, Guionnet and Husson generalize~\eqref{eq:spherical-1} by giving asymptotics for
\begin{equation}\label{eq:GuiHus22}
\frac{1}{d_n}\log\int_{\bf{U}(k,n)}e^{d_n(\langle \pi_n(a)v_1,v_1\rangle+ \cdots + \langle \pi_n(a)v_k,v_k\rangle)}\ dm_{\bf{U}(k,n)}([v_1,\dots,v_k])
\end{equation}
as $n\to\infty$. Using this, those authors also generalize a number of previous random-matrix large deviations results involving single vectors to allow $k$-tuples of vectors.  To adapt the approach of the present section to the analysis of~\eqref{eq:GuiHus22}, one must first generalize the function $p$ from~\eqref{eq:Leg-transf} as follows.  Let $k\ge 1$, let $Q$ be a fixed positive semidefinite $k$-by-$k$ matrix with $\tr_k Q = 1$, and now for $b \in \rmM_k(\A)_{\rm{sa}}$ let
\[p_Q(b) := \sup\big\{\langle \phi,b\rangle + \log \Delta \phi:\ \phi \in \S_k(\A),\ \phi(1) = Q\big\}.\]
That is, we impose an additional linear constraint on $\phi(1)$.  I believe the work of the present section can be adapted to give formulas for $p_Q$, but several of the required ingredients become more complicated.  In particular, in place of the scalar variable $z$ as the argument for $F_b$ and $U_b$, one should now allow self-adjoint elements of the C*-subalgebra $\rmM_k\otimes 1_\A$ of $\rmM_k(\A)$. This leads to the family of states
\begin{multline*}
\g_{b,Z} := \frac{1}{\|\l(Z\otimes 1 - b)^{-1/2}\xi\|^2}\Phi^\l_{\l(Z\otimes 1 - b)^{-1/2}\xi} \\ (Z \in \rmM_{k,\rm{sa}},\ Z\otimes 1 - b\ \hbox{strictly positive}).
\end{multline*}
Tuning $Z$ here should give one enough control to fix $\phi(1)$ for a $p_Q$-tangent state $\phi$.  To determine the asymptotic in~\eqref{eq:GuiHus22}, one would take $Q = I_k$ and $b = a \otimes I_k$.  In general, one can show using Proposition~\ref{prop:FK-det-properties}(e) that $\Delta \phi \le \det \phi(1)$, so $p_Q = -\infty$ if $Q$ is singular.  One the other hand, if $Q$ is nonsingular, then one can reduce to the case $Q = I_k$ by transforming $\phi$ into $\psi := Q^{-1/2}\phi Q^{-1/2}$, using Proposition~\ref{prop:FK-det-properties}(e), and adjusting $a$ itself accordingly.  However, I have not looked further into adaptation of Theorem~\ref{thm:AP-equilibrium} to $p_{I_k}$.

The asymptotic analysis of~\eqref{eq:GuiHus22} is used in~\cite{GuiHusRek25} for the proof of the some related large deviations principles.  We discuss these more fully in Section~\ref{sec:op-norm-LDP} below.  In~\cite{GuiHusRek25}, the steps involved in conditioning on the `profile' of a $k$-tuple of vectors are the analog of fixing the matrix $Q$ above.

\subsection*{\emph{Notes and further references}}

Since its introduction in~\cite{FugKad51,FugKad52}, the Fuglede--Kadison determinant has played an increasingly important role in various connections of operator algebras to others parts of mathematics.  A succinct survey of some of these is given in~\cite{delaHar13}. A more thorough exposition with a view towards $L^2$-invariants in topology is given in~\cite[Sec. 3.2]{Luc02}; see also~\cite[Chap. 13]{Luc02} on the `approximation' and `determinant' conjectures.

For finite-dimensional but non-commutative algebras of matrices, some analogs of Theorem~\ref{thm:AP-equilibrium} are proved in~\cite{BakWoe95}, where they are used to characterize special solutions to some matrix completion problems related to those discussed in Chapter~\ref{chap:PSD}; see also~\cite[Sec. 1.4]{BakWoeMCetc}.

A general introduction to Borel--Stieltjes transforms and log-potentials in random matrix theory is given in~\cite[Sec. 2.4]{AndGuiZei--book} and the references for that section. The significance of the R-transform in free probability is discussed further in many sources, including~\cite{GuiMai05} and~\cite[Subs. 5.3.3]{AndGuiZei--book}.

\chapter{Random almost periodic sequences and notions of entropy}\label{chap:random-AP}

In this chapter we generalize AP entropy by allowing the individual representations in an AP sequence to be random.

\section{Random AP sequences}

Formally, assume that $\A$ is separable and unital, and let $\Rep_n(\A)$ denote the space of unital $\ast$-representations of $\A$ on $\bbC^n$ for each $n$. This is a space of maps $\A \to \B(\bbC^n)$, and we endow it with the topology of pointwise convergence as in~\cite[Subs. 3.5.2]{Dix--Cstar}.  This topology is Polish because $\A$ is separable~\cite[Prop. 3.7.1]{Dix--Cstar}.

Fix a background probability space $(\Omega,\F,\bbP)$. On this space, a \textbf{random almost periodic} (`\textbf{AP}') \textbf{sequence} is a sequence of random variables $\pi_n$ taking values in $\Rep_{d_n}(\A)$ for some divergent sequence of positive integers $(d_n)_{n \ge 1}$.  We always write $\bbE$ for expectation with respect to $\bbP$.

We only ever consider quantities that involve one value of $n$ at a time.  As a result, the properties we consider depend only on of individual distributions of each $\pi_n$; the coupling of those distributions is unimportant.  We assume that every $\pi_n$ is defined on the same background probability space only to lighten notation.  The reader wanting a definite choice may always assume that they are independent.

For example, let $\G$ be the group freely generated by a set $S$ of size $r$, and let $\A := C^\ast \G$.  Let $(d_n)_{n\ge 1}$ be a divergent sequence of positive integers, and let $\nu_n$ be a Borel probability measure on $\rmU(d_n)$ for each $n$.  Finally, for each $n$, choose a random element $\pi_n$ of $\Rep_{d_n}(\G)$ such that the generators $\pi_n(s) \in \rmU(d_n)$, $s \in S$, are independent and all have distribution $\nu_n$.  By the universal property of $C^\ast \G$ for unitary representations of $\G$, each $\pi_n$ extends uniquely to a random unitary representation of $C^\ast \G$.

Let $\bspi := (\pi_n)_{n\ge 1}$.  In this section we introduce two variants of AP entropy that account for the randomness in $\bspi$: `annealed' and `zeroth-order' AP entropy.  We establish some first properties of these which hold for any random AP sequence.

In Part~\ref{part:free} we study these quantities and their consequences in a special case: $\A = C^\ast \G$ for a group $\G$ which is freely generated by a set $S$ of size $r$, and we generate $\pi_n \in \Rep_{d_n}(\G)$ by choosing the generators $\pi_n(s) \in \rmU(n)$, $s \in S$, independently at random from Haar measure.  Several special features of this example lead to much more complete results.

\section{Annealed almost periodic entropy}\label{sec:ann-APent}

We now fix a random AP sequence $\bspi = (\pi_n)_{n\ge 1}$ for the rest of this chapter.  When needed, the dimension of $\pi_n$ is denoted by $d_n$.

Recall the ball-volume function $v$ from~\eqref{eq:ball-vol}.

\begin{dfn}[Annealed AP entropy]\label{dfn:annAPent}
Let $\phi \in \B(\A,\rmM_k)_+$. Its \textbf{annealed almost periodic} (`\textbf{AP}') \textbf{entropy along $\bspi$} is the quantity
\begin{equation}\label{eq:AnnAPent}
\rmh^\ann_{\bspi}(\phi) := \inf_O \limsup_{n\to\infty} \frac{1}{d_n}\log \bbE\frac{\vol_{2kd_n}\X(\pi_n,O)}{v(d_n)^k},
\end{equation}
where the infimum runs over all neighbourhoods of $\phi$.

Similarly, the \textbf{lower AP entropy of $\phi$ along $\bspi$} is the quantity
\begin{equation}\label{eq:lowerAnnAPent}
\ul{\rmh}^\ann_{\bspi}(\phi) := \inf_O \liminf_{n\to\infty} \frac{1}{d_n}\log \bbE\frac{\vol_{2kd_n}\X(\pi_n,O)}{v(d_n)^k},
\end{equation}
where the infimum runs over all neighbourhoods of $\phi$.
\end{dfn}

The term `annealed' is borrowed from the statistical physics of disordered materials: see, for instance,~\cite{MezParVir--book}.  Here it means that we apply the expectation directly to a random cardinality or a random volume in phase space.  Other options are available: for example, we briefly discuss the associated `quenched' average in Section~\ref{sec:three-entropy-cors} below. The annealed average is often easiest to study because Fubini's theorem allows us to exchange $\bbE$ and $\vol_{2kd_n}$ in the formula~\eqref{eq:AnnAPent}: this fact becomes crucial during our work in Part~\ref{part:free}.

If $\rmh^\ann_{\bspi}(\phi)$ and $\ul{\rmh}^\ann_{\bspi}(\phi)$ coincide, then this says that the `$\limsup$' and `$\liminf$' in their definitions become closer and closer as $O$ shrinks around $\phi$.  If this holds, then it can be seen as a rather rough mode of convergence for the expression inside the $\limsup$ of~\eqref{eq:AnnAPent}, regarded as a function of $O$.  A similar kind of convergence characterizes large deviations principles (see Appendix~\ref{chap:LDP-prelims}).  This connection becomes closer in Part~\ref{part:free}, where we prove that $\rmh^\ann_{\bspi}$ and $\ul{\rmh}^\ann_{\bspi}$ coincide for a uniformly random AP sequence for free groups and relate this fact to certain true large definitions principles for those random AP sequences.  We introduce notation for both $\rmh^\ann_{\bspi}$ and $\ul{\rmh}^\ann_{\bspi}$ for convenience during that work later.

In the rest of this section we largely focus on $\rmh^\ann_{\bspi}$.  The same reasoning gives analogous results for $\ul{\rmh}^\ann_{\bspi}$, but we omit these unless they are needed later.

In case each $\pi_n$ is deterministic, $\rmh^\ann_{\bspi}(\phi)$ simplifies back to Definition~\ref{dfn:APent}.  In Definition~\ref{dfn:annAPent}, as previously, we may restrict $O$ to lie in any chosen neighbourhood base at $\phi$ without changing the value of $\rmh^\ann_{\bspi}(\phi)$, by monotonicity.

Since $\rmh^\ann_{\bspi}(\phi)$ is an infimum over neighbourhoods of $\phi$, Lemma~\ref{lem:upper-semicts} gives the following extension of Lemma~\ref{lem:APent-usc}.

\begin{lem}\label{lem:hann-usc}
Each restriction $\rmh^\ann_{\bspi}|\B(\A,\rmM_k)_+$ is upper semicontinuous. \qed
\end{lem}

In ergodic theory, the analogous generalization from deterministic to annealed sofic entropy appears in the second displayed equation after~\cite[Definition 4]{Bowen10c}.

We can also extend Proposition~\ref{prop:APent-properties} with only small changes to the proofs.

\begin{prop}\label{prop:hann-properties}
Let $\phi \in \B(\A,\rmM_k)_+$ be positive definite.
\begin{itemize}
\item[a.] We have $\rmh^\ann_{\bspi}(\phi) \le \log \det \phi(1)$. In particular, if $\phi(1)$ is singular, then $\rmh^\ann_{\bspi}(\phi) = -\infty$.
\item[b.] If $Q$ is an invertible $k$-by-$k$ matrix and
\[\psi(a) := (Q^\rm{T})^\ast \phi(a) Q^\rm{T} \qquad (a \in \A),\]
then
\[\rmh^\ann_{\bspi}(\psi) = 2\log|\det Q| + \rmh^\ann_{\bspi}(\phi).\]
The analogous formula holds if $\rmh^\ann_{\bspi}$ is replaced by $\ul{\rmh}^\ann_{\bspi}$ on both sides.
\end{itemize}
\end{prop}

\begin{proof}
\emph{Part (a).}\quad Let $h > \log \det \phi(1)$.  Theorem~\ref{thm:types-1}(b) (the `method of types' for log-determinants) gives a neighbourhood $O$ of $\phi(1)$ in $\rmM_{k+}$ such that the sets
\[T(d,O) := \{X \in \rmM_{d,k}:\ X^\ast X \in O\}\]
satisfy
\begin{equation}\label{eq:phi-1-upper-bound}
\frac{\vol_{2dk}T(d,O)}{v(d)^k} \le e^{hd + o(d)} \qquad \hbox{as}\ d\to\infty.
\end{equation}
Now let
\[U := \{\psi \in \B(\A,\rmM_k)_+:\ \psi(1) \in O\}.\]
Then
\[\X(\pi_n,U)\subset T(d_n,O)\]
for every $n$ and every value of $\pi_n$. Therefore, by~\eqref{eq:phi-1-upper-bound}, the neighbourhood $U$ witnesses the upper bound $\rmh^\ann_{\bspi}(\phi) \le h$.  By the arbitrariness in $h$, this completes the proof.

\vspace{7pt}

\emph{Part (b).}\quad Let $O$ be any neighbourhood of $\psi$.  Then Lemma~\ref{lem:lin-maps} gives a neighbourhood $U$ of $\phi$ such that
\[(I_{d_n}\otimes Q)[\X(\pi_n,U)]\subset \X(\pi_n,O)\]
for every $n$ and every value of $\pi_n$, 
and hence
\begin{align*}
\vol_{2kd_n}\X(\pi_n,O) &\ge |\det (I_{d_n}\otimes Q)|^2\cdot \vol_{2kd_n}\X(\pi_n,U) \\
&= |\det Q|^{2d_n}\cdot \vol_{2kd_n}\X(\pi_n,U)
\end{align*}
for every $n$ and every value of $\pi_n$. The determinant here is squared because we must regard $I_{d_n}\otimes Q$ as a linear map in $2kd_n$ real dimensions for the purpose of computing volumes (see, for instance,~\cite[Subs. 1.3.5]{RudinFTUB}). Inserting this into Definition~\ref{dfn:annAPent}, we obtain
\[\limsup_{n\to\infty}\frac{1}{d_n}\log \bbE\frac{\vol_{2kd_n}\X(\pi_n,O)}{v(d_n)^k} \ge 2\log |\det Q| + \rmh^\ann_{\bspi}(\phi),\]
and similarly with the limit infimum on the left and the quantity $\ul{\rmh}^\ann_{\bspi}(\phi)$ on the right. By the arbitrariness of $O$, this gives the desired inequalities for both $\rmh^\ann_{\bspi}(\phi)$ and $\ul{\rmh}^\ann_{\bspi}(\phi)$ in one direction.  It holds in the other direction by applying the same argument with $\phi$ and $\psi$ switched and with $Q^{-1}$ in place of $Q$.
\end{proof}

Our next result generalizes Lemma~\ref{lem:use-of-spherical} in two ways.  First, it is a version for annealed AP entropy along a random AP sequence.  Secondly, it allows $k > 1$ by introducing the Haar measures on the sets $\bf{U}(k,d_n)$ of orthonormal tuples (see Appendix~\ref{sec:lin-alg} for this notation): these are the tuples whose types are unital.

\begin{lem}\label{lem:normalized3}
Let $\phi \in \B(\A,\rmM_k)_+$ be unital, and let $\cal{O}$ be any base of neighbourhoods around $\phi$ in $\S_k^\rm{u}(\A)$.  Then
\begin{equation}\label{eq:normalized3}
\rmh^\ann_{\bspi}(\phi) = \inf_{O\in \cal{O}} \limsup_{n \to\infty} \frac{1}{d_n}\log \bbE m_{\bf{U}(k,d_n)}\X(\pi_n,O).
\end{equation}
The analogous formula holds for $\ul{\rmh}^\ann_{\bspi}$ if `$\limsup$' is replaced by `$\liminf$'.
\end{lem}

The proof is similar to Lemma~\ref{lem:use-of-spherical}.  The main difference is that, in order to handle $\rmU(k,d_n)$ with $k > 1$, we need the more complicated integration formula from Proposition~\ref{prop:int-form} in place of integration in polar coordinates.

\begin{proof}
Consider the collection $\cal{W}$ of all sets that have the form
\[W = \{\psi \in \B(\A,\rmM_k)_+:\ \psi(1) \in U,\ \psi(1)^{-1/2}\cdot \psi\cdot \psi(1)^{-1/2} \in O\}\]
for some open neighbourhood $U$ of $I_k$ in $\rmM^\circ_{k+}$ and some $O \in \cal{O}$.  Since $\cal{O}$ is a base of neighbourhoods around $\phi$ in $\S^\rm{u}_k(\A)$, this $\cal{W}$ is a base of neighbourhoods around $\phi$ in $\B(\A,\rmM_k)_+$.  Therefore, as remarked following Definition~\ref{dfn:annAPent}, we may express $\rmh^\ann_{\bspi}(\phi)$ as an infimum over $\cal{W}$.

However, if $W$ is the set above, then its special form turns into the equation
\[1_{\X(\pi,W)}(VQ^{1/2}) = 1_{\X(\pi,O)}(V)\cdot 1_U(Q) \qquad (V \in \bf{U}(k,d),\ Q \in \rmM_{k+})\]
for any $d$-dimensional representation $\pi$. As a result, Proposition~\ref{prop:int-form} gives
\begin{align*}
&\vol_{2kd}\X(\pi,W)  \\
&= v(k,d)\int_{\rmM^\circ_{k+}} (\det Q)^{d-k} \int_{\bf{U}(k,d)} 1_{\X(\pi,W)}(VQ^{1/2})\ dm_{\bf{U}(k,d)}(V)\ d\vol_{k^2}(Q) \\
&= v(k,d)\cdot m_{\bf{U}(k,d)}\X(\pi,O)\cdot \int_U (\det Q)^{d-k}\ d\vol_{k^2}(Q).
\end{align*}

Finally, for any $\eps > 0$, we may choose $U$ so small that
\[e^{-\eps} < \det Q < e^\eps \qquad \hbox{for every}\ Q \in U.\]
Substituting these inequalities into the integral formula above, they lead to
\[e^{-d\eps}\cdot \vol_{k^2} U\cdot m_{\bf{U}(k,d)}\X(\pi,O) \le \frac{\vol_{2kd}\X(\pi,W)}{v(k,d)} \le e^{d\eps}\cdot \vol_{k^2} U\cdot m_{\bf{U}(k,d)}\X(\pi,O).\]
Recall that $O$ and $W$ range over neighbourhood bases of $\phi$ in $\S_k^\rm{u}(\G)$ and $\B(\A,\rmM_k)_+$ respectively.  Therefore substituting these inequalities and also the asymptotic~\eqref{eq:ckn-asymp} into Definition~\ref{dfn:annAPent}, and then letting $\eps \downarrow 0$, we obtain~\eqref{eq:normalized3} and also its analog for $\ul{\rmh}^\ann_{\bspi}$.
\end{proof}

If $\phi$ is not unital but $\phi(1)$ is invertible, then we can apply Lemma~\ref{lem:normalized3} to the map $\phi(1)^{-1/2}\cdot \phi \cdot \phi(1)^{-1/2}$ and combine the result with Proposition~\ref{prop:APent-properties}(b).

\section{Zeroth-order AP entropy}\label{sec:0-APent}

Consider $\phi \in \B(\A,\rmM_k)_+$ and a random AP sequence $\bspi$.  Heuristically, our next notion of entropy estimates the probability of asymptotic association as $n\to\infty$.

\begin{dfn}\label{dfn:0ent1}
The \textbf{zeroth-order almost periodic} (`\textbf{AP}') \textbf{entropy of $\phi$ along $\bspi$} is
\[\rmh^0_{\bspi}(\phi) := \inf_O \limsup_{n\to\infty} \frac{1}{d_n}\log \bbP(\X(\pi_n,O) \ne \emptyset),\]
where the infimum runs over all neighbourhoods of $\phi$.  Similarly, the \textbf{lower zeroth-order AP entropy of $\phi$ along $\bspi$} is
\[\ul{\rmh}^0_{\bspi}(\phi) := \inf_O \liminf_{n\to\infty} \frac{1}{d_n}\log \bbP(\X(\pi_n,O) \ne \emptyset),\]
where the infimum runs over all neighbourhoods of $\phi$.
\end{dfn}

As for annealed AP entropy, having notation for both $\rmh^0_{\bspi}(\phi)$ and its lower version $\ul{\rmh}^0_{\bspi}(\phi)$ is convenient during certain proofs later about cases when they coincide.  However, we regard $\rmh^0_{\bspi}$ as the primary notion, and focus on this in the remainder of this section, observing that analogous results hold for $\ul{\rmh}^0_{\bspi}$ with essentially the same proofs.

The term `zeroth-order' invokes the R\'enyi entropy of order zero of a random variable. This is simply the logarithm of the probability that the random variable is not zero: see, for instance,~\cite[eqn. (17.98)]{CovTho06}.  As with R\'enyi entropy, we could also define `order-$p$ entropy' for any other $p > 0$, recovering our original annealed entropy $\rmh^\ann_{\bspi}$ when $p=1$.  But this does not seem to provide much additional insight: see the discussion in Section~\ref{sec:three-entropy-cors} below.

Since $\rmh^0_{\bspi}(\phi)$ is an infimum over neighbourhoods of $\phi$, Lemma~\ref{lem:upper-semicts} gives the following partner of Lemma~\ref{lem:hann-usc}.

\begin{lem}\label{lem:hzero-usc}
Each restriction $\rmh^0_{\bspi}|\B(\A,\rmM_k)_+$ is upper semicontinuous. \qed
\end{lem}

It turns out that $\rmh^0_{\bspi}(\phi)$ is much less sensitive as a function of $\phi$ than $\rmh^\ann_{\bspi}(\phi)$.  We prove shortly that $\rmh^0_{\bspi}(\phi)$ depends only on the minimal dilation $\pi_\phi$ and is monotone under approximate containment of representations.  These properties motivate an extension of $\rmh^0_{\bspi}$ to general representations in Definition~\ref{dfn:0ent2} below.

\begin{lem}\label{lem:h0-almost-assoc}
If $\psi$ is approximately associated to $\pi_\phi$, then $\rmh^0_{\bspi}(\psi)\ge \rmh^0_{\bspi}(\phi)$.
\end{lem}

\begin{proof}
By Lemma~\ref{lem:typ-trans}, if $U$ is any neighbourhood of $\psi$ then there is a neighbourhood $O$ of $\phi$ such that
\[\X(\pi,O)\ne \emptyset \quad \Rightarrow \quad \X(\pi,U)\ne \emptyset\]
for any representation $\pi$.  With this choice of $O$, it follows that
\[\bbP(\X(\pi_n,O)\ne \emptyset) \le \bbP(\X(\pi_n,U)\ne \emptyset) \qquad (n\ge 1).\]
Taking logarithms, normalizing, letting $n\to\infty$, and then taking the infimum over $U$, this turns into the desired inequality.
\end{proof}

Lemma~\ref{lem:h0-almost-assoc} inspires our more general definition of $\rmh^0_{\bspi}$.

\begin{dfn}\label{dfn:0ent2}
For any separable representation $\pi$, its \textbf{zeroth-order AP entropy along $\bspi$} is
\[\rmh^0_{\bspi}(\pi) = \inf\big\{\rmh^0_{\bspi}(\Phi^\pi_{x_1,\dots,x_k}):\ k\ge 1,\ x_1,\dots,x_k \in H_\pi\big\}.\]
Its \textbf{lower zeroth-order AP entropy along $\bspi$} is defined analogously using $\ul{\rmh}^0$.
\end{dfn}

Further properties of $\rmh^0_{\bspi}$ flow from Definition~\ref{dfn:0ent2} rather similarly to the traditional development of Kolmogorov--Sinai entropy in ergodic theory, which can be defined as a supremum over partitions.  See, for example,~\cite[Secs. 4.4--6]{Walters--book}.  Let us stress, however, that $\rmh^0_{\bspi}$ is defined as an infimum rather than a supremum, even though Definition~\ref{dfn:0ent2} and Kolmogorov--Sinai entropy follow the same sign convention for an `entropy'.

Sometimes other formulas for zeroth-order entropy are convenient.

\begin{lem}\label{lem:0-ent-alt-dfn}
	We have
	\begin{align}
\rmh^0_{\bspi}(\pi) & = \inf\Big\{\rmh^0_{\bspi}(\phi):\ \phi \in \bigcup_{k \ge 1}\ol{\S_k(\pi)}\Big\} \label{eq:0-ent-alt-1}\\
&= \inf_U\limsup_{n\to\infty} \frac{1}{d_n}\log \bbP(\pi_n \in U) \label{eq:0-ent-alt-2},
\end{align}
where the infimum on the second line runs over all q-neighbourhoods of $\pi$.  The analogous formulas hold if `$\rmh^0$' is replaced with `$\ul{\rmh}^0$' and `$\limsup$' is replaced with `$\liminf$'.
\end{lem}

\begin{proof}
Definition~\ref{dfn:0ent2} may be written in the equivalent form
\[\rmh^0_{\bspi}(\pi) = \inf\Big\{\rmh^0_{\bspi}(\phi):\ \phi \in \bigcup_{k \ge 1}\S_k(\pi)\Big\}.\]
The formula~\eqref{eq:0-ent-alt-1} follows from this together with the upper semicontinuity from Lemma~\ref{lem:hzero-usc}.

To prove formula~\eqref{eq:0-ent-alt-2}, first observe that we may restrict attention to sets $U$ that lie in any choice of neighbourhood base at $\phi$, because the infimum is monotone in $U$.  Having done so, Lemma~\ref{lem:lower-Vietoris-simplify} tells us that it suffices to consider q-neighbourhoods that have the form
\[U = \{\rho:\ \ol{\S_k(\rho)}\ \hbox{meets}\ O\}\]
for some positive integer $k$ and some open subset $O$ of $\S_k(\A)$ that meets $\S_k(\pi)$.  This choice of $U$ satisfies
\[\bbP(\pi_n \in U) = \bbP(\S_k(\pi_n)\ \hbox{meets}\ O) = \bbP(\X(\pi_n,O)\ne \emptyset).\]
Therefore the infimum over all neighbhourhoods $U$ of this form turns into the infimum of the expression in Definition~\ref{dfn:0ent1} over $\phi \in \bigcup_k \S_k(\A)$, as required.

The formulas for $\ul{\rmh}^0_{\bspi}$ are proved in the same way.
\end{proof}

The next two results allow us to restrict which vectors in a representation we consider when computing $\rmh^0_{\bspi}$.  The first is analogous to the Kolmorogov--Sinai generator theorem~\cite[Thm. 4.17]{Walters--book}.

\begin{lem}\label{lem:h0-cyclic}
If the tuple $x_1$, \dots, $x_k$ is cyclic for $\pi$, then
\[\rmh^0_{\bspi}(\pi) = \rmh^0_{\bspi}(\Phi^\pi_{x_1,\dots,x_k}).\]
Equivalently, if $\phi \in \S_k(\A)$, then $\rmh^0_{\bspi}(\pi_\phi) = \rmh^0_{\bspi}(\phi)$.
\end{lem}

\begin{proof}
The inequality ``$\le$'' follows at once from the definition of $\rmh^0_{\bspi}$.  On the other hand, if $y_1$, \dots, $y_\ell$ is any other tuple in $H_\pi$, then $\Phi^\pi_{y_1,\dots,y_\ell}$ is associated to the GNS representation of $\Phi^\pi_{x_1,\dots,x_k}$, since the latter equals the whole of $\pi$ by cyclicity.  Therefore Lemma~\ref{lem:h0-almost-assoc} shows that $\rmh^0_{\bspi}(\Phi^\pi_{y_1,\dots,y_\ell})$ is bounded below by $\rmh^0_{\bspi}(\Phi^\pi_{x_1,\dots,x_k})$.
\end{proof}

\begin{rmk}
With Lemma~\ref{lem:h0-cyclic} in hand, we use the notations $\rmh^0_{\bspi}(\phi)$ and $\rmh^0_{\bspi}(\pi_\phi)$ interchangeably in the sequel. \fin
\end{rmk}

The next lemma generalizes Lemma~\ref{lem:h0-cyclic} for representation that do not have finite cyclic tuples.  It is an analog of~\cite[Thms. 4.21 and 4.22]{Walters--book}.

\begin{lem}\label{lem:h0-nearly-cyclic}
	Let $\pi$ be a representation of $\A$, and let $S$ be a subset of $H_\pi$ such that the linear span of $\{\pi(a)x:\ a \in \A, x \in S\}$ is dense in $H_\pi$.  Then
	\[\rmh^0_{\bspi}(\pi) = \inf\big\{\rmh^0_{\bspi}(\Phi^\pi_{x_1,\dots,x_k}):\ k\ge 1,\ x_1,\dots,x_k \in S\big\}.\]
\end{lem}

\begin{proof}
The inequality ``$\le$'' follows at once from the definition of $\rmh^0_{\bspi}$.

For the reverse direction, suppose that $Y = [y_1,\dots,y_\ell]$ is a tuple in $H_\pi$.

For any tuple $X = [x_1,\dots,x_k]$ in $H_\pi$, let $M_X$ be the closed $\pi$-invariant subspace of $H_\pi$ generated by $X$. By our assumption on $S$, there are tuples $X_1$, $X_2$, \dots drawn from $S$ and further $\ell$-tuples $Y_1$, $Y_2$, \dots such that $Y_n$ lies in $M_{X_n}$ and $Y_n \to Y$, and hence also $\Phi^\pi_{Y_n}\to \Phi^\pi_Y$ by Lemma~\ref{lem:unif-cts}.  By Lemma~\ref{lem:hzero-usc} and Lemma~\ref{lem:h0-cyclic}, it follows that
\[\rmh^0_{\bspi}(\Phi^\pi_Y) \ge \limsup_{n\to\infty} \rmh^0_{\bspi}(\Phi^\pi_{Y_n}) \ge \limsup_{n\to\infty} \rmh^0_{\bspi}(\pi^{M_{X_n}}) = \limsup_{n\to\infty} \rmh^0_{\bspi}(\Phi^\pi_{X_n}).\]
Since each $X_n$ is a tuple drawn from $S$, this proves the inequality ``$\ge$''.
\end{proof}

Using Lemma~\ref{lem:0-ent-alt-dfn}, we can improve Lemma~\ref{lem:h0-almost-assoc} to the following monotonicity for an arbitrary pair of representations $\pi$ and $\rho$.

\begin{cor}\label{cor:h0-almost-contained}
If $\rho$ is approximately contained in $\pi$, then $\rmh^0_{\bspi}(\rho) \ge \rmh^0_{\bspi}(\pi)$, and similarly for $\ul{\rmh}^0_{\bspi}$.  In particular, the functions $\rmh^0_{\bspi}$ and $\ul{\rmh}^0_{\bspi}$ on representations are invariant under approximate equivalence.
\end{cor}

\begin{proof}
If $\phi$ is approximately associated to $\rho$, then it is also approximately associated to $\pi$, so this follows from formula~\eqref{eq:0-ent-alt-1}.
\end{proof}

\begin{rmk}
In general $\rmh^0_{\bspi}$ is \emph{not} an invariant of weak equivalence, because it can sometimes detect multiplicities of subrepresentations. This becomes clear from Proposition~\ref{prop:additivity} below. \fin
\end{rmk}

Part~\ref{part:free} below is dedicated to annealed and zeroth-order AP entropy and related quantities for uniformly random AP sequences for free groups.  By contrast, zeroth-order AP entropy turns out to be somewhat degenerate for uniformly random \emph{permutation} representations: see Section~\ref{sec:rndm-perms}.

\section{The three-entropy formula}\label{sec:three-entropy}

Some results about AP entropy require the convergence of the traces on $\A$ pulled back from an AP sequence to $\tau$.  For annealed and zeroth-order AP entropy along $\bspi$, this assumption can sometimes be replaced by convergence of the random traces $\tr_{d_n}\circ \pi_n$ in probability sufficiently fast.  Specifically, assume that $\tau$ is a tracial state on $\A$, let $\l$ be its GNS representation, and let $\Delta$ be the Fuglede--Kadison determinant defined by $\tau$ as in Section~\ref{sec:FK}.  In this section, we work from the following additional assumption: for every neighbourhood $V$ of $\tau$ in $\A^\ast_+$ and every $c > 0$, we have
\begin{equation}\label{eq:traces-fast}
\bbP(\tr_{d_n}\circ \pi_n \not\in V) = o(e^{-cd_n}) \qquad \hbox{as}\ n\to\infty.
\end{equation}
For uniformly random AP sequences for free groups, this holds by an asymptotic freeness theorem of Voiculescu: see Theorem~\ref{thm:asymp-free2} below.

Our first consequence of~\eqref{eq:traces-fast} is a probabilistic generalization of Proposition~\ref{prop:APent-properties-trace}.

\begin{lem}\label{lem:h0-sing-only}
Assume that~\eqref{eq:traces-fast} holds, and let $\pi$ be a separable unital representation of $\A$.
\begin{enumerate}
\item[a.] If $\phi:\A\to\rmM_k$ is approximately associated to $\l^{\oplus \infty}$, and $O$ is a neighbourhood of $\phi$ in $\B(\A,\rmM_k)_+$, then
\[\bbP(\X(\pi_n,O) \ne \emptyset) \to 1.\]
If $\pi \lesssim_{\rm{a}} \l^{\oplus \infty}$, then $\rmh_{\bspi}^0(\pi) = \ul{\rmh}_{\bspi}^0(\pi) = 0$.
\item[b.] In general, let $\pi^M$ be the subrepresentation such that $\pi^M \lesssim \l^{\oplus \infty}$ and $\pi^{M^\perp} \spoon \l$ (see Proposition~\ref{prop:Leb-reps}).  Then $\rmh_{\bspi}^0(\pi) = \rmh_{\bspi}^0(\pi^{M^\perp})$ and $\ul{\rmh}_{\bspi}^0(\pi) = \ul{\rmh}_{\bspi}^0(\pi^{M^\perp})$.
\end{enumerate}
\end{lem}

\begin{proof}
\emph{Part (a).}\quad For this $\phi$ and $O$, Corollary~\ref{cor:APent-properties-trace}(a) gives an integer $d_0$ and a neighbourhood $V$ of $\tau$ such that
\[d_n\ge d_0\quad \hbox{and}\quad \tr_{d_n}\circ \pi_n \in V \qquad \Rightarrow \qquad \X(\pi_n,O) \ne \emptyset.\]
By~\eqref{eq:traces-fast}, the probability that $\tr_{d_n} \circ \pi_n \in V$ tends to $1$ as $n\to\infty$, hence so does the probability that $\X(\pi_n,O) \ne \emptyset$.  Now the conclusions when $\pi \lesssim_{\rm{a}} \l^{\oplus \infty}$ follow from formula~\eqref{eq:0-ent-alt-1} for $\rmh^0_{\bspi}$ and its analog for $\ul{\rmh}^0_{\bspi}$.

\vspace{7pt}

\emph{Part (b).}\quad The inequality $\rmh_{\bspi}^0(\pi) \le \rmh_{\bspi}^0(\pi^{M^\perp})$ holds by Corollary~\ref{cor:h0-almost-contained}, and similarly for $\ul{\rmh}^0_{\bspi}$.

To prove the reverse inequality, we may assume that $\rmh_{\bspi}^0(\pi^{M^\perp}) > -\infty$.  Let $\phi:\A\to\rmM_k$ be associated to $\pi$, and let $O$ be a neighbourhood of $\phi$ in $\B(\A,\rmM_k)_+$.

By projecting the associating vectors of $\phi$ to $M$ and $M^\perp$, we obtain the Lebesgue decomposition $\phi_{\rm{ac}} + \phi_{\rm{sing}}$ of $\phi$ with respect to $\tau$ (see Corollary~\ref{cor:Leb}).  Now Corollary~\ref{cor:APent-properties-trace}(b) gives a positive integer $d_0$ and neighbourhoods $V$ of $\tau$ and $U$ of $\phi_{\rm{sing}}$.  Once $d_n \ge d_0$, the conclusion of that corollary yields
\begin{align*}
\bbP(\X(\pi_n,O)\ne \emptyset) & \ge \bbP(\tr_{d_n}\circ \pi_n \in V\ \hbox{and}\ \X(\pi_n,U)\ne \emptyset)\\
&\ge \bbP(\X(\pi_n,U)\ne \emptyset) - \bbP(\tr_{d_n}\circ \pi_n \not\in V).
\end{align*}
By~\eqref{eq:traces-fast}, the second term here decays faster than any exponential in $d_n$, leaving
\[\limsup_{n \to \infty}\frac{1}{d_n}\log \bbP(\X(\pi_n,O)\ne \emptyset) \ge \limsup_{n \to \infty}\frac{1}{d_n}\log \bbP(\X(\pi_n,U)\ne \emptyset)\]
for any such map $\phi$ and neighbourhood $O$.  The quantity on the right-hand side is at least $\rmh_{\bspi}^0(\pi^{M^\perp})$ for every $U$.  Therefore, taking the infimum over $\phi$ and $O$, formula~\eqref{eq:0-ent-alt-1} gives $\rmh^0_{\bspi}(\pi) \ge \rmh^0_{\bspi}(\pi^{M^\perp})$. To prove the result for $\ul{\rmh}^0$, replace `$\limsup$' with `$\liminf$' on both sides above.
\end{proof}

We have now met three notions of entropy for a map $\phi$ in $\B(\A,\rmM_k)_+$:
\begin{itemize}
	\item[i.] AP entropy along a fixed AP sequence, which always equals either $-\infty$ or a Fuglede--Kadison log-determinant from Theorem~\ref{thm:APE4};
	\item[ii.] annealed AP entropy along a random AP sequence $\bspi$ (or its lower version);
	\item[iii.] zeroth-order AP entropy along a random AP sequence $\bspi$ (or its lower version).
\end{itemize}

When~\eqref{eq:traces-fast} holds, these three notions are related by the following formulas.

\begin{thm}[Three-entropy formula]\label{thm:three-entropy}
If~\eqref{eq:traces-fast} holds and $\phi \in \B(\A,\rmM_k)_+$, then
	\begin{equation}\label{eq:three-entropies} 
\rmh_{\bspi}^\ann(\phi) = \rmh_{\bspi}^0(\phi) + \log \Delta \phi
\end{equation}
and
\begin{equation}\label{eq:three-entropies-2}
\ul{\rmh}^\ann_{\bspi}(\phi) = \ul{\rmh}_{\bspi}^0(\phi) + \log \Delta \phi
\end{equation}
\end{thm}

Before proving Theorem~\ref{thm:three-entropy}, let us describe a probabilistic intuition for it when $\phi$ is unital.  For simplicity, let $k=1$, and assume that $\rmh^\ann_{\bspi}(\phi) = \ul{\rmh}^\ann_{\bspi}(\phi)$ and $\rmh^0_{\bspi}(\phi) = \ul{\rmh}^0_{\bspi}(\phi)$ (perhaps after passing to a subsequence, for example).  In that case we can write $\rmh^{\rm{ann}}_{\bspi}$ using formula~\eqref{eq:normalized3} from Lemma~\ref{lem:normalized3}.  Now let $n$ be large, let $v_n$ be a uniform random vector in $\bf{S}^{2n-1}$ which is independent from $\pi_n$, and consider the event
\[E_n := \{\Phi^{\pi_n}_{v_n}\approx \phi,\ \tr_{d_n}\circ \pi_n \approx \tau\},\]
where we gloss over the quality of approximation here.  Because of~\eqref{eq:traces-fast}, we can give~\eqref{eq:normalized3} the heuristic interpretation
\[\bbE[\bbP(E_n\mid \pi_n)] \approx e^{\hann_{\bspi}(\phi)n},\]
On the one hand, the event $E_n$ is contained in the event
\[F_n := \{\Phi^{\pi_n}_u \approx \phi\ \hbox{for \emph{some}}\ u \in \bf{S}^{2n-1},\ \tr_{d_n}\circ \pi_n \approx \tau\},\]
which depends on only $\pi_n$ and not $v_n$.  Because of~\eqref{eq:traces-fast} again, this event satisfies
\[\bbP(F_n) \approx e^{\rmh^0_{\bspi}(\phi)n}.\]
On the other hand, for any $\pi_n$ such that $F_n$ holds, Theorem~\ref{thm:APE4} gives
\[\bbP(E_n \mid \pi_n) \approx e^{(\log \Delta \phi_{\rm{ac}})n},\]
using that $v_n$ is independent from $\pi_n$ and uniformly distributed.  Putting the above approximations together, they become
\[e^{\hann_{\bspi}(\phi)n} \approx \bbE[\bbP(E_n\mid \pi_n)] \approx \bbP(F_n)\cdot e^{(\log \Delta \phi_{\rm{ac}} )n} \approx e^{(\rmh^0_{\bspi}(\phi) + \log \Delta \phi_{\rm{ac}})n}.\]

The full proof of Theorem~\ref{thm:three-entropy} is just a precise version of this argument.  The main points are to be careful about the precise approximations needed in different places and to allow general normalized volumes in place of $m_{\bf{S}^{2n-1}}$.

\begin{proof}[Proof of Theorem~\ref{thm:three-entropy}]
We prove~\eqref{eq:three-entropies} and put the alterations needed for~\eqref{eq:three-entropies-2} in [brackets].

\vspace{7pt}
	
	\emph{Step 1:.}\quad Assume that the right-hand side of~\eqref{eq:three-entropies} is finite, and let $h_0 < \rmh_{\bspi}^0(\pi_\phi)$ and $h_1 < \log \Delta \phi_{\rm{ac}}$.

Let $O$ be any neighbourhood of $\phi$.  For this choice of $h_1$ and $O$, Theorem~\ref{thm:APE4}(b) gives a positive integer $d_0$, a neighbourhood $V$ of $\tau$, and another neighbourhood $O'$ of $\phi$.  As a result, for all sufficiently large $n$, we have
	\begin{align*}
		\bbE\frac{\vol_{2kd_n}\X(\pi_n,O)}{v(d_n)^k} &\ge \bbE\Big[\frac{\vol_{2kd_n}\X(\pi_n,O)}{v(d_n)^k}\,;\, \tr_{d_n}\circ \pi_n \in V,\ \X(\pi_n,O') \ne \emptyset\Big] \\
	&\ge e^{h_1d_n}\cdot \bbP(\tr_{d_n}\circ \pi_n \in V,\ \X(\pi_n,O') \ne \emptyset)\\
	&\ge e^{h_1d_n}\cdot \big(\bbP(\X(\pi_n,O') \ne \emptyset) - \bbP(\tr_{d_n}\circ \pi_n \not\in V)\big).
	\end{align*}
By the definition of zeroth-order AP entropy and~\eqref{eq:traces-fast}, the last line above is at least
\[e^{(h_1 + h_0)d_n} - o(e^{(h_1 + h_0)d_n})\]
for infinitely many $n$ [for all sufficiently large $n$].  Inserting this lower bound into Definition~\ref{dfn:annAPent}, we arrive at $\rmh_{\bspi}^\ann(\phi) \ge h_0 + h_1$.  By the arbitrariness of $h_0$ and $h_1$, this proves the inequality `$\ge$' in~\eqref{eq:three-entropies} [in~\eqref{eq:three-entropies-2}].

\vspace{7pt}

\emph{Step 2.}\quad Let $h_0 > \rmh_{\bspi}^0(\phi)$ and $h_1 > \log \Delta \phi_{\rm{ac}}$, and in addition let ${R > \max_i \phi_{ii}(1)}$.

On the one hand, the definition of $\rmh_{\bspi}^0$ gives a neighbourhood $O_0$ of $\phi$ such that $\bbP(\X(\pi_n,O_0) \ne \emptyset) \le e^{h_0 d_n}$ for all sufficiently large $n$.

On the other hand, having made our choice of $h_1$, Theorem~\ref{thm:APE4}(a) gives a positive integer $d_0$, a neighbourhood $V$ of $\tau$, and a neighbourhood $O$ of $\phi$.  By shrinking $O$ further if necessary, we may also assume that $O \subset O_0$ and that every $\psi \in O$ satisfies $\max_i\psi_{ii}(1) < R$.  The last property implies that $\X(\pi,O)$ is contained in $B_R(0)^k$ for every $d$-dimensional representation $\pi$, where $B_R(0)$ is the radius-$R$ ball in $\bbC^d$.

Once $d_n \ge d_0$, we can now estimate the relevant expected volume like this:
\begin{align*}
	&\bbE\frac{\vol_{2kd_n}\X(\pi_n,O)}{v(d_n)^k} \\ &=  \bbE\Big[\frac{\vol_{2kd_n}\X(\pi_n,O)}{v(d_n)^k}\,;\,\tr_{d_n}\circ \pi_n \not\in V\Big] + \bbE\Big[\frac{\vol_{2kd_n}\X(\pi_n,O)}{v(d_n)^k}\,;\,\tr_{d_n}\circ \pi_n \in V\Big] \\
	&\le R^{2kd_n}\cdot \bbP(\tr_{d_n}\circ \pi_n \not\in V) + \bbE\Big[\frac{\vol_{2kd_n}\X(\pi_n,O)}{v(d_n)^k}\,;\,\tr_{d_n}\circ \pi_n \in V,\ \X(\pi_n,O_0) \ne \emptyset\Big]
\end{align*}
We can insert the event $\{\X(\pi_n,O_0) \ne \emptyset\}$ into the last expectation here, because if $\X(\pi_n,O_0) = \emptyset$ then also $\X(\pi_n,O) = \emptyset$ and so $\vol_{2kd_n}\X(\pi_n,O) = 0$.

Since $R$ is a fixed real number, by~\eqref{eq:traces-fast} the first term above decays faster than any exponential in $d_n$.  On the other hand, on the event $\{\tr_{d_n}\circ \pi_n \in V\}$, our appeal to Theorem~\ref{thm:APE4} gives
\[\frac{\vol_{2kn}\X(\pi_n,O)}{v(d_n)^k} \le e^{h_1d_n}\]
for all sufficiently large $n$, and so the second term above is at most
\[e^{h_1d_n}\bbP(\X(\pi_n,O_0) \ne \emptyset)\]
for all sufficiently large $n$.  On the other hand, by our choice of $h_0 > \rmh_{\bspi}^0(\phi)$ and $O_0$, the expression above is less than or equal to $e^{(h_0 + h_1)d_n}$ for all sufficiently large $n$ [for infinitely many $n$].  Inserting this into Definition~\ref{dfn:annAPent}, and recalling the arbitrariness of $h_0$ and $h_1$, this proves the inequality `$\le$' in~\eqref{eq:three-entropies} [in~\eqref{eq:three-entropies-2}].
\end{proof}


If $\phi \in \B(\A,\rmM_k)_+$, then Theorem~\ref{thm:three-entropy} expresses $\rmh_{\bspi}^\ann(\phi)$ in terms of $\rmh_{\bspi}^0(\phi)$ and $\log \Delta \phi_{\rm{ac}}$.  However, it can happen that $\rmh_{\bspi}^\ann(\phi)$ and $\log \Delta \phi_{\rm{ac}}$ both equal $-\infty$, and then we cannot simply re-arrange to find $\rmh_{\bspi}^0(\phi)$.  Instead, we can apply Theorem~\ref{thm:three-entropy} to the perturbations
\begin{equation}\label{eq:conv-comb-with-phi}
\phi_t := \tau \otimes I_k + t\phi \qquad (0 \le t < 1).
\end{equation}
This leads to a probabilistic variant of Corollary~\ref{cor:h-pi-and-strong-quot}.

\begin{lem}\label{lem:mollify}
If~\eqref{eq:traces-fast} holds, then the maps in~\eqref{eq:conv-comb-with-phi} satisfy
\begin{equation}\label{eq:mollify}
\rmh_{\bspi}^0(\phi) \le \rmh_{\bspi}^\ann(\phi_t)\le \rmh_{\bspi}^0(\phi) + k \log (1 + t\cdot \tr_k\phi(e)) \qquad (0 < t < 1).
\end{equation}
As a result, either $\rmh_{\bspi}^0(\phi)$ equals $-\infty$ and so does $\rmh_{\bspi}^\ann(\phi_t)$ for every $t$, or these quantities are all finite and
\[\rmh_{\bspi}^\ann(\phi_t) \to \rmh_{\bspi}^0(\phi) \qquad \hbox{as}\ t \downarrow 0.\]

The same conclusions hold if $\rmh^\ann_{\bspi}$ and $\rmh_{\bspi}^0$ are replaced with $\ul{\rmh}^\ann_{\bspi}$ and $\ul{\rmh}_{\bspi}^0$.
\end{lem}

\begin{proof}
We prove this for $\rmh^\ann_{\bspi}$ and $\rmh_{\bspi}^0$.  The proof for $\ul{\rmh}^\ann_{\bspi}$ and $\ul{\rmh}_{\bspi}^0$ is analogous.

Since $\tau\otimes I_k$ is absolutely continuous with respect to $\tau$, and the Lebesgue decomposition from Corollary~\ref{cor:Leb} is linear, we have $(\phi_t)_{\rm{ac}} = \tau \otimes I_k + t \phi_{\rm{ac}}$ and $(\phi_t)_{\rm{sing}} = t\phi_{\rm{sing}}$ for all $t$.  Therefore two appeals to Lemma~\ref{lem:h0-sing-only}(b) give
\[\rmh_{\bspi}^0(\phi_t) = \rmh_{\bspi}^0(t \phi_{\rm{sing}}) = \rmh_{\bspi}^0(\phi_{\rm{sing}}) = \rmh_{\bspi}^0(\phi)\]
when $t \in (0,1)$.  Inserting this in Theorem~\ref{thm:three-entropy}, we obtain
\begin{equation}\label{eq:three-entropies-again}
\rmh_{\bspi}^\ann(\phi_t) = \rmh_{\bspi}^0(\phi) + \log\Delta(\tau \otimes I_k + t \phi_{\rm{ac}}).
\end{equation}
Now~\eqref{eq:mollify} follows from the inequalities
\[0= \log \Delta(\tau\otimes I_k) \le \log \Delta(\tau \otimes I_k + t \phi_{\rm{ac}}) \le k \log (1 + t\cdot \tr_k\phi(e)),\]
which are both provided by Proposition~\ref{prop:FK-det-properties}(b).
\end{proof}

Here is a first application of Lemma~\ref{lem:mollify}.

\begin{cor}\label{cor:h0}
If~\eqref{eq:traces-fast} holds, and if $\rmh^\ann_{\bspi}$ and $\ul{\rmh}^\ann_{\bspi}$ agree on $\B(\A,\rmM_k)_+$ for some positive integer $k$, then so do $\rmh^0_{\bspi} = \ul{\rmh}^0_{\bspi}$.
\end{cor}

\begin{proof}
If $\phi \in \B(\A,\rmM_k)_+$, then the hypothesis and Lemma~\ref{lem:mollify} give
\[\ul{\rmh}_{\bspi}^0(\phi) = \lim_{t\downarrow 0} \ul{\rmh}^\ann_{\bspi}(I_k\otimes \tau + t\phi) = \lim_{t\downarrow 0} \rmh^\ann_{\bspi}(I_k\otimes \tau + t\phi) = \rmh^0_{\bspi}(\phi).\]
\end{proof}

In Proposition~\ref{prop:annealed-formula-1} we use Corollary~\ref{cor:h0} to prove that $\ul{\rmh}^0_{\bspi} = \rmh^0_{\bspi}$ when $\bspi$ is a uniformly random AP sequence for a free group.  Later, in Chapter~\ref{chap:prob-lims}, this conclusion is an essential step towards one of our main large deviations principles in that setting.  We also meet another application of Lemma~\ref{lem:mollify} to those examples in Section~\ref{sec:additivity}.

\begin{rmk}\label{rmk:other-orders}
Besides $\rmh^\ann_{\bspi}$ and $\rmh_{\bspi}^0$, one could consider entropies of `other orders'.  However, if~\eqref{eq:traces-fast} holds, then these do not seem to lead to anything new.  Indeed, if $\phi\in \B(\A,\rmM_k)_+$ and $p > 0$, then essentially the same proof as for Theorem~\ref{thm:three-entropy} gives the generalization
\[\inf_O\limsup_{n\to\infty}\frac{1}{d_n}\log \bbE\Big[\Big(\frac{\vol_{2kd_n}\X(\pi_n,O)}{v(d_n)^k}\Big)^p\Big] = \rmh_{\bspi}^0(\phi) + p\cdot \log\Delta \phi_{\rm{ac}}.\]
Theorem~\ref{thm:three-entropy} is the case $p=1$.  If $\log\Delta \phi_{\rm{ac}}$ is finite, then we recover $\rmh_{\bspi}^0(\phi)$ as $p\downarrow 0$, again justifying the name `zeroth-order'.  So random AP entropies of other `orders' are simply obtained from $\rmh_{\bspi}^0$ and $\rmh^\ann_{\bspi}$ by linear interpolation. \fin
\end{rmk}

\section{Legendre transform of annealed AP entropy}\label{sec:annAP-Legendre}

We now extend the work of Section~\ref{sec:AP-Legendre} by studying the Legendre transform of $\rmh^\ann_{\bspi}$.  We assume throughout this section that $\tr_{d_n}\circ \pi_n \to \tau$.

\begin{dfn}\label{dfn:annAP-pressure}
Let $a$ be a self-adjoint element of $\rmM_k(\A)$.  Its \textbf{annealed almost periodic} (`\textbf{AP}') \textbf{pressure along $\bspi$} is the quantity
\[\rmp^\ann_{\bspi}(a) := \sup\big\{\langle \phi,a\rangle + \rmh^\ann_{\bspi}(\phi):\ \phi \in \S_k(\A)\big\}.\]
An element $\phi$ of $\S_k(\A)$ is an \textbf{$\rmh^\ann_{\bspi}$-equilibrium map for $a$} if it achieves the supremum on the right-hand side above.  When $k=1$ we write \textbf{equilibrium state} rather than `equilibrium map'.
\end{dfn}

The upper semicontinuity of $\rmh^\ann_{\bspi}$ (Lemma~\ref{lem:hann-usc}) guarantees that the $\rmh^\ann_{\bspi}$-equilibrium maps for $a$ form a closed and nonempty subset of $\S_k(\A)$.  However, $\rmh_{\bspi}^\ann$ need not be concave in general, so now we cannot always recover it by taking another Legendre transform of $\rmp^\ann_{\bspi}$.  This failure of concavity really comes from the possible misbehaviour of $\rmh^0_{\bspi}$: we discuss some concrete examples in Section~\ref{sec:additivity} below.

As in Section~\ref{sec:AP-Legendre}, we now focus on the scalar-valued case $k=1$ to lighten notation.  Let $a \in \A_{\rm{sa}}$.  Let $\phi \in \S_1(\A)$, and let its Lebesgue decomposition be $\phi_{\rm{ac}} + \phi_{\rm{sing}}$.  Let $\pi$ be the GNS representation of $\phi$, and let its corresponding decomposition be $\pi_{\rm{ac}} \oplus \pi_{\rm{sing}}$.  The next lemma derive an alternative formula for $\rmp^\ann_{\bspi}(a)$ in terms of the quantities $p_{\frak{I}}(a)$ defined in~\eqref{eq:Leg-transf-2} for certain ideals $\frak{I}$ of $\A$.  Recall that $\rm{Rep}(\A)$ is a set of representatives of all equivalence classes of separable unitary representations of $\A$.

\begin{lem}\label{lem:juxtaposition}
For any $a \in \A_{\rm{sa}}$, we have
\begin{equation}\label{eq:juxtaposition}
\rmp^\ann_{\bspi}(a) = \sup\{p_{\ker \rho}(a) + \rmh^0_{\bspi}(\rho):\ \rho \in \rm{Rep}(\A),\ \rho \gtrsim \l\}.
\end{equation}
If $\phi$ is a state and $\pi$ is its GNS representation, then $\phi$ is an $\rmh^\ann_{\bspi}$-equilibrium state for $a$ if and only (i) $\pi$ is an element of $\rm{Rep}(\A)$ that attains the supremum on the right-hand side of~\eqref{eq:juxtaposition} and also (ii) $\phi$ is also a $p_{\ker \pi}$-tangent state at $a$.
\end{lem}

\begin{proof}
Let $\rho \in \rm{Rep}(\A)$ satisfy $\rho \gtrsim \l$, let $K := \ol{\S_1(\rho)}$, and let $\psi \in K$.  Then Theorem~\ref{thm:three-entropy} and formula~\eqref{eq:0-ent-alt-1} from Lemma~\ref{lem:0-ent-alt-dfn} give
\[\rmp^\ann_{\bspi}(a) \ge \psi(a) + \rmh^\ann_{\bspi}(\psi) = \psi(a) + \log\Delta \psi + \rmh^0_{\bspi}(\psi) \ge \psi(a) + \log\Delta \psi + \rmh^0_{\bspi}(\rho).\]
Taking the supremum over all $\psi \in K$, Lemma~\ref{lem:K-and-I} turns this into the inequality ``$\ge$'' in~\eqref{eq:juxtaposition}.

On the other hand, if $\phi$ is an $\rmh^\ann_{\bspi}$-equilibrium state for $a$, then we have
\begin{align*}
\rmp^\ann_{\bspi}(a) &= \phi(a) + \rmh^\ann_{\bspi}(\phi) \\
&= \phi(a) + \log \Delta \phi + \rmh^0_{\bspi}(\pi_{\rm{sing}}) \qquad \hbox{(by Theorem~\ref{thm:three-entropy} and Lemma~\ref{lem:h0-sing-only}(b))}\\
&\le p_{\ker \pi_\phi}(a) + \rmh^0_{\bspi}(\pi_{\rm{sing}}) \qquad \hbox{(comparing with~\eqref{eq:Leg-transf-2})}.
\end{align*}
This proves the inequality ``$\le$'' in~\eqref{eq:juxtaposition}.  Because we have already proved the opposite inequality, equality must actually throughout, and this also verifies conditions (i) and (ii) for $\phi$.
\end{proof}

Lemma~\ref{lem:juxtaposition} expresses $\rmp^\ann_{\bspi}$ as the pointwise supremum of the family of functions $p_{\ker \rho} + \rmh^0_{\bspi}(\rho)$, each of which has the form in~\eqref{eq:Leg-transf-2} except for a vertical offset.  This lets us extract stability conditions for an $\rmh^\ann_{\bspi}$-equilibrium state by applying Theorem~\ref{thm:AP-equilibrium}.

Let $a \in \A_{\rm{sa}}$, let $\mu_a$ be the spectral measure of $a$ on $\bbR$ determined by $\tau$, and use it to define $F_a$, $U_a$, and the states $\g_{a,z}$ as in Section~\ref{sec:AP-Legendre}.  As for Theorem~\ref{thm:AP-equilibrium}, let us now also assume that $\l(a)$ is not a multiple of the identity operator (condition~\eqref{eq:non-mult} from Section~\ref{sec:AP-Legendre}), so $\mu_a$ is not a point mass. The analysis below can be modified in case $\l(a)$ is a multiple of the identity, but it degenerates in several places; we leave these details to the reader.

For convenience of application later, we formulate the next corollary for a general positive scalar multiple $\b a$, recalling the relations~\eqref{eq:F-U-dilate} and~\eqref{eq:gamma-dilate} and using Corollary~\ref{cor:AP-equilibrium} rather than Theorem~\ref{thm:AP-equilibrium}.

\begin{cor}\label{cor:equilibrium-struct}
Let $\b > 0$, let $\phi$ be a $\rmh^\ann_{\bspi}$-equilibrium state for $\b a$, and let $\pi$ be the GNS representation of $\phi$.
\begin{enumerate}
\item[a.] If $\phi$ is $\l$-normal, then there is a unique value $z \in [\s_{\max}(\l(a)),\infty)$ which satisfies $F_a(z) = \b$, and we have $\phi = \g_{a,z}$ and $\rmp^\ann_{\bspi}(\b a) = \b z-1 + U_a(z) - \log \b$.
\item[b.] If $\phi$ is not $\l$-normal, and we let $m := \s_{\max}(\pi(a))$, then also
\[m = \s_{\max}(\pi_{\rm{sing}}(a)) = \phi(a) +1/\b,\]
this value satisfies $F_a(m) < \b$, and we have
\[\phi = \frac{F_a(m)}{\b}\g_{a,m} + \Big(1-\frac{F_a(m)}{\b}\Big)\phi_0\]
for some $\l$-singular state $\phi_0$ which is associated to $\pi_{\rm{sing}}$ by an eigenvector of $\pi_{\rm{sing}}(a)$ with eigenvalue $m$.  In this case
\[\rmp^\ann_{\bspi}(\b a) = \b m-1 + U_a(m) - \log \b + \rmh^0_{\bspi}(\phi_0).\]
\item[c.] If $\phi$ is unique and not $\l$-normal, then $\pi_{\rm{sing}}$ is irreducible.
\item[d.] If $\phi$ is not $\l$-normal and $\rho$ is any separable representation, then
\[\b(1-F_a(m))\s_{\max}(\rho(a)) + \rmh_{\bspi}^0(\rho) \le \b(1-F_a(m))\s_{\max}(\pi(a)) + \rmh^0_{\bspi}(\pi_{\rm{sing}}).\]
In particular, $\phi$ maximizes $\rmh^0_{\bspi}(\phi)$ given the value of $\s_{\max}(\pi(a))$.
\end{enumerate}
\end{cor}

\begin{proof}
By Lemma~\ref{lem:juxtaposition}, $\phi$ is also a $p_{\ker \pi}$-tangent state at $a$.  Now conclusions (a--c) are all given by Corollary~\ref{cor:AP-equilibrium}, with the evaluations of $\rmp^\ann_{\bspi}$ following by Lemma~\ref{lem:juxtaposition}.

It remains to prove part (d).  Write $\phi$ in the form given by part (b).  Towards a contradiction, suppose that some representation $\rho$ violates the desired inequality. Then we can find a unit vector $x \in H_\rho$ such that the state $\psi = \Phi^\rho_x$ satisfies
\[\b(1-F_a(m))\psi(a) + \rmh_{\bspi}^0(\psi) > \b(1-F_a(m))\s_{\max}(\pi(a)) + \rmh^0_{\bspi}(\pi_{\rm{sing}}).\]
Since also $\phi_0(\b a) = \b\s_{\max}(\pi(a))$, it would follow that the state
\[ \frac{F_a(m)}{\b}\g_{a,m} + \Big(1- \frac{F_a(m)}{\b}\Big)\psi\]
contradicts the optimality of $\phi$.  This proves conclusion (d).
\end{proof}

In particular, Corollary~\ref{cor:equilibrium-struct} shows that the absolutely continuous part of an $\rmh^\ann_{\bspi}$-equilibrium state $\phi$ for $\b a$ is always a scalar multiple of $\gamma_{a,z}$ for some $z$: either $\phi = \g_{a,z}$ and $F_a(z) = \b$ if $\phi$ is $\l$-normal, or otherwise $\phi_{\rm{ac}} = \frac{F_a(m)}{\b}\g_{a,m}$ and $F_a(m) < \b$ for $m = \s_{\max}(\pi(a))$.

While Corollary~\ref{cor:equilibrium-struct} limits the possible structure of any single $\rmh^\ann_{\bspi}$-equilibrium state for $\b a$, it does not describe the set of those equilibrium states as a whole.  For instance, if $\b a$ has multiple $\rmh^\ann_{\bspi}$-equilibrium states, then it could be that some are $\l$-normal and others are not.  However, the next lemma shows that at most one value of $\b$ can exhibit this possibility.

\begin{lem}\label{lem:maxs-monotone}
Let $0 < \b_1 < \b_2$, and let $\phi$ and $\psi$ be $\rmh^\ann_{\bspi}$-equilibrium states for $\b_1 a$ and for $\b_2 a$, respectively.  Let $\pi$ and $\rho$ be their respective GNS representations.
\begin{enumerate}
\item[a.] We have $\s_{\max}(\pi(a)) \le \s_{\max}(\rho(a))$.
\item[b.] If $\rho$ is $\l$-normal, then $\pi$ is $\l$-normal.
\end{enumerate}
\end{lem}

\begin{proof}
\emph{Part (a).}\quad Let $m_1 := \s_{\max}(\pi(a))$ and $m_2 := \s_{\max}(\rho(a))$. Towards a contradiction, assume that $m_1 > m_2$.  Since we must have $\log \Delta \phi > - \infty$, it follows that $\pi \gtrsim \l$, and similarly for $\rho$.  So, if $m_1 > m_2$, then we must also have that $\phi$ is not $\l$-normal and $m_1 > \s_{\max}(\l(a))$.  By Corollary~\ref{cor:equilibrium-struct}(b), this requires that $\b_1 > F_a(m_1)$.

Since $\phi$ is an $\rmh^\ann_{\bspi}$-equilibrium state at $\b_1 a$, Lemma~\ref{lem:juxtaposition} gives
\[\rmp^\ann_{\bspi}(\b_1 a) = p_{\ker \pi}(\b_1 a) + \rmh^0_{\bspi}(\pi) \ge p_{\ker \rho}(\b_1 a) + \rmh^0_{\bspi}(\rho).\]
By Lemma~\ref{lem:diff-of-ps}, the difference $p_{\ker \pi}(\b a) - p_{\ker \rho}(\b a)$ is strictly increasing in $\b$ for $\b > F_a(m_1)$.  This range includes the interval $(\b_1,\b_2)$, so the inequality above implies that
\[p_{\ker \pi}(\b_2 a) + \rmh^0_{\bspi}(\pi) > p_{\ker \rho}(\b_2 a) + \rmh^0_{\bspi}(\rho).\]
Because of this, another appeal to Lemma~\ref{lem:juxtaposition} shows that $\psi$ cannot be a $\rmh^\ann_{\bspi}$-equilibrium state at $\b_2 a$.  This is the desired contradiction.

\vspace{7pt}

\emph{Part (b).}\quad As argued above, both of $\pi$ and $\rho$ contain $\l$.  Therefore, if $\rho$ is $\l$-normal, then part (a) implies that $\s_{\max}(\pi(a)) = \s_{\max}(\rho(a)) = \s_{\max}(\l(a))$.  In addition, if $\rho$ is $\l$-normal, then Corollary~\ref{cor:equilibrium-struct}(a) gives $F_a(z) = \b_2$ for some $z \ge \s_{\max}(\l(a))$.  It follows that
\[F_a(\s_{\max}(\pi(a))) = F_a(\s_{\max}(\l(a))) \ge F_a(z) = \b_2 > \b_1.\]
At the scalar multiple $\b_1 a$, this is incompatible with the conclusions in part (b) of Corollary~\ref{cor:equilibrium-struct}, so $\pi$ must also be $\l$-normal.
\end{proof}

Similarly to the critical value from Section~\ref{sec:AP-Legendre}, let us define the \textbf{annealed critical value} for $a$ along $\bspi$ by
\[\b_{\rm{c}}^\ann := \inf\big\{\b \ge 0:\ \hbox{some $\rmh^\ann_{\bspi}$-equilibrium state for $\b a$ is not $\l$-normal}\big\},\]
with the usual convention that $\inf \emptyset = +\infty$. Beware that this depend on both $a$ and the sequence $\bspi$, although we suppress these in our notation. 

Some consequences of Lemma~\ref{lem:maxs-monotone} are easiest to see in the following special case.

\begin{cor}\label{cor:maxs-monotone}
Assume that $\l(a)$ is not a multiple of the identity and also that $\b a$ has a unique $\rmh^\ann_{\bspi}$-equilibrium state $\phi_{\b a}$ for every $\b > 0$. Let $\pi_{\b a}$ be the GNS representation of $\phi_{\b a}$. Then
\[(\phi_{\b a})_{\rm{ac}} = \frac{F_a(z(\b))}{\b}\g_{a,z(\b)}\]
for the unique parameter value $z(\b) := \phi_{\b a}(a) + 1/\b$.  In addition:
\begin{itemize}
\item $z(\b) = F_a^{-1}(\b)$ for $\b \in (0,\b_{\rm{c}}^\ann)$, where this function of $\b$ is continuous and strictly decreasing;
\item $z(\b) = \s_{\max}(\pi_{\b a}(a))$ for $\b \in (\b_{\rm{c}}^\ann,\infty)$, where this function of $\b$ is non-decreasing. \qed
\end{itemize}
\end{cor}

The quantity $z(\b)$ need not be strictly increasing on $[\b_{\rm{ac}}^\ann,\infty)$ in Corollary~\ref{cor:maxs-monotone}. For example, suppose that $\G$ is a countable group and $\bspi$ is a deterministic AP sequence which converges strongly to $\l$ (recall Examples~\ref{ex:Green} and~\ref{ex:Green2}). Then the $\rmh^\ann_{\bspi}$-equilibrium states are already described by Corollary~\ref{cor:AP-equilibrium}. Inspecting the two possibilities listed there, we find in this example that $\b_{\rm{c}} = \b_{\rm{c}}^\ann$ and $z(\b) = \s_{\max}(\l(a))$ for all $\b \ge \b_{\rm{c}}^\ann$.

More generally, if $\b a$ may have mutiple $\rmh^\ann_{\bspi}$-equilibrium states for some values of $\b$, we can still consider the sets
\[E_\b := \{\s_{\max}(\pi_\phi(a)):\ \phi\ \hbox{an $\rmh^\ann_{\bspi}$-equilibrium state for}\ \b a\}.\]
If $0 \le \b_1 < \b_2$, then Lemma~\ref{lem:maxs-monotone} implies that every value in $E_{\b_1}$ is less than or equal to every value in $E_{\b_2}$, and so the set $E_\b$ is a singleton for all but countably many values of $\b$.  (Note, however, that the $\rmh^\ann_{\bspi}$-equilibrium state for $\b a$ need not be unique even if $E_\b$ is a singleton.)

\part{RANDOM AP SEQUENCES FOR FREE GROUPS}\label{part:free}

\chapter{Positive definite functions on free groups}\label{chap:free-prelims}

Much of Part~\ref{part:free} studies random positive definite functions on a finitely generated free group that are obtained by fixing a vector in $\bbC^n$ and then choosing a representation uniformly at random.  Before introducing randomness, we need some preliminaries about the geometry of these groups, and also a way of parametrizing those positive definite functions.  That parametrization generalizes the classical Verblunsky coefficients of a positive definite function on $\bbZ$.

Most of the results in Section~\ref{sec:group-geom} are standard, but the point of view in Section~\ref{sec:restrict-extend} is somewhat new.

Throughout this chapter, $\G$ is a group freely generated by a finite set $S$.

\section{Geometry of finitely generated free groups}\label{sec:group-geom}

The generating set $S\cup S^{-1}$ defines a Cayley graph on the vertex set $\G$.  It is a tree because $\G$ is free on $S$~\cite[Prop. I.15]{Ser--T}.  To be more precise, we have a choice between the left and right Cayley graphs.  We mostly use the \emph{left} Cayley graph, and later references to `the Cayley graph' all imply this choice.  The edges of this graph are the unordered pairs $\{g,sg\}$ for $g \in \G$ and $s \in S\cup S^{-1}$.  However, at a few points below we need the right Cayley graph instead, and this is always made explicit.

The graph metric of the Cayley graph is a right-invariant metric on $\G$.  For a group element $g$, its \textbf{length} $|g|$ is its distance from $e$ in this metric.  Equivalently, it the minimal length of a word in $S\cup S^{-1}$ that evaluates to $g$.  We write $B_n$ and $S_n$ for the closed ball and sphere of radius $n$ around $e$, respectively.  In particular, $S_1 = S\cup S^{-1}$.

Our main uses for this graph boil down to the following. We wish to interpret paths in the graph as sequences of moves around an orbit of a point under an action of $\G$, where each individual step is implemented by a generator or its inverse.  Since our actions are all from the left, this requires that we use the left Cayley graph.  (By contrast, the right Cayley graph is important as a graph on which $\G$ acts naturally by graph automorphisms.)

If $\G$ is the free group on $S$, then every element of $\G$ has a unique representation as a reduced word in $S\cup S^{-1}$.  If elements $g$ and $h$ of $\G$ are represented by reduced words $v$ and $w$ respectively, then $g$ is adjacent to $h$ in the left Cayley graph if and only if one of those words is an extension of the other by a single letter on the left.  More generally, if $v$ is an extension of $w$ by any number of letters on the left (resp. right), then we call $h$ a \textbf{suffix} (resp. \textbf{prefix}) of $g$.  For any $g$, its suffixes (resp. prefixes) are the vertices along the shortest path from $e$ to $g$ in the left (resp. right) Cayley graph.

If $E$ is a subset of $\G$, then its \textbf{exterior boundary} is the set of elements of $\G$ that are not in $E$ but are adjacent to $E$ in the left Cayley graph.  Similarly, the \textbf{interior boundary} of $E$ is the exterior boundary of $\G\setminus E$.

Assume now that $\G$ is freely generated by $S$. The following class of finite subsets of $\G$ plays a central role throughout the remaining chapters.

\begin{dfn}\label{dfn:grounded}
	A finite subset $F$ of $\G$ is \textbf{grounded} if it contains $e$ and is connected in the left Cayley graph.
\end{dfn}

Grounded sets may be visualized as the finite subtrees of the left Cayley graph that are rooted at the identity.  They (or their analogs for the right Cayley graph) appear naturally in many works on free groups.  For instance, they offer revealing choices of transversals to subgroups in Schreier's theorem: see, for example,~\cite[Prop. I.16]{Ser--T}.  They also arise implicitly in the procedure of `splitting' observables that Bowen uses for proving some of the properties of annealed sofic entropy (the `f-invariant') in~\cite{Bowen10free}. Our term for them follows~\cite[Sec. 14]{Oza13}.

If $F$ is a grounded subset of $\G$, and $g$ is any element of the exterior boundary of $F$, then $F\cup \{g\}$ is also grounded.  We often abbreviate such a union to $F\cup g$ in the sequel. Because the Cayley graph of $\G$ is a tree and $F$ is a connected subset of it, in this case there is a unique element $s$ of $S\cup S^{-1}$ such that $g \in sF$, and then we call $F\cup g$ an \textbf{enlargement of $F$ in direction $s$}. By removing interior boundary points one at a time and then reversing the process, any grounded set may be reached from $\{e\}$ by a sequence of enlargements.

Several proofs about grounded sets in the sequel are by \textbf{induction on the grounded set}. For a desired assertion, this means we prove (i) that it holds for $\{e\}$, and then (ii) that if it holds for a grounded set $F$ then it also holds for any enlargement of $F$.  This implies the assertion for any grounded set by the assembly procedure described above.

In some of our inductive proofs of this kind, we must keep careful track of the directed edges of the Cayley graph that are contained in each grounded set.  These directed edges are labeled by the generators from $S$ that give rise to them: the directed edges within $F$ that are labeled by generator $t$ emanate from the vertices in $t^{-1}F\cap F$.  The next lemma tracks how these families of edges grow when $F$ itself is enlarged.

\begin{lem}\label{lem:shift-enlargement}
If $F$ is a grounded set and $F' = F \cup g$ is an enlargement in direction $s \in S\cup S^{-1}$, then
\[F'\cap sF' = (F\cap sF) \cup \{g\},\]
\[F'\cap s^{-1}F' = (F\cap s^{-1}F)\cup\{s^{-1}g\},\]
and
\[F'\cap tF' = F\cap tF \quad \hbox{for all}\ t \in (S\cup S^{-1})\setminus \{s,s^{-1}\}.\]
\end{lem}

\begin{proof}
	Let $t \in S\cup S^{-1}$, and suppose that
	\[h \in (F'\cap tF')\setminus (F\cap tF).\]
	In particular, $h$ lies in both $F'$ and $tF'$, but must actually lie in at least one of $F'\setminus F = \{g\}$ and $tF'\setminus tF = \{tg\}$. This gives us two cases:
	\begin{itemize}
	\item If $h=g$, then we must also have $g \in tF'$, and hence $t^{-1}g \in F$.  Since $g$ is an external boundary point of $F$, this is possible only if $t=s$.
	\item If $h = tg$, then we must also have $tg \in F'$, and hence $tg \in F$.  Since $g$ is an external boundary point of $F$, this is possible only if $t=s^{-1}$.
	\end{itemize}
\end{proof}

There are many ways one can begin with $\{e\}$ and grow a sequence of grounded sets that exhaust $\G$ by enlargements.  At a few points below we need such sequences explicitly.  A \textbf{grounded enumeration} is an enumeration
\[e = g_0, g_1, g_2, \dots\]
of $\G$ such that $g_{i+1}$ is an external boundary point of $\{g_0,\dots,g_i\}$ for every $i$.  In any grounded enumeration, $\{g_0,\dots,g_i\}$ is a grounded set for every $i$, and each of these sets is an enlargement of its predecessor for $i\ge 1$.  For example, any enumeration that satisfies
\begin{equation}\label{eq:growing-length}
0=|g_0| \le |g_1| \le |g_2|\le \cdots
\end{equation}
is grounded -- we refer to these examples as \textbf{length-first} enumerations.

An even more specialized way to select an enumeration that satisfies~\eqref{eq:growing-length} is to place a total order on $S\cup S^{-1}$, and then apply the \textbf{length-lexicographic} ordering to elements of $\G$ by regarding them as reduced words.  In this ordering, $g$ precedes $h$ if and only if either $|g| < |h|$, or $|g| = |h|$ and the first letter of $g$ that differs from the corresponding letter of $h$ precedes that letter of $h$.

\section{Verblunsky coefficients}\label{sec:restrict-extend}

Let $\phi \in \S^\rm{u}_k(\G)$.  For any finite $F\subset \G$, we can consider the $F$-by-$F$ block matrix
\begin{equation}\label{eq:phi-QF}
\phi[F] := [\phi(g^{-1}h):\ g,h \in F],
\end{equation}
where each block lies in $\rmM_k$. These belong to the following sets.

\begin{dfn}
	For any finite subset $F$ of $\G$, we write $\S^\rm{u}_k(F)$ for the set of all positive semidefinite elements $Q$ of $\rmM_F(\rmM_k)$ that (i) have all diagonal blocks $Q(g,g)$ equal to $I_k$, and (ii) satisfy the symmetry
\begin{equation}\label{eq:Toe}
	Q(g_1,h_1) = Q(g_2,h_2) \quad \hbox{whenever}\ g_1,g_2,h_1,h_2 \in F\ \hbox{and}\ g_1^{-1}h_1 = g_2^{-1}h_2.
\end{equation}
	We refer to the elements of $\S^\rm{u}_k(F)$ as \textbf{partial positive definite functions over $F$}.
	
	We write $\S_k^\circ(F)$ for the subset of nonsingular members of $\S^\rm{u}_k(F)$.

If $\phi \in \S^\rm{u}_k(\G)$, then it is \textbf{nonsingular} if $\phi[F] \in \S^\circ_k(F)$ for every finite $F\subset \G$.
\end{dfn}

Starting from $\phi$, the matrices $\phi[F]$ in~\eqref{eq:phi-QF} are consistent in that $\phi[F]$ and $\phi[F']$ have the same submatrix indexed by $F\cap F'$.  On the other hand, given a consistent family of matrices $Q_F \in \S^\rm{u}_k(F)$ indexed by some upwards-directed family of finite subsets $F$ that cover $\G$, we can set $\phi(g) := Q_F(e,g)$ for any $F$ that contains $\{e,g\}$, and so produce an element of $\S^\rm{u}_k(\G)$ that gives rise to this family of matrices via~\eqref{eq:phi-QF}.  In this way, $\S^\rm{u}_k(\G)$ is identified with the space of consistent families of such $F$-by-$F$ block matrices.

For general locally compact groups, a classic question of abstract harmonic analysis asks when a partial positive definite function over a subset has a positive definite extension to the whole group.  When $\G = \bbR$ and $F$ is an interval, a classic result of Krein~\cite{Kre40} shows that this is always possible. The corresponding result for discrete intervals in $\bbZ$ is similar but simpler.  On the other hand, even when $\G = \bbZ^2$ and $F$ is a box, examples of Rudin~\cite{Rud63,Rud70} show that extension may be impossible, and similar examples seem to be available over most other groups. See~\cite{BakTim11} for recent advances and a selection of further references.

However, generalizing that result about $\bbZ$, extension is always possible when $\G$ is a free group and $F$ is a grounded subset. That is, for any such $F$ and any positive integer $k$, the map
\begin{equation}\label{eq:big-surj}
\S^\rm{u}_k(\G) \to \S^\rm{u}_k(F)
\end{equation}
defined by~\eqref{eq:phi-QF} is surjective.  Indeed, this is true even if we replace $\rmM_k$ with $\B(H)$ for any complex Hilbert space $H$.  This is essentially a result of Bakonyi and Timotin~\cite[Thm. 4.3]{BakTim07}, although they do not consider general grounded sets explicitly.  The variant of their proof in~\cite[Lem. 25]{Oza13} does allow this generality.  That section of~\cite{Oza13} offered strong inspiration for several of the arguments in this chapter and also Chapter~\ref{chap:LDP-proof}.

The key step in~\cite[Sec. 14]{Oza13} is showing that the submatrix map
\begin{equation}\label{eq:little-surj}
\S^\rm{u}_k(F\cup g)\to \S^\rm{u}_k(F)
\end{equation}
is surjective when $F$ is grounded and $F\cup g$ is an enlargement of it.  This implies the surjectivity of~\eqref{eq:big-surj} by extending an element of $\S^\rm{u}_k(F)$ through a sequence of enlargements that exhausts $\G$.

This surjectivity can be recast as an example of the completion problem for $3$-by-$3$ block positive semidefinite matrices: see Section~\ref{sec:three-block-Gram}.  In~\cite[Lem. 24]{Oza13}, Ozawa simply chooses the central completion of the relevant partial Gram matrix at each stage (see the remarks at the end of Section~\ref{sec:three-block-Gram}).

Theorem~\ref{mainthm:LDP} is a large deviations principle for certain random elements of the spaces $\S^\rm{u}_k(F)$.  It is proved by induction on $F$.  Along a sequence of grounded sets $F$, a sequence of these random positive semidefinite matrices is revealed one enlargement at a time.  To describe the distributions of these random matrices, we need not just the surjectivity of the map~\eqref{eq:little-surj} for each $F$ and $F\cup g$, but also a parametrization of all possible pre-images of an element of $\S^\rm{u}_k(F)$.  We obtain this using Proposition~\ref{prop:three-block-completion}.

In the rest of this section we explain this parametrization in more detail.  Suppose that $F$ is a grounded set and that $F\cup g$ is an enlargement of $F$ in direction $s$.  Let $Q \in \S^\rm{u}_k(F)$, and consider all possible extensions $Q' \in \S^\rm{u}_k(F\cup g)$.  Then $Q$ is the submatrix of $Q'$ obtained by removing the block-row and block-column indexed by $g$.  The symmetry~\eqref{eq:Toe} dictates some elements of this row and column: for example, we must have $Q'(g,h) = Q(s^{-1}g,s^{-1}h)$ whenever $h \in F\cap sF$.  It turns out that these are the only entries of $Q'$ that are determined by $Q$ through equation~\eqref{eq:Toe}.  This fact is contained within the proof of~\cite[Lem. 25]{Oza13}.

\begin{lem}\label{lem:Toe}
If $h,a,b \in F$ satisfy $a^{-1}b = h^{-1}g$, then $h \in F\cap sF$.
\end{lem}
	
\begin{proof}
Since $F$ is grounded, it contains $e$ and is connected in the left Cayley graph.  Therefore, in addition to $a$ and $b$, $F$ contains all their suffixes.  We may therefore cancel a common prefix and so assume that $a$ and $b$ begin with different letters.

Now there are two cases.  First, if $s$ is the first letter of the reduced word of $h$, then $s^{-1}h$ is a suffix of $h$, so $s^{-1}h \in F$ by connectedness.

On the other hand, if $s$ is not the first letter of $h$, then $h$ and $g$ have no prefix in common.  In this case there is no cancellation between the reduced words of $h^{-1}$ and $g$ when they are multiplied, as we have already arranged for $a$ and $b$.  Since $b\in F$ and $a^{-1}b = h^{-1}g$, this is possible only if $b$ is a proper suffix of $g$, and therefore $h$ is a proper suffix of $a$.  Because of the equality $h^{-1}g = a^{-1}b$, this properness implies that $s^{-1}h$ is still a suffix of $a$ (now possibly equal to $a$), and so again it lies in $F$ by connectedness.
\end{proof}	
	
Now imagine ordering the set $F\cup g$ so that $F\setminus sF$ comes first, $F\cap sF$ comes next, and $g$ comes last.  Use this ordering to write $Q$ and $Q'$ as block matrices:
\begin{equation}\label{eq:K-ext}
Q = \left[\begin{array}{cc} Q_{11} & Q_{12}\\ Q_{12}^\ast & Q_{22}\end{array}\right] \quad \hbox{and} \quad Q' := \left[\begin{array}{ccc} Q_{11} & Q_{12} & R\\ Q_{12}^\ast & Q_{22} & Q_{23} \\ R^\ast & Q_{23}^\ast & I_k\end{array}\right].
\end{equation}
Let $Q^?$ be the partial matrix obtained from $Q'$ by replaced $R$ with $?$.

The possible extensions $Q'$ of $Q$ are the positive semidefinite completions of $Q^?$.  If $Q^?$ is partially nonsingular, then these completions are parametrized by contractions from $\bbC^{\oplus k}$ to $\bbC^{\oplus k|F\setminus sF|}$ according to Proposition~\ref{prop:three-block-completion}.  This partial nonsingularity holds if and only if $Q$ is nonsingular, since the two relevant two-by-two-block submatrices of $Q^?$ are both submatrices of $Q$ by translation-invariance.

\begin{dfn}\label{dfn:Verb}
If $Q \in \S^\circ_k(F)$ and $Q'$ is an extension of it in $\S^\rm{u}_k(F\cup g)$, then the \textbf{Verblunsky coefficient} of $Q'$ over $Q$ is the element of $\Xi(k,k|F\setminus sF|)$ that parametrizes it according to Proposition~\ref{prop:three-block-completion}.
	
If $\phi \in \S^\rm{u}_k(\G)$, $F$ is a grounded set, $\phi[F]$ is nonsingular, and $F\cup g$ is an enlargement of $F$, then the \textbf{Verblunksy coefficient} of $\phi$ from $F$ to $F\cup g$ is the Verblunsky coefficient of $\phi[F\cup g]$ over $\phi[F]$.
		\end{dfn}

The use of this term is extended from the classical case when $\G = \bbZ$ and $k=1$. By Bochner's theorem, a unital positive definite function $\phi$ on $\bbZ$ is the Fourier--Stieltjes transform of a probability measure $\mu$ on the unit circle $\bf{K}$ in $\bbC$.  As a result, if $F = \{1,2,\dots,m\}$, then $\phi[F]$ is the Toeplitz matrix $[\langle z^j,z^i\rangle_{L^2(\mu)}]_{i,j=1}^m$.  The associated sequence $(p_n)_{n\ge 0}$ of orthogonal polynomials is obtained by applying the Gram--Schmidt procedure to the sequence $1$, $z$, $z^2$, \dots in $L^2(\mu)$.  These polynomials are related by a recursive formula due to Szeg\H{o}~\cite[Thm. 1.5.2]{SimOPUCI}.  Provided $\mu$ is not finitely supported, so $1$, $z$, \dots are linearly independent in $L^2(\mu)$, that recursion introduces an infinite sequence of coefficients $\a_n$ called the Verblunsky coefficients.  They lie in the open disk in the complex plane, and their geometric interpretation (explained in~\cite[Sec. 1.5]{SimOPUCI}) shows that they are precisely the special case of our Definition~\ref{dfn:Verb}.  The generalization with $\G = \bbZ$ but $k > 1$ has also been studied fairly completely: see the introduction and further references in~\cite[Sec. 2.13]{SimOPUCI}.

Our work in this section is a generalization for free groups of this classical parametrization of positive definite functions on $\bbZ$ by their Verblunsky coefficients.  This point of view is also taken by Bakonyi and Timotin in~\cite{BakTim07} (they refer to `Szeg\H{o} parameters' rather than Verblunsky).  Our parametrization is a technical variation on theirs.  Some of our other results later are also free-group generalization of known theorems about orthogonal polynomials.

Our next result is the free-group generalization of Verblunsky's original theorem in this area: compare~\cite[Thm. 1.7.11]{SimOPUCI}. 

\begin{cor}\label{cor:Verb}
Fix a grounded enumeration $g_0,g_1,\dots$ of $\G$, let $F_n := \{g_0,g_1,\dots,g_n\}$ for each $n$, and assume that $F_{n+1}$ is an enlargement of $F_n$ in direction $s_n$.  Given $\phi \in \S^\circ_k(\G)$, let $C_n$ be the Verblunsky coefficient of $\phi$ from $F_n$ to $F_{n+1}$.  Then the map
\begin{equation}\label{eq:Verb-map}
	\S^\circ_k(\G) \to \prod_{n\ge 0}\Xi^\circ(k,k|F_n\setminus s_nF_n|):\phi\mapsto (C_0,C_1,\dots)
	\end{equation}
	is a bijection.
	\end{cor}

\begin{proof}
First suppose we know the sequence $C_0,C_1$, \dots, for some $\phi \in \S^\circ_k(\G)$.  Then $\phi(e,e) = I_k$, and once $\phi[F_n]$ is known, it and the contraction $C_n$ determine $\phi[F_{n+1}]$.  So the given map is injective.

On the other hand, given any such sequence $C_0,C_1,\dots$, we can construct a compatible sequence of elements $Q_n \in \S_k^\circ(F_n)$ by setting $Q_0 := I_k$ and then recursively appealing to Proposition~\ref{prop:three-block-completion}.   As explained at the start of the section, these all come from a single element of $\S^\circ_k(\G)$, and that element is nonsingular by construction.  So the given map is surjective.
	\end{proof}

By allowing generalized Schur complements, the inverse of the map~\eqref{eq:Verb-map} can be extended to parametrize the whole of $\S^\rm{u}_k(\G)$, not just $\S_k^\circ(\G)$.  But this is more complicated, particularly when the topologies on those spaces of contractions must be taken into account, so instead we find ways to work around this case in the sequel.  The necessary modifications are already a little fiddly when $\G = \bbZ$.  See, for instance,~\cite[Intro.]{BreSimZei18}, where the authors have to allow finitely supported measures because they are working towards a large deviations principle on the space of \emph{all} probability measures on $\bbT$.

We also sometimes need the continuity given by the following consequence of Corollary~\ref{cor:three-block-completion}.

\begin{lem}\label{lem:pdf-completion}
Let $F$ be a grounded set and let $F\cup g$ be an enlargement of $F$ in direction $s$.  Given $Q \in \S^\rm{u}_k(F\cup g)$ such that $Q[F]$ is nonsingular, let $C$ be its Verblunsky coefficient over $Q[F]$.  Then the map
\begin{align*}
\big\{Q \in \S^\rm{u}_k(F\cup g):\ Q[F]\ \hbox{nonsingular}\big\} &\to \S_k^\circ(F)\times \Xi(k,k|F\setminus sF|):\\ Q &\mapsto (Q[F],C)
\end{align*}
is a homeomorphism, and it maps $\S_k^\circ(F\cup g)$ onto $\S_k^\circ(F)\times \Xi^\circ(k,k|F\setminus sF|)$.
\end{lem}

\begin{proof}
We identify $\S^\rm{u}_k(F\cup g)$ as the subset of $\rmM_{k(|F| + 1)+}$ defined by the symmetries~\eqref{eq:Toe}; restrict the map from Corollary~\ref{cor:three-block-completion} to this subset; and then identify its image as described above.
\end{proof}

The use of Verblunsky coefficients to parametrize completely positive maps on free groups is a great advantage in the sequel.  It has no clear analog in the ergodic theory of stationary processes over free groups.  This enables new proofs about annealed AP entropy that are quite different from their counterparts for annealed sofic entropy, including in cases where those ergodic theoretic proofs do not extend.


For $\G = \bbZ$, Verblunsky's thoerem is the beginning of a long-running search for implications between properties of a measure $\mu$ on $\bf{K}$ and properties of its Verblunsky coefficients.  We discuss some possible analogs of these over free groups in Subsection~\ref{subs:harm-an} below.

The tempered positive definite functions on a free group admit various characterizations.  One of the most elementary is due to Haagerup: see~\cite[Thm. 3.1]{Haa79}.  The next theorem is a less precise version of that one.

\begin{thm}\label{thm:Haa79}
If $\G$ is freely generated by $S$ and $\phi:\G\to \bbC$ is positive definite, then $\phi$ is tempered if and only if
\[\sum_{g \in \G,\ |g|=n}|\phi(n)|^2 = e^{o(n)} \qquad \hbox{as}\ n\to\infty.\]
\qed
\end{thm}

As an alternative to Theorem~\ref{thm:Haa79}, we see that Theorem~\ref{mainthm:tempered} provides a characterization of temperedness in terms of the Verblunsky coefficients of $\phi$, rather than the values it takes.  This is because we can express $\hann(\phi)$ in terms of those coefficients: see Corollary~\ref{cor:first-expansion} below.  We do not use Theorem~\ref{thm:Haa79} in any of the general proofs in this work, but it does help with the analysis of some concrete examples.  For positive definite functions on a more general locally compact group, a variation on Theorem~\ref{thm:Haa79} is due to Cowling, Haagerup and Howe~\cite{CowHaaHow88}.  From a different point of view, temperedness of positive definite functions can also be characterized in terms of `boundary realizations': see~\cite{Kuh94} and~\cite[Prop. 2.2]{HebKuhSte22}.  It would be interesting to know whether this last characterization has any closer relation to AP entropy and Verblunsky coefficients.

\section{Haagerup functions and central extensions}\label{sec:Haa}

Let $\phi(e) := 1$, and let $\phi(s) = \ol{\phi(s^{-1})}$ be a complex number of modulus at most one for each $s \in S$.  Now extend $\phi$ by multiplying:
\[\phi(s_1\cdots s_n) = \phi(s_1)\cdots \phi(s_n) \qquad \hbox{for any reduced word}\ s_1\cdots s_n\ \hbox{in}\ S\cup S^{-1}.\]
A function on $\G$ constructed this way is called a \textbf{Haagerup function}.  Haagerup himself studied the examples of the form $\phi(g) = e^{-t|g|}$ in~\cite{Haa79}; we refer to these as \textbf{radial}.  Then the larger class was introduced and named in~\cite{DeMFigTal80}, which starts with a proof that they are positive definite~\cite[Thm. 1]{DeMFigTal80}.  They are perhaps the simplest positive definite functions on free groups after the regular character.  Note that they depend on the choice of the free generating set $S$, and they are never characters apart from the trivial examples of the regular character and $1_G$.  They generalize~\cite[Exs. 1.6.4 and 2.1.6]{SimOPUCI} from the classical case $\G = \bbZ$. We refer to them several times later to illustrate our main results.  A thorough account of Haagerup functions can be found in the monograph~\cite{FigTalSte94}.

Suppose that the Haagerup function $\phi$ is associated to $\pi = \pi_\phi$ by the cyclic unit vector $z$, and set $z_g := \pi(g)z$ for all $g \in \G$.  The analysis of these examples begins with the following lemma.

\begin{lem}\label{lem:like-Markov}
Let $F$ be a finite connected subset of the \emph{right} Cayley graph of $(\G,S)$, and suppose that $g \in F$ but $gs \in \G\setminus F$ for some $s \in S\cup S^{-1}$.  Then the orthogonal projection of $z_{gs}$ onto $\rm{span}\{z_h:\ h \in F\}$ is equal to $\phi(s)z_g$.
\end{lem}

\begin{proof}
Let $h \in F$, and let
\[h,\ ht_1,\ ht_1t_2,\ \dots,\ ht_1\cdots t_\ell = g\]
be the unique path in the right Cayley graph from $h$ to $g$.  Since $F$ is connected, this path lies in $F$, and therefore it does not contain $gs$ and so $s \ne t_\ell^{-1}$.  It follows that
\[\langle z_{gs},z_h\rangle = \langle \pi(h^{-1}gs)z,z\rangle = \phi(t_1\cdots t_\ell s) = \phi(t_1\cdots t_\ell)\phi(s) = \langle \phi(s)z_g,z_h\rangle.\]
\end{proof}

Now let $r = |S|$ and pick a length-first enumeration $e = g_0$, $g_1$, \dots of $\G$, such that $S = \{g_1,\dots,g_r\}$ and $S^{-1} = \{g_{r+1},\dots,g_{2r}\}$. Let $F_n := \{g_0,g_1,\dots,g_n\}$ for each $n$, and for $n\ge 1$ suppose that $F_n$ is an enlargement of $F_{n-1}$ in direction $s_n$.  Finally, for each $n\ge 1$, let $h_n$ be the prefix of $g_n$ of length $|g_n|-1$, and let $t_n := h_n^{-1}g_n$.

\begin{cor}\label{cor:Haa-Verb1}
Continue to assume that $\phi$ is a Haagerup function.  If $n\ge 1$, then the orthogonal projection of $z_{g_n}$ onto $\rm{span}(\pi(F_{n-1})z)$ equals $\phi(t_n)z_{h_n}$.
\end{cor}

\begin{proof}
Let $r := |g_n|$.  Since we use a length-first enumeration, we have
\[B_{r-1} \subset F_n \subset B_r\qquad \hbox{and hence also} \qquad B_{r-1} \subset F_n^{-1} \subset B_r.\]
Since inversion in $\G$ exchanges the left and right Cayley graphs, this shows that $F_n$ is connected in both of these Cayley graphs, and also that $h_n$ lies in $F_{n-1}$ because $|h_n| < |g_n|$.  Therefore, in the right Cayley graph, $h_n$ is the unique neighbour of $g_n$ that belongs to $F_{n-1}$. Now we can apply Lemma~\ref{lem:like-Markov}.
\end{proof}

\begin{cor}\label{cor:Haa-Verb2}
If $\phi$ is a Haagerup function, then for the enumeration above the Verblunsky coefficient from $F_{n-1}$ to $F_n$ equals $\phi(g_n)$ for $n=1,2,\dots,r$, and all subsequent Verblunsky coefficients vanish.
\end{cor}

\begin{proof}
There are three cases:
\begin{itemize}
\item If $1 \le n \le r$, then $g_n = t_n \in S$, $h_n = e$, and $F_{n-1}\cap g_nF_{n-1} \subset S\cap g_nS= \emptyset$.  By Corollary~\ref{cor:Haa-Verb1}, the orthogonal projection of $\pi(g_n)z$ onto the span of $\pi(F_{n-1})z$ is equal to $\phi(g_n)z$, whereas the span of $\pi(F_{n-1}\cap g_nF_{n-1})z$ is trivial.  The resulting Verblunsky coefficient is $\phi(g_n)$ by Proposition~\ref{prop:three-block-completion}. 
\item If $r+1 \le n \le 2r$, then $g_n = t_n \in S^{-1}$ and $h_n = e$.  Since $g_n^{-1} \in S$, this inverse already belongs to $F_{n-1}$, so this time we have $F_{n-1}\cap g_nF_{n-1} = \{e\}$.  By Corollary~\ref{cor:Haa-Verb1}, the orthogonal projection of $z_{g_n}$ onto $\rm{span}(\pi(F_{n-1})z)$ equals $\phi(g_n)z$, which lies in $\rm{span}(\pi(F_{n-1}\cap g_nF_{n-1})z)$ and so equals the orthogonal projection of $z_{g_n}$ onto that smaller subspace. Since these projections agree, the Verblunsky coefficient is zero, by Proposition~\ref{prop:three-block-completion}.
\item Finally, if $n > 2r$, then $s_n$ is the first letter of the reduced word of $g_n$, and also of $h_n$ because in this case $|h_n| \ge 1$.  Overall, this shows that $h_n \in F_n\cap s_n F_n$, and $h_n$ separates $g_n$ from the rest of $F_n$ in the right Cayley graph.  Therefore, by Corollary~\ref{cor:Haa-Verb1}, the two relevant orthogonal projections are again both equal to $\phi(t_n)z_{h_n}$, and so once again the Verblunsky coefficient vanishes. 
\end{itemize}
\end{proof}

Corollary~\ref{cor:Haa-Verb2} actually holds for any grounded enumeration, but the general case requires some more complicated calculations because then the sets $F_n$ need not be connected in the right Cayley graph.  One has to consider the projection of $z_{g_n}$ to the span of all the vectors in $\pi(F_n)z$ which correspond to the points of each right-connected component of $F_n$ that are closest $g_n$.  Concerning a slightly different generalization of Veblunsky coefficients, Bakonyi and Timotin prove a variant of Corollary~\ref{cor:Haa-Verb2} in~\cite[Thm. 6.3]{BakTim07}.

When $\G = \bbZ$, Haagerup functions are the first-order Bernstein--Szeg\H{o} polynomials from the study of orthogonal polynomials on the unit circle: the family for which all but the first Verblunsky coefficients vanish.  See~\cite[Ex. 1.6.2 and Sec. 1.7]{SimOPUCI}.  In keeping with our interpretation of `conditioning' as projecting to an orthogonal subspace (see Example~\ref{ex:1D-cond-ent}), one can also regard Haagerup functions as the representation theoretic analog of free-group Markov processes, as already noted by Burton and Juschenko in~\cite{BurJus22}.  See~\cite{Bowen10b,Sew14} for an introduction to Markov chains on free groups.  The reference~\cite[Sec. 5]{BakTim07} discusses this `independence' phenomenon further for positive definite functions on free groups, and extends it to $\rmM_k$-valued positive definite functions.


More generally, for any grounded set $F$ and $Q \in \S_k(F)$, one way to obtain an extension of $Q$ in $\S_k(\G)$ is to pick a grounded enumeration $g_0$, $g_1$, \dots, such that $F = \{g_0,\dots,g_N\}$ for some $N$, and then take all subsequent Verblunsky coefficients to be zero (so all our one-step extensions starting from $Q$ are central completions, as described in Section~\ref{sec:three-block-Gram}).  We call this the \textbf{central extension} of $Q$ along this enumeration.  Examples constructed this way include the Haagerup functions when $F = \{e\}\cup S$.  As defined here, the central extension of $Q$ depends on a choice of grounded enumeration, but we give an alternative characterization using entropy in Proposition~\ref{prop:Burg} below, and that implies that the central extension depends only on $Q$ itself. Central extensions are also defined and studied by Bakonyi and Timotin in~\cite[Secs. 5 and 6]{BakTim07}, but they work with a different generalization of Verblunsky coefficients, so some translation is required.  See~\cite[Thm. 5.3]{BakTim07} for their version of the result that the central extension does not depend on a choice of enumeration of the group.

When $\G = \bbZ$ and $F = \{0,1,\dots,N\}$, this construction reduces to the higher-order Bernstein--Szeg\H{o} polynomials~\cite[Ex. 1.6.2]{SimOPUCI}.  By a generalization of~\cite[Thm. 1.7.8]{SimOPUCI}, higher-order Bernstein--Szeg\H{o} polynomials can be used to approximate an arbitrary positive definite function, and this fact plays an important role in several of the results in~\cite{SimOPUCI}. For finite-valued stochastic processes, the analog of this is the approximation of general processes by $n$-step Markov processes~\cite[Sec. 16.8]{CovTho06}.  This fact also has a version for processes over free groups which plays a role in the theory of annealed sofic entropy~\cite{Bowen10b,Sew14}.  Closing this circle of connections, for any positive definite function $\phi$ on $\G$, we can let $\phi_n$ be the central extension of $\phi[B_n]$ for each $n$.  Then $\phi_n[B_n] = \phi[B_n]$ by construction, so $\phi_n\to \phi$ as $n\to\infty$.  This provides a canonical approximation to $\phi$ using functions with only finitely many nonzero Verblunsky coefficients.

When $|S| \ge 2$, qualitative features of Haagerup functions depend on how close they are to the regular character.  One of the most basic results is the following.

\begin{prop}\label{prop:Haa-ac}
Let $\phi$ be a Haagerup function, and let
\[r := \sum_{s \in S}\frac{|\phi(s)|^2}{1 + |\phi(s)|^2}.\]
\begin{enumerate}
\item[a.] If $r < 1/2$, then $\phi$ is square summable over $\G$ and is $\l$-normal.
\item[b.] If $r=1/2$, then $\phi$ is $\l$-normal but not square summable.
\item[c.] If $r > 1/2$, then $\phi$ is not tempered. \qed
\end{enumerate}
\end{prop}

All three parts of this proposition can be found in~\cite{KuhSte95}.  If $\phi$ is non-vanishing, then the fact that $\phi$ is square-summable if and only if $r <1/2$ follows by applying~\cite[Lem. 2.1.5]{FigTalSte94} to $|\phi|^2$ (beware that they sum over the whole of $S\cup S^{-1}$, so their threshold is $1$ rather than $1/2$).  If $|\phi|$ is not strictly positive, then we can make the same argument over the subgroup $\{g:\ \phi(g)\ne 0\}$ of $\G$, which is generated by its intersection with $S$.  When $\phi$ is square summable, this implies $\l$-normality by a general property of positive definite functions on groups~\cite[Lem. 14.1.1]{Dix--Cstar}, completing the proof of (a).  On the other hand, part (c) follows by using the multiplicativity of $\phi$ to check Haagerup's criterion from Theorem~\ref{thm:Haa79}.  Finally, the $\l$-normality of $\phi$ in part (b) in the most delicate conclusion: this is~\cite[Thm. 4]{KuhSte95}.


%
%
%

If a Haagerup function $\phi$ is $\l$-normal and satisfies $\phi(s)\ge 0$, then it arises from the Green's function of a certain nearest-neighbour difference equation on the Cayley graph of $\G$ (recall Example~\ref{ex:Green}, and see~\cite{FigTalSte94}).  We return to this connection in Section~\ref{sec:Haa-as-equil}, where we relate it to an identification of Haagerup functions as the $\rmh^\ann_{\bspi}$-equilibrium states of self-adjoint elements of $\bbC[\G]$ that are supported on $S\cup S^{-1}$, where $\bspi$ is uniformly random AP sequence (introduced in Section~\ref{sec:unif-rndm-AP} below).  Using this fact, Corollary~\ref{cor:equilibrium-struct} gives us a new way to recover several known facts about the structure of the GNS representations of Haagerup functions in different regimes, and to study some other representations related to them.

\begin{rmk}\label{rmk:Bozejko}
The multiplicative formula that defines a Haagerup function has a natural matrix-valued generalization using contraction matrices as ingredients.  It appears first in~\cite[Sec. 7]{Boz89}.  In~\cite[Sec. 6]{BakTim07}, Bakonyi and Timotin identify these examples as central extensions as well.  They also appear in Lehner's proof of a matrix-valued Ackemann--Ostrand formula in~\cite{Leh99}, which generalizes the scalar-valued proof that uses Haagerup functions from~\cite{PicPyt88} (see also Remark~\ref{rmk:history} below).  Our later results about entropy for these matrix-valued generalizations of Haagerup functions should generalize, but may involve some new complications. \fin
\end{rmk}

\subsection*{\emph{Notes and further references}}

Even for locally compact Abelian groups, problems of characterizing and extending locally-defined positive definite functions quickly become rich and subtle, and have been studied widely.  The monograph~\cite{JorPedTia16} is mostly dedicated to such questions, with a number of sections that allow non-Abelian Lie groups as well.

\chapter{Uniformly random representations}\label{chap:random-orbits}

Our main results in Part~\ref{part:free} involve annealed and zeroth-order AP entropy for `uniformly random' AP sequences for free groups.  The present chapter lays foundations for this work in the form of some calculations with the distributions of these particular random AP sequences.  At the heart of these is a free-group generalization of the Killip--Nenciu theorem: see Theorem~\ref{thm:dil-dist} and Corollary~\ref{cor:KilNenfree} below.

\section{Uniformly random AP sequences}\label{sec:unif-rndm-AP}

Let $\G$ be the group freely generated by a set $S$ of size $r$, and let $\A := C^\ast \G$.  Let $(d_n)_{n\ge 1}$ be a divergent sequence of positive integers, and let $\nu_n$ be a Borel probability measure on $\rmU(d_n)$ for each $n$.  Finally, for each $n$, let $\pi_n$ be a random element of $\Rep_{d_n}(\G)$ such that the generators $\pi_n(s)$, $s \in S$, are independent and all have distribution $\nu_n$.  By the universal property of $C^\ast \G$, each $\pi_n$ extends uniquely to a random unitary representation of $C^\ast \G$.

Within this construction, two further special cases stand out:
\begin{itemize}
\item Let $d_n = n$, and let $\nu_n$ be the Haar probability measure of $\bf{U}(d_n)$.  We call the resulting example a \textbf{uniformly random AP sequence} for $\G$ (or for $C^\ast \G$).
\item Alternatively, let $d_n = n$, and let $\nu_n$ be the uniform distribution on the finite set of $n$-by-$n$ permutation matrices.  Then $\pi_n$ is the linear representation defined by a uniformly random permutation representation of $\G$.  We call the resulting example a \textbf{uniformly random permutation AP sequence} for $\G$ (or for $C^\ast \G$).
\end{itemize}

Both of these special cases have a large literature, although the terminology of `AP sequences' is not standard.  Almost all of our work below concerns uniformly random AP sequences for free groups.  By contrast, it turns out that variants of AP entropy do not shed much light on uniformly random permutation AP. This is because the most important fluctuations in behaviour that arise from permutations are only polynomially unlikely in $n$, not exponentially unlikely: see Section~\ref{sec:rndm-perms}.

The template above also allows other natural variations.  We do not explore them in detail in this book, but discuss a few as possible directions for further research in Chapter~\ref{chap:further}.  That chapter also discusses some examples of possible interest for other groups or other C*-algebras. 

The next two results are versions of Voiculescu's fundamental discovery of `asymptotic freeness' for independent high-dimensional random matrices; see~\cite[Thms. 3.8 and 3.9]{Voi91} (and note that the first of these covers rather more than we formulate below).  Let $\tau$ be the regular tracial state on $C^\ast \G$.

\begin{thm}\label{thm:asymp-free1}
If $\bspi$ is a uniformly random AP sequence, then the expectation $\bbE(\tr_n\circ \pi_n)$ converges to $\tau$ in $\S_1(\G)$ as $n\to\infty$. \qed
	\end{thm}

Combined with standard measure concentration estimates, Theorem~\ref{thm:asymp-free1} actually implies the following, whose specific form we need during at a few points later.  

\begin{thm}\label{thm:asymp-free2}
If $\bspi$ is a uniformly random AP sequence, then for every neighbourhood $V$ of $\tau$ in $\S_1(\G)$ there is a positive constant $c$ such that
\[\bbP(\tr_n\circ \pi_n \not\in V) = O(e^{-cn^2}).\]
\end{thm}

\begin{proof}
The topology of $\S_1(\G)$ agrees with the topology of pointwise convergence.  Recalling the usual sub-base for this topology, it suffices to prove that, for each $g \in \G\setminus \{e\}$ and $\eps > 0$, there exists $c > 0$ such that
	\[\bbP(|\tr_n \pi_n(g)| \ge \eps) = O(e^{-cn^2}).\]
By Theorem~\ref{thm:asymp-free1}, this follows if we find $c > 0$ such that
\[\bbP(|\tr_n \pi_n(g) - \bbE (\tr_n \pi_n(g))| \ge \eps/2) = O(e^{-cn^2}).\]
Finally, this follows from Theorem~\ref{thm:Un-conc} because, for each $g \in \G$, the map
\[(\pi_n(s):\ s \in S) \mapsto \tr_n\pi_n(g)\]
is $O(|g|/n)$-Lipschitz for the Hilbert--Schmidt metric on $\bf{U}(n)^S$.
	\end{proof}

Some more recent results about random AP sequences have concerned strong convergence. The next is contained within a theorem of Collins and Male~\cite[Thm. 1.4]{ColMal14}.

\begin{thm}\label{thm:ColMal}
For a finitely-generated free group, the approximate equivalence classes of uniformly random AP sequences sq-converge to the class of the regular representation in probability. \qed
\end{thm}

The paper~\cite{ColMal14} is about strong convergence, but since the limit is a regular representation this implies sq-convergence by Corollary~\ref{cor:perfect}.  In~\cite{ColMal14}, the main result is deduced from a celebrated strong-convergence result of Haagerup and Thorbj\o rnsen for independent GUE random matrices~\cite{HaaTho05}.  See~\cite{CGVTvH,CGVvH} for other proofs of Theorem~\ref{thm:ColMal}, and the surveys~\cite{Magee--survey,vanHan--strong-survey} for broader discussion of known exmaples of strong convergence.  In the group theoretic examples, the limits are usually left regular representations of finitely generated groups, we can actually conclude strong-quotient convergence because of Corollary~\ref{cor:perfect}.

\section{Conditioning on orbit patches}\label{sec:random-orbits}

This section develops some auxiliary calculations about conditioning on grounded subsets of orbits in a random representation.  At this point the underlying principles are quite general, so we formulate them for random homomorphisms from $\G$ to an arbitrary compact metrizable group $G$ with a homogeneous space $X$.  In Section~\ref{sec:random-Verb} we return to the special case of uniformly random AP sequences.


Let $G$ be a compact metric group and let $m$ be its Haar probability measure.  More generally, if $gK$ is any coset of a closed subgroup $K$ of $G$, let $m_{gK} = \delta_g \ast m_K$ be the translate of $m_K$ to that coset.

The results of this section are phrased as identifying disintegrations of certain measures over certain maps.  Disintegrations are also called `regular conditional probabilities', and are classical constructs in measure theoretic probability: see~\cite[Secs. V.6--8]{ParPMMS}, for example.

Suppose now that $L\subset K$ is a nested pair of closed subgroups of $G$.  Fix a coset $gK$, and write $gK/L$ for the set of cosets of $L$ that it contains: that is,
\[gK/L := \{gkL \in G/L:\ k \in K\}.\]

\begin{lem}\label{lem:cosets-disint}
The family of measures
\[m_{gkL} \qquad (gkL \in gK/L)\]
is a disintegration of $m_{gK}$ over the map
\[f:gK \to gK/L:gh \mapsto ghL.\]
\end{lem}

\begin{proof}
Following the definition in~\cite[Sec. V.8]{ParPMMS}, we must check two criteria.  Firstly, for any $gkL \in gK/L$, we have
\[f^{-1}\{gkL\} = \{gk':\ gk'L = gkL\} = gkL,\]
and $m_{gkL}$ is supported on this fibre by definition.  Secondly, we must check that
\[\int m_{gkL}\,dm_K(k) = m_{gK}.\]
Translating from the left by $g^{-1}$, we may reduce this to the case $g = e$, so it remains to check that
\[\int m_{kL}\,dm_K(k) = m_K.\]
This is a standard relation between a compact group and a closed subgroup: see, for instance,~\cite[Thm. 2.51]{FolAHA}, which is formulated for all unimodular locally compact groups.
\end{proof}

Now suppose that $G$ acts on a homogeneous space $X$ from the left.  Given a point $x$ of $X$, let $G(x)$ denote its stabilizer.  More generally, given a tuple $\bf{x}$ in $X^I$ for some index set $I$, let $G(\bf{x})$ be the mutual stabilizer $\bigcap_iG(x_i)$.  In addition, given another tuple $\bf{y}$ in $X^I$, let $H(\bf{x};\bf{y})$ be the set of elements $g$ of $G$ such that
\[gx_i = y_i \qquad \forall i\in I.\]
If $\bf{y} = \bf{x}$ then this set is simply $G(\bf{x})$.  In other cases it may be empty.  If it is not empty, then it is equal to the left coset $hG(\bf{x})$ and to the right coset $G(\bf{y})h$ for any single element $h$ in $H(\bf{x};\bf{y})$.  When it is nonempty, we also write $m(\,\cdot\,|\,\bf{x};\bf{y})$ for $m_{H(\bf{x};\bf{y})}$.  Observe that
\begin{equation}\label{eq:measure-id}
m(\,\cdot\,|\,\bf{x};h\bf{y}) = \delta_h \ast m(\,\cdot\,|\,\bf{x};\bf{y}).
\end{equation}

We may apply Lemma~\ref{lem:cosets-disint} to the cosets $H(\bf{x};\bf{y})$ as follows.  Assume that $H(\bf{x};\bf{y})$ is nonempty, let $z \in X$, and let $w = gz$ for some $g \in H(\bf{x};\bf{y})$. Then $H(\bf{x},z;\bf{y},w)$ is still nonempty, because it contains $g$.  For each such $g$, the set $H(\bf{x},z;\bf{y},w)$ is equal to the set $gG(\bf{x},z)$, and it is contained in $H(\bf{x};\bf{y}) = gG(\bf{x})$.  Therefore Lemma~\ref{lem:cosets-disint} applies to tell us that the family
\[m(\,\cdot\mid \bf{x},z;\bf{y},w) \qquad (w \in H(\bf{x};\bf{y})z)\]
is a disintegration of $m(\,\cdot\mid \bf{x};\bf{y})$ over the orbit-map $g\mapsto gz$.

We need a slightly more general version of this result that allows larger Cartesian products, but it is proved in the same way.

%

\begin{cor}\label{cor:cond-ingredient}
Let $I_1$, \dots, $I_r$ be finite index sets, and for each $i$ let $\bf{x}_i$ and $\bf{y}_i$ be two $I_i$-tuples such that $H(\bf{x}_i;\bf{y}_i)$ is nonempty.  Let $j \in \{1,2,\dots,r\}$ and let $z \in X$.  Then the kernel
\[\prod_{i<j}m(\,\cdot\mid \bf{x}_i;\bf{y}_i)\times m(\,\cdot\mid \bf{x}_j,z;\bf{y}_j,w)\times \prod_{i > j}m(\,\cdot\mid \bf{x}_i;\bf{y}_i) \qquad (w \in H(\bf{x}_j;\bf{y}_j)z)\]
is a disintegration of the measure
\[\prod_{i=1}^r m(\,\cdot \mid \bf{x}_i;\bf{y}_i)\]
over the map
\[(g_1,\dots,g_r)\mapsto g_jz.\]
\end{cor}

\begin{proof}
The case $r=1$ is outlined above. For the general case, in the group $G^r$, apply Lemma~\ref{lem:cosets-disint} to the large coset
\[\prod_{i=1}^r H(\bf{x}_i;\bf{y}_i)\]
and the smaller cosets
\[\prod_{i<j}H(\bf{x}_i;\bf{y}_i)\times H(\bf{x}_j,z;\bf{y}_j,w)\times \prod_{i > j}H(\bf{x}_i;\bf{y}_i) \qquad (w \in H(\bf{x}_j;\bf{y}_j)z).\]
\end{proof}

With the above preparations complete, we can return to random actions of a free group. Suppose now that $\G$ is freely generated by a finite set $S$ of size $r$, and that a random homomorphism $\pi:\G \to G$ is obtained by choosing $g_s$ for $s \in S$ independently at random from $m$ and then setting $\pi(s) := g_s$.

Where we apply Corollary~\ref{cor:cond-ingredient} below, the tuples $\bf{x}_i$ and $\bf{y}_i$ are indexed by certain subsets of $\G$. The following terminology helps keep track of them.

\begin{dfn}\label{dfn:possible}
	If $\pi \in \rm{Hom}(\G,G)$ and $x \in X$, then the \textbf{orbit map at $x$} is
	\[\pi^{(x)}(g) := \pi(g)x \qquad (g \in \G).\]
	If $F \subset \G$, then the resulting \textbf{orbit patch} is the restriction $\pi^{(x)}|F$, which we regard as a tuple in $X^F$.

If we fix $F$ and $x$, then the set of all orbit patches for different possible choices of $\pi$ is a subset $Y_{F,x}$ of $X^F$.  It is closed as a continuous image of the compact space $\rm{Hom}(\G,G)$.  We refer to the tuples in $Y_{F,x}$ as \textbf{$F$-possible starting from $x$}.
\end{dfn}

Now consider a grounded set $F$ and a tuple $\bf{y} \in Y_{F,x}$.  For each $t \in S$, the sets
\[t^{-1}F\cap F \quad \hbox{and} \quad F\cap tF\]
are naturally identified by left-translation by $t$.  With this identification, we can consider the cosets $H(\bf{y}_{t^{-1}F\cap F};\bf{y}_{F\cap tF})$ for each $t$.  The fact that $\bf{y}$ is $F$-possible implies that these cosets are all nonempty: if $\bf{y} = \pi^{(x)}|F$ then $\pi(t)$ must lie in $H(\bf{y}_{t^{-1}F\cap F};\bf{y}_{F\cap tF})$ for every $t$.

Here is the main result of this section. It describes the conditional distribution of the generators of a random action given an orbit patch.

\begin{prop}\label{prop:condition-action}
Fix $x \in X$ and let $F$ be a grounded subset of $\G$.  Let $\pi$ be a random action in which the generators $\pi(s)$, $s \in S$, are independent and distributed according to $m$. Then a regular conditional distribution for $(\pi(s):\ s \in S)$ given the orbit patch $\pi^{(x)}|F$ is given by
\begin{equation}\label{eq:prod-meas}
m_F(\,\cdot \mid \bf{y}) := \prod_{s \in S} m(\,\cdot\,|\,\bf{y}_{F\cap s^{-1}F};\bf{y}_{s F\cap F}) \qquad (\bf{y} \in Y_{F,x}).
\end{equation}
\end{prop}

\begin{proof}
We prove this by induction on the grounded set $F$.

When $F = \{e\}$, the set $Y_{F,x}$ is the singleton $\{x\}$.  The left-hand side of~\eqref{eq:prod-meas} is simply equal to $m^{\times S}$, and so is the right-hand side because $F\cap sF$ and $s^{-1}F\cap F$ are both empty for every $s$.

So now suppose that the result is known for some grounded set $F$, and let $F' = F\cup g$ be an enlargement of it in direction $s$.  Then, for any action $\pi$, we have $\pi(g)x = \pi(s)\pi(s^{-1}g)x$, and so $\pi^{(x)}|F'$ can be identified with $(\pi^{(x)}|F,\,\pi(s)\pi(s^{-1}g)x)$.  Accordingly, $Y_{F',x}$ can be identified with a subset of $Y_{F,x}\times X$.  For $\bf{y} \in Y_{F,x}$, let us write
\[Y^\bf{y}_{F',x} := \{w \in X:\ (\bf{y},w) \in Y_{F',x}\}.\]

To extend the result from $F$ to $F\cup g$, by the tower property of conditional expectation~\cite[Thm. 4.5(vi)]{Var--prob}, it suffices to prove the following:
\begin{quote}
If $\pi$ is a random action in which the generators are distributed according to $m_F(\,\cdot\mid \bf{y})$, then a regular conditional distribution for ${(\pi(s):\ s \in S)}$ given the image $\pi(s)y_{s^{-1}g}$ is given by
\begin{equation}\label{eq:prod-meas-2}
m_{F'}(\,\cdot\mid \bf{y},w) \qquad (w \in Y^\bf{y}_{F',x}).
\end{equation}
\end{quote}

To do this, we first identify the measures appearing in~\eqref{eq:prod-meas-2} more carefully.  Writing an element of $Y_{F',x}$ as $\bf{y}' := (\bf{y},w)$, consider the factor measures
\[m(\,\cdot\mid \bf{y}'_{F'\cap t^{-1}F'};\bf{y}'_{tF'\cap F'}) \qquad (t \in S).\]
If $t=s^{\pm 1}$ (according as $s$ lies in $S$ or $S^{-1}$), then either the first or the second equality of Lemma~\ref{lem:shift-enlargement} gives
\[m(\,\cdot\mid \bf{y}'_{F'\cap t^{-1}F'};\bf{y}'_{tF'\cap F'}) = m(\,\cdot\mid \bf{y}_{F\cap s^{-1}F},y_{s^{-1}g};\bf{y}_{sF\cap F},w).\]
On the other hand, if $t \ne s^{\pm 1}$, then the third equality of Lemma~\ref{lem:shift-enlargement} gives
\[m(\,\cdot\mid \bf{y}'_{F'\cap t^{-1}F'};\bf{y}'_{tF'\cap F'}) = m(\,\cdot\mid \bf{y}_{F\cap t^{-1}F};\bf{y}_{tF\cap F}).\]
(This is the key point where we need $F$ to be grounded.)

Combining the identifications above, we see that the asserted conditional distribution is precisely the one given by Corollary~\ref{cor:cond-ingredient} when we let $r = |S|$, index all the data by $S$ rather than by $\{1,2,\dots,r\}$, and identify
\[I_t,\ \bf{x}_t, \hbox{and}\ \bf{y}_t \quad \hbox{with} \quad t^{-1}F\cap F,\ \bf{y}_{t^{-1}F\cap F},\ \hbox{and}\ \bf{y}_{F\cap tF}\]
for each $t \in S$.  This continues the induction.
\end{proof}

\begin{rmk}
In the proof above, the assumption that $F$ is grounded enables us to apply Lemma~\ref{lem:shift-enlargement}.  Proposition~\ref{prop:condition-action} is not true if we omit that assumption.  To see this more clearly, consider a simplified version of Proposition~\ref{prop:condition-action} that partitions the space of actions according to orbit patches but ignores conditional distributions.  That version tells us that, if $\bf{y}$ is $F$-possible starting from $x$, then
\[\big\{(\pi(s):\ s \in S) \in G^S: \pi(g)x = y_g\ \forall g\in F\big\} = \prod_{s \in S}H(\bf{y}_{s^{-1}F\cap F};\bf{y}_{F\cap sF}).\]
The left-hand side here is contained in the right-hand side for any subset $F$ of $\G$, but the reverse inclusion can fail if $F$ is not grounded.  For example, if $F$ is actually well-separated in the sense that no two of its element are adjacent in the left Cayley graph of $\G$, then the left-hand set above may still impose many restriction on the generators $(\pi(s):\ s \in S)$, but $F\cap sF = \emptyset$ for every $s$, and so the right-hand set above is the whole of $G^S$. \fin
\end{rmk}

\begin{rmk}
Our use of grounded sets and other calculations in this section have several precedents in the literature.  For example, special cases or similar arguments appear when one studies the orbits of single points in random actions of $\G$ by permutations of a finite set.  Examples can be found in~\cite{Nic94,LinPud10,LarSha12,ParPud15}. \fin
\end{rmk}

\section{Distribution of Verblunsky coefficients}\label{sec:random-Verb}

Now consider again a random $n$-dimensional unitary representation $\pi$ for some large $n$, and fix an orthonormal $k$-tuple $V$ in $\bbC^{\oplus n}$.  We can turn the results of the previous section into a description of the distribution of $(\Phi^\pi_V)[F]$ for each grounded set $F$. In the next chapter this enables our proof of Theorem~\ref{mainthm:LDP}.

Let $Q^\pi[F] := (\Phi^\pi_V)[F]$ to lighten notation.  The description is recursive from one grounded set $F$ to an enlargement $F\cup g$ in direction $t$.  When $Q^\pi[F]$ is nonsingular, let
\[C^\pi_{F,g} \in \Xi(k,k|F\setminus tF|)\]
be the Verblunsky coefficient of $Q^\pi[F\cup g]$ over $Q^\pi[F]$ (Definition~\ref{dfn:Verb}). This is a random element of that space of contractions once we condition on the event that it is well-defined.  For notation, recall the distributions $\s_{n,\ell,k}$ on the space of contractions $\Xi(k,\ell)$ defined in~\eqref{eq:projected-contraction-dist}. 

\begin{thm}\label{thm:dil-dist}
If $n\ge k|F|$, then the following hold:
\begin{itemize}
\item[a.] $Q^\pi[F]$ is nonsingular almost surely, so $C^\pi_{F,g}$ is defined almost surely;
\item[b.] $C^\pi_{F,g}$ is independent from $Q^\pi[F]$;
\item[c.] $C^\pi_{F,g}$ has distribution $\s_{n-k|F\cap tF|,k|F\setminus tF|,k}$.
\end{itemize}
\end{thm}

Notice that we need part (a) of this theorem for parts (b) and (c) to make sense.

\begin{proof}
We prove this by specializing the results of the previous section to the case when $G = \bf{U}(n)$ and $X = \bf{U}(k,n)$.

For a grounded set $F$ and an enlargement $F\cup g$ in direction $t$, we prove the following implications among the parts of Theorem~\ref{thm:dil-dist}:
\begin{multline*}
	\hbox{part (a) for $F$} \ \ \Rightarrow \ \ \hbox{parts (b) and (c) for $(F,g)$} \ \ \Rightarrow \ \ \hbox{part (a) for $F\cup g$}.
	\end{multline*}
When $F = \{e\}$, we have $Q^\pi[F] = I_k$, so part (a) is immediate in this case.  Starting from there, the above implications imply the result in full by induction on $F$.

So now assume that $n\ge k|F|$, and suppose we already know that $Q^\pi[F]$ is nonsingular almost surely.  Therefore the Verblunsky coefficient $C := C^\pi_{F,g}$ is defined almost surely.  It takes values in $\Xi := \Xi(k,k|F\setminus tF|)$. Let us also assume that $t \in S$; the case $t \in S^{-1}$ is analogous. 

Let $(\Omega,\F,\bbP)$ be the background probability space, and let $\cal{G}$ be the sigma-subalgebra of $\F$ generated by the random orbit patch $\pi^{(V)}|F$.  Since $Q^\pi[F]$ is a function of this orbit patch, the law of total probability gives
\[\bbP\{Q^\pi[F] \in A,\ C \in B\} = \bbE\big[\bbP(C\in B\mid \cal{G});\ Q^\pi[F] \in A\big]\]
for any Borel subsets $A\subset \S^\rm{u}_k(F)$ and $B \subset \Xi$.
Therefore parts (b) and (c) for $F$, $g$ follow if we show that the constant $\s_{n-k|F\cap tF|,\,k|F\setminus tF|,\,k}(B)$ is a version of $\bbP(C\in B\mid \cal{G})$.

Now, $C$ is a function of $Q^\pi[F\cup g]$, and this in turn is a function of $(\pi^{(V)}|F,\,\pi(t))$.  Therefore Proposition~\ref{prop:condition-action} asserts that a version of $\bbP(C\in B\mid \cal{G})$ is given by
\begin{equation}\label{eq:C-indep}
m(\{\pi(t):\ C \in B\}\mid \bf{y}_{F\cap t^{-1}F};\bf{y}_{tF\cap F}) \qquad \hbox{when}\ \bf{y} := \pi^{(V)}|F \in Y_{F,V}.
\end{equation}
To use this expression, fix $\bf{y} \in Y_{F,V}$ and also some $U_0 \in H(\bf{y}_{F\cap t^{-1}F};\bf{y}_{tF\cap F})$, and let $U$ be drawn at random from the Haar probability measure on $\bf{U}((\bf{y}_{F\cap tF})^\perp)$.  Under the conditional measure in~\eqref{eq:C-indep}, $\pi(t)$ has the same distribution as $UU_0$, by the definitions in Section~\ref{sec:random-orbits}, and then $Q^\pi[F\cup g]$ has the same distribution as the Gram matrix of the combined tuple $[\bf{y}_{F\setminus tF},\ \bf{y}_{F\cap tF},\ UU_0y_{t^{-1}g}]$, where only the last entry is random.

If it happens that $Q^\pi[F] \in \S^\circ_k(F)$, then the combined tuple $[\bf{y}_{F\setminus tF},\ \bf{y}_{F\cap tF}]$ is linearly independent, and so is the combined tuple $[\bf{y}_{F\cap tF},\ UU_0y_{t^{-1}g}]$ because it is contained in the translate $U\bf{y}$.  We can therefore apply Proposition~\ref{prop:law-of-contraction}, which tells us that the left-hand side of~\eqref{eq:C-indep} is equal to $\s_{n-k|F\cap tF|,k|F\setminus tF|,k}(B)$, as required.  Since $Q^\pi[F] \in \S^\circ_k(F)$ lies in $Y^\circ_{F,V}$ almost surely by part (a) for $F$, this completes the proof of parts (b) and (c) for $(F,g)$.
	
Finally, by Lemma~\ref{lem:a-s-strict}, $\s_{n-k|F\cap tF|,\,k|F\setminus tF|,\,k}$ is supported on the subset ${\Xi^\circ(k,k|F\setminus tF|)}$ of proper contractions provided we have
	\[n-k|F\cap tF| \ge k + k|F\setminus tF|, \quad \hbox{or equivalently} \quad n\ge k(|F| + 1).\]
	In this case, the dilation $Q^\pi[F\cup g]$ of $Q^\pi[F]$ is still almost surely nonsingular, by Proposition~\ref{prop:three-block-completion}.  So if $n\ge k(|F|+1)$ then we can now also conclude part (a) for $F\cup g$, so the induction continues.
\end{proof}

The independence given by part Theorem~\ref{thm:dil-dist}(b) is a great advantage of working with Verblunsky coefficients during the later proof of the rest of Theorem~\ref{mainthm:LDP}.  It has the following corollary.  Fix any grounded enumeration $e = g_0$, $g_1$, $g_2$, \dots of $\G$, and let $F_i := \{g_0,\dots,g_i\}$ for each $i\ge 0$.  Suppose that $F_{i+1}$ is an enlargement of $F_i$ in direction $s_{i+1}$ for each $i$.

\begin{cor}\label{cor:KilNenfree}
Let $\pi$ be a uniform random $n$-dimensional representation of $\G$, and let $m := \lfloor n/k\rfloor$.  The resulting Gram matrix $Q^\pi[F_i]$ is nonsingular almost surely for all $i =0,1,2\dots,m$, and the corresponding random Verblunsky coefficients
\begin{equation}\label{eq:first-m-Verblunskys}
C^\pi_{F_0,g_1},\ C^\pi_{F_1,g_2},\ \dots,\ C^\pi_{F_m,g_{m+1}}
\end{equation}
are defined almost surely, are independent, and have respective distributions
\[\s_{n - k|F_i\cap s_{i+1}F_i|,k|F_i\setminus s_{i+1}F_i|,k} \qquad \hbox{for}\ i=0,1,2,\dots,m.\]
\qed
\end{cor}

If $k=r=1$, then~\eqref{eq:first-m-Verblunskys} is the sequence of the first $m$ Verblunsky coefficients of the random positive definite function on $\bbZ$ given by the orbit of $e_1$ under a uniform random $n$-by-$n$ unitary matrix.  In this case the independence of the first $n$ of these random Verblunksy coefficients, together with their individual distributions, is a noted result of Killip and Nenciu~\cite[Prop. 3.3]{KilNen04}.  For $r= 1$ but $k \ge 2$, their result is generalized in~\cite[Thm. 3.2]{GamRou14}. Corollary~\ref{cor:KilNenfree} is a further generalization of those results to random representations of higher-rank free groups.

One could also re-write Theorem~\ref{mainthm:LDP} directly in terms of the sequences of Verblunsky coefficients that appear in Corollary~\ref{cor:KilNenfree}.  The rate function has a simple sum form in terms of those coefficients: see Corollary~\ref{cor:first-expansion} below.

\subsection*{\emph{Notes and further references}}

A single random unitary matrix with distribution $m_{\bf{U}(n)}$ is one of the most well-studied models in random matrix theory, where this distribution is traditionally called the `circular unitary ensemble': see, for instance,~\cite[Sec. 10.3]{MehRM}.  Its analysis is driven by Weyl's integration formula for central functions on unitary groups~\cite[Prop. 4.1.6]{AndGuiZei--book}.  Certain Toeplitz determinants appear inside this formula, leading to another connection with orthogonal polynomials on the unit circle: see~\cite[pp67--9]{SimOPUCI} or~\cite[Prop. 4.1.6 or Rmk. 4.1.7]{AndGuiZei--book}.

However, I do not see an analog of this connection in the setting of non-Abelian free groups.  On the whole, the interaction of several independent random unitary matrices is not amenable to such exact calculations.  The alternative point of view provided by Voiculescu's free probability theory has been key to uncovering many phenomena concerning such random tuples.  Many aspects of this work can be found in textbooks such as~\cite{AndGuiZei--book} or~\cite{HiaPetSL}.

The result of Killip and Nenciu asserts that the Verblunsky coefficients of a single uniformly random unitary matrix are independent. This is an important step towards finding random-matrix models of various natural ensembles of $n$-point subsets of the unit circle. These development and further references are discussed on~\cite[p321]{AndGuiZei--book}.

\chapter{Annealed and zeroth-order AP entropy}\label{chap:LDP-proof}

In this chapter we prove the large deviations principle in Theorem~\ref{mainthm:LDP}. Then we use it to prove Theorem~\ref{mainthm:annealed-formula} and also some other formulas for annealed AP entropy along a uniformly random AP sequence.

\section{The chosen-tuple large deviations principle}\label{sec:completed-LDP}

We now adopt the notation of Theorem~\ref{mainthm:LDP}, and also the notation for log-determinant entropy from Chapter~\ref{chap:PSD}.  Using this notation, let us define
\[h_F(q) := \left\{\begin{array}{ll} \rmH_q(F) - \sum_{s \in S}\rmH_q(F\cap sF) &\quad \hbox{if $q$ is nonsingular} \\ -\infty & \quad \hbox{if $q$ is singular}\end{array}\right.\]
for any grounded set $F$ and $q \in \S^\rm{u}_k(F)$.  Later we manipulate this quantity using the chain rules from Section~\ref{sec:block-Gram}.  Also, if $\phi:\G\to \rmM_k$ is positive definite and $F$ is a finite subset of $\G$, then we adapt the notation from Chapter~\ref{chap:PSD} by writing
\[\rmH_\phi(F) := \log \det \phi[F],\]
and similarly for conditional log-determinant entropy and mutual information.

As in Section~\ref{sec:random-Verb}, let $Q^{\pi_n}[F] := (\Phi^{\pi_n}_V)[F]$ for any $n$ and any fixed grounded subset $F$ of $\G$. We deduce Theorem~\ref{mainthm:LDP} from a corresponding large deviations principle for these finite matrices.

\begin{thm}\label{thm:pre-LDP}
	If $F$ is a grounded subset of $\G$, then $Q^{\pi_n}[F]$ obeys the large deviations principle in the space $\S^{\rm{u}}_k(F)$ with rate function $-h_F$.
\end{thm}

The proof of Theorem~\ref{thm:pre-LDP} is by induction on the grounded set $F$.  The main ingredient is Theorem~\ref{thm:dil-dist}. For an enlargement $F\cup g$ of $F$, the basic idea for the inductive step is straightforward:
\begin{itemize}
\item the inductive hypothesis gives us the LDP for $Q^{\pi_n}[F]$;
\item Theorem~\ref{thm:dil-dist} and Corollary~\ref{cor:matrix-LDP1} tell us that the next Verblunsky coefficient $C$ is independent from $Q^{\pi_n}[F]$ and satisfies its own LDP;
\item Lemma~\ref{lem:LDP-product} tells us how to combine these for the pair $(Q^{\pi_n}[F],C)$;
\item finally, since $Q^{\pi_n}[F\cup g]$ is parametrized uniquely by $Q^{\pi_n}[F]$ and $C$, applying the contraction principle should continue the induction.
\end{itemize}

There is a slight complication in the last of these steps.  This is because ${Q^{\pi_n}[F\cup g]}$ is not a continuous image of the pair $(Q^{\pi_n}[F],C)$ until we exclude the event that $Q^{\pi_n}[F]$ is singular.  For this reason, the proof also involves passing between the desired target spaces $\S^\rm{u}_k(F)$ and $\S^\rm{u}_k(F\cup g)$ and their open subsets $\S_k^\circ(F)$ and $\S_k^\circ(F\cup g)$.  This is where we need Lemma~\ref{lem:LDP-open-subset}.

\begin{proof}[Proof of Theorem~\ref{thm:pre-LDP}]
\emph{Step 1.}\quad In view of part Theorem~\ref{thm:dil-dist}(a), we may choose Borel probability measures $\mu_{F,n}$ on the open subsets $\S_k^\circ(F)$ of $\S^\rm{u}_k(F)$ such that $\mu_{F,n}$ agrees with the distribution of $Q^{\pi_n}[F]$ for all sufficiently large $n$.  In addition, we have $-h_F(q) = \infty$ for any grounded set $F$ and $q \in \S^\rm{u}_k(F)\setminus \S_k^\circ(F)$.  Therefore, by Lemma~\ref{lem:LDP-open-subset}, it suffices to prove that $(\mu_{F,n})_{n\ge 1}$ obeys the LDP on $\S_k^\circ(F)$ with rate function $-h_F|\S_k^\circ(F)$.

\vspace{7pt}

\emph{Step 2.}\quad The rest of the proof is an induction on $F$.  First, since $\phi$ is unital, we have
\[h_{\{e\}}(\phi[\{e\}]) = \rmH(I_k) - 0 = 0,\]
and the required large deviations principle is vacuous because $\S_k^{\rm{u}}(\{e\}) = \{I_k\}$.

So now suppose the result is known for a grounded set $F$ and consider an enlargement $F\cup g$ in direction $s$.  Combining this inductive hypothesis with Corollary~\ref{cor:matrix-LDP1} and applying Lemma~\ref{lem:LDP-product}, the product measures
\[\mu_{F,n}\times \s_{n - k|F\cap sF|,k|F\setminus sF|,k} \qquad (n\ge 1)\]
obey the LDP on the space
\begin{equation}\label{eq:open-prod}
\S_k^\circ(F)\times \Xi^\circ(k,k|F\setminus sF|)
\end{equation}
with the rate function
\[-h_F(q) -\log\det(I_k - C^\ast C) \qquad ((q,C) \in \S_k^\circ(F)\times \Xi^\circ(k,k|F\setminus sF|)).\]

\vspace{7pt}

\emph{Step 3.}\quad Under the inverse of the homeomorphism in Lemma~\ref{lem:pdf-completion}, the product set~\eqref{eq:open-prod} is identified with $\S_k^\circ(F\cup g)$.  Moreover, Theorem~\ref{thm:dil-dist} shows that $\mu_{F\cup g,n}$ is the image of $\mu_{F,n}\times \s_{n - k|F\cap sF|,k|F\setminus sF|,k}$ under this inverse homeomorphism.  Therefore, by the identity~\eqref{eq:cond-mut-inf-contraction}, the measures $\mu_{F\cup g,n}$ obey the LDP on ${\S_k^\circ(F\cup g)}$ with rate function
\[-h_F(q[F]) + \rmI_q(g\,;\,F\setminus sF\mid F\cap sF) \qquad (q \in \S_k^\circ(F\cup g)).\]

To finish the proof we show that the expression above equals $-h_{F\cup g}(q)$.  Since $q$ and hence $q[F]$ are nonsingular, we have
\begin{multline}\label{eq:change}
h_{F\cup g}(q) - h_F(q[F]) \\ = \rmH_q(F\cup g) - \rmH_q(F) - \sum_{t\in S}\big(\rmH_q((F\cup g)\cap t(F\cup g)) - \rmH_q(F\cap tF)\big),
\end{multline}
where all terms are finite. Lemma~\ref{lem:shift-enlargement} tells us the difference between
\[(F\cup g)\cap t(F\cup g) \qquad \hbox{and} \qquad F\cap tF\]
for each $t \in S$.  Using this and entropy identities, we can express the right-hand side of~\eqref{eq:change} in terms of conditional log-determinant entropies.  First, by the third equality from Lemma~\ref{lem:shift-enlargement}, all terms in the sum with $t \ne s^{\pm 1}$ cancel.  Next, if $s \in S$ and $t = s$, then by the first equality from Lemma~\ref{lem:shift-enlargement} and the conditional chain rule~\eqref{eq:cond-chain2} the remaining differences become
\[\rmH_q(g\mid F) - \rmH_q(g\mid F\cap tF) = - \rmI_q(g\,;\,F\setminus sF\mid F\cap sF),\]
as required in this case.  Finally, if $s \in S^{-1}$ and $t = s^{-1}$, then by the second equality from Lemma~\ref{lem:shift-enlargement} the remaining differences become
\begin{align*}
\rmH_q(g\mid F) - \rmH_q(tg\mid F\cap tF) &= \rmH_q(g\mid F) - \rmH_q(g\mid sF\cap F)\\
&= - \rmI_q(g\,;\,F\setminus sF\mid F\cap sF),
\end{align*}
where the first equality holds by the invariance of $q$ under translation by $s = t^{-1}$, and the second follows by~\eqref{eq:cond-chain2} as before.  Once again this verifies the desired equality.  This continues the induction and so completes the proof.
\end{proof}

\begin{rmk}
	It is worth comparing our appeal to Lemma~\ref{lem:LDP-open-subset} with the exposition of the case of positive definite functions on $\bbZ$ in~\cite{BreSimZei18}, which does not involve a version of that lemma.  Instead, those authors introduce a slightly larger space of finite or infinite sequences of Verblunsky coefficients that can be used to parametrize positive definite functions without assuming nonsingularity, then apply the Killip--Nenciu result to prove a large deviations principle there, and finish with a simple appeal to the contraction principle.
	
	We could take a similar approach in our setting.  However, over a non-cyclic free group, the resulting space of Verblunsky coefficients allowing singularity would be much more complicated than over $\bbZ$, and it seems easier to proceed via Lemma~\ref{lem:LDP-open-subset} as above. \fin
\end{rmk}

\begin{proof}[Proof of Theorem~\ref{mainthm:LDP}]
The whole of $\Phi^{\pi_n}_V$ is a random element of the infinite-dimensional compact convex set $\S^\rm{u}_k(\G)$, which we can regard as the inverse limit of its finite-dimensional projections $\S^\rm{u}_k(F)$ as $F$ ranges over grounded subsets of $\G$.  Since grounded sets form an upwards-directed cover of the whole of $\G$, the first formula in~\eqref{eq:LDP-formula-inf} follows from Theorem~\ref{thm:pre-LDP} by an application of Lemma~\ref{lem:LDP-inf-product}(b).  Since the balls $B_n$ are all grounded and also form an upwards-directed cover of $\G$, the same reasoning gives the analogous formula when $F$ is restricted to the set of balls.  By Lemma~\ref{lem:LDP-inf-product}(a), the quantity appearing inside that infimum is non-decreasing in $n$, so the infimum is equal to the limit as $n\to\infty$.
\end{proof}

Because of Theorem~\ref{mainthm:LDP} and Varadhan's asymptotic from~\eqref{eq:Var}, we obtain a probabilistic interpretation of the annealed AP pressure along the uniformly random AP sequence $\bspi$ in terms of another family of `spherical integrals'. If $a$ is a self-adjoint element of $C^\ast \G$, then the result is
\begin{equation}\label{eq:spherical-2}
\frac{1}{n}\log \bbE[e^{n\langle \pi_n(a)v_n,v_n\rangle}]\to \rmp^\ann(a),
\end{equation}
where $v_n$ is any sequence of unit vectors in $\bf{S}^{2n-1}$ (all choices give the same left-hand side by the same symmetry used in Lemma~\ref{lem:LDP-to-A}), and where $\rmp^\ann$ is the annealed AP pressure along our uniformly random AP sequence (see Definition~\ref{dfn:annAP-pressure}).  The left-hand side of~\eqref{eq:spherical-2} is another example of a `spherical integral', related to~\eqref{eq:spherical-1} but with an expectation over the random representation $\pi_n$ rather than an explicit integral over the sphere.

\section{Formulas for annealed AP entropy}\label{sec:ann-AP-forms}

Fix $\bspi = (\pi_n)_{n\ge 1}$ to be a uniformly random AP sequence for our rank-$r$ free group $\G$, and turn to the functionals
\[\hann := \rmh^\ann_{\bspi},\quad \ul{\rmh}^\ann := \ul{\rmh}^\ann_{\bspi}, \quad \rmh^0 := \rmh^0_{\bspi}, \quad \ul{\rmh}^0 := \ul{\rmh}^0_{\bspi}, \quad \hbox{and} \quad \rmp^\ann := \rmp_{\bspi}^\ann.\]
In addition, throughout this section, we fix a positive definite function $\phi:\G \to \rmM_k$ (it need not be unital now), and abbreviate
\[H(F) := \rmH_\phi(F) \qquad \hbox{and} \qquad I(F\,;\,G) := \rmI_\phi(F\,;\,G)\]
for any finite disjoint subsets $F$ and $G$ of $\G$, and similarly for their conditional versions.  For any grounded set $F$, we again write $h_F(\phi)$ for the quantity introduced in the paragraph before Theorem~\ref{mainthm:LDP}.

We now turn the results of the preceding section into formulas for $\hann$ and $\ul{\rmh}^\ann$, proving in particular that these are always equal.  These formulas include those in Theorem~\ref{mainthm:annealed-formula}.

The connection from Theorem~\ref{mainthm:LDP} to Theorem~\ref{mainthm:annealed-formula} goes through the following.

\begin{lem}\label{lem:LDP-to-A}
If $\phi$ is unital and $\cal{O}$ is any base of neighbourhoods of $\phi$ in $\S^{\rm{u}}_k(\G)$, then
\[\hann(\phi) = \inf_O\limsup_{n\to\infty}\frac{1}{n}\log \bbP(\Phi^{\pi_n}_V \in O).\]
The analogous formula holds for $\ul{\rmh}^\ann$ if `$\limsup$' is replaced by `$\liminf$'.
\end{lem}

\begin{proof}
Fix a dimension $n$ and an orthonormal tuple $V = [e_1,\dots,e_k]$.  Starting from the definition of $m_{\bf{U}(k,n)}$ as a pushforward of $m_{\bf{U}(n)}$, we have
\begin{align*}
\bbE m_{\bf{U}(k,n)}\X(\pi_n,O) &= \bbE \int_{\bf{U}(n)}1_{\X(\pi_n,O)}(UV)\ dm_{\bf{U}(n)}(U) \\
&= \int_{\bf{U}(n)}\bbP(UV \in \X(\pi_n,O))\ dm_{\bf{U}(n)}(U) \\
&= \int_{\bf{U}(n)}\bbP(V \in \X(U^\ast\pi_n U,O))\ dm_{\bf{U}(n)}(U) \\
&= \bbP(V \in \X(\pi_n,O)) \\
&= \bbP(\Phi^{\pi_n}_V \in O),
\end{align*}
where the second equality holds by Tonelli's theorem, and the fourth holds because the law of $\pi_n$ is invariant under conjugation by any fixed element of $\bf{U}(n)$.  Now both conclusions follow by inserting this calculation into Lemma~\ref{lem:normalized3}.
\end{proof}

\begin{prop}\label{prop:annealed-formula-1}
If $\phi:\G\to\rmM_k$ is positive definite, then $\hann(\phi)$ and $\ul{\rmh}^\ann(\phi)$ are both equal to the quantity in formula~\eqref{eq:LDP-formula-1} from Theorem~\ref{mainthm:annealed-formula}. In particular, $\ul{\rmh}^\ann = \hann$ and $\ul{\rmh}^0 = \rmh^0$.
\end{prop}

\begin{proof}
Assume first that $\phi$ is unital, so $\hann(\phi)$ and $\ul{\rmh}^\ann(\phi)$ are given by Lemma~\ref{lem:LDP-to-A}.  In that probability formula, the $\limsup$ and the $\liminf$ are both controlled by the large deviations principle from Theorem~\ref{mainthm:LDP}, giving
	\[\hann(\phi) = \ul{\rmh}^\ann(\phi) = \lim_{n\to\infty}h_{B_n}(\phi).\]
This shows that $\ul{\rmh}^\ann = \hann$.  We then obtain also $\ul{\rmh}^0 = \rmh^0$ by Theorem~\ref{thm:asymp-free2} and Corollary~\ref{cor:h0}.

The right-hand side above agrees with~\eqref{eq:LDP-formula-1} when $\phi$ is unital.  Now suppose that $\phi(e)$ is nonsingular but not necessarily unital, and let $\phi_1 := \phi(e)^{-1/2}\cdot \phi \cdot \phi(e)^{-1/2}$. Then Proposition~\ref{prop:hann-properties}(b) gives
	\begin{equation}\label{eq:reduce-to-normalized}
	\hann(\phi) = \log \det \phi(e) + \hann(\phi_1)
	\end{equation}
	and similarly with $\ul{\rmh}^\ann$ on both sides.  On the other hand, for any finite subset $F$ of $\G$, we have
	\[\phi[F] = \rm{diag}(\underbrace{\phi(e)^{1/2},\dots,\phi(e)^{1/2}}_{F})\cdot \phi_1[F]\cdot \rm{diag}(\underbrace{\phi(e)^{1/2},\dots,\phi(e)^{1/2}}_{F}),\]
	and hence
	\[H(F) = |F|\log \det \phi(e) + \log\det\,\phi_1[F]\]
	(where both sides may equal $-\infty$).  Applying this equality to the balls $B_n$ and their intersections $B_n\cap sB_n$, we find that
\[	h_{B_n}(\phi) = H(B_n) - \sum_{s \in S}H(B_n\cap sB_n) = \log \det \phi(e) + h_{B_n}(\phi_1),\]
where as usual either both sides are finite or both sides are $-\infty$, and where we have used the calculation
\[|B_n| - \sum_{s \in S}|B_n\cap sB_n| = \Big(1 + 2r\cdot \frac{(2r-1)^n-1}{2r-2}\Big) - r\cdot \Big(2\cdot \frac{(2r-1)^n-1}{2r-2}\Big) = 1.\]
Comparing this with~\eqref{eq:reduce-to-normalized}, we see that the unital case of identity~\eqref{eq:LDP-formula-1} implies the nonsingular case.

Finally, if $\phi(e)$ is singular, then $\hann(\phi)=-\infty$, and the explicit convention in Theorem~\ref{mainthm:annealed-formula} says that~\eqref{eq:LDP-formula-1} is also equal to $-\infty$ in this case.
	\end{proof}

We need a few more steps to reach formula~\eqref{eq:LDP-formula-2}.  Formula~\eqref{eq:LDP-formula-2} is the direct analog of Bowen's original formula for annealed sofic entropy (then called the `f-invariant')~\cite{Bowen10free}.
Towards the proof of formula~\eqref{eq:LDP-formula-2}, suppose again that $F$ is a grounded subset of $\G$, and let
\begin{equation}\label{eq:G-from-F}
G := F\cup \bigcup_{s \in S\cup S^{-1}}sF.
\end{equation}
In the left Cayley graph of $\G$, $G$ is the union of $F$ and its exterior boundary, and therefore $G$ is still grounded.  For example, if $F$ is the ball $B_n$ in the left Cayley graph, then $G = B_{n+1}$. In general, $G$ has the following key relation to $F$:
\begin{equation}\label{eq:GtF}
G\cap sG = F\cup sF \qquad \hbox{for each}\ s \in S\cup S^{-1}.
\end{equation}
Indeed, $G$ contains all of $s^{-1}F$, $F$, and $sF$ by definition, and so $sG$ also contains $F$: this proves the inclusion `$\supset$' in~\eqref{eq:GtF}.  On the other hand, if $g \in (G\cap sG)\setminus F$, then $g$ belongs to the exterior boundary of $F$.  In this case, $s^{-1}g$ cannot still lie in $G$ unless it actually lies in $F$, so that $g \in sF$.  This proves the inclusion `$\subset$' in~\eqref{eq:GtF}.

Because of~\eqref{eq:GtF}, the function $h_G$ has the alternative formula
\begin{equation}\label{eq:hG-alt}
h_G(\phi) = H(G) - \sum_{s \in S}H(F\cup sF)
\end{equation}
whenever it is finite. In addition, let us now define
\begin{equation}\label{eq:h-prime}
h'_F(\phi) := \sum_{s \in S}H(F\cup sF) - (2r-1)\cdot H(F),
\end{equation}
taking this to be $-\infty$ in case $\phi[F\cup sF]$ is singular for any $s$.

\begin{lem}\label{lem:F-G-ineqs}
If $F$ is any grounded set and $G$ is defined by~\eqref{eq:G-from-F}, then
\[h'_F(\phi) \ge h_G(\phi) \ge h'_G(\phi).\]

Suppose that $\phi$ is associated to the representation $\pi$ by the cyclic vector $v$, and let $P$ be the orthogonal projection from $H_\pi$ to the subspace $\rm{span}\{\pi(g)v:\ g \in F\}$.  If the values $h'_F(\phi)$ and $h_G(\phi)$ are finite, then they are equal if and only if the sets of vectors
\[\{P^\perp \pi(g)v:\ g \in sF\setminus F\} \qquad \hbox{for} \qquad s \in S\cup S^{-1}\]
are pairwise orthogonal.
\end{lem}

\begin{proof}
\emph{Step 1.}\quad Assume that $h'_G(\phi)$ is finite.  Then $\phi[G\cup sG]$ is nonsingular for every $s$, and so $h_G(\phi)$ is also finite.  Using the original definition of $h_G$, we now obtain
\begin{align*}
h_G(\phi) - h'_G(\phi) &= 2rH(G) -\sum_{S \in S}\big(H(G\cap sG) + H(G\cup sG)\big)\\
&= \sum_{s \in S}\big(2H(G) - H(G\cap sG) - H(G\cup sG)\big).
\end{align*}
By the strong subadditivity inequality~\eqref{eq:strong-subadd} and translation invariance, we have
\[H(G\cap sG) + H(G\cup sG) \le H(G) + H(sG) = 2H(G),\]
so $h_G(\phi) - h'_G(\phi)$ is a sum of non-negative terms.

\vspace{7pt}

\emph{Step 2.}\quad Now assume that $h_G(\phi)$ is finite. Then $\phi[G]$ is nonsingular, so $h_F'(\phi)$ is also finite since $G$ contains $F\cup sF$ for every $s$. This time writing $h_G$ as in~\eqref{eq:hG-alt}, we have
\begin{equation*}
h'_F(\phi) - h_G(\phi) = 2\sum_{s \in S}H(F \cup sF) - H(G) - (2r-1)\cdot H(F).
\end{equation*}
The difference $G\setminus F$ is the disjoint union of $sF\setminus F$ as $s$ ranges over $S\cup S^{-1}$. Therefore an iterated appeal to the subadditivity inequality~\eqref{eq:strong-subadd} from Corollary~\ref{cor:cond-chain} gives
\begin{equation}\label{eq:from-Sew}
H(G\setminus F\mid F) \le \sum_{s \in S\cup S^{-1}}H(sF\setminus F\mid F).
\end{equation}
Adding $2rH(F)$ and using the chain rule (Proposition~\ref{prop:chain1}), this becomes
\[H(G) + (2r-1)\cdot H(F) \le \sum_{s\in S\cup S^{-1}}H(F \cup sF) = 2\sum_{s\in S}H(F \cup sF),\]
where the second equality holds by translation-invariance.  Re-arranging, this asserts that $h'_F(\phi) - h_G(\phi)$ is non-negative.

Finally, we obtain the condition for the equality $h'_F(\phi) = h_G(\phi)$ by concatenating the conditions for equality in that iterated appeal to Corollary~\ref{cor:cond-chain}. 
\end{proof}

\begin{rmk}
When $F = B_n$ and so $G = B_{n+1}$, the inequalities in Lemma~\ref{lem:F-G-ineqs} have some overlap with those in~\cite[Sec. 3]{Seward--freestab} (in the special case in which the sigma-subalgebra `$\Sigma$' of that paper is trivial).  That paper largely inspired the proofs in this section.  To be specific, for these Cayley-graph balls,~\eqref{eq:from-Sew} is essentially the same as the first inequality in~\cite[Lem. 3.1]{Seward--freestab}, and its use to prove that $h_{B_{n+1}} \le h_{B_n}'$ is roughly the same as in the first paragraph of the proof of~\cite[Lem. 3.2]{Seward--freestab}.  However, beyond that our routes diverge.  The analog of the second inequality in~\cite[Lem. 3.2]{Seward--freestab} does not appear in the proof above.  In the notation above, the second paragraph of the proof of~\cite[Lem. 3.2]{Seward--freestab} essentially uses the second inequality from~\cite[Lem. 3.1]{Seward--freestab} to show that $h_{B_{n+1}}' \le (h_{B_{n+1}} + h_{B_n}')/2$, rather than showing $h_{B_{n+1}}' \le h_{B_{n+1}}$ as we do. \fin
\end{rmk}

\begin{prop}\label{prop:annealed-formula-2}
If $\phi:\G\to\rmM_k$ is positive definite, then $\hann(\phi)$ equals the quantity in formula~\eqref{eq:LDP-formula-2} from Theorem~\ref{mainthm:annealed-formula}.
\end{prop}

\begin{proof}
Applying Lemma~\ref{lem:F-G-ineqs} with $F = B_n$ and $G = B_{n+1}$, it shows that the sequences $h_{B_n}(\phi)$ and $h'_{B_n}(\phi)$ must have the same limit as $n\to\infty$.  The limit of the former sequence equals $\hann(\phi)$ by Proposition~\ref{prop:annealed-formula-1}, and the limit of the latter sequence is precisely~\eqref{eq:LDP-formula-2}.
\end{proof}

Putting Propositions~\ref{prop:annealed-formula-1} and~\ref{prop:annealed-formula-2}, we have completed the proof of Theorem~\ref{mainthm:annealed-formula}.

\begin{prob}
When $\phi$ is unital, Theorem~\ref{thm:pre-LDP} says that $-h_F$ gives the rate function in the large deviations principle obeyed by the random element $(\Phi^\pi_V)[F]$ of $\S^\rm{u}_k(F)$.  The function $-h_F'$ should have an analogous interpretation, but this time for the random $S$-tuples
\begin{equation}\label{eq:tuple-of-Phis}
\big((\Phi^\pi_V)[F\cup sF]:\ s \in S\big)
\end{equation}
taking values in the space
\[\Big\{(Q_s:\ s \in S) \in \prod_{s \in S}\S^\rm{u}_k(F\cup sF):\ Q_s[F]\ \hbox{is the same for all $s$}\Big\}.\]
If this is true, then both inequalities in Lemma~\ref{lem:F-G-ineqs} become instances of the contraction principle.

I believe one can prove this along the same lines as Theorem~\ref{thm:pre-LDP} without any really new ideas, but working with the tuples in~\eqref{eq:tuple-of-Phis} seems more complicated than working with $(\Phi^\pi_V)[F]$ alone.
\end{prob}

We can also write $\hann$ as an infinite series of contributions that appear as one enlarges a grounded set one element at a time.  Let $e = g_0$, $g_1$, $g_2$, \dots be a grounded enumeration of $\G$ (see Section~\ref{sec:group-geom}), let $F_n = \{g_0,\dots,g_n\}$ for each $n\ge 0$, and let $s_n \in S\cup S^{-1}$ be such that $F_{n+1}$ is an enlargement of $F_n$ in direction $s_n$ for every $n$.

\begin{cor}\label{cor:first-expansion}
If $\phi:\G\to\rmM_k$ is positive definite, then
\[\hann(\phi) = H(B_0) - \sum_{n=0}^\infty I(g_{n+1}\,;\,F_n\setminus s_nF_n\mid F_n\cap s_nF_n).\]
If $C_n$ is the Verblunsky coefficient of $\phi$ from $F_n$ to $F_n\cup g_{n+1}$ for each $n$, then the quantity above is also equal to
\begin{equation}\label{eq:hann-Verblunskys}
H(B_0) + \sum_{n=0}^\infty \rmH(I_k - C_n^\ast C_n).
\end{equation}
\end{cor}

\begin{proof}
When $\phi$ is unital, this is already implicit in the proof of Theorem~\ref{thm:pre-LDP} above.  Repeating the calculation in Step 3 of that proof gives
\[h_{F_n}(\phi) = -\sum_{i=0}^{n-1} I(g_{i+1}\,;\,F_i\setminus s_iF_i\mid F_i\cap s_iF_i)\]
by induction on $n$. Moreover, the $i^{\rm{th}}$ term of this sum is equal to
\[\log\det (I_k - C_i^\ast C_i),\]
by equation~\eqref{eq:cond-mut-inf-contraction}.  These partial sums are non-increasing in $n$ and converge to the desired infinite sum as $n\to\infty$.

For non-unital $\phi$, we introduce the extra term $H(B_0)$, and then the result follows by using Proposition~\ref{prop:hann-properties}(b) in the same way as in the proof of Proposition~\ref{prop:annealed-formula-1}.
\end{proof}

\begin{rmk}
The partial sums of the infinite series above depend on our particular choice of enumeration of $\G$, but their limit does not.  Presumably this fact can be proved directly by re-arranging log-determinant entropies using the chain rules from Proposition~\ref{prop:chain1}, but I have not found such a proof.

It seems a little surprising that these partial sums telescope into the relatively simple form in~\eqref{eq:LDP-formula-1}.  In the analogous story in ergodic theory, a similar limit of closed-form expressions is found for annealed sofic entropy.  But in that case a different proof can be given that handles a whole ball in $\G$ (or other grounded set) at once, so the appearance of a closed form is less surprising. This is the original derivation in~\cite{Bowen10c}.  I have not found an analog of that ergodic theoretic proof for annealed AP entropy, because it relies on certain bijections between sets to estimate their cardinalities. For us these would become maps between subsets of vector spaces that have no obvious reason to preserve volumes. 

On the other hand, our use of Verblunsky coefficients has no predecessor in ergodic theory. They make the proof of Theorem~\ref{mainthm:LDP} much easier, especially because of the independence given by Theorem~\ref{thm:dil-dist}.  One could try to turn our proof of Theorem~\ref{mainthm:LDP} back into a derivation of annealed sofic entropy in ergodic theory.  But without the independence given by Theorem~\ref{thm:dil-dist}, this would presumably require some kind of conditional LDP for the conditional distribution of $Q^{\pi_n}[F\cup g]$ given $Q^\pi[F]$ (or their ergodic theory analogs).  Conditional versions of LDPs appear in various places in the literature: see, for instance,~\cite{Cha97}, or~\cite{DupEll--WCLD}, or the notion of `maxingales' discussed in~\cite[Chap. 2]{Puh--LDIP}.  But there does not seem to be a canonical form for such conditional LDPs, and uses of them tend to involve longer and more delicate estimates than we have needed above.  \fin 
\end{rmk}

\begin{ex}
If $\phi$ is a Haagerup function, then Corollaries~\ref{cor:first-expansion} and~\ref{cor:Haa-Verb2} show that
\begin{equation}\label{eq:hann-of-Haa}
\hann(\phi) = \sum_{s \in S}\log (1 -|\phi(s)|^2).
\end{equation}
We use this special case of Corollary~\ref{cor:first-expansion} several times later. \qed
\end{ex}

One could obtain a variant of Corollary~\ref{cor:first-expansion} by using the increments in formula~\eqref{eq:LDP-formula-2} rather than~\eqref{eq:LDP-formula-1}, but I do not know that this has any advantages. We leave it aside here.

As explained in Section~\ref{sec:Haa}, any positive definite function may be approximating by central extensions from finite-radius balls.  By then perturbing the nonzero Verblunsky coefficients of that approximation a little further if necessary, and considering a formula such as~\eqref{eq:hann-Verblunskys}, we obtain the following.

\begin{cor}\label{cor:finite-ent-dense}
In $\S_k(\G)$, the subset of elements $\phi$ with $\hann(\phi) > -\infty$ is dense. Consequently, for any self-adjoint element $a$ of $\rmM_k(C^\ast \G)$, the annealed AP pressure along a uniformly random AP sequence (recall Definition~\ref{dfn:annAP-pressure}) satisfies
\[\frac{1}{\b}\rmp^\ann(\b a) \to \s_{\max}(a) \qquad \hbox{as}\ \b\to\infty.\]
\qed
\end{cor}

Note that Corollary~\ref{cor:finite-ent-dense} can fail for other examples of random AP sequences.  For example, if $\bspi$ is deterministic and $\phi$ is not asymptotically associated to $\bspi$, then $\hann$ must equal $-\infty$ throughout some neighbourhod of $\phi$.

For any positive definite function $\phi$, averaging~\eqref{eq:hG-alt} and~\eqref{eq:h-prime} gives
\begin{equation}\label{eq:pre-Sew}
\frac{h_G(\phi) + h'_F(\phi)}{2} = \frac{H(G) - (2r-1)\cdot H(F)}{2}.
\end{equation}
By Lemma~\ref{lem:F-G-ineqs}, the averages on the left-hand side of~\eqref{eq:pre-Sew} still converge to $\hann(\phi)$ as $F$ grows along a grounded enumeration.  But on the right-hand side all the terms that involve shifting $F$ or $G$ by individual generators have canceled out.  As a result, we can now derive another non-negative series expansion for $\hann(\phi)$.  It is crucial during our proof of Theorem~\ref{mainthm:tempered} below.

Let $g_n = e$, $g_1$, $g_2$, \dots be any length-first ordering of $\G$ as described in Section~\ref{sec:group-geom}, and for each positive integer $N$ let
\begin{equation}\label{eq:past}
P(g_N):= \{g_0,g_1,\dots,g_{N-1}\}:
\end{equation}
that is, the set of predecessors of $g_N$ in this ordering.  Observe that $P(e) = \emptyset$ and that
\[P(g_N) = B_n \quad \hbox{when}\quad N = |B_n|.\]

\begin{cor}\label{cor:Sew}
If $\phi:\G\to\rmM_k$ is positive definite, then $\hann(\phi)$ is equal to
	\begin{multline}\label{eq:Sew}
	H(B_0) - \frac{1}{2}\sum_{n=0}^\infty \sum_{s_{n+1}\cdots s_1 \in S_{n+1}}\big(H(s_n\cdots s_1\,|\,P(s_n\cdots s_1)) \\ \qquad \qquad \qquad \qquad \qquad \qquad - H(s_{n+1}\cdots s_1\mid P(s_{n+1}\cdots s_1))\big).
	\end{multline}
	\end{cor}

This is the analog of a series expansion due to Seward~\cite[Thm. 1.7]{Seward--freestab}.  For this reason we call it the \textbf{Seward expansion} of $\hann(\phi)$ corresponding to the given ordering of $\G$.  (Seward works throughout with a length-lexicographic ordering, but he needs that more specialized assumption only at other points in his paper.)  The Seward expansion proves more convient then either~\eqref{eq:LDP-formula-1} or~\eqref{eq:LDP-formula-2} during our proof of Theorem~\ref{mainthm:tempered} later.

\begin{proof}
For any $n\ge 0$, iterating the chain rule from Proposition~\ref{prop:chain1} gives
\begin{equation*}
H(S_{n+1}\mid B_n) = \sum_{s_{n+1}\cdots s_1 \in S_{n+1}}H(s_{n+1}\cdots s_1\mid P(s_{n+1}\dots s_1)).
\end{equation*}
When $n=0$, we combine this with~\eqref{eq:pre-Sew} to obtain
\begin{align*}
h_{B_1}(\phi) + h'_{B_0}(\phi) &= H(B_1) - (2r-1)\cdot H(B_0)\\
&= 2H(B_0) + H(S_1\mid B_0) -2r\cdot H(B_0)\\
&= 2H(B_0) - \sum_{s \in S_1}\big(H(B_0) - H(s\mid P(s))\big).
\end{align*}
Similarly, for $n\ge 1$, formula~\eqref{eq:pre-Sew} gives
\begin{align*}
&(h_{B_{n+1}}(\phi) + h'_{B_n}(\phi)) - (h_{B_n}(\phi) + h_{B_{n-1}}'(\phi)) \\ &= H(S_{n+1}\mid B_n) - (2r-1)\cdot H(S_n\mid B_{n-1})\\
&= \sum_{s_{n+1}\cdots s_1 \in S_{n+1}}H(s_{n+1}\cdots s_1\mid P(s_{n+1}\dots s_1)) \\
&\qquad \qquad \qquad \qquad - (2r-1)\cdot\sum_{s_n\cdots s_1 \in S_n}H(s_n\cdots s_1\mid P(s_n\dots s_1)).
\end{align*}
Since every reduced word in $S_n$ has $(2r-1)$ neighbours in $S_{n+1}$, the difference above is equal to
\[\sum_{s_{n+1}\dots s_1 \in S_{n+1}}\big(H(s_{n+1}\cdots s_1\,|\,P(s_{n+1}\cdots s_1)) - H(s_n\cdots s_1\mid P(s_n\cdots s_1))\big).\]
So~\eqref{eq:Sew} is the infinite series of increments of the sequence of values
\[(h_{B_{n+1}}(\phi) + h'_{B_n}(\phi))/2,\]
whose limit equals $\hann(\phi)$.
\end{proof}

\begin{rmk}
Different length-first orderings of $\G$ give different Seward expansions for $\hann$.  I do not see a simple way to translate directly between one of the resulting Seward expansions and the other. \fin
\end{rmk}

\section{Maximal entropy, central extensions, and Haagerup functions}\label{sec:Haa-ent}

The next result is a free-group extension of the maximum-entropy property of the Bernstein--Szeg\H{o} approximation to a positive definite sequence~\cite[Sec. 2.3 Notes]{SimOPUCI}, or equivalently of Burg's maximum entropy theorem for one-dimensional stationary Gaussian processes~\cite[Sec. 12.6]{CovTho06}.  We deduce it from formula~\eqref{eq:LDP-formula-1} of Theorem~\ref{mainthm:annealed-formula} along with Corollary~\ref{cor:first-expansion}.

Let $F$ be a grounded set, let $N := |F|-1$, and let $\phi:\G\to\rmM_k$ be positive definite.  Fix a grounded enumeration $g_0$, $g_1$, \dots of $\G$ such that $F = \{g_0,\dots,g_N\}$, and let $\psi$ be the positive definite function obtained by setting $\psi[F] := \phi[F]$ and then taking all subsequent Verblunsky coefficients of $\psi$ along this enumeration to be zero.  That is, $\psi$ is the central extension of $\phi[F]$ as defined in Section~\ref{sec:Haa}.

\begin{prop}\label{prop:Burg}
These functions satisfy $\hann(\psi) \ge \hann(\phi)$, and if $\phi[F]$ is nonsingular then equality holds only if $\psi = \phi$.
\end{prop}

\begin{proof}
Along this grounded enumeration, let $C_0$, $C_1$, \dots be a sequence of Verblunsky coefficients for $\phi$.  By assumption, a sequence of Verblunsky coefficients for $\psi$ is given by $C_0$, $C_1$, \dots, $C_{N-1}$, $0$, $0$, \dots.  Therefore Corollary~\ref{cor:first-expansion} gives
\begin{multline*}
\hann(\phi) = \rmH_\phi(\{e\}) + \sum_{n=0}^\infty \rmH(I_k - C_n^\ast C_n) \\
\le \rmH_\psi(\{e\}) + \sum_{n=0}^{N-1} \rmH(I_k - C_n^\ast C_n) = \hann(\psi).
\end{multline*}
If $\phi[F]$ is nonsingular, then the third and fourth expressions here are finite, and they can equal the second expression only if $\rmH(I_k - C_n^\ast C_n) = 0$ for every $n\ge N$, and hence $C_N = C_{N+1} = \dots = 0$.
\end{proof}


The function $\psi$ is initially defined from $\phi[F]$ using a particular choice of enumeration $g_0$, $g_1$, \dots.  But Proposition~\ref{prop:Burg} characterizes $\psi$ solely in terms of the functional $\hann$, so in fact it does not depend on the choice of enumeration.

For stationary processes over free groups, Seward remarks in~\cite{Seward--freestab} on how the vanishing of all but finitely many terms in his expansion characterizes the special classes of tree-indexed Markov process or their generalization to larger steps (see~\cite{Bowen10b} or~\cite{Sew14} for background on these).  An analog of this fact should hold for positive definite functions and central extensions in terms of the Seward expansion from Corollary~\ref{cor:Sew}, with a proof based on Proposition~\ref{prop:Burg} and the non-positivity of the terms in the expansion~\eqref{eq:Sew}.  But we do not explain this further here.

Proposition~\ref{prop:Burg} simplifies the identification of equilibrium states for elements of the group ring $\bbC[\G]$, reducing it to a finite-dimensional optimization problem.

\begin{cor}
Let $a$ be a self-adjoint element of $\bbC[\G]$, and suppose it vanishes outside the grounded set $F$.  If $\phi$ is an $\hann$-equilibrium state for $a$, then $\phi$ is equal to the central extension of $\phi[F]$.
\end{cor}

\begin{proof}
Any equilibrium state has finite $\hann$-value, so $\phi$ is nonsingular.  Let $\psi$ be the central extension of $\phi[F]$.  Our assumption on $a$ gives $\psi(a) = \phi(a)$.  On the other hand, Proposition~\ref{prop:Burg} gives $\hann(\psi) \ge \hann(\phi)$, with equality only if $\phi = \psi$.
\end{proof}

The following variant of Proposition~\ref{prop:Burg} isolates the Haagerup functions.

\begin{lem}\label{lem:Haa-max}
Let $\phi:\G\to\bbC$ be positive definite, and let $\psi$ be the Haagerup function which satisfies $\psi(s) = \phi(s)$ for all $s \in S$.  Then $\hann(\psi) \ge \hann(\phi)$, and if $\max_s|\phi(s)| < 1$ then this is an equality only if $\phi = \psi$.
\end{lem}

\begin{proof}
If $|\phi(s)| = 1$ for some $s \in S$, then $\phi[\{e,s\}]$ and $\psi[\{e,s\}]$ are both singular, and so $\hann(\phi) = \hann(\psi) = -\infty$.  So assume now that $\max_s|\phi(s)| < 1$, which implies that $\hann(\psi) > -\infty$ by formula~\eqref{eq:hann-of-Haa}.

To complete the proof, we assume further that $\phi$ maximizes $\hann(\phi)$ subject to the given values $\phi(s)$ for $s \in S$, and deduce from this that $\phi$ must equal the Haagerup function $\psi$.  Let $\phi$ be associated to the representation $\pi$ by the cyclic vector $v$.

First, applying Proposition~\ref{prop:Burg} to the grounded set $B_1$, the assumed maximality of $\hann(\phi)$ implies that $\phi$ is the central extension of $\phi[B_1]$ and so $\hann(\phi) = h_{B_1}(\phi)$.

Now apply Lemma~\ref{lem:F-G-ineqs} with $F =B_0$ and $G = B_1$.  It gives $h'_F(\phi) \ge h_G(\phi)$, and equality holds if and only if the vectors $P^\perp \pi(s)v$ for $s \in S\cup S^{-1}$ are orthogonal, where $P$ is the orthogonal projection onto $\bbC v$.  From this last condition, it follows that
\[\phi(s^{-1}t) = \langle \pi(t)v,\pi(s)v\rangle = \langle \pi(t)v,v\rangle\langle v,\pi(s)v\rangle = \phi(s^{-1})\phi(t) = \psi(s^{-1}t)\]
for any distinct $s,t \in S\cup S^{-1}$. Therefore $\phi[B_1] = \psi[B_1]$, and so $\phi = \psi$ since they are both central extensions from $B_1$.
\end{proof}

\begin{rmk}\label{rmk:Bozejko-central}
Continuing from Remark~\ref{rmk:Bozejko}, a matrix-valued generalization of Lemma~\ref{lem:Haa-max} should be straightforward, but we leave it aside here.  Actually, I believe that several variations on Proposition~\ref{prop:Burg} can be extracted from some of the other formulas that we have proved for $\hann$, and perhaps a more general formulation would cover all of them, but we do not pursue this either. \fin
\end{rmk}

\section{Haagerup functions as equilibrium states}\label{sec:Haa-as-equil}

Using Lemma~\ref{lem:Haa-max}, in this section we identify Haagerup functions as equilibrium states for certain explicit elements of $\bbC[\G]$.  The theory from Section~\ref{sec:annAP-Legendre} then leads to a more explicit formula for the zeroth-order AP entropy of these functions, and reveals more of the structure of their GNS representations.  The main conclusions are given in Proposition~\ref{prop:h0-for-Haa} below.  Some parts of it are known results, but the connections to entropy seem to be new.  We discuss this history more completely in Remark~\ref{rmk:history} following that proposition.

Throughout the remainder of this section, let
\begin{equation}\label{eq:nearest-neighbour-a}
a = \sum_{s \in S}a_s(e^{i \theta_s}\delta_s + e^{-i\theta_s}\delta_{s^{-1}}),
\end{equation}
where $a_s \ge 0$ and $\theta_s \in [0,2\pi)$ for each $s\in S$.

\begin{ex}\label{ex:adj}
The element $a$ in~\eqref{eq:nearest-neighbour-a} is the \textbf{adjacency element} of $\bbC[\G]$ if $\theta_s = 0$ and $a_s = 1$ for every $s$.  Every radial element of the form in~\eqref{eq:nearest-neighbour-a} is a scalar multiple of the adjacency element. \fin
\end{ex}

\begin{prop}\label{prop:Haa-as-equilibrium}
The element $a$ in~\eqref{eq:nearest-neighbour-a} has a unique $\hann$-equilibrium state $\phi$, and it is the Haagerup function which satisfies
\[\phi(s) = \left\{\begin{array}{ll} 0 & \quad \hbox{if}\ a_s = 0\\ e^{-i\theta_s}\frac{\sqrt{1 + 4a_s^2}-1}{2a_s} & \quad \hbox{if}\ a_s > 0\end{array}\right. \qquad (s \in S).\]
This function satisfies
\[\hann(\phi) = \sum_{s \in S}\log\frac{\sqrt{1+4a_s^2} - 1}{2a_s^2}\]
and
\begin{equation}\label{eq:Haa-pann}
\rmp^\ann(a) = \sum_{s \in S}\Big(\sqrt{1 + 4a_s^2}-1 + \log\frac{\sqrt{1+4a_s^2} - 1}{2a_s^2}\Big).
\end{equation}
\end{prop}

\begin{proof}
For this choice of $a$ and any state $\phi$, the value $\phi(a)$ depends only on $\phi|S$.  Therefore, by Lemma~\ref{lem:Haa-max}, $\phi$ can be a $\hann$-equilibrium state for $a$ only if it is a Haagerup function.  It remains to identify $\phi|S$.   By~\eqref{eq:hann-of-Haa}, it must optimize the expression
\[\phi(a) + \hann(\phi) = \sum_{s \in S}\big(a_s(e^{i\theta_s}\phi(s) + e^{-i\theta_s}\ol{\phi(s)}) + \log(1- |\phi(s)|^2) \big).\]
We can now optimize each term in this sum separately.  Firstly, whenever $a_s \ne 0$, the optimal choice of phase for $\phi(s)$ is such that $e^{i\theta_s}\phi(s) \in [0,\infty)$. Now calculus finds the optimal value of $|\phi(s)|$: by differentiating and solving a quadratic equation, it satisfies
\[2a_s - \frac{2|\phi(s)|}{(1-|\phi(s)|^2)} = 0 \qquad \Rightarrow \qquad |\phi(s)| = \left\{\begin{array}{ll} 0 & \quad \hbox{if}\ a_s = 0\\ \frac{\sqrt{1 + 4a_s^2}-1}{2a_s} & \quad \hbox{if}\ a_s > 0,\end{array}\right.\]
where at the last step we know we must choose $|\phi(s)|$ lying in $[0,1)$.  To finish we insert these values back into~\eqref{eq:hann-of-Haa} and then the definition of $\rmp^\ann$.
\end{proof}

With $a$ as in~\eqref{eq:nearest-neighbour-a}, one can also find the formula~\eqref{eq:Haa-pann} directly from its spherical-integral interpretation in~\eqref{eq:spherical-2}.  Indeed, for this choice of $a$, that spherical integral separates as a product:
\begin{align*}
\bbE[e^{n\langle \pi_n(a)v_n,v_n\rangle}] &= \bbE\Big[\exp\Big(n\sum_sa_s\langle (e^{i\theta_s}\pi_n(s)+e^{-i\theta_s}\pi_n(s)^\ast)v_n,v_n\rangle\Big)\Big]\\
&= \prod_{s \in S}\bbE[e^{na_s\rm{Re}\langle e^{i\theta_s}\pi_n(s)v_n,v_n\rangle}],
\end{align*}
where the second equality holds because the random unitary matrices $\pi_n(s)$ for $s \in S$ are independent for each fixed $n$. Since $e^{i\theta_s}\pi_n(s)v_n$ is uniformly distributed on $\rmS^{2n-1}$ for each $n$ and $s$, one can now find the asymptotic behaviour of each factor above by elementary estimates on the sphere, arriving back at the summands in~\eqref{eq:Haa-pann}.

The formula for $\phi$ in terms of $a$ in Proposition~\ref{prop:Haa-as-equilibrium} can be inverted explicitly, giving the following alternative characterization of the class of Haagerup functions.

\begin{cor}\label{cor:Haa-as-equilibrium}
If $\phi$ is a nonsingular Haagerup function, then $\phi$ is the $\hann$-equilibrium state for a unique element $a$ of the form in~\eqref{eq:nearest-neighbour-a}, and $a$ is specified by the conditions
\[\phi(s) = e^{-i\theta_s}|\phi(s)| \qquad \hbox{and} \qquad a_s = \frac{|\phi(s)|}{1-|\phi(s)|^2} \qquad \hbox{for each}\ s \in S.\]
\qed
\end{cor}

Henceforth, we write $\phi_a$ for the Haagerup function associated to $a$ by Proposition~\ref{prop:Haa-as-equilibrium}.  The group-ring element $a$ in~\eqref{eq:nearest-neighbour-a} is parametrized by the coefficients ${(a_se^{i\theta_s}:\ s \in S)}$, which can be regarded as forming a vector in $\bbC^r$ where $r := |S|$.  Then $\phi_a$ is parametrized by $(\phi(s):\ s \in S)$, which forms a tuple in $\bbD^r$, where $\bbD$ is the open unit disk in $\bbC$.  By Corollary~\ref{cor:Haa-as-equilibrium}, this sets up a homemorphism between these spaces through the Fenchel duality given by the functional $\hann$.  See~\cite[Sec. I.6]{SimSMLG} for an introduction to Fenchel duality in general.

Consider again a fixed choice of $a$, and now let $r_1$ be the number of $s \in S$ for which $a_s \ne 0$.  If $r_1 = 0$ then $\phi_{\b a} = \tau$ for every $\b$, and if $r_1 = 1$ then we are essentially working within a cyclic subgroup of $\G$.  To avoid these degenerate cases, let us assume henceforth that $r_1 \ge 2$.

We next introduce a parameter $\b > 0$ and compare the Haagerup functions $\phi_{\b a}$ with the general description of $\hann$-equilibrium states from Corollary~\ref{cor:equilibrium-struct}.  Proposition~\ref{prop:Haa-as-equilibrium} gives a formula for $\phi_{\b a}$ with $\b a_s$ in place of $a_s$ everywhere.  Checking that formula against Proposition~\ref{prop:Haa-ac}, we see that $\phi_{\b a}$ is (i) $\l$-normal when $\b$ is sufficiently small, but (ii) not tempered when $\b$ is sufficiently large (because $r_1 \ge 2$).  Therefore, as in Section~\ref{sec:annAP-Legendre}, Lemma~\ref{lem:maxs-monotone} promises a critical value $\b_{\rm{c}}^\ann$ such that $\phi_{\b a}$ is $\l$-normal for $\b < \b_{\rm{c}}^\ann$ and not $\l$-normal for $\b > \b_{\rm{c}}^\ann$.  A more careful use of Proposition~\ref{prop:Haa-ac} gives $\b_{\rm{c}}^\ann$ exactly.  Substituting from Proposition~\ref{prop:Haa-as-equilibrium} and simplifying gives
\[\frac{|\phi_{\b a}(s)|^2}{1 + |\phi_{\b a}(s)|^2} = \frac{2\b^2a_s^2 + 1 - \sqrt{1 + 4\b^2a_s^2}}{4\b^2a_s^2 + 1 - \sqrt{1 + 4\b^2a_s^2}}= \frac{1}{2}\Big(1 - \frac{1}{\sqrt{1 + 4\b^2 a_s^2}}\Big).\]
Therefore the criterion from Proposition~\ref{prop:Haa-ac} re-arranges to show that $\b_{\rm{c}}^\ann$ is the unique root of the equation
\begin{equation}\label{eq:from-Haa-ac}
\sum_{s \in S}\Big(1 - \frac{1}{\sqrt{1 + 4\b^2 a_s^2}}\Big) = 1.
\end{equation}
Notice that this equation has a root because we assume that $r_1 \ge 2$. We obtain the same formula for $\b_{\rm{c}}^\ann$ again in a different way below.

Proposition~\ref{prop:Haa-as-equilibrium} shows that $\phi_{\b a}$ is unique for every $\b$, so we can also apply Corollary~\ref{cor:maxs-monotone}.  It gives that $(\phi_{\b a})_{\rm{ac}}$ is a nonzero multiple of $\g_{a,z(\b)}$, where
\begin{equation}\label{eq:z-beta}
z(\b) = \phi_{\b a}(a) + \frac{1}{\b} = \sum_{s \in S}\frac{\sqrt{1 + 4\b^2 a_s^2} - 1}{\b} + \frac{1}{\b} \qquad (\b > 0).
\end{equation}
In the range $\b < \b_{\rm{c}}^\ann$, this is an explicit fomula for $F_a^{-1}(\b)$.  (If $a$ is the radial element then it can be re-arranged into a formula for $F_a$: see Example~\ref{ex:adj2} below.)  Since the expression in~\eqref{eq:z-beta} is a continuous function of $\b$, we also have $F_a(z(\b_{\rm{c}}^\ann)) = \b_{\rm{c}}^\ann < \infty$, and so $\phi_{\b a}$ still equals the $\l$-normal state $\g_{a,z(\b)}$ when $\b$ equals this critical value (compare with Proposition~\ref{prop:Haa-ac}(b)) .

Using the chain rule and simplifying,~\eqref{eq:z-beta} gives
\begin{equation}\label{eq:z-beta-deriv}
z'(\b) = -\frac{1}{\b^2}\Big(\sum_{s \in S}\Big(\frac{1}{\sqrt{1 + 4 \b^2a_s^2}} - 1\Big) + 1\Big) \qquad (\b > 0).
\end{equation}
This formula shows that $z'(\b)$ is continuous, and that it changes sign from negative to positive at the unique value of $\b$ which satisfies equation~\eqref{eq:from-Haa-ac}. Comparing with Corollary~\ref{cor:maxs-monotone}, this again identifies $\b_{\rm{c}}^\ann$ as the root of that equation.

\begin{prop}\label{prop:h0-for-Haa}
Let $a$ be as in~\eqref{eq:nearest-neighbour-a} with $r_1 \ge 2$, define $z(\cdot)$ as in~\eqref{eq:z-beta}, and assume that $\b > \b_{\rm{c}}^\ann$.  Let $\phi_{\b a}$ be the $\hann$-equilibrium state from Proposition~\ref{prop:Haa-as-equilibrium}, and let $\pi_{\b a}$ be the GNS representation of $\phi_{\b a}$.
\begin{enumerate}
\item[a.] We have $\s_{\max}(\pi_{\b a}(a)) = z(\b)$.
\item[b.] Let $\b_1 := F_a(z(\b))$, so $\b_1$ is the unique value in $(0,\b_{\rm{c}}^\ann)$ such that $z(\b_1) = z(\b)$. Then
\[\phi_{\b a} = \frac{\b_1}{\b}\g_{a,z(\b)} + \Big(1 - \frac{\b_1}{\b}\Big)\psi_{\b a},\]
where (i) $\psi_{\b a}$ is a $\l$-singular positive definite function, (ii) the GNS representation $\rho_{\b a}$ of $\phi_{\b a}$ is irreducible, and (iii) $\psi_{\b a}$ is associated to $\rho_{\b a}$ by an eigenvector of $\rho_{\b a}(a)$ with eigenvalue $z(\b)$.
\item[c.] We have
\begin{equation}\label{eq:h0-formula}
\rmh^0(\rho_{\b a}) = \rmh^0(\phi_{\b a}) = \sum_{s \in S}\log \frac{\sqrt{1 + 4\b^2a_s^2}-1}{\sqrt{1 + 4\b_1^2a_s^2} - 1} + (2r-1)\log\frac{\b_1}{\b}.
\end{equation}
\item[d.] If another representation $\kappa$ satisfies $\s_{\max}(\kappa(a)) \ge z(\b)$, then $\rmh^0(\kappa) \le \rmh^0(\phi_{\b a})$.
\end{enumerate}
\end{prop}

\begin{proof}
If $\b > \b_{\rm{c}}^\ann$, then $\phi_{\b a}$ is unique and not $\l$-normal.  Now parts (a) and (b) here are special cases of Corollary~\ref{cor:equilibrium-struct}(b) and (c).

To prove part (c), first observe that $\phi_{\b_1 a}$ is $\l$-normal because $\b_1 < \b_{\rm{c}}^\ann$, and so $\phi_{\b_1 a} = \g_{a,z(\b_1)} = \g_{a,z(\b)}$.  Therefore $\rmh^\ann(\phi_{\b_1a}) = \log \Delta \g_{a,z(\b)}$, and so, for $\b > \b_{\rm{c}}^\ann$, the formula for $\phi_{\b a}$ from part (b) gives
\[\log \Delta ((\phi_{\b a})_{\rm{ac}}) = \log \Delta \Big(\frac{\b_1}{\b}\g_{a,z(\b)}\Big) = \rmh^\ann(\phi_{\b_1a}) + \log \frac{\b_1}{\b}.\]
Consequently, we can re-arrange the three-entropy formula like this:
\begin{align}
\rmh^0(\rho_{\b a}) = \rmh^0(\phi_{\b a}) &= \hann(\phi_{\b a}) - \log \Delta ((\phi_{\b a})_{\rm{ac}}) \nonumber \\
&= \hann(\phi_{\b a}) - \hann(\phi_{\b_1a}) - \log\frac{\b_1}{\b} \label{eq:re-arrange-for-h0}.
\end{align}
Inserting the formula for $\hann$ from Proposition~\ref{prop:Haa-as-equilibrium} and combining some of the logarithms, this becomes~\eqref{eq:h0-formula}.

Finally, part (d) is a special case of Corollary~\ref{cor:equilibrium-struct}(d).
\end{proof}

\begin{rmk}\label{rmk:history}
When $\b < \b_{\rm{c}}^\ann$, the function $\phi_{\b a}$ is a multiple of the Green's function $G(e,\cdot,z(\b))$ on $\G$, as discussed in Example~\ref{ex:Green}.  As such, its careful study goes back at least to~\cite{Aom84,GerWoe86}, and even further in the radial case (see Example~\ref{ex:adj2} below).  Those works already contain several of the calculations that appear above, including an important role for the function we call $z(\cdot)$.  They also provide interpretations of their results in terms of random walks on $\G$, but we do not address these here.  In~\cite{Aom84}, Aomoto focused on the algebraic nature of $G(e,\cdot,z)$ as a function of $z$, and extended it by analytic continuation to allow $z$ in a large region of the complex plane.

These results were revisited and extended by Fig\`a-Talamanca and Steger in~\cite{FigTalSte94}.  They also explored a family of positive definite functions obtained by analytic continuation.  Their family includes our $\phi_{\b a}$ for every $\b > 0$.  They then studied new families of representations obtained as subrepresentations of the GNS representations of those functions. In particular, the representations $\rho_{\b a}$ from Proposition~\ref{prop:h0-for-Haa}(b) for $\b > \b_{\rm{c}}^\ann$ belong to the `complementary series' from~\cite[Sec. 3.3]{FigTalSte94} (see also~\cite[Chap. 5]{FigTalPicHAFG} in the radial case).  This means they belong to the segments outside the spectral radius ``$\pm\rho$'' depicted in~\cite[Fig. 2]{FigTalSte94}.  These new representations are constructed in~\cite[Chap. 3]{FigTalSte94} and then shown to be irreducible in~\cite[Thm. 4.1.2]{FigTalSte94}.  The decomposition from Proposition~\ref{prop:h0-for-Haa}(b) is implicit there, and is also made more explicit in the proof of~\cite[Thm. 3]{KuhSte95}.  Those references realize these representations on a Hilbert space constructed as a certain completion of the locally constant functions on the boundary $\partial\G$.  However, this completion does not naturally include $C(\partial \G)$, and these are not `boundary representations' in the usual sense (which are always tempered).

Our route to these representations through Fenchel duality and our proof of irreducibility from the uniqueness of $\hann$-equilibrium states seem to be quite different from those older arguments.  It would be interesting to know whether the study of $\hann$ can shed any further light on the whole family of representations considered in~\cite{FigTalSte94}.

In~\cite{PicPyt88}, the positive definite functions $\phi_{\b a}$ are used to prove that the values $z(\b)$ are a family of upper bounds on the operator norm $\|\l(a)\|$ when $a$ has the form in~\eqref{eq:nearest-neighbour-a}.  These upper bounds are strong enough to give a new proof of the Akemann--Ostrand formula for this norm from~\cite{AkeOst76}.  The estimates in~\cite{PicPyt88} effectively include the determination of $\b_{\rm{c}}^\ann$ and a proof that $\phi_{\b a}$ is square summable below this threshold (except they work with the parameter ``$s$'' which is the inverse of our $\b$). \fin
\end{rmk}

By Corollary~\ref{cor:Haa-as-equilibrium}, Proposition~\ref{prop:h0-for-Haa}(c) gives a procedure for finding the zeroth-order AP entropy of any nonsingular Haagerup function.  The value of this procedure depends on two key properties:
\begin{itemize}
\item The function $z(\b)$ can be found explicitly and is continuous and strictly increasing on the interval $[\b_{\rm{c}}^\ann,\infty)$.  Via Corollary~\ref{cor:finite-ent-dense} and the intermediate value theorem, it follows that $z$ takes all values in the range $[\s_{\max}(\l(a)),\s_{\max}(a))$ on that interval.
\item For each $\b$, the annealed AP entropy of $\phi_{\b a}$ is given by an explicit continuous function of $\b$.  In other settings, this could still be true for a particular choice of $\phi_{\b a}$, even if they are not the unique $\hann$-equilibrium states.
\end{itemize}
The key idea in the proof of Proposition~\ref{prop:h0-for-Haa}(c) is that the formula for $\hann$ may first be applied at the auxiliary value $\b_1$ to obtain $\log \Delta \g_{a,z(\b_1)}$, and then this result can be used to re-arrange the analogous formula at $\b$ to obtain $\rmh^0(\phi_{\b a}$).  In this way we can arrive at~\eqref{eq:h0-formula} without needing a formula for either $U_a$ or $F_a$ directly.  Notice that we also obtain a formula for $\log \Delta \g_{a,z(\b)}$ during the course of the proof.  I do not known of a more direct way to determine values of $\rmh^0$ for non-trivial examples. It still may be difficult to obtain an explicit formula for $\b_1$ as a function of $\b$, but we can do this in the special case in Example~\ref{ex:adj2} below.  This use of the auxiliary value $\b_1$ in order to find quantities of interest at $\b$, particularly through the relation~\eqref{eq:re-arrange-for-h0}, seems to be a quirk of the structure of annealed AP pressure.   In the specific example of Haagerup functions, these pairs of Haagerup functions satisfying the same resolvent condition with respect to $a$ already play a role in~\cite{FigTalSte94,KuhSte95}, their values given by different roots of the same algebraic equation. I do not know of similar phenomena in any other areas of infinite-dimensional convex analysis.

In Proposition~\ref{prop:h0-for-Haa}, we think of $a$ as fixed, and regard the formula~\eqref{eq:h0-formula} as a function of $\b$.  But in principle it can be re-arranged in various ways:
\begin{itemize}
\item We can absorb $\b$ into $a$ and regard $a$ itself as the variable.  This then becomes a formula for $\rmh^0(\phi_a)$.
\item We can fix $a$, but regard~\eqref{eq:h0-formula} as a function of the quantity $z = \s_{\max}(\pi_{\b a}(a))$ by understanding $\b_1$ and $\b$ as respectively the smaller and larger positive roots of the equation $z(\b) = z$.  This makes sense whenever $\s_{\max}(\l(a)) < z < \s_{\max}(a)$ (and also when $z = \s_{\max}(\l(a))$ if we take $\b_1 = \b$ in that case).  In this form, Proposition~\ref{prop:h0-for-Haa}(d) tells us that we have found the maximal value of $\rmh^0(\kappa)$ among all representation $\kappa$ such that $\s_{\max}(\kappa(a)) = z$.  This interpretation is key to deriving a large deviations principle in Section~\ref{sec:op-norm-LDP} below.
\end{itemize}

\begin{ex}[\emph{Continuation of Example~\ref{ex:adj}}]\label{ex:adj2}
If $a$ is the adjacency element of $\bbC[\G]$, then $\mu_a$ equals the Kesten--McKay law for a $2r$-regular tree from~\cite{Kes59,McK81}:
\[\mu_a(dx) = \frac{2r\sqrt{4(2r-1)-x^2}}{2\pi(4r^2-x^2)}\,dx \qquad (-2\sqrt{2r-1} \le x \le 2\sqrt{2r-1}).\]

The resulting Haagerup functions $\phi_{\b a}$ are radial for this choice of $a$.  Spectral and other properties of radial Haagerup functions were explored before their non-radial generalizations, including by Kesten~\cite{Kes59}, Cartier~\cite{Cartier72,Cartier73}, and Haagerup~\cite{Haa79}.  Their results and others are covered by the monograph~\cite{FigTalPicHAFG}, which includes further references.  More recently, a range of calculations from~\cite{FigTalPicHAFG} have been generalized in~\cite{Yam26b} in order to define and study a more abstract family of commutative *-algebras.

Let us assume that $r \ge 2$ and specialize Proposition~\ref{prop:h0-for-Haa} to this choice of $a$.  First~\eqref{eq:z-beta} becomes
\[z(\b) = \frac{r\sqrt{1 + 4\b^2}-(r-1)}{\b} \qquad (\b > 0).\]
Minimizing this function using calculus, we find that
\[\b_{\rm{c}}^\ann = \frac{\sqrt{2r-1}}{2(r-1)} \qquad \hbox{and} \qquad z(\b_{\rm{c}}^\ann) = 2\sqrt{2r-1}.\]
This determination of $z(\b_{\rm{c}}^\ann)$ is Kesten's formula for the spectral radius of the adjacency element~\cite{Kes59}.  See Figure~\ref{fig:zbeta}.

\begin{figure}[h!tbp]
\centering
\includegraphics[width=14cm]{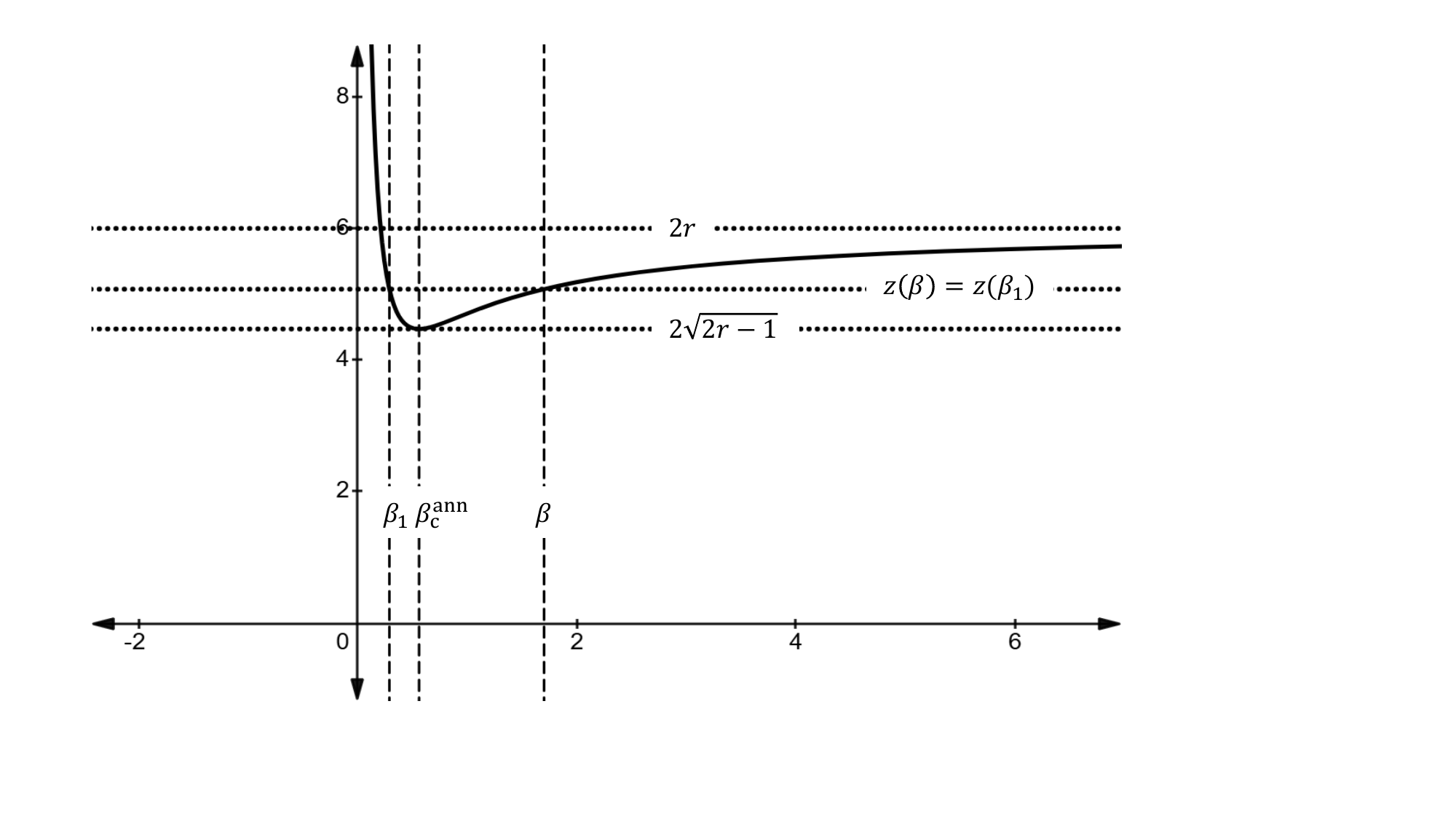}
\caption{\label{fig:zbeta} The function $z(\b)$ for the adjacency element when $r=3$, showing the relationship between $\b$ and $\b_1$.}
\end{figure}

Abbreviating $z(\b)$ to $z$, the formula for $z$ can be re-arranged like this:
\begin{equation}\label{eq:quadratic}
(z\b + (r-1))^2 = r^2(1 + 4\b^2) \qquad \Rightarrow \qquad (4r^2-z^2)\b^2 - 2(r-1)z\b + (2r-1) = 0.
\end{equation}
Since $\b > 0$, solving this quadratic equation and simplifying gives
\[\b = \left\{\begin{array}{ll} ((r-1)z \pm r\sqrt{z^2-4(2r-1)})\big/(4r^2-z^2) &\quad \hbox{if}\ 2\sqrt{2r-1} \le z < 2r\\
4r(r-1)\big/(2r-1) &\quad \hbox{if}\ z = 2r\\
(r\sqrt{z^2-4(2r-1)}-(r-1)z)\big/(z^2 - 4r^2) &\quad \hbox{if}\ z > 2r.\end{array}\right.\]
There is a transition from a pair of positive solutions for $z < 2r$ to a solitary positive solution for $z \ge 2r$. If we are given $\b > \b_{\rm{c}}^\ann$, then we can use the quadratic equation in~\eqref{eq:quadratic} to find the corresponding value $\b_1$, since it is the other root of the equation satisfied by $\b$.  This is the result:
\[\b_1 = \frac{2r-1}{\b(4r^2-z(\b)^2)} = \frac{(2r-1)\b}{2r(r-1)\sqrt{1+4\b^2} - (2r^2-2r+1)} \qquad (\b > \b_{\rm{c}}^\ann).\]
The relation between $\b$ and $\b_1$ is also shown in Figure~\ref{fig:zbeta}.

Now the formula from Proposition~\ref{prop:Haa-as-equilibrium} gives
\[\hann(\phi_{\b a}) = r\log \frac{\sqrt{1+4\b^2}-1}{2\b^2},\]
and, when $\b > \b_{\rm{c}}^\ann$ and $\b_1$ is given in terms of $\b$ as above, the formula from Proposition~\ref{prop:h0-for-Haa}(c) gives
\[\rmh^0(\phi_{\b a}) = r\log \frac{\sqrt{1 + 4\b^2}-1}{\sqrt{1 + 4\b_1^2} - 1} + (2r-1)\log\frac{\b_1}{\b}.\]
Finally, for $\b \le \b_{\rm{c}}^\ann$ the quantities $\log \Delta \phi_{\b a}$ and $\hann(\phi_{\b a})$ are equal, and for $\b > \b_{\rm{c}}^\ann$ their difference is equal to $\rmh^0(\phi_{\b a})$ by the three-entropy formula.  All three entropy quantities are plotted in Figure~\ref{fig:three-ent} as functions of $\b$ when $r=3$.

\begin{figure}[h!tbp]
\centering
\includegraphics[width=15cm]{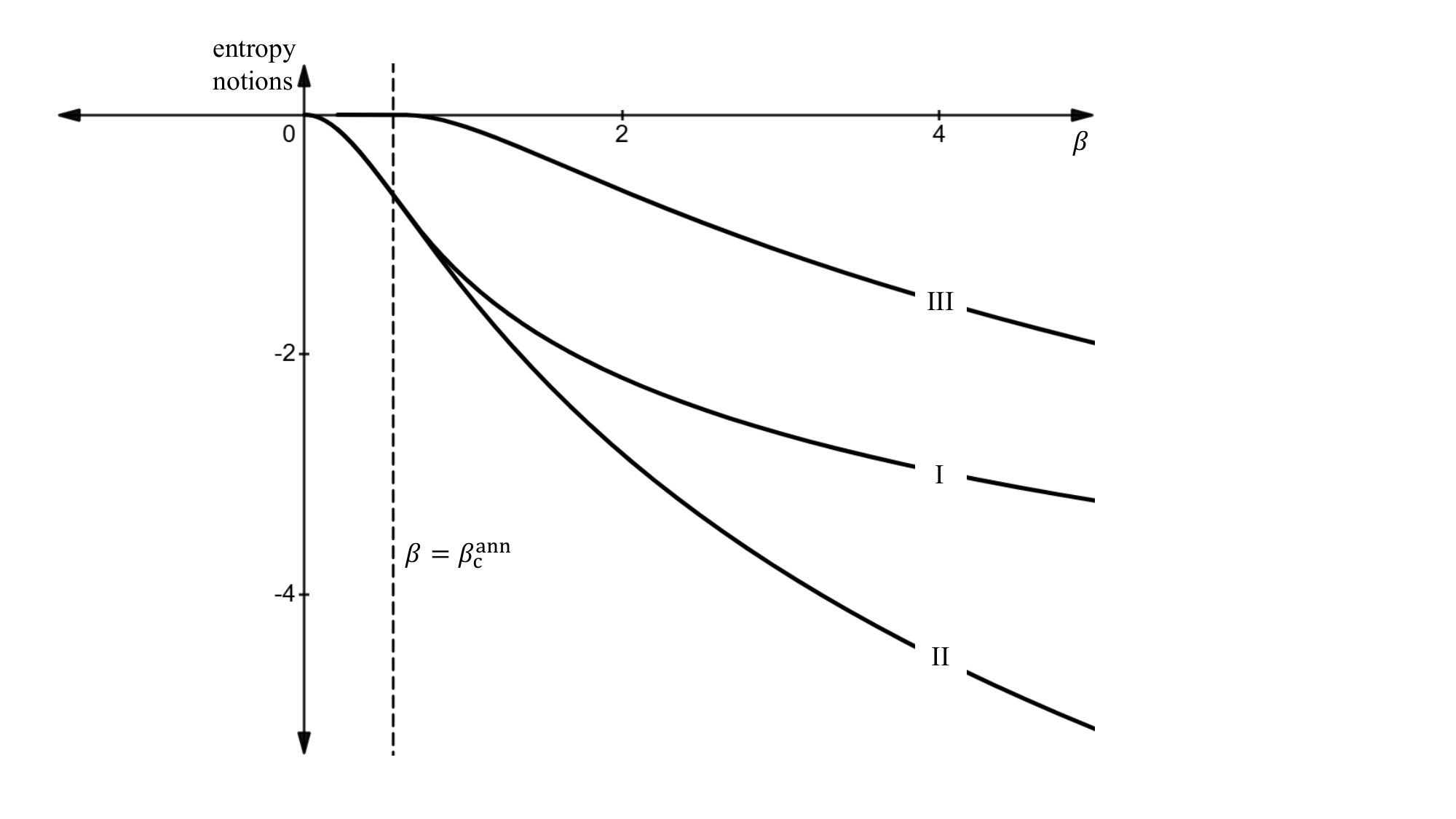}
\caption{\label{fig:three-ent} Three entropy notions for the $\hann$-equilibrium states of multiples of the adjacency element when $r=3$: $\log \Delta \phi_{\b a}$ (curve I), $\hann(\phi_{\b a})$ (curve II), and $\rmh^0(\phi_{\b a})$ (curve III).}
\end{figure}

Since we have found $\b$ as a function of $z$, this analysis also gives an explicit formula for $F_a(z)$, and then we can obtain a formula for $U_a(z)$ from this and the formula for $\log \Delta \g_{a,z}$.  These formulas can be checked by comparing with the definitions of $F_a$ and $U_a$ as integrals against the Kesten--McKay law, but we omit the details. \fin
\end{ex}

\begin{prob}\label{prob:multi-fn}
Haagerup functions have a natural matrix-valued generalization: see Remarks~\ref{rmk:Bozejko} and~\ref{rmk:Bozejko-central}. How do the results of this section generalize to these?
\end{prob}

Some further insight into this problem might be found in Lehner's matrix-valued generalization of the Akemann--Ostrand formula in~\cite{Leh99}.

\subsection*{\emph{Notes and further references}}

In the classical case, when $\G = \bbZ$, Burg's maximum entropy theorem is a widely used result of prediction theory in applied mathematics, particularly when augmented with a procedure such as Levinson's algorithm.  See~\cite[Sec. 12.6]{CovTho06} for an introduction or~\cite[Chap. II and Sec. XV.6]{FoiFraCL} for a thorough account.

Along a uniformly random AP sequence, the annealed AP pressure $\rmp^\ann$ is an analog of the Bethe free energy from the study of statistical physics and combinatorial optimization over trees and sparse random graphs: see, for instance,~\cite[Sec. 14.2]{MezMon09}.  It would be interesting to know whether other ideas from that field have natural variants in the setting of free-group representations, such as message-passing algorithms. 

\chapter{Further properties of the entropy notions}

This chapter proves two more fundamental properties of annealed and zeroth-order AP entropy.  The first is the additivity of $\hann$ and $\rmh^0$ for diagonal joinings of positive definite functions.  The second identifies those representations with vanishing $\rmh^0$ as those approximately contained in the regular represetation, leading to Theorem~\ref{mainthm:tempered}.

\section{Additivity}\label{sec:additivity}

\begin{prop}\label{prop:hann-additivity}
If $\phi \in \B(\G,\rmM_k)_+$ and $\psi \in \B(\G,\rmM_\ell)_+$ for some $k$ and $\ell$, then
\[\hann(\rm{diag}(\phi,\psi)) = \hann(\phi) + \hann(\psi).\]
\end{prop}

\begin{proof}
Consider the first of the formulas for $\hann(\phi)$ given by Theorem~\ref{mainthm:annealed-formula}: the limit as $n\to\infty$ of the expression
\begin{equation}\label{eq:hann-psi}
\log\det\,\phi[B_n] - \sum_{s \in S}\log\det\,\phi[B_n\cap sB_n].
\end{equation}
If we apply this to the diagonal joining $\rm{diag}(\phi,\psi)$, then every term is additive, and hence so is $\hann$ as a whole.
\end{proof}

Any of our other formulas for $\hann$ could be used in the proof of Proposition~\ref{prop:hann-additivity}(b) instead.  But it seems to be hard to prove this lemma directly from the interpretation of $\hann$ as the negative large deviations rate function in Theorem~\ref{mainthm:LDP}.

Proposition~\ref{prop:hann-additivity} can be seen as a version of Corollary~\ref{cor:AP-additive-concave}(a) for $\hann$.  We now obtain a version for $\rmh^0$ as well.

\begin{cor}\label{cor:h0-additivity}
If $\phi \in \B(\G,\rmM_k)_+$ and $\psi \in \B(\G,\rmM_\ell)_+$ for some $k$ and $\ell$, then
\[\rmh^0(\rm{diag}(\phi,\psi)) = \rmh^0(\phi) + \rmh^0(\psi).\]
\end{cor}

\begin{proof}
Observe that
\[\tau\otimes I_{k+\ell} + t\rm{diag}(\phi,\psi) = \rm{diag}(\tau\otimes I_k + t\phi,\tau\otimes I_\ell + t\psi)\]
for any $t \ge 0$.  Therefore Proposition~\ref{prop:hann-additivity} and Lemma~\ref{lem:mollify} give
\begin{align*}
\rmh^0(\rm{diag}(\phi,\psi)) &= \lim_{t\downarrow 0}\rmh^\ann(\tau\otimes I_{k+\ell} + t\rm{diag}(\phi,\psi))\\
&= \lim_{t\downarrow 0}\rmh^\ann(\tau\otimes I_k + t\phi) + \lim_{t\downarrow 0}\rmh^\ann(\tau\otimes I_\ell + t\psi)\\
&= \rmh^0(\phi) + \rmh^0(\psi).
\end{align*}
\end{proof}


The conclusion of Corollary~\ref{cor:h0-additivity} can be written heuristically as
\begin{multline*}
\bbP\big(\pi_n \ \hbox{approx. contains}\ \rm{diag}(\phi,\psi)\big) \\ \approx \bbP(\pi_n \ \hbox{approx. contains}\ \phi)\cdot \bbP(\pi_n \ \hbox{approx. contains}\ \psi).
\end{multline*}
That is, to leading exponential order, the appearance of a tuple in $\pi_n$ that is approximately typical for $\phi$ neither helps nor hinders the chance of finding another \emph{orthogonal} tuple that is approximately typical for $\psi$.

Combining Corollary~\ref{cor:h0-additivity} with other properties of zeroth-order entropy, we find that it is actually additive in the following stronger sense.  Its proof bears some resemblance with the proof that Kolmogorov--Sinai entropy is additive for Cartesian products of measure-preserving systems~\cite[Theorem 4.23]{Walters--book}.

\begin{prop}\label{prop:additivity}
Let $(\rho_i:\ i\in I)$ be any family of separable representations and let $\rho$ be their direct sum.  Then
\[\rmh^0(\rho) = \inf\Big\{\sum_{i\in J}\rmh^0(\rho_i):\ J\subset I,\ J\ \hbox{finite}\Big\}.\]
\end{prop}

In the sequel, we usually write simply
\[\sum_{i \in I}\rmh^0(\rho_i)\]
for the infimum on the right-hand side in Proposition~\ref{prop:additivity}.

\begin{proof}
Let $H_i$ be the space of $\rho_i$ for each $i$, so the space of $\rho$ is
\[H := \bigoplus_{i \in I}H_i.\]
Let $S$ be the set of vectors in $H$ that are non-zero in at most one coordinate. This satisfies the hypotheses of Lemma~\ref{lem:h0-nearly-cyclic}, so
\[\rmh^0(\rho) = \inf\big\{\rmh^0(\Phi^\rho_X):\ X\ \hbox{is a tuple drawn from}\ S\big\}.\]

If $X$ is a tuple drawn from $S$, then it uses only finitely many of the direct summands in $H$.  Therefore, permuting the entries of $X$ if necessary, we may write it as the concatenation of some sub-tuples, say $X_{i_1}$, \dots, $X_{i_\ell}$, so that $X_{i_j}$ is a tuple of vectors nonzero in only the $i_j^{\rm{th}}$ coordinate.  Since these sub-tuples lie in orthogonal sub-representations of $\pi$, the assumed additivity of $\rmh^0_{\bspi}$ gives
\[\rmh^0(\Phi^\rho_X) = \sum_{r=1}^\ell \rmh^0(\Phi^\rho_{X_{i_r}}).\]
By taking the infimum over all possible choices of $i_1$, \dots, $i_r$ and then $X_{i_1}$, \dots, $X_{i_r}$, this completes the proof.
\end{proof}

Proposition~\ref{prop:hann-additivity} is an essential ingredient in the rest of this chapter.  This additivity result follows fairly quickly from the formulas in Section~\ref{sec:ann-AP-forms}, but it does seem to be a special property of uniformly random AP sequences.  I do not see a clear method for proving it in greater generality.  On the other hand, once we know the additivity of $\hann$, the additivity of $\rmh^0$ (Proposition~\ref{prop:additivity}) follows quickly via Lemma~\ref{lem:mollify}, which is a much more general argument.

\section{Temperedness}\label{sec:tempered}

This section reaches the proof of Theorem~\ref{mainthm:tempered} following Proposition~\ref{prop:HS-close} and Theorem~\ref{thm:big-tempered} below.  Let $\|\cdot\|_{\rm{HS}}$ be the Hilbert--Schmidt norm for operators on $H$ (see Section~\ref{sec:lin-alg}).

\begin{prop}\label{prop:HS-close}
Let $\phi :\G\to \bbC$ be nonzero and positive definite.  If $\hann(\phi) > -\infty$ and $\pi = \pi_\phi$, then there is another representation $\rho$ of $\G$ on $H_\pi$ such that $\rho \simeq \l$ and
\[\sum_{s \in S\cup S^{-1}}\|\pi(s) - \rho(s)\|_{\rm{HS}}^2 \le 4(\log \phi(e) - \hann(\phi)).\]
\end{prop}

\begin{proof}
First, letting $\phi_1 := \phi/\phi(e)$, we have $\pi_{\phi_1} \simeq \pi_\phi$, and $\hann(\phi_1) = \hann(\phi) - \log \phi(e)$ by the scalar-valued case of Proposition~\ref{prop:hann-properties}(b).  We may therefore assume that $\phi$ itself is normalized. Suppose that $\phi$ is associated to $\pi$ by the cyclic vector $v$.

Now let $g_0$, $g_1$, \dots be a length-lexicographic ordering of $\G$, and for each $g \in \G$ let $P(g)$ be its set of predecessors in this ordering as in~\eqref{eq:past}.  These satisfy
\begin{equation}\label{eq:pasts-contained}
sP(g) \subset P(sg)
\end{equation}
whenever $s \in S\cup S^{-1}$, $g \in \G$, and $|sg| > |g|$.  For each $g \in \G$, let
\[M(g):= \rm{span}\{\pi(h)v:\ h \in P(g)\},\]
and let $Q_g$ be the orthogonal projection of $H_\pi$ onto $M(g)$.  From~\eqref{eq:pasts-contained} we obtain $\pi(s)M(g) \subset M(sg)$, and hence
\begin{equation}\label{eq:pasts-contained-2}
Q_{sg}^\perp = Q_{sg}^\perp\big(\pi(s)Q_g^\perp \pi(s^{-1})\big)
\end{equation}
whenever $|sg| > |g|$.

The finiteness of entropy implies that the vectors $\pi(g_n)v$ for $n=0,1,\dots$ are linearly independent, and together they span $H_\pi$.  We may therefore produce an orthonormal basis for $H_\pi$ by Gram--Schmidt orthonormalization:
\[w_g:= \frac{Q_g^\perp\pi(g)v}{\|Q_g^\perp\pi(g)v\|} \qquad (g \in \G).\]
Now define $\rho$ by setting $\rho(h)w_g := w_{hg}$ for all $g,h \in \G$ and extending by linearity.  Since the vectors $w_g$ are orthonormal, $\rho \simeq \l$.

Fix $s \in S\cup S^{-1}$, and let us use the basis $(w_g:\ g \in \G)$ to compare $\pi(s)$ and $\rho(s)$.  Let $g \in \G$, and suppose that $|sg| > |g|$.  From the definitions of $w_g$ and the relation~\eqref{eq:pasts-contained-2}, we have
\[Q_{sg}^\perp\pi(s)w_g = \frac{Q_{sg}^\perp\pi(s)Q_g^\perp\pi(g)v}{\|Q_g^\perp \pi(g)v\|} = \frac{Q_{sg}^\perp\pi(sg)v}{\|Q_g^\perp \pi(g)v\|}.\]
Comparing this with $\rho(s)w_g = w_{sg}$, which lies in the image of $Q_{sg}^\perp$, we obtain
\begin{multline*}
\langle \pi(s)w_g,\rho(s)w_g\rangle = \langle Q_{sg}^\perp\pi(s)w_g,\rho(s)w_g\rangle \\ = \frac{\langle Q_{sg}^\perp\pi(sg)v,Q_{sg}^\perp \pi(sg)v\rangle}{\|Q_g^\perp \pi(g)v\|\|Q_{sg}^\perp \pi(sg)v\|} = \frac{\|Q_{sg}^\perp \pi(sg)v\|}{\|Q_g^\perp \pi(g)v\|}.
\end{multline*}
Therefore
\begin{multline}\label{eq:HS-first}
\|(\pi(s) - \rho(s))w_g\|^2 = 2\Big(1 - \frac{\|Q_{sg}^\perp \pi(sg)v\|}{\|Q_g^\perp \pi(g)v\|}\Big) \\ \le 2 \log \frac{\|Q_g^\perp \pi(g)v\|}{\|Q_{sg}^\perp \pi(sg)v\|} = \rmH_\phi(g\,|\,P(g)) - \rmH_\phi(sg\mid P(sg)),
\end{multline}
where the inequality holds because $\log x^{-1} \ge 1-x$ whenever $0 < x \le 1$, and where the final expression comes from the calculation in Example~\ref{ex:1D-cond-ent}.

Now let us write
\[\sum_{s\in S\cup S^{-1}}\|\pi(s) - \rho(s)\|_{\rm{HS}}^2 = \sum_{g,s}\|(\pi(s) - \rho(s))w_g\|^2\]
and separate the terms of the right-hand side according as $|sg| > |g|$ or $|sg| < |g|$.
\begin{itemize}
\item By~\eqref{eq:HS-first}, the terms of the first kind satisfy
\[\sum_{g,s:\ |sg| > |g|}\|(\pi(s) - \rho(s))w_g\|^2 \le \sum_{g,s:\ |sg| > |g|} \big(\rmH_\phi(g\,|\,P(g)) - \rmH_\phi(sg\mid P(sg))\big).\]
Since $k=1$ and $\phi(e) = 1$, the right-hand side here is precisely the Seward expansion for $-2\hann(\phi)$ from Corollary~\ref{cor:Sew}.
\item On the other hand, if $|sg| < |g|$, then we can let $h:= sg$ and re-arrange like this:
\[\|(\pi(s) - \rho(s))w_g\|^2 = \|w_g - \pi(s^{-1})w_h\|^2 = \|(\rho(s^{-1}) - \pi(s^{-1}))w_h\|^2.\]
Summing over these pairs $(g,s)$, we obtain
\[\sum_{g,s:\ |sg| < |g|}\|(\pi(s) - \rho(s))w_g\|^2 = \sum_{h,s^{-1}:\ |s^{-1}h| > |h|}\|(\rho(s^{-1}) - \pi(s^{-1}))w_h\|^2,\]
so the terms of the second kind have the same sum as the terms of the first kind, and so this sum is also at most $-2\hann(\phi)$.
\end{itemize}
Adding these two upper bounds completes the proof.
\end{proof}

\begin{thm}\label{thm:big-tempered}
Let $\pi$ be a separable representation of $\G$.
\begin{enumerate}
\item[a.] If $\pi$ is cyclic and disjoint from $\l$, and if $\rmh^0(\pi) > -\infty$, then for every $\eps > 0$ there is a representation $\rho$ on $H_\pi\oplus H_\l$ such that $\rho \simeq \l$ and
\begin{equation}\label{eq:cyclic-HS-close}
\sum_{s \in S\cup S^{-1}}\|(\pi\oplus \l)(s) - \rho(s)\|^2_{\rm{HS}} \le 4(\eps - \rmh^0(\pi)).
\end{equation}
\item[b.] If $\rmh^0(\pi) > -\infty$, then $\pi$ is contained in a Hilbert--Schmidt perturbation of an inflation of $\l$.
\item[c.] If $\rmh^0(\pi) = 0$, then $\pi \lesssim_{\rm{a}} \l$.
\end{enumerate}
\end{thm}

\begin{proof}
\emph{Part (a).}\quad Suppose that $\pi$ is the GNS representation of the state $\phi$, and let $\psi$ be $\tau + \eps\phi$.  Since $\pi\spoon \l$, Corollary~\ref{cor:dom-contained} gives $\pi_\psi \simeq \pi \oplus \l$.

From Lemmas~\ref{lem:mollify} and~\ref{lem:h0-sing-only}(b), we obtain $\hann(\psi) \ge \rmh^0(\pi_\psi) = \rmh^0(\pi)$. On the other hand, $\log \psi(e) = \log (1+\eps) \le \eps$. Now~\eqref{eq:cyclic-HS-close} follows by applying Proposition~\ref{prop:HS-close} to the state $\psi$.

\vspace{7pt}

\emph{Part (b).}\quad By Proposition~\ref{prop:Leb-reps} and Lemma~\ref{lem:h0-sing-only}(b), we may first reduce to the case when $\pi$ is nonzero and disjoint from $\l$.  Then, decomposing further, we can write $\pi$ as $\bigoplus_{i\in I} \pi_i$, where $I \subset \bbN$ and each $\pi_i$ is cyclic and disjoint from $\l$.

Since $\rmh^0(\pi)$ is finite, so is $\rmh^0(\pi_i)$ for every $i$.  Therefore, for each $i$, part (a) gives a representation $\rho_i$ on $H_{\pi_i}\oplus H_\l$ such that $\rho_i \simeq \l$ and
\[\sum_{s\in S\cup S^{-1}} \|(\pi_i \oplus \l)(s) - \rho_i(s)\|_{\rm{HS}}^2 \le 4(2^{-i}\eps - \rmh^0(\pi_i)).\]
Summing over $i$, Proposition~\ref{prop:additivity} turns this into
\begin{equation}\label{eq:4eps-minus}
\sum_{s \in S\cup S^{-1}}\Big\|\bigoplus_{i \in I} (\pi_i\oplus \l)(s) - \bigoplus_{i\in I} \rho_i(s)\Big\|_{\rm{HS}}^2 \le 4\sum_{i \in I} 2^{-i}\eps - 4\sum_{i \in I}\rmh^0(\pi_i) \le 4\eps - 4\rmh^0(\pi).
\end{equation}

So $\bigoplus_i((\pi_i\oplus \l)(s) - \rho_i(s))$ is of Hilbert--Schmidt class for every $s \in S$. This fact extends to any other group element because the Hilbert--Schmidt operators form a self-adjoint ideal.  So $\pi$ is contained in $\bigoplus_{i \in I} (\pi_i\oplus \l)$, and this is a Hilbert--Schmidt perturbation of $\bigoplus_{i\in I} \rho_i \simeq \l^{\oplus I}$.

\vspace{7pt}

\emph{Part (c).}\quad As in part (b), we may first assume that $\pi$ is nonzero and disjoint from $\l$, and then decompose $\pi$ into cyclic summands $\pi_i$. If $\rmh^0(\pi) = 0$, then the right-hand side of~\eqref{eq:4eps-minus} is just $4\eps$, which we can make arbitrarily small.  So in this case $\bigoplus_{i \in I} (\pi_i\oplus \l)$ is an arbitrarily small Hilbert--Schmidt perturbation of $\l^{\oplus I}$.  Since the Hilbert--Schmidt norm dominates the operator norm (see inequality~\eqref{eq:op-HS-ineq}), this verifies condition (iv) in Theorem~\ref{thm:Voi-main} to show that $\pi\oplus \l^{\oplus I} \lesssim_{\rm{a}} \l^{\oplus I}$.  Finally, $\l^{\oplus I} \simeq_{\rm{a}} \l$ by Theorem~\ref{thm:get-sum} and the fact that $\l$ contains no nonzero compact operators (Lemma~\ref{lem:mix}).
\end{proof}

\begin{proof}[Proof of Theorem~\ref{mainthm:tempered}]
When $\phi$ is tempered, the first part of Theorem~\ref{mainthm:tempered} is given by Lemma~\ref{lem:h0-sing-only}(a), recalling that the hypothesis~\eqref{eq:traces-fast} is provided by Theorem~\ref{thm:asymp-free2} for a uniformly random AP sequence. On the other hand, if $\phi$ is not tempered, then $\pi_\phi$ is not approximately contained in $\pi$, and so $\rmh^0(\phi) = \rmh^0(\pi_\phi) < 0$ by Theorem~\ref{thm:big-tempered}(c).  Now the definition of $\rmh^0(\phi)$ provides a suitable neighbourhood of $\phi$.
\end{proof}

A previous version of the present paper gave a proof of Theorem~\ref{mainthm:tempered} based on Haagerup's criterion for temperedness in Theorem~\ref{thm:Haa79}, before the proof given above via Proposition~\ref{prop:HS-close} was discovered. That previous proof was considerably longer.  This is because, even if $\phi$ itself is a simple example such as a Haagerup function, the Verblunsky coefficients of the convex combinations $\phi_t$ from Lemma~\ref{lem:mollify} may be much more complicated, and one must estimate these well enough to make contact with the sums over spheres in Theorem~\ref{thm:Haa79}. It could still be interesting to look for more relationships between $\hann$ and other conditions for temperedness.

If $\phi$ has only finitely many nonzero Verblusnky coefficients, then the proof of Proposition~\ref{prop:HS-close} actually realizes $\pi_\phi$ as a finite-rank perturbation of $\l$.  In the further special case when $\phi$ is a Haagerup function, this realization is already in the literature.  Indeed, the original reference~\cite{DeMFigTal80} proves that a Haagerup function is positive definite by associating it explictly to a certain finite-rank perturbation of $\l$, and a quick inspection shows that it is the same pertubration as ours.  We do not write these details out explicitly here.  For a radial Haagerup function, this family of finite-rank perturbations of $\l$ is studied further in~\cite{Pyt91} and~\cite{Yam26}, where several other consequences are derived.  It would be interesting to know whether similar ideas can shed further light on the general class of positive definite functions with finite annealed AP entropy.

Since $\pi_\phi$ is a compact perturbation of $\l$, if it is not $\l$-normal, then it follows for the $\l$-singular irreducible summand in Proposition~\ref{prop:h0-for-Haa}(b) that $\rho_{\b a}(b)$ is compact whenever $b \in C^\ast\G$ and $\l(b) = 0$.  Because $\rho_{\b a}$ is irreducible, this actually implies that
\begin{equation}\label{eq:ker-in-stiff}
\rho_{\b a}[\ker \l] = \frak{K}(H_{\rho_{\b a}}).
\end{equation}

\section{Further consequences and discussion}\label{sec:three-entropy-cors}

\subsection{Harmonic analysis on free groups}\label{subs:harm-an}

Let $\phi \in \S^\rm{u}_k(\G)$, let $g_0$, $g_1$, \dots be a grounded enumeration of $\G$, and let $C_0$, $C_1$, \dots be the resulting Verblunsky coefficients for $\phi$.  By Theorem~\ref{thm:asymp-free2}, our uniformly random AP sequence $\bspi$ satisfies the fast convergence in probability in~\eqref{eq:traces-fast}, so Theorem~\ref{thm:three-entropy} gives a three-entropy formula.  Combined with Corollary~\ref{cor:first-expansion}, that formula gives
\begin{equation}\label{eq:free-Verb}
\sum_{n=0}^\infty\rmH(I_k - C_n^\ast C_n) = \rmh^0(\phi) + \log\Delta \phi_{\rm{ac}}.
\end{equation}
In addition, if $\phi$ is tempered, then $\rmh^0(\phi) = 0$ by Lemma~\ref{lem:h0-sing-only}(a), leaving only the log-determinant on the right-hand side of~\eqref{eq:free-Verb}.


If $\G = \bbZ$, then all positive definite functions are tempered, so $\rmh^0$ is identically zero.  The remaining terms in~\eqref{eq:free-Verb} are precisely Verblunsky's form of Szeg\H{o}'s limit theorem for determinants when $k=1$ (see~\cite[Thm. 2.3.1]{SimOPUCI}), or its matrix extension when $k > 1$ (see~\cite[Thm. 2.13.5]{SimOPUCI}).  These classical theorems have many different proofs.  Among them, the proof via large deviations is a relatively recent discovery of Gamboa, Nagel and Rouault~\cite{GamNagRou16}. See also the survey~\cite{BreSimZei18}, which gives references to subsequent refinements and variants on this work by those same three authors.

Our work here generalizes these results, and also that large-deviations proof, to finitely generated free groups.  However, when $\G$ has at least two generators, this group is not amenable.  In this case some positive definite functions $\phi$ are not tempered, and for these Theorem~\ref{thm:big-tempered} tells us $\rmh^0(\phi) < 0$.  So the free-group generalization of Verblunsky's form of Szeg\H{o}'s theorem has an additional term which `feels' how far $\phi$ is from being tempered.  Theorem~\ref{thm:big-tempered}(a) makes this intuition more precise.

Now consider the following special case.  Let $a \in C^\ast \G$, let $\l$ be the regular representation with its usual tracial vector $\xi$, and let $\phi := \Phi^\pi_{\l(a)\xi}$.  Then $k=1$ and $\phi$ is tempered by construction.  Since $k=1$, the resulting Verblunsky coefficients are single vectors $c_n$, and~\eqref{eq:free-Verb} turns into
\begin{equation}\label{eq:free-Verb2}
\log\Delta |a| = \frac{1}{2}\sum_{n=0}^\infty \log(1 - \|c_n\|^2).
\end{equation}
This seems to be a new formula for the Fuglede--Kadison determinant of an element of the reduced C*-algebra $C_{\rm{r}}^\ast \G$.  A special case of it underlies our work in~\ref{sec:Haa-as-equil} for the Haagerup functions $\phi_{\b a}$ in the phase when $\b \le \b_{\rm{c}}^\ann$. It could be interesting to look for more connections between the expansion in~\eqref{eq:free-Verb2} and other calculations of Fuglede--Kadison determinants (see~\cite{BenAri22} or~\cite{MaiSpe24}, for example), or for a more direct derivation of~\eqref{eq:free-Verb2} that does not go through $\hann$ and $\rmh^0$.


In the proof of Proposition~\ref{prop:HS-close}, our use of the basis $(w_g:\ g \in \G)$ has a classical precedent in the theory of orthogonal polynomials on the unit circle, corresponding to the case $\G = \bbZ$.  This is the `CMV' matrix representation of a single unitary operator due to Cantero, Moral and Vel\'azquez~\cite{CanMorVel03}; see also~\cite[Sec. 4.2]{SimOPUCI}.  The CMV representation enables proofs of several relationships between the Verblunsky coefficients of a positive definite function on $\bbZ$ and properties of the associated unitary operator, including estimates on that operator as a compact perturbation of the shift on $\ell^2(\bbZ)$.  Much of this work can be found in~\cite[Sec. 4.3]{SimOPUCI}, which also lists original references.  However, those proofs for $\G = \bbZ$ depend heavily on the special band structure of the CMV representing matrix, and this is replaced by a much more complicated picture when we use the basis $(w_g:\ g \in \G)$ to study a non-amenable free group.  Our results are consequently less complete.  The next problem asks about a strengthening of Proposition~\ref{prop:HS-close} that would be a step closer to~\cite[Thm. 4.3.5]{SimOPUCI}. 

\begin{prob}
Fix $\phi \in \S^{\rm{u}}_1(\G)$ and a grounded enumeration of $\G$, and let $(c_n)_{n \ge 1}$ be the resulting sequence of Verblunsky coefficients.  If $\|c_n\| \to 0$, is $\pi_\phi$ contained in a compact perturbation of the regular representation?
\end{prob}

On $\bbZ$, a positive definite function $\phi$ is the Fourier--Stieltjes transform of a measure $\mu$ on $\bbT$.  Many other properties of $\mu$ can be characterized in terms of its Verblunsky coefficients, and several chapters of Simon's books~\cite{SimOPUCI,SimOPUCII} are dedicated to results of this kind.  It would be interesting to look for other such results that generalize to a finitely generated but non-amenable free group $\G$.

An example of such a result that may be promising is Rakhmanov's theorem.  Fix again a positive definite function $\phi$ on $\G$ and a grounded enumeration of $\G$, and let $(c_n)_{n \ge 1}$ be the resulting Verblunsky coefficients of $\phi$, regarded as vectors since $k=1$.  If $\G = \bbZ$ and $\phi = \hat{\mu}$, then Rakhmanov's theorem says that, if $d\mu_{\rm{ac}}/dm$ is positive $m$-almost everywhere, then the Verblunsky coefficients of $\phi$ tend to $0$: see~\cite[p.5]{SimOPUCI} and~\cite[Chap. 9]{SimOPUCII}.  This hypothesis on $\mu$ is equivalent to $\pi_\phi \gtrsim \l$, and in this form it makes sense over any free group.  Does it still imply that $\|c_n\| \to 0$?

Another enhancement of Szeg\H{o}'s formula is the `higher-order Szeg\H{o} theorem' of~\cite[Sec. 2.8]{SimOPUCI}; see also~\cite{Rou21} for a matrix-valued version.  Here a suitable question over general free groups is less clear to me.

In a different direction, it could also be worth considering alternatives to the Verblunsky coefficients that we define and use in this book.  Here are two examples:
\begin{itemize}
\item In~\cite{BakTim07}, Bakonyi and Timotin provided a different way of extending a partial positive definite function on $\G$.  They also proceed along a sequence of larger and larger finite subsets of $\G$, but of a different kind.  Their work also introduces its own analog of `Verblunsky coefficients' over free groups.
\item Each term in the Seward expansion~\eqref{eq:Sew} is another conditional log-determinant mutual information.  Therefore each term may be expressed using a contraction matrix arising from Proposition~\ref{prop:three-block-completion} as in~\eqref{eq:cond-mut-inf-contraction}.  These contraction matrices could also provide alternatives to the version of Verblunsky coefficients that we use elsewhere in this paper.
\end{itemize}

\begin{prob}
What is the joint distribution of the coefficients of Bakonyi and Timotin for a uniform random representation?  Can those coefficients be used to give another expression for $\hann$?  Alternatively, what is the joint distribution of the contraction matrices that one obtains from the terms in the Seward expansion?
\end{prob}

Many other distinguished families of positive definite functions and free-group representations could also be investigated.

\begin{prob}
Is it true that $\rmh^0(1_H) = -\infty$ for any non-Abelian subgroup $H$ of $\G$?  What about examples constructed from invariant random subgroups (see~\cite[Sec. 15.F]{BekdelaHarUDC})?
\end{prob}

\begin{prob}
Is that true that ${\rmh^0(\pi) = -\infty}$ whenever $\pi$ is finite dimensional?
\end{prob}


One can combine positive definite maps or representations in various standard ways.  Apart from direct sums, I do not know how $\hann$ or $\rmh^0$ behaves under these.

\begin{prob}
Given an inclusion of discrete groups, one can use induction to extend representations of the smaller to the larger, and this construction can be carried out at the level of positive definite functions~\cite[Chap. 6]{FolAHA}.  If they are both free groups, how do $\hann$ and $\rmh^0$ behave under this construction?  The answer could be related to Seward's work~\cite{Sew14} in the ergodic theory setting.
\end{prob}

\begin{prob}
If $\phi$ and $\psi$ are positive definite, then so is $\phi\cdot \psi$: it is associated to the tensor product representation $\pi_\phi\otimes \pi_\psi$ of $\G$~\cite[Prop. 13.4.9]{Dix--Cstar}.  Is there any simple relation between $\hann(\phi\cdot \psi)$ or $\rmh^0(\phi\cdot \psi)$ and the two separate ingredients?
\end{prob}

\subsection{The structure of finite-entropy representations}


If $\rmh^0(\pi) > - \infty$, then Theorem~\ref{thm:big-tempered}(b) is quite a substantial constraint on the `global' structure of $\pi$.  For example, for any fixed $a \in C^\ast \G$, its image $\pi(a)$ can have some exceptional eigenvalues, but it must satisfy $\s_{\rm{ess}}(\pi(a)) \subset \s_{\rm{ess}}(\l(a))$.  

If $\pi \lesssim_{\rm{a}} \l$, then we can decompose $\pi$ into cyclic representations and apply the first part of Theorem~\ref{mainthm:tempered} to each summand in order to show that the implication in Theorem~\ref{thm:big-tempered}(c) can be reversed.

\begin{prob}
Does the reverse implication hold in Theorem~\ref{thm:big-tempered}(b)?
\end{prob}

If we combine the reverse of Theorem~\ref{thm:big-tempered}(c) with the forward direction of Theorem~\ref{thm:big-tempered}(b), we see that, if $\pi \lesssim_{\rm{a}} \l$, then $\pi$ is contained in a representation $\pi'$ such that $\pi'(g) - \l^{\oplus \infty}(g)$ is Hilbert--Schmidt for every $g \in \G$. This implies that $\pi'(a) - \l^{\oplus \infty}(a)$ is compact (although not necessarily Hilbert--Schmidt) for every $a \in C^\ast \G$. In this way, it recovers another aspect of Voiculescu's result in~\cite[Thm 1.5]{Voi76}.  In fact, the Weyl--von Neumann theorem for individual self-adjoint operators already gives perturbations of Hilbert--Schmidt class~\cite{vonNeu35}.  Very possibly an analogous conclusion about the Hilbert--Schmidt class can be drawn for representations from the original proof of Voiculescu's theorem, but I have not checked this.

If $\phi$ is associated to $\l$ and $\Delta \phi > 0$ (so $\hann(\phi) > -\infty$), then applying Proposition~\ref{prop:HS-close} gives a new copy $\rho$ of the regular representation on the Hilbert space $H_\l$, but I do not see a simple direct relationship between $\rho$ and $\l$.

Consider again a separable representation $\pi$ with $\rmh^0(\pi) > -\infty$.  By Theorem~\ref{thm:big-tempered}(b), we may enlarge it to be a Hilbert--Schmidt (and hence compact) perturbation of $\l^{\oplus \infty}$.  We can now examine $\pi$ further using more ideas from the proof of Voiculescu's theorems from Section~\ref{sec:approx-equiv}.  Let $\I := \pi^{-1}[\K(H_\pi)]$, and let $M := \ol{\rm{span}}(\pi(\I)H_\pi)$.  Then $\pi$ decomposes into $\pi^M$ and $\pi^{M^\perp}$.  A direct application of Theorem~\ref{thm:get-sum} gives $\pi^{M^\perp} \lesssim_{\rm{a}} \l$.  On the other hand, as in the proof of~\cite[Cor. 1.6]{Voi76} or~\cite[Prop. 2.10]{Had81}, $\pi^M$ decomposes as a countable direct sum of finite inflations of irreducible representations, say
\[\pi^M = \bigoplus_i \kappa_i^{\oplus m_i},\]
with the property that $\kappa_i|\I$ is a non-degenerate representation by compact operators for every $i$.  Now another appeal to  Lemma~\ref{lem:h0-sing-only}(a) and Proposition~\ref{prop:additivity} gives
\begin{equation}\label{eq:form-for-h0}
\rmh^0(\pi) = \rmh^0(\pi^M) = \sum_i m_i\cdot \rmh^0(\kappa_i).
\end{equation}
If $\pi(\A)$ does not contain any nonzero compact operators, then $\rmh^0(\pi)$ must be either $0$ or $-\infty$ according as $\pi \lesssim_{\rm{a}} \l$ or not.  

Proposition~\ref{prop:h0-for-Haa}(b) illustrates the decomposition~\eqref{eq:form-for-h0} when $\phi$ is a Haagerup function which is not $\l$-normal: we find the single additional summand $\rho_{\b a}$, which is singular, irreducible, and satisfies~\eqref{eq:ker-in-stiff}.

Still regarding~\eqref{eq:form-for-h0}, I do not know whether the following is possible.

\begin{prob}\label{prob:single-irred}
Is there a single irreducible representation $\kappa$ of $\G$ such that $\kappa \gtrsim_{\rm{a}} \l$ and $\rmh^0(\kappa) > -\infty$?
\end{prob}

For such a $\kappa$, the right-hand side of~\eqref{eq:form-for-h0} would have just a single summand.


By Corollary~\ref{cor:AP-additive-concave}(b), if $\rmh_{\bspi}$ fails to be concave for a deterministic AP sequence $\bspi$, it must be because two positive functionals can be asymptotically associated to $\bspi$ while some convex combination of them cannot.  Along a uniformly random AP sequence, $\hann$ also fails to be concave in general.  In view of the three-entropy formula from Theorem~\ref{thm:three-entropy}, the origin of this failure is the non-concavity, or even discontinuity, of $\rmh^0$: that is, of the `log-probability of asymptotic association'.  In the first place, if $\tau$ is the regular tracial state, then a simple variant of Lemma~\ref{lem:mollify} gives $\hann((1-t)\tau + t\phi) \to \rmh^0(\phi)$ as $t\downarrow 0$ for any $\phi \in \S_1(\G)$, whereas $\hann(\tau) = 0$.  In addition, if $\phi$ and $\psi$ are mutually singular, then Corollaries~\ref{cor:dom-contained} and~\ref{cor:h0-additivity} give
\[\rmh^0(t\phi + (1-t)\psi) = \rmh^0(\pi_{\phi} \oplus \pi_{\psi}) = \rmh^0(\phi) + \rmh^0(\phi) \qquad (0 < t < 1)\]
--- a sum, not a convex combination.  If $\rmh^0(\phi)$ and $\rmh^0(\psi)$ are both non-zero, then this formula also implies discontinuities in $\rmh^0|\S_1(\G)$ as $t$ converges to $0$ or $1$.  Because of Theorem~\ref{thm:big-tempered}, this failure of concavity appears as soon as one leaves the subset of tempered elements of $\S_1(\G)$.  This is a proper subset of that space provided $r\ge 2$, so we can attribute this failure to the non-amenability of non-Abelian free groups.

By Corollary~\ref{cor:h0-almost-contained}, we can also regard $\rmh^0$ as a function on $\Rep^\sim_{\rm{a}}(\G)$ rather than $\S_1(\G)$.  In this form, the following is still a natural question.

\begin{prob}\label{prob:h0-cts}
Is $\rmh^0$ continuous as a function on $\Rep^\sim_{\rm{a}}(\G)$?
\end{prob}

\subsection{Graded Fredholm modules and generated C*-algebras}

In addition to the use of CMV representations in~\cite[Chap. 4]{SimOPUCI}, the proof of Proposition~\ref{prop:HS-close} is also partly inspired by Cuntz' and Connes' simplified proof of the Pimsner--Voiculescu result that $C^\ast_{\rm{r}}\G$ has no non-trivial projections.  This proof begins by constructing a certain graded Fredholm module for $C^\ast \G$ which intertwines $\l$ and $\l^{\oplus 2}$ up to finite-rank perturbations.  See, for instance,~\cite[Sec. IV.5]{ConnesNG} or~\cite{CohenFigTal88}. 

Consider again Theorem~\ref{thm:big-tempered}(c), and for simplicity assume that $\pi$ is cyclic, say equal to $\pi_\phi$ for a normalized state $\phi$.  Let $\pi' := \pi \oplus \l$. Then the proof of that theorem gives a unitary $U:H_{\pi'} \to H_\l$ such that the difference $\pi'(g) - U^\ast \l(g) U$ is of Hilbert--Schmidt class for every $g \in \G$.  It follows that the pair
\[\Big(\pi'(g)\oplus \l(g),\Big[\begin{array}{cc}0 & U^\ast \\ U & 0\end{array}\Big]\Big)\]
is a graded Fredholm module over $C^\ast \G$ (see~\cite[Part IV]{ConnesNG}, for example).  

It would be interesting to know whether zeroth-order entropy has any interesting interaction with this structure.  If $\rmh^0(\pi) > -\infty$, then Theorem~\ref{thm:big-tempered} leads to a short exact sequence of C*-algebras
\[0 \to \pi(C^\ast\G)\cap \K(H_\pi) \to \pi(C^\ast \G) \to C^\ast_{\rm{r}}\G \to 0.\]
Can the graded Fredholm module above shed any additional light on the middle C*-algebra here?  If $\pi$ is the GNS representation of a radial Haagerup function, then a rather complete description of $\pi(C^\ast \G)$ is obtained in~\cite{Yam26}. In addition to the final result in~\cite[Cor. 2.13]{Yam26}, see in particular~\cite[Prop. 2.7 and Cor. 2.8]{Yam26} for arguments that use the description of $\pi$ as a finite-rank perturbation of $\l$.

One can show that if $\rmh^0(\pi) > -\infty$ then $\pi(C^\ast \G)$ has the `matricial field' property Blackadar and Kirchberg in~\cite{BlaKir97}.  We discuss this further in Section~\ref{sec:sq-LDP}.

\subsection{Quenched AP entropy}\label{subs:other-variants}




In Remark~\ref{rmk:other-orders} we consider the possibility of random AP entropy of `other orders'.  Another alternative could be `quenched' rather than `annealed' AP entropy.  This should capture the high-probability behaviour of the random volumes $\vol_{2kn}\X(\pi_n,O)$, which may not be reflected accurately by the annealed average because the expectation in~\eqref{eq:AnnAPent} may be dominated by large values on a low-probability event.  Similar phenomena appear often in the study of statistical physics models of disordered materials~\cite{MezParVir--book}.  A quenched average would be accomplished by setting
\[\rmh^{\rm{quen}}(\phi) := \inf_O \limsup_{n\to\infty}\frac{1}{n}\bbE \log\frac{\vol_{2kn}\X(\pi_n,O)}{v(n)^k}.\]
That is, we move the expectation outside the logarithm compared to~\eqref{eq:AnnAPent}.  This suppresses the tail of the distribution of the random volume $\vol_{2kn}\X(\pi_n,O)$, and so prevents the expectation from being dominated by rare events.

However, since Theorem~\ref{thm:big-tempered} gives $\rmh^0(\phi) < 0$ if $\phi$ is not tempered, we find that another variation on the proof of Theorem~\ref{thm:three-entropy} gives
\[\rmh^{\rm{quen}}(\phi) = \left\{\begin{array}{ll} \log \Delta \phi_{\rm{ac}} &\quad \hbox{if}\ \phi\ \hbox{is tempered} \\ -\infty &\quad \hbox{otherwise}.\end{array}\right.\]

So for a uniformly random AP sequence we do not obtain a really new notion of entropy from the quenched definition.  However, the reasoning above depend on a key feature of uniformly random AP sequences: if $\rmh^0(\phi) = 0$, then actually
\[\bbP(\X(\pi_n,O) \ne \emptyset) \to 1 \qquad \hbox{as}\  n\to\infty\]
for any neighrbourhood $O$ of $\phi$.  This can fail for some other natural examples, including uniformly random \emph{permutation} representations of free groups: see Section~\ref{sec:rndm-perms} below.  In some of those examples the behaviour of $\rmh^{\rm{quen}}_{\bspi}$ could be more interesting.

\chapter{More probabilistic limit laws}\label{chap:prob-lims}

\section{A new proof of the Collins--Male theorem}

Our proof of Theorem~\ref{mainthm:tempered} does not use Theorem~\ref{thm:ColMal} (the Collins--Male theorem), so we obtain a new proof of that theorem.

\begin{proof}[Proof of Theorem~\ref{thm:ColMal}]
By Lemma~\ref{lem:lower-Vietoris-simplify} and Lemma~\ref{lem:Vietoris-simplify}, it suffices to consider an sq-neighbourhood of $\t{\l}$ of the form $\cal{W}(k,U) \cap \cal{W}'(\ell,O)$ for some open subsets $U$ of $\S_k(\G)$ and $O$ of $\S_\ell(\G)$.  Since this intersection contains $\t{\l}$, the set $U$ must meet $\S_k(\l)$, and $O$ must contain all tempered positive definite maps.

On the one hand, since $U$ meets $\S_k(\l)$, it contains a tempered map. Therefore the first part of Theorem~\ref{mainthm:tempered} gives
	\[\bbP(\pi_n \in \cal{W}(k,U)) = \bbP(\X(\pi_n,U) \ne \emptyset) \to 1.\]
	
On the other hand, any $\phi \in \S_\ell(\G)\setminus O$ is not tempered, so the second part of Theorem~\ref{mainthm:tempered} provides a neighbourhood $V_\phi$ of $\phi$ and a positive value $c_\phi$ such that
\[\bbP(\S_\ell(\pi_n) \ \hbox{meets}\ V_\phi) \le e^{-c_\phi n + o(n)}.\]
By compactness, the closed set $\S_\ell(\G)\setminus O$ has a finite subcover $V_{\phi_1}\cup \cdots \cup V_{\phi_m}$, and now it follows that
\[\bbP(\t{\pi_n} \in \cal{W}'(\ell,O)) \ge 1 - \sum_{j=1}^m\bbP(\S_\ell(\pi_n)\ \hbox{meets}\ V_\phi) \to 1.\]
\end{proof}

At this point, I expect one can quickly give a new proof of the Haagerup--Thorbj\o rnsen theorem by reversing the spectral-calculus steps in Collins and Male's original proof in~\cite{ColMal14}, but I have not checked this carefully.

\section{Large deviations in the strong-quotient topology}\label{sec:sq-LDP}

For any separable representation $\pi$ we have $\rmh^0(\pi) = \ul{\rmh}^0(\pi)$ by Proposition~\ref{prop:annealed-formula-1}.  As usual, this agreement between `$\limsup$' and `$\liminf$' can be interpreted as a kind of large deviations principle.  By formula~\eqref{eq:0-ent-alt-2} from Lemma~\ref{lem:0-ent-alt-dfn}, this instance applies to the uniformly random AP sequence $\bspi$ and the quotient topology:
\begin{itemize}
\item[a.] If $\rmh^0(\pi) > h > -\infty$, then for any q-neighbourhood $U$ of $\t{\pi}$ we have
\begin{equation}\label{eq:h0-a}
\bbP(\t{\pi_n} \in U) \ge e^{hn - o(n)}.
\end{equation}
\item[b.] If $h >\rmh^0(\pi)$, then $\t{\pi}$ has a q-neighbourhood $U$ such that
\begin{equation}\label{eq:h0-b}
\bbP(\t{\pi_n} \in U) \le e^{hn + o(n)}.
\end{equation}
\end{itemize}
However, the quotient topology on $\Rep^\sim_{\rm{a}}$ is far from metrizable, so this theorem does not fall within the usual scope of large deviations theory (see Definition~\ref{dfn:LDP}).  Comparing with results for real-valued random variables,~\eqref{eq:h0-a} and~\eqref{eq:h0-b} are more like an estimate on the probabilities of `upper tails' than an estimate on the probabilities of landing close to specific values.

However, by combining with~\eqref{eq:h0-a} and~\eqref{eq:h0-b} with Proposition~\ref{prop:additivity} and Theorem~\ref{mainthm:tempered}, we can now prove the stronger large deviations principle in Theorem~\ref{mainthm:sq-LDP}.
We require one more lemma as preparation.  Intuitively, it says that $\rmh^0$ is `coercive' for approximate containment of finite-entropy representations.

\begin{lem}\label{lem:coerce}
If $\pi \gtrsim_\rm{a} \rho \gtrsim_\rm{a} \l$ and $\rmh^0(\pi) = \rmh^0(\rho) > - \infty$, then $\pi \simeq_\rm{a} \rho$.
\end{lem}

\begin{proof}
Theorem~\ref{thm:Had-main}(iii) gives $\pi \simeq_{\rm{a}} \rho \oplus \pi_1$ for some representation $\pi_1$.  Now the additivity from Proposition~\ref{prop:additivity} and our finite-entropy assumption give $\rmh^0(\pi_1) = 0$, and hence $\pi_1 \lesssim_\rm{a} \l$ by Theorem~\ref{mainthm:tempered}.  Therefore $\pi \lesssim_{\rm{a}} \rho \oplus \l$.  On the other hand, since $\rho \gtrsim_\rm{a} \l$ and $\l$ is contains no nonzero compact operators (Lemma~\ref{lem:mix}), Theorem~\ref{thm:get-sum} gives $\rho\oplus \l\simeq_\rm{a} \rho$. 
\end{proof}

\begin{proof}[Proof of Theorem~\ref{mainthm:sq-LDP}]
\emph{Step 1.}\quad If $\pi$ does not approximately contain $\l$, then there are a positive integer $k$, an element $\phi \in \S_k(\l)$, and a neighbourhood $U$ of $\phi$ such that $\ol{\S_k(\pi)}$ is disjoint from $\ol{U}$.  By Corollary~\ref{cor:APent-properties-trace}(a), the regular tacial state $\tau$ has a neighbourhood $V$ such that
\[\tr_n\circ \pi_n \in V \qquad \Rightarrow \qquad \X(\pi_n,U) \ne \emptyset.\]
We therefore obtain
\[\bbP\big(\S_k(\pi_n) \subset \S_k(\G)\setminus \ol{U}\big) = \bbP(\X(\pi_n,\ol{U}) = \emptyset) \le \bbP(\tr_n\circ \pi_n \not\in V),\]
and this decays faster than any exponential in $n$ by Theorem~\ref{thm:asymp-free2}.  Since the left-hand event displayed above is defined by an sq-neighbourhood of $\t{\pi}$, this gives the required large deviations upper bound if $\pi$ does not approximately contain $\l$.

\vspace{7pt}

\emph{Step 2.}\quad Next, let $\pi$ be any separable representation of $\G$, and let $h > \rmh^0_{\bspi}(\pi)$.  Then $\t{\pi}$ has a q-neighbourhood $U$ as in~\eqref{eq:h0-b}. Since $U$ is also an sq-neighbourhood, this proves the large deviations upper bound in all remaining cases.

\vspace{7pt}

\emph{Step 3.}\quad  Finally we prove the large deviations lower bound when $\pi \gtrsim_{\rm{a}} \l$ and $\rmh^0(\pi)$ is finite.  This is essentially a more careful version of our proof of Theorem~\ref{thm:ColMal} in the previous section.  By Lemma~\ref{lem:Vietoris-simplify}, it suffices to consider an sq-neighbourhood of $\t{\pi}$ of the form $W := U\cap \cal{W}'(\ell,G)$ for some q-neighbourhood $U$ of $\t{\pi}$, positive integer $\ell$, and open set $G \subset \S_\ell(\G)$ which contains $\ol{\S_\ell(\pi)}$. Let 
\[K := \{\t{\rho} \in \Rep^\sim_{\rm{a}}(\G):\ \ol{\S_\ell(\rho)}\ \hbox{meets}\ \S_\ell(\G)\setminus G\}.\]
This is sq-closed, hence sq-compact, and we have $W = U\setminus K$.

Now consider some $\t{\rho} \in K$.  Proposition~\ref{prop:Leb-reps-2} gives subrepresentatations $\pi' \lesssim \pi$ and $\rho' \lesssim \rho$ so that $\pi' \spoon \rho'$ and so that $\kappa:= \pi' \oplus \rho'$ contains both $\pi$ and $\rho$.   Since $G$ contains $\ol{\S_\ell(\pi)}$ but not $\ol{\S_\ell(\rho)}$, we must have $\rho \not\lesssim_{\rm{a}} \pi$, and hence also $\kappa \not\simeq_{\rm{a}} \pi$. Therefore $\rmh^0(\kappa) < \rmh^0(\pi)$ by Lemma~\ref{lem:coerce}.  Now~\eqref{eq:h0-b} gives $c_\rho > 0$ and a q-neighbourhood $W_\rho$ of $\t{\kappa}$ such that
\begin{equation}\label{eq:two-intersect-small}
\bbP(\t{\pi_n} \in W_\rho) \le e^{-c_\rho n + o(n)}e^{\rmh^0(\pi)n}.
\end{equation}
Having chosen $W_\rho$, Lemma~\ref{lem:two-q-nbhds} gives q-neighbourhoods $O_\rho$ of $\t{\rho}$ and $U_\rho$ of $\t{\pi}$ such that $O_\rho\cap U_\rho \subset W_\rho$.

Since $K$ is sq-compact, it has a finite subcover $O_{\rho_1}$, \dots, $O_{\rho_p}$, and now we have
\[W \supset U \cap U_{\rho_1}\cap \cdots \cap U_{\rho_p}\setminus (O_{\rho_1} \cup \cdots \cup O_{\rho_p}).\]
Therefore
\begin{align*}
\bbP(\t{\pi_n} \in W) &\ge \bbP(\t{\pi_n} \in U \cap U_{\rho_1}\cap \cdots \cap U_{\rho_p}) - \sum_{j=1}^p\bbP(\t{\pi_n} \in O_{\rho_j}\cap U_{\rho_j})\\
&\ge \bbP(\t{\pi_n} \in U \cap U_{\rho_1}\cap \cdots \cap U_{\rho_p}) - \sum_{j=1}^p\bbP(\t{\pi_n} \in W_{\rho_j}).
\end{align*}
On the right-hand side here, $U \cap U_{\rho_1}\cap \cdots \cap U_{\rho_p}$ is still a q-neighbourhood of $\t{\pi}$, so the first term above is bounded below by $e^{\rmh^0(\pi)n - o(n)}$ by~\eqref{eq:h0-a}.  Combining this with~\eqref{eq:two-intersect-small}, it gives
\[\bbP(\t{\pi_n} \in W) \ge e^{\rmh^0(\pi)n - o(n)} - \sum_{j=1}^p e^{-c_{\rho_j} n + o(n)}e^{\rmh^0(\pi)n} = e^{\rmh^0(\pi)n - o(n)},\]
as required.
\end{proof}

\begin{rmk}
By the definition of the strong-quotient topology, Theorem~\ref{mainthm:sq-LDP} can be regarded as a large deviations principle for the sequence of $k$-summaries $Z_k(\pi_n)$ in the space $\prod_{k\ge 1}\cal{K}(\S_k(\G))$, which is a product of hyperspaces.  Other examples of hyperspace-valued large deviations principles have appeared in the literature from time to time, particularly in works of Cerf on Minkowski sums of independent and identically distributed random compact sets~\cite{Cer99}, percolation clusters~\cite{Cer00a,Cer00b}, and the Wulff crystal~\cite{Cer06}. \fin
\end{rmk}


One potential application of Theorem~\ref{mainthm:sq-LDP} is new examples of representations that are approximable by finite-dimensional representations in the strong-quotient topology.  For example, if $\pi$ is a separable representation with $\pi \gtrsim_{\rm{a}} \kappa$ and ${\rmh^0(\pi) > -\infty}$, and $U$ is any sq-neighbourhood of $\t{\pi}$, then Theorem~\ref{mainthm:sq-LDP} gives
\[\bbP(\t{\pi_n} \in U) \ge e^{\rmh^0(\pi)n - o(n)},\]
and this is strictly positive for all sufficiently large $n$.  In particular, this event must be non-empty for all sufficiently large $n$.  From this and the metrizability of the strong-quotient topology, it follows that there is deterministic AP sequence $(\pi'_n)_{n\ge 1}$ whose approximate equivalence classes sq-converge to $\t{\pi}$.  As Lewis Bowen has pointed out to me, this implies that the generated C*-algebra $\pi(C^\ast \G)$ has the `matricial field' property of Blackadar and Kirchberg~\cite{BlaKir97}, because it follows that the composition
\[C^\ast \G \stackrel{(\pi'_1,\pi'_2,\dots)}{\longrightarrow} \prod_{n\ge 1}\rmM_n \stackrel{\rm{quotient}}{\longrightarrow} \prod_{n\ge 1}\rmM_n \Big\slash \big\{(A_n)_{n\ge 1}:\ \|A_n\|\to 0\big\}\]
has the same kernel as $\pi$.

\begin{rmk}\label{rmk:top-free-ent}
	Theorem~\ref{mainthm:sq-LDP} is worth comparing with Voiculescu's conjectured variational principle between `topological free entropy' and `free capacity' from~\cite{Voi02}.  When specialized to representations of a free group $\G$, the unproven parts of Voiculescu's principle reflect the following expected heuristic phenomena: 
	\begin{itemize}
	\item Given a character $\chi$ on $\G$ and a small neighbourhood $U$ of $\chi$, if we choose an $n$-dimensional representation $\pi_n$ of $\G$ uniformly at random, then the probability of the event $\{\tr_n\circ \pi_n \in U\}$ decays as a well-defined exponential with speed $n^2$, not $n$, and with a constant that usually blows up as we shrink $U$ around $\chi$;
	\item On this event, `most' of those representations $\pi_n$ also satisfy $\S_\bullet(\pi_n) \approx \ol{\S_\bullet(\pi_\chi)}$ in the Vietoris topology.
	\end{itemize}
	Both of these assertions remain open problems in free probability.
	
	When $\chi$ is the character of the regular representation $\l$, these conclusions are implied by the Collins--Male theorem (Theorem~\ref{thm:ColMal}).  Theorem~\ref{mainthm:sq-LDP} does not imply any other new cases, because in the regime covered by Theorem~\ref{mainthm:sq-LDP} the characters of the random representations $\pi_n$ are still typically close to regular, and the `excess' of $\S_\bullet(\pi_n)$ over $\ol{\S_\bullet(\l)}$ is caused by only a small subset of vectors in the unit sphere of $\bbC^n$.  This is also why the exponents that appear in  Theorem~\ref{mainthm:sq-LDP} are only of order $n$, whereas the more global constraints on representations required for Voiculescu's conjecture would lead to exponents of order $n^2$. But Theorem~\ref{mainthm:sq-LDP} does have the same spirit of providing probabilistic control over the sequence $\S_\bullet(\pi_n)$ in a new regime.	 \fin
	\end{rmk}

\section{Large deviations for largest eigenvalues}\label{sec:op-norm-LDP}


Large deviations principles have been found for the largest eigenvalues of many popular random matrix ensembles.  If $\bspi$ is a uniformly random AP sequence and $a$ is a self-adjoint element of $C^\ast \G$, then we obtain such a principle for $\s_{\max}(\pi_n(a))$ directly from Theorem~\ref{mainthm:sq-LDP} and the contraction principle (see~\eqref{eq:contraction}). Lemma~\ref{lem:sq-to-norm} gives the continuity required for the contraction principle.

\begin{cor}\label{cor:norm-upper-tails}
As $n\to\infty$, the random variables $\s_{\max}(\pi_n(a))$ obey a large deviations principle with speed $n$ and rate function
\[J_a(x) := \left\{\begin{array}{ll} \inf\big\{-\rmh^0(\pi):\ \t{\pi} \in \Rep^\sim_{\rm{a}}(\G),\ \s_{\max}(\pi(a)) = x\big\} &\quad \hbox{if}\ x \ge \s_{\max}(\l(a))\\ \infty &\quad \hbox{if}\ x < \s_{\max}(\l(a)).\end{array}\right.\]
\qed
\end{cor}

Conversely, we can regard Theorem~\ref{mainthm:sq-LDP} as an abstract generalization of large deviations principles for spectra of single operators.  To see this, consider a positive definite function $\phi$ such that (i) $\pi_\phi$ is contained in a Hilbert--Schmidt perturbation of $\l$ but (ii) $\pi_\phi$ is not approximately contained in $\l$ itself. If $U$ is a small q-neighbourhood of $\t{\pi_\phi}$, then the rare event that $\t{\pi_n} \in U$ is a representation-theoretic generalization of the rare event that a single random matrix has an exceptional outlying eigenvalue.

As a large deviations rate function, $J_a$ is lower semicontinuous.  In addition:
\begin{itemize}
\item $J_a(x) = \infty$ if $x < \s_{\max}(\l(a))$ because of the first part of Theorem~\ref{mainthm:tempered};
\item $J_a(\s_{\max}(\l(a))) = 0$ because $\rmh^0(\l) = 0$;
\item $J_a(x) > 0$ when $x > \s_{\max}(\l(a))$ because of the second part of Theorem~\ref{mainthm:tempered};
\item $J_a(x) = \infty$ if $x > \s_{\max}(a)$ because the infimum above is empty.
\end{itemize}

\begin{prob}
Under what general conditions is $J_a$ (i) continuous or (ii) convex on $(\s_{\max}(\l(a)),\s_{\max}(a))$?
\end{prob}

A standard alternative formulation of a large deviations principle is in terms of closed and open sets: see inequalities~\eqref{eq:LDP2-1} and~\eqref{eq:LDP2-2}.  Used in that form, Theorem~\ref{mainthm:sq-LDP} gives upper and lower estimates on the upper-tail probabilities of largest eigenvalues:
\[\bbP(\s_{\max}(\pi_n(a)) \ge y) \le \exp\big(-\inf_{x \ge y}J_a(x)\cdot n + o(n)\big)\]
and
\[\bbP(\s_{\max}(\pi_n(a)) > y) \ge \exp\big(-\inf_{x > y}J_a(x)\cdot n - o(n)\big).\]
The leading exponents agree if $y$ is a point of continuity of $J_a$.

\begin{rmk}
We do not learn anything about the lower tails of $\s_{\max}(\pi_n(a))$ from Theorem~\ref{mainthm:sq-LDP}, except that these decay faster than any exponential in $n$.  This can also be deduced directly from a suitable application of Theorem~\ref{thm:Un-conc}. 
Instead, the lower tails should be controlled by the conjectured variational principle for the `topological' version of Voiculescu's free entropy: see~\cite{Voi02}, and compare with Remark~\ref{rmk:top-free-ent}.  \fin
\end{rmk}

To obtain a more explicit formula for $J_a$, we look for representations which maximize $\rmh^0(\pi)$ subject to a given value of $\s_{\max}(\pi(a))$.  By Corollary~\ref{cor:equilibrium-struct}(d), these can sometimes be identified as the GNS representations of the $\hann$-equilibrium states of $\b a$ for some $\b > 0$.

\begin{cor}
Let $a$ be a self-adjoint element of $C^\ast \G$, let $x > \s_{\max}(\l(a))$, and suppose that for some $\b > 0$ the element $\b a$ has a $\hann$-equilibrium state $\phi$ which satisfies $\s_{\max}(\pi_\phi(a)) = x$.  Then
\[J_a(x) = - \rmh^0(\phi) = \log \Delta \phi_{\rm{ac}} - \hann(\phi).\]
\qed
\end{cor}

For example, if $a$ has the nearest-neighbour form in~\eqref{eq:nearest-neighbour-a}, then we can proceed like this:
\begin{itemize}
\item each $\b a$ has a unique equiliubrium state $\phi_{\b a}$ given by Proposition~\ref{prop:Haa-as-equilibrium};
\item for $\b > \b_{\rm{c}}^\ann$ this satisfies $\s_{\max}(\phi_{\b a}) = z(\b)$ as in Proposition~\ref{prop:h0-for-Haa}(a), and $z(\cdot)$ restricts to an increasing homeomorphism from $(\b_{\rm{c}}^\ann,\infty)$ to $(\s_{\max}(\l(a)),\s_{\max}(a))$;
\item the zeroth-order AP entropy $\rmh^0(\phi_{\b a})$ is given by the formula in Proposition~\ref{prop:h0-for-Haa}(c), leading to
\begin{equation}\label{eq:rate-eg}
J_a(x) =  -\rmh^0(\phi_{\b_+(x)a}) = - \sum_{s \in S}\log \frac{\sqrt{1 + 4\b_+(x)^2a_s^2}-1}{\sqrt{1 + 4\b_-(x)^2a_s^2} - 1} - (2r-1)\log\frac{\b_-(x)}{\b_+(x)} 
\end{equation}
when $\s_{\max}(\l(a)) < x < \s_{\max}(a)$, where $\b_{\pm}(x)$ are the two positive roots of the equation $z(\b) = x$.
\end{itemize}
If we can find $\b_\pm(x)$ as explicit functions of $x$, then the last step above turns into an explicit formula for $J_a(x)$. For example, this is possible if $a$ is the adjacency element (see Example~\ref{ex:adj}), for which the re-arrangements we need are already performed in Example~\ref{ex:adj2}:
\[\b_{\pm}(x) = \frac{(r-1)x \pm r\sqrt{x^2-4(2r-1)}}{4r^2-x^2} \qquad (2\sqrt{2r-1} < x < 2r).\]
For the adjacency element, the function $J_a$ also simplifies further because all the terms in the sum over $S$ in~\eqref{eq:rate-eg} are the same.  The resulting rate function is plotted when $r=3$ in Figure~\ref{fig:rate-function}.  It appears to be convex in this example, but I have not checked this carefully.

\begin{figure}[h!tbp]
\centering
\includegraphics[width=16cm]{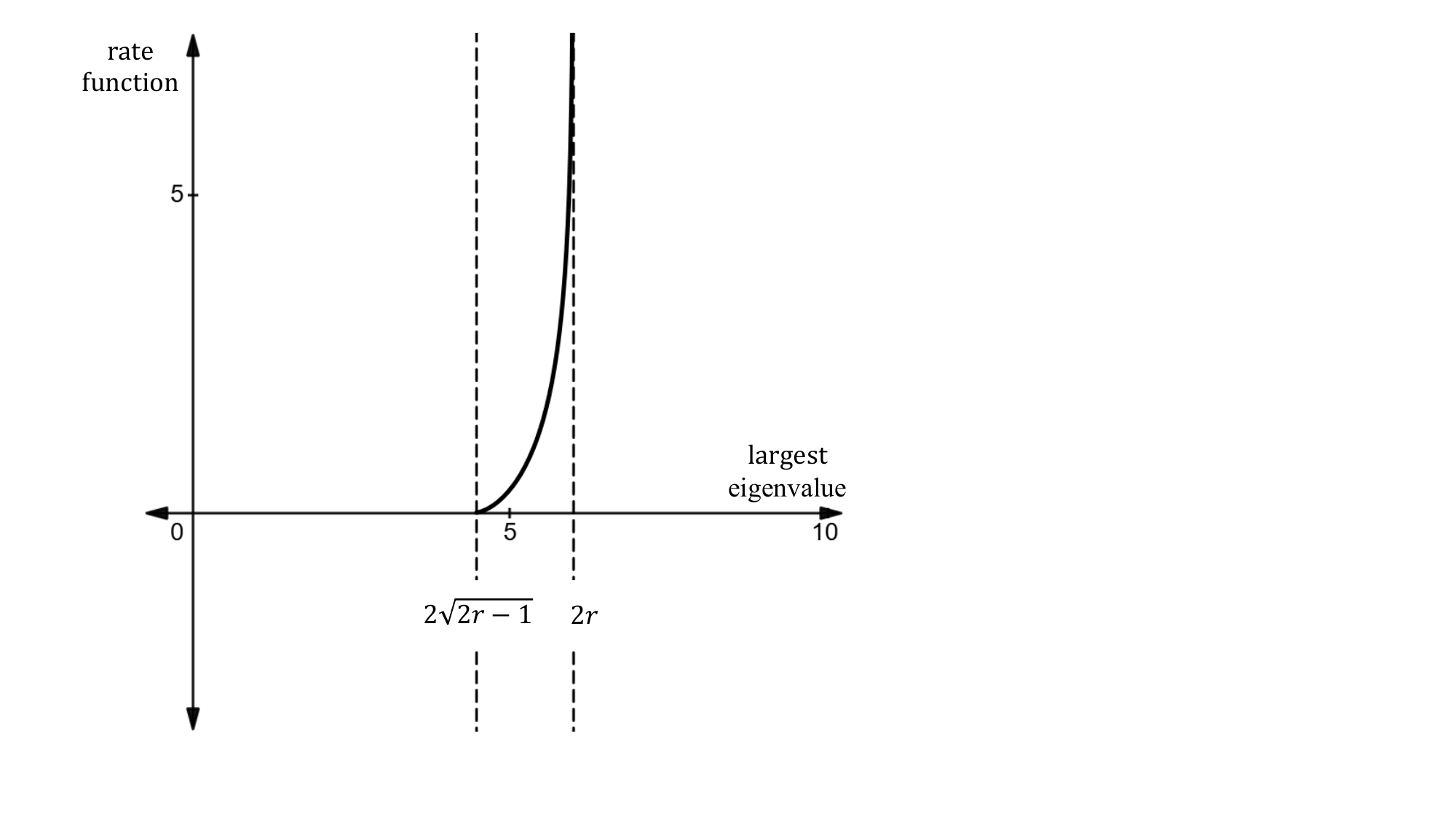}
\caption{\label{fig:rate-function} The rate function $J_a$ when $r=3$ and $a$ is the adjacency element.}
\end{figure}

I expect the route we have taken to deriving~\eqref{eq:rate-eg} can be generalized to a variety of other self-adjoint elements of $\bbC[\G]$, and also matrix-valued analogs.  The key requirements are the same as for the analysis of Haagerup functions in Section~\ref{sec:Haa-as-equil}, particularly the two bullet points in the paragraph that follows Remark~\ref{rmk:history}.  The first of these requires some continuity of the quantity $\rmh^0(\phi_{\b a})$ as a function of $\b$, so this may be related to Problem~\ref{prob:h0-cts}.

For a general self-adjoint element $a$, if the $\hann$-equilibrium state $\phi$ for $\b a$ is unique and it satisfies $\s_{\max}(\pi_\phi(a)) = x$, then an application of Lemma~\ref{lem:cond-lim} provides a conditional limit theorem to accompany Corollary~\ref{cor:norm-upper-tails}: for any sq-neighbourhood $U$ of $\pi_\phi$ there is an $\eps > 0$ such that
\[\bbP(\tilde{\pi_n} \in U \mid x-\eps < \s_{\max}(\pi_n(a)) <  x+ \eps) \to 1 \qquad \hbox{as}\ n\to\infty.\]
We do not explain this more fully here. 

Our work above has several steps in common with previous proofs of large deviations principles for largest eigenvalues, although the organization of those steps can vary.  These connections begin with the basic principle for the Gaussian orthogonal ensemble obtained in~\cite{BenDemGui01} (see also~\cite[Subs. 2.6.2]{AndGuiZei--book}).  They continue in more recent works such as~\cite{CooDucGui24}, which finds an example of a non-convex rate function, and~\cite{GuiHusRek25}, which studies a scenario closely related to ours.  In~\cite{GuiHusRek25}, the authors study random Kronecker matrices which are similar to $\pi_n(a)$ when $a$ is supported on $S\cup S^{-1}$, except that (i) each random Kronecker matrix is constructed from independent GUE ingredients, rather than independent uniformly random unitary matrices, and (ii) they additionally allows tensor products with matrix coefficients of fixed dimension.  The main theorem in~\cite{GuiHusRek25} is a large deviations principle for the largest eigenvalues of those matrices.  The proof in that paper involves a careful analysis of certain tilts of the original random matrix model using aymptotics for some related spherical integrals. These asymptotics make contact with our spherical-integral interpretations of AP pressure and annealed AP pressure described in Section~\ref{sec:AP-Legendre} and at the end of Section~\ref{sec:completed-LDP}.  This way of using asymptotics for spherical integrals has yielded several large deviations principles for largest eigenvalues of different random matrix models in recent years:~\cite{GuiHusRek25} contains further references. 

Beyond nearest-neighbour elements of $\bbC[\G]$ as in~\eqref{eq:nearest-neighbour-a}, one can imagine various steps towards more general results about largest eigenvalues:
\begin{itemize}
\item Similarly to~\cite{GuiHusRek25}, one could consider elements of $\rmM_k\otimes \bbC[\G]$ of the form
\[a = \sum_s (A_s\otimes \delta_s + A^\ast_s \otimes \delta_{s^{-1}})\]
for some $k$-by-$k$ matrices $A_s$.  For these, a generalization of Proposition~\ref{prop:Haa-as-equilibrium} shows that the $\hann$-equilibrium states are $\rmM_k$-valued generalizations of Haagerup functions.  At that point the optimization problems that one must solve become more complicated.  The main issues are those in Problem~\ref{prob:multi-fn}.  Some of the necessary ideas are already present in~\cite{GuiHusRek25}: for example, some of their estimates are conditional on the `profiles' of relevant vectors, which in our setting amounts to studying the refined functions $p_Q$ introduced at the end of Section~\ref{sec:AP-Legendre}.
\item Alternatively, one could study a `higher-degree' element $a$ of $\bbC[\G]$, meaning that its support is not contained in $S\cup S^{-1}$.  This task might reduce to the nearest-neighbour, matrix-valued case above via a linearization construction such as Pisier's from~\cite[Prop. 6]{Pis96}. 
\item A general self-adjoint element $a$ of $C^\ast \G$ can be approximated in norm by self-adjoint elements of $\bbC[\G]$, but I do not know whether this offers useful estimates on $J_a$.  How exotic can this function be in general?
\item Finally, one could investigate the large deviations of a tuple of values like
\[\big(\s_{\max}(\pi_n(a_1)),\dots,\s_{\max}(\pi_n(a_k))\big)\]
for several self-adjoint elements $a_1$, \dots, $a_k$ of $\bbC[\G]$.  Abstractly, a large deviations principle does follow for this from Theorem~\ref{mainthm:sq-LDP} in the same way as Corollary~\ref{cor:norm-upper-tails}.  But a more explicit description of the rate function seems to lie beyond the methods used in this book, even if each $a_i$ is supported on $S\cup S^{-1}$, and particularly if they do not commute.  It could be that one encounters interesting infinite-dimensional issues related to the characterization of possible spectra of sums of self-adjoint matrices in terms of honeycombs or hives~\cite{KnuTao99,KnuTao01}.  The recent work~\cite{NarShe24} studies hives of sums of certain pairs of independent random self-adjoint matrices, but for general $a_1$, \dots $a_k$ as above our model would provide less independence.  I have not pursued this any further. 
\end{itemize}

%
%
%
%

\chapter{Related objects and open problems}\label{chap:further}

This final chapter describes a few more directions that could be explored next.  Some of these are rather open-ended.

\section{Uniformly random permutations}\label{sec:rndm-perms}

The proof of Lemma~\ref{lem:LDP-to-A} depends on the invariance of the law of $\pi_n$ under conjugation by arbitrary elements of $\bf{U}(n)$.  This is why annealed averages are simpler than other statistics when studying uniformly random AP sequences.  Less symmetric distributions on finite-dimensional representations do not admit the same simplification, and the resulting zeroth-order and annealed AP entropy functionals can behave quite differently.

To consider an example, let $\G$ be freely generated by $S = \{s_1,\dots,s_r\}$, and for each $n$ let $\pi_n$ be a uniformly random permutation representation of $\G$ on $\bbC^n$.

For this group, Lubotzky and Shalom proved in~\cite{LubotSha04} that any nonempty q-open set of separable representations $U$ contains some finite permutation representation $\rho$, say on $\{1,\dots,m\}$.  For each $n \ge m$, let $V_n$ be the subspace of $\bbC^n$ spanned by the first $m$ coordinates, so this has a natural identification with $\bbC^m$. Now consider the events
\[E_n := \{\pi_n:\ \pi_n(s_i)|V_n= \rho(s_i)\ \hbox{for}\ i=1,2,\dots,m\} \qquad (n\ge m).\]
Then
\[\bbP(\t{\pi_n} \in U) \ge \bbP(E_n) = \frac{((n-m)!)^r}{(n!)^r}\ge \frac{1}{n^{mr}}.\]
This decays more slowly than any decaying exponential function of $n$ as $n\to\infty$. Since $r$ is fixed throughout, and $m$ depends only on $U$, we may finish with an infimum over $U$ to conclude that $\rmh^0_{\bspi}(\pi) = 0$ for every separable unital representation $\pi$.

On the other hand, an analog of Theorem~\ref{thm:asymp-free2} does hold when $\bspi$ is a sequence of uniformly random permutation representations, with convergence rate bounded by $O(1/(n!)^c)$ for positive constants $c$.  This is still faster than any exponential in $n$, so Theorem~\ref{thm:three-entropy} applies to give
\[\rmh^\ann_{\bspi}(\phi) = \rmh^0_{\bspi}(\phi) + \log \Delta \phi_{\rm{ac}} = \log \Delta \phi_{\rm{ac}}.\]
So in this case the annealed AP entropy always just equals to the relevant Fuglede--Kadison determinant.  I do not see any more direct proof of this fact starting from the definition of annealed AP entropy, or any simple alternative formula for the annealed AP entropy. 


The analog of Theorem~\ref{thm:ColMal} also holds for uniformly random permutation representations.  This is a major result of Bordenave and Collins~\cite[Thm. 3]{BordCol19}:

\begin{thm}\label{thm:BordCol}
The approximate equivalence classes of uniformly random permutation AP sequences for a finitely-generated free group sq-converge to the class of the regular representation in probability. \qed
\end{thm}

Once again~\cite{BordCol19} itself is about strong convergence, and we have coupled it to Corollary~\ref{cor:perfect} to conclude strong-quotient convergence.  Theorem~\ref{thm:BordCol} generalizes Friedman's earlier theorem from~\cite{Fri08} about spectral gaps of random regular graphs.

As shown above, if $U$ is a non-empty quotient-open set of separable representations, then the probabilities $\bbP(\t{\pi_n} \in U)$ do not decay faster than polynomially in $n$, but Theorem~\ref{thm:BordCol} shows that they do indeed converge to $0$ if $U$ does not contain $\t{\l}$.

\begin{prob}
Is this decay governed by a well-defined power-law in $1/n$?  If so, is there some kind of analog of AP entropy that computes or estimates those powers?
\end{prob}



\section{Other models of random unitary representations}

\subsection*{\emph{Planted models and models with a potential}}

Consider still the free group $\G$ on $S = \{s_1,\dots,s_r\}$.  Other choices of probability distribution for the generators $\pi_n(s)$ of a random representation may be more promising than random permutations for extending the theory from the uniformly random case.

One possibility is to `plant' a small number of vectors with prescribed behaviour in each of the representations $\pi_n$.  Technically, this would mean choosing the generators $(\pi_n(s):\ s \in S)$ initially from a distribution that has the form we have already met on the right-hand side of~\eqref{eq:prod-meas}. So it could be that this model of randomness requires few new ideas compared to our work in previous chapters.  Among the benefits of this generalization, one could hope to find a definition and an interpretation of `relative' annealed AP entropy.

In combinatorics and theoretical computer science, analogous `planted-partition' or `stochastic block' models of sparse random graphs have been widely studied: see~\cite{BolSco04} for an early example.  More recently, these have also found a role in the ergodic theory of free group actions, for instance in an interpretation of relative annealed sofic entropy~\cite{Shr23}.  For stochastic block models of random graphs, the typical behaviour of the largest eigenvalues of the associated nonbacktracking matrix was established in~\cite{BordLelMas18}, where it was applied to the `community detection' problem in the underlying graph.  However, for these random graphs themselves, large deviations in the behaviour of those largest eigenvalues should be dominated by events that involve only a bounded number of vertices, as in Section~\ref{sec:rndm-perms}, so this falls well outside the scope of the present book.

%

Another alternative would be to choose each $\pi_n(s)$ independently from a perturbation of the Haar measure by some potential.  The potential could be a symmetric function of the spectrum of $\pi_n(s)$, or could also refer to some particular chosen vectors in $\bbC^n$.  Notice that, depending on the potential, this could entail studying a different limiting tracial state $\tau$ on $C^\ast \G$.  For single random unitary or self-adjoint operators, several models of this kind are studied in~\cite{GamNagRou16,GamNagRou17,GamNagRou22}.  

\subsection*{\emph{Tensor products}}

One can also combine a random choice of unitary-matrix generators with operations such as tensor products: see~\cite{BordCol24}, for example.  A different construction with tensor products is essential to Hayes' recent proof of the Peterson--Thom conjecture in~\cite{Hay22}.  That proof was initially conditional on strong asymptotic freeness of certain random tensor product representations, but this ingredient has now been provided: see~\cite{BelCap22} and also~\cite{BordCol--tensors}.

\begin{prob}
Is any version of annealed AP entropy available in settings such as these?  If so, can it be computed or estimated?
\end{prob}

\subsection*{\emph{Orthogonal or symplectic representations}}

Most of our work should be straightforward to adapt to independent random orthogonal matrices, and I doubt there are any surprises here.  This is because the main results in Part~\ref{part:free} do not depend on the spectral theorem for unitary matrices anywhere. The case of symplectic matrices is less clear to me.  Note that the key calculations from Section~\ref{sec:random-orbits} already apply to any of these settings.

\subsection*{\emph{Analogies with free probability}}

Formally, our only direct appeal to free probability in Part~\ref{part:free} is through Theorems~\ref{thm:asymp-free1} and~\ref{thm:asymp-free2}.  But other points of resemblance are widespread.  Do other known ideas or constructions from the study of free entropy in free probability have adaptations to the study of annealed AP entropy?  See~\cite[Chap. 6]{HiaPetSL} for an introduction to that theory.  

\subsection*{\emph{Fixed matrices}}

Many of the known results about asymptotic freeness or strong asymptotic freeness gain extra strength because one can also allow a sequence of fixed matrices alongside the random ones, and show that the whole partially random collections of matrices still satisfy `as much (strong) asymptotic freeness as possible'.  See the full results in~\cite{ColMal14}, for example.

\begin{prob}
Is there a variant of annealed AP entropy that can help to analyse these models?
\end{prob}

\section{Other groups, algebras, or limits}

To construct our random finite-dimensional unitary representations of a free group, we choose a unitary matrix for each free generator independently from the Haar measure $m_{\bf{U}(n)}$.  As pointed out to me by Uri Bader, the distribution of the resulting random representation does not depend on the chosen free generating set for the group.  This is because any two sets of free generators are related by an automorphism of the free group, and any such automorphism is a composition of Nielsen transformations~\cite{Nie24}.  A direct check shows that each Nielsen transformation converts one tuple of independent Haar-distributed unitary matrices into another, and so the same is true of any automorphism of the group.  See also~\cite{Gel08} and the references listed there for more on this measure-preserving action of $\rm{Aut}(\G)$.

It would be natural to study uniformly random finite-dimensional representations of free products of other cyclic groups, and it should be straightforward to develop a resulting theory of annealed AP entropy.  For the analogous story in ergodic theory involving permutations and annealed sofic entropy, this generalization is presented in~\cite{AusBowShr}.  The formula for annealed sofic entropy for such groups is connected to a classical entropy-like functional introduced by Kikuchi (see~\cite[Section 3]{AusBowShr}), and I expect the analog of this story holds for annealed AP entropy as well.  The construction of Haagerup functions on free groups has been generalized to other free products and studied by various authors, starting with Bo\.{z}ejko in~\cite{Boz86}.  This should provide a basic source of examples for more general free-product groups, and also natural candidates for equilibrium states of simple group-ring elements.

However, beyond this class of free-product groups, it is more difficult to define `natural' measures on representation varieties.  Some progress of this kind is due to Goldman for surface groups~\cite{Gol84}.  See~\cite{MulPen02,BureLaw22,Mag22,MagPud24} for recent results about representations chosen at random from these measures.  When such a natural finite measure exists, one can sensibly use it to define `annealed' AP entropies for those groups.  I do not know how tractable or useful the resulting theory might be.  Perhaps more fundamentally, I do not know whether these measures satisfy analogs of the main properties we have used for the Haar measures on $\bf{U}(n)^r$, such as the measure concentration from Theorem~\ref{thm:Un-conc}. 

For a Cartesian product $\G_1 \times \G_2$ of two other groups, we would need a random way of choosing a pair of finite-dimensional representations of $\G_1$ and $\G_2$ which commute.  One possibility is to form Kronecker products of separate random representations of $\G_1$ and $\G_2$ (see~\cite[Sec. 7.3]{FolAHA}).  Studying these would bring us back to our questions about tensor products above.  In general, questions about whether such tensor-product representations can effectively `see' all representations of $\G_1\times \G_2$ have long-known connections to Connes' embedding problem and related phenomena.  Much more on these topics can be found in~\cite{PisOST,BroOza08,Oza13}, and the references cited there.




Finally, several works by Popesu have studied a notion of `entropy' for representations of free semigroups defined by a certain limit of associated log-determinants.  Popescu has used these to develop a theory of Toeplitz matrices and dilations for such semigroups: see~\cite{Popescu96,Popescu01} and the further references given there.  But the semigroups of operators in those works are highly non-invertible and very different from unitary representations of groups.  I have not found a link between Popescu's entropy in those papers and the quantities studied here.  

\bibliography{bibfile}{}
\bibliographystyle{abbrv}

\begin{small}
\noindent Mathematics Institute\\ Zeeman Building\\ University of Warwick\\ Coventry CV4 7AL\\ United Kingdom\\ \href{mailto:Tim.Austin@warwick.ac.uk}{\texttt{Tim.Austin@warwick.ac.uk}}
\end{small}

\end{document}